\documentclass{article}
\usepackage{graphicx} 
\usepackage{authblk}

\usepackage[a4paper,top=3cm,bottom=3cm,left=3cm,right=3cm,marginparwidth=1cm]{geometry}

\usepackage{upgreek}
\usepackage{amsmath,amssymb,amsthm,amsfonts}
\usepackage{mathtools}
\usepackage{dsfont}
\usepackage{graphicx}
\usepackage{stmaryrd}
\usepackage{subcaption}
\usepackage{float}
\usepackage{bbm}
\usepackage{bm}
\usepackage{enumitem}
\usepackage{pifont}
\usepackage{crossreftools}
\usepackage{comment}
\usepackage{centernot}
\usepackage{mathrsfs}
\usepackage{tikz}
\usepackage{cancel}
\usepackage[normalem]{ulem}

\usepackage[colorinlistoftodos]{todonotes}
\usepackage[allcolors=black]{hyperref}

\usepackage[backend=biber, style=alphabetic]{biblatex}
\newcommand{\PP}{{\mathbb{P}}}

\newcommand{\EE}{{\mathbb{E}}}

\newcommand{\LL}{{\mathcal{L}}}
\newcommand{\XX}{{\mathcal{X}}}
\newcommand{\eps}{{\epsilon}}
\newcommand{\HH}{{\mathbb{H}}}

\newcommand{\vol}{{\operatorname{vol}}}
\newcommand{\diam}{{\operatorname{diam}}}
\newcommand{\ZZ}{{\mathcal{Z}}}
\newcommand{\Var}{{\mathbb{V}\!\operatorname{ar}}}
\newcommand{\dist}{{\operatorname{dist}}}
\newcommand{\der}{{\operatorname{d}}}
\newcommand{\Int}{{\operatorname{Int}}}

\newcommand{\Area}{{\operatorname{Area}}}
\def\cZ{\mathcal{Z}}
\def\lamn{\Lambda_N}

\def\1{\mathbb{I}}
\def\cH{\mathcal{H}}
\def\w{\mathrm{w}}
\def\f{\mathrm{f}}
\def\Ne {\mathtt{NoEx}}
\def\cont{\mathtt{Con_2}}
\def\sB{\mathscr{B}}
\def\fB{\mathsf{B}}
\def\sfb{\mathsf{b}}
\def\sfa{\mathsf{a}}
\def\sfc{\mathsf{c}}
\def\sfd{\mathsf{d}}
\def\sfe{\mathsf{e}}
\def\fC{\mathsf{C}}
\def\fD{\mathsf{D}}
\def\tn{|\tau|}

\def\cC{\mathcal{C}}
\newcommand{\mcc}{\mathcal}
\newcommand{\mbb}{\mathbb}
\newcommand{\mrr}{\mathrm}
\newcommand{\mbf}{\mathbf}
\newcommand{\mss}{\mathscr}
\def\bpartial{\bar{\partial}}
\def\cA{\mathcal{A}}
\newcommand{\oloop}{\operatorname{oloop}}

\newtheorem{theorem}{Theorem}[section]
\newtheorem{corollary}[theorem]{Corollary}
\newtheorem{lemma}[theorem]{Lemma}

\newtheorem{claim}[theorem]         {Claim}
\newtheorem{proposition}[theorem]{Proposition}

\theoremstyle{definition}
\newtheorem{definition}[theorem]{Definition}
\theoremstyle{remark}
\newtheorem{remark}[theorem]{Remark}
\theoremstyle{remark}

\title{Scaling limits of the single-curve interface and outermost loops in the planar random field Ising model}
\author[1]{Fenglin Huang}
\author[2]{L\'eonie Papon}
\author[3]{Aoteng Xia}
\affil[1]{Peking University}
\affil[2]{TU Wien}
\affil[3]{University of Cologne}
\date{\today}

\begin{document}

\maketitle

\begin{abstract}
    We prove that the interface separating $+1$ and $-1$ spins in the near-critical planar random field Ising model (RFIM) with Dobrushin boundary conditions has a scaling limit, whose law is conformally covariant and almost surely absolutely continuous with respect to SLE$_3$. The limiting curve can be seen as a massive version of SLE$_3$ in the sense of Makarov and Smirnov, but in a random environment. We then show that the outermost spin loops of the  near-critical planar RFIM with $+1$ boundary conditions have subsequential limits and that any of these limits is almost surely singular with respect to CLE$_3$. This dichotomy between absolute continuity of the single interface and singularity of the outermost loops reflects the fact that a single interface does not explore enough of the magnetization field of the near-critical RFIM to detect the singularity of this field with respect to the critical Ising magnetization field, whereas the outermost spin loops do.
\end{abstract}

\section{Introduction}

\subsection{Main results}

Chordal SLE$_\kappa$ curves form a one-parameter family of planar curves, characterized by conformal invariance and a Markov property, that were introduced by Schramm \cite{Schramm_SLE}. They have been shown to arise as the scaling limits of interfaces in many planar statistical mechanics models at criticality when the boundary conditions are chosen appropriately \cite{percolation, LERW, HE, DGFF, cvgIsingSLE}.

However, many questions arise when looking at near-critical perturbations of these models, which are obtained by sending some of the parameters to their critical values at an appropriate rate. In the scaling limit, these perturbations typically introduce a finite correlation length and break conformal invariance: observables of the model that are conformally invariant in the critical regime are often only conformally covariant in the near-critical regime. In \cite{massiveSLE}, Makarov and Smirnov asked whether SLE-type curves could nevertheless describe the scaling limits of interfaces in the near-critical regime. The idea is that, on a lattice, the discrete interface of the near-critical model can be seen as a perturbation of the interface of the critical one. In the scaling limit, one can therefore expect the near-critical limiting interface to be in some sense a perturbation of an SLE$_{\kappa}$ curve. This near-critical limiting interface would depend on a mass parameter, arising from the near-critical perturbation, which led Makarov and Smirnov to call the (conjectural) limiting laws of near-critical interfaces massive SLEs \cite{massiveSLE}. In particular, these massive SLEs should still enjoy a Markovian property similar to that of SLE$_{\kappa}$ and the perturbation would only affect their transformation law under conformal maps: the mass parameter would have to be changed when mapping one domain to another via a conformal map and as a result, massive SLE should be conformally covariant in law, instead of conformally invariant as SLE.

Makarov and Smirnov originally asked about scaling limits of interfaces in near-critical models in the context of deterministic near-critical perturbations. However, one can ask the same question when the near-critical perturbation is random, that is, for statistical mechanics models in a random environment.
A prime example of such a model is the planar random field Ising model (RFIM), which is an Ising model in which each spin also interacts with a random magnetic field (see below for a more precise definition). This model was introduced by Imry and Ma, who asked whether adding a random magnetic field to the Ising model would destroy the phase transition \cite{PhysRevLett.35.1399}. They conjectured that this would be the case in dimension $2$ but not in dimensions $3$ and higher. In dimension $3$, this conjecture was proved in \cite{Imbrie85} (for $\beta = \infty$, or equivalently at zero temperature) and in \cite{Bricmont-3d} (for large $\beta$) using the renormalization group theoretic approach. More recently, a new and simpler proof (without renormalization group) was provided in \cite{DZ21} (with critical inputs from \cite{Chalker83, FFS84}), and building on this, this result was later extended to the entire low temperature regime in \cite{DLX22}. In dimension $2$, this conjecture was proved by \cite{Aizenmann-2d}. Recent progress in this direction mainly concerns quantitative bounds on the decay rate: sharper and sharper bounds were progressively obtained in a series of works \cite{Cha18, AP19, DX21, AHP20}, and exponential decay was ultimately proved in \cite{DX21, AHP20} (which rules out BKT transition).

Of particular relevance to us is the discovery of a near-critical regime for the planar RFIM and the construction of its partition function by Caravenna, Zygouras and Sun \cite{CSZ16}: they showed that when the underlying Ising model is at critical temperature, one can tune the strength of the noise defining the magnetic perturbation in such a way that the corresponding planar RFIM has a non-trivial scaling limit. Their result also hinted that the critical disorder strength of the random field Ising model is $\delta^{\frac{7}{8}}$ where $\delta$ is the meshsize of the lattice. Under such disorder strength, a continuum RFIM was constructed in \cite{continuum-2d-RFIM} via scaling limits of the magnetization field. Later, in \cite{DHX23}, a rigorous phase transition was established for the spontaneous magnetization around the disorder strength $\delta^{\frac{7}{8}}$, and exponential decay was shown in the strong disorder regime. Together, these results point to 
$\delta^{\frac{7}{8}}$ as the natural candidate for the critical disorder strength, identifying a near-critical window in which the interfaces of the planar RFIM are expected to exhibit nontrivial behavior. This opens up the possibility of investigating their scaling limit and connecting it to the framework originally proposed by Makarov and Smirnov, which is the main goal of this paper.
We note that a different type of random near-critical perturbation, namely a random \textit{temperature} perturbation of the critical planar Ising model, was recently investigated in \cite{mahfouf2025nearcriticalrandombond,averous2025nearcriticalrandombondfkpercolation}.

Our first main result is about the interface separating $+1$ and $-1$ spins in the planar near-critical RFIM with Dobrushin boundary conditions. When there is no external field, this interface is known to converge to an SLE$_3$ curve \cite{cvgIsingSLE}. Following Makarov and Smirnov \cite{massiveSLE}, in the presence of a random external field, one should expect the limiting interface to be related in some way to SLE$_3$. This is what the result below establishes. Below, we denote by $S_{\delta}(x)$ the square centered at $x$ with side-length $\delta$.

\begin{theorem}\label{theorem-single-interface}
    Let $D$ be a bounded, open, and simply connected domain with piecewise smooth boundary, and let $a,b \in \partial D$. Assume that $(D_{\delta};a_{\delta},b_{\delta})_{\delta}$ is a sequence of connected subsets of $\delta \mathbb{Z}^2$ with two marked boundary points $a_{\delta}, b_{\delta}$ that satisfies the assumptions of Section \ref{sec_assumptions_domain}. Let $\lambda: \bar D \to (0,\infty)$ be a bounded and continuous function. Let $H:=(H_x)_{x \in \mathbb{Z}^2}$ be i.i.d. Gaussian random variables with mean $0$ and variance $1$. Denote the law of $H$ by $\PP$. Define $h:=(h_x)_{x \in D_{\delta}}$ via $h_{x}=H_{x/\delta}$. Let $\nu_{\delta}^{\pm,\lambda, h}$ be the law of the interface separating $+1$ and $-1$ spins in the random-field Ising model on $D_{\delta}$ with external field $\delta^{\frac{7}{8}}\lambda(x)h_{x}$ and Dobrushin boundary conditions. Then, as $\delta \to 0$,
    \begin{equation*}
        (\delta^{-1}\sum_{x \in D_{\delta}}h_{x}\mathbb{I}_{S_{\delta}(x)},\nu_{\delta}^{\pm,\lambda, h}) \to (W, \nu_{W,\lambda}^{(D,a,b)})
    \end{equation*}
    weakly, where $W$ is a two-dimensional white noise and $\nu_{W,\lambda}^{(D,a,b)}$ is a probability measure on curves in $D$ that join $a$ to $b$. Moreover, $\nu_{W,\lambda}^{(D,a,b)}$ is conformally covariant in law and $\PP$-almost surely, $\nu_{W,\lambda}^{(D,a,b)}$ is absolutely continuous with respect to the law $\nu_{\operatorname{SLE}_3}^{(D,a,b)}$ of SLE$_{3}$ in $D$ from $a$ to $b$.
\end{theorem}

The topology in which convergence of $(\nu_{\delta}^{\pm,\lambda,h})_{\delta}$ is obtained is described in Section \ref{sec_topology_curves}, while the topology of convergence for $(\delta^{-1}\sum_{x \in D_{\delta}}h_{x}\mathbb{I}_{S_{\delta}(x)})$ is recalled in Section \ref{sub:proof-characterization-interface}. Besides, absolute continuity of $\nu_{W,\lambda}^{(D,a,b)}$ with respect to $\nu_{\operatorname{SLE}_3}^{(D,a,b)}$ is established by proving an explicit expression for the Radon-Nikodym derivative of $\nu_{W,\lambda}^{(D,a,b)}$ with respect to $\nu_{\operatorname{SLE}_3}^{(D,a,b)}$, stated in Section \ref{sub:proof-characterization-interface}. The conformal covariance of $\nu_{W,\lambda}^{(D,a,b)}$ follows from the conformal covariance of this Radon-Nikodym derivative, as explained in Section \ref{sec-conformal-covariance}. As $\nu_{W,\lambda}^{(D,a,b)}$ transforms covariantly under conformal maps and is absolutely continuous with respect to $\nu_{\operatorname{SLE}_3}^{(D,a,b)}$, the limiting interface can be considered as a massive version of SLE$_3$ in the sense of Makarov and Smirnov.

In Theorem \ref{thm: sle continuity critical main} below, the absolute continuity of $\nu_{W,\lambda}^{(D,a,b)}$ with respect to $\nu_{\operatorname{SLE}_{3}}^{(D,a,b)}$ established in Theorem \ref{theorem-single-interface} is complemented at the discrete level by large-deviation-type bounds on the discrete Radon-Nikodym derivative of $\nu_{\delta}^{\pm,\lambda, h}$ with respect to the law $\nu_{\delta}^{\pm}$ of the interface of the critical Ising model in $D_{\delta}$ with Dobrushin boundary conditions, when $\lambda \equiv c$, $c \in \mathbb{R}$. These bounds show in particular that when $\lambda$ is constant, $\nu_{W,\lambda}^{(D,a,b)}$ and $\nu_{\operatorname{SLE}_3}^{(D,a,b)}$ are in fact mutually absolutely continuous, see Remark \ref{remark-mutual-abs-continuity} for further details.

In the critical planar Ising model, it is also natural to look at another set of interfaces, namely the collection of nested spin loops separating $+1$ and $-1$ spins that appear when imposing $+1$ boundary conditions to the model. This collection of spin loops has been shown to converge to nested CLE$_3$, while the outermost loops in this collection have been proved to converge to CLE$_3$ \cite{cvg-nested-loops-Ising}. CLE$_3$ (conformal loop ensemble with parameter $3$) is a loop version of SLE$_3$ that was originally constructed in \cite{CLE_She} and then given an alternative construction and a characterization via conformal invariance and a domain Markov property in \cite{CLE_Markovian}. Throughout this paper, following \cite{cvg-nested-loops-Ising}, CLE$_3$ denotes the non-nested conformal loop ensemble consisting only of the outermost loops, whereas nested CLE$_3$ denotes the full loop ensemble including all nesting levels. Just as in the case of the spin-spin interface of Theorem \ref{theorem-single-interface}, one may wonder if the outermost spin loops of the near-critical planar RFIM converge to some ``near--critical" version of CLE$_3$ and if so, what can be said about this scaling limit. The next theorem provides a partial answer to this question.

\begin{theorem}\label{theorem-singularity-outermost-loops}
    Let $\Lambda:= [-1,1]^2$ and for $\delta > 0$, set $\Lambda_{\delta} := \Lambda \cap \delta \mathbb{Z}^2$. Let $H:=(H_x)_{x \in \mathbb{Z}^2}$ be i.i.d Gaussian random variables with mean $0$ and variance $1$. Denote the law of $H$ by $\PP$. Define $h:=(h_x)_{x \in \Lambda_{\delta}}$ via $h_{x}=H_{x/\delta}$. Let $\nu_{\delta}^{\operatorname{oloop},h}$ be the law of the collection of outermost loops separating $+1$ and $-1$ spins in the random-field Ising model on $\Lambda_{\delta}$ with external field $\delta^{\frac{7}{8}}h_{x}$ and $+1$ boundary conditions. Then, the joint law of $(\delta^{-1}\sum_{x \in \Lambda_{\delta}}h_{x}\mathbb{I}_{S_{\delta}(x)},\nu_{\delta}^{\operatorname{oloop},h})_{\delta}$ is tight as $\delta \to 0$ and any subsequential limit of $(\nu_{\delta}^{\operatorname{oloop},h})_{\delta}$ is $\PP$--almost surely singular with respect to the law of CLE$_3$ in $\Lambda$.
\end{theorem}

The topology in which tightness of $(\nu_{\delta}^{\operatorname{oloop},h})_{\delta}$ is obtained is described in Section \ref{sec-topology-loops}. Besides, note that Theorem \ref{theorem-singularity-outermost-loops} is stated for a specific choice of domain, namely the unit square $\Lambda$. Our proof could in fact be generalized to bounded, open, and simply connected domains with a piecewise smooth boundary. However, proving Theorem \ref{theorem-singularity-outermost-loops} in this greater generality would require more technical work without providing further insights into the mechanism at play behind the singularity of any subsequential limit of $(\nu_{\delta}^{\operatorname{oloop},h})_{\delta}$ with respect to CLE$_3$. This is why we chose to restrict ourselves to square domains in the statement of Theorem \ref{theorem-singularity-outermost-loops}.

We also note that in Section \ref{sec-singularity-nested-loops}, a result analogous to that of Theorem \ref{theorem-singularity-outermost-loops} is proved when one considers the whole collection of nested spin loops: their laws are tight and any subsequential limit is $\PP$-almost surely singular with respect to the law of nested CLE$_3$ in $D$, where in this case $D \subset \mathbb{C}$ can be any bounded, open and simply connected domain with smooth boundary. We refer the reader to Proposition \ref{prop-singularity-nested-loops} for a precise statement.

We remark that Theorem \ref{theorem-singularity-outermost-loops} contrasts with Theorem \ref{theorem-single-interface}: the law of the limit of the spin-spin interface of the near-critical RFIM is almost surely absolutely continuous with respect to SLE$_3$ whereas the law of any subsequential limit of the outermost spin loops of the near-critical RFIM is almost surely singular with respect to CLE$_3$. This dichotomy can be understood, at least at a heuristic level, as follows. On a lattice of meshsize $\delta$, both in the near-critical RFIM and in the critical Ising model, the interfaces -- whether single interface or collection of loops --  are measurable with respect to the magnetization field $\mathcal{M}_{\delta}$ defined by, for $f$ a smooth function with compact support, $\mathcal{M}_{\delta}(f) := \delta^{\frac{15}{8}}\sum_{x}f_x \sigma_x$. The limit of $(\mathcal{M}_{\delta})_{\delta}$ as $\delta \to 0$ can be constructed, both for the critical Ising model \cite{GarbanMagnetization} and for the near-critical RFIM \cite{continuum-2d-RFIM}. However, the law of the limit of $(\mathcal{M}_{\delta})_{\delta}$ in the near-critical RFIM is almost surely singular with respect to the limiting law of $(\mathcal{M}_{\delta})_{\delta}$ in the critical Ising model \cite{continuum-2d-RFIM}. Therefore, when looking at the interfaces, it is not clear whether one should expect them to be singular or not with respect to the critical ones. Our results roughly say the following: the answer depends on how much the magnetization field the interface can see. In the case of a single interface, the curve discovers too little of the magnetization field to witness singularity in the scaling limit, while the outermost spin loops explore already enough of the magnetization field to detect the singularity in the scaling limit.

We note that a similar dichotomy has been conjectured by Garban and Kupiainen in the case of the near-critical Ising model (deterministically) perturbed in the temperature direction: the scaling limit of a single interface should be absolutely continuous with respect to SLE$_3$ while the scaling limit of the spin loops should be singular with respect to nested CLE$_3$ \cite{garban2025energyfieldcriticalising}.

As a final comment, let us mention that the results of this paper do not cover the case when the random external field of the RFIM is rescaled differently than $\delta^{\frac{7}{8}}$. In the case when the random field is rescaled by $\delta^{\frac{7}{8}}g(\delta)$ with $g(\delta) \to 0$ as $\delta \to 0$, it is quite straightforward to see that the RFIM interfaces have the same limits as the interfaces of the critical Ising model. When the random field is rescaled by $\delta^{\frac{7}{8}}f(\delta)$ with $f(\delta) \to \infty$ as $\delta \to 0$, what the scaling limits of the interfaces should be is less clear. It is of course natural to expect that the resulting spin-loop ensembles remain singular with respect to the critical Ising scaling limits, since the influence of the random field becomes stronger rather than weaker in this regime. There is nevertheless one specific setting in which one may be able to make a precise conjecture about the scaling limit of the interfaces in this regime, provided one looks at the \textit{annealed} RFIM. Indeed, in the annealed RFIM, each spin locally wants to align with a random variable that has probability $1/2$ to be positive and probability $1/2$ to be negative. This local behavior dominates over the nearest-neighbor interaction if the noise variables are not properly rescaled. Now, if the RFIM is defined on the triangular lattice, since the noise variables are independent, this implies that the annealed law of the model looks like critical percolation on the triangular lattice, whose interfaces are known to converge to SLE$_{6}$ and CLE$_{6}$ \cite{percolation, Camia-percolation}. Therefore, when the scaling of the noise is of order $\delta^{\frac{7}{8}} f(\delta)$ with $f(\delta) \to \infty$, one may expect the interfaces of the annealed RFIM on the triangular lattice to also converge to SLE$_{6}$ and CLE$_{6}$.

\subsection{Comparison with other near-critical perturbations of the critical planar Ising model}

There are several natural ways to perturb the critical planar Ising model. In this paper, we focus on the random-field perturbation, but it is useful to compare it with perturbations in other natural directions, in particular, the magnetic field direction and the temperature direction. In the magnetic-field direction, the most relevant comparison is the deterministic near-critical perturbation. In the temperature direction, both deterministic and random near-critical perturbations have recently been studied.

Let us first discuss the deterministic magnetic field perturbation. In this case, the near-critical regime corresponds to rescaling the external field by $\delta^{15/8}$. Under Dobrushin boundary conditions, it was shown in \cite{papon2024interfacescalinglimitcritical} that the spin interface converges in law, as $\delta\to0$, to a measure $\nu_\lambda^{(D,a,b)}$ which is absolutely continuous with respect to $\nu_{\operatorname{SLE}_3}^{(D,a,b)}$ and conformally covariant. Moreover, the deterministic magnetic field perturbation remains soft also at the level of spin loops. Indeed, when the model has $+1$ boundary conditions, the Radon-Nikodym derivative of the law of the spin-loop ensemble under the perturbed Ising measure with respect to its critical counterpart is explicit. Using the same argument as in the proof of \cite[Proposition~3.1]{papon2024interfacescalinglimitcritical}, one obtains uniform $L^p$ bounds, for every $p\ge1$, on this Radon-Nikodym derivative. Consequently, any subsequential scaling limit of the spin-loop ensemble in the deterministic magnetic field perturbation is absolutely continuous with respect to $\operatorname{CLE}_3$.

The random-field perturbation has a different nature. It still couples to the
spin field, but the perturbing object is now a random field rather than a
deterministic function. This changes the near-critical scale from
$\delta^{15/8}$ to $\delta^{7/8}$ and, more importantly, changes the variance
structure of the perturbation. Theorem~\ref{theorem-singularity-outermost-loops} shows that the scaling limit
of the outermost spin loops is already singular with respect to
$\operatorname{CLE}_3$. In particular, the random--field perturbation leaves a
detectable signature in a much coarser observable than the full loop ensemble or
the full crossing configuration. This should be contrasted with the
deterministic magnetic field perturbation, for which even the full spin-loop
ensemble remains absolutely continuous with respect to its critical counterpart.

Let us finally discuss perturbations in the temperature direction. For the deterministic temperature perturbation, the near-critical scale is $\beta-\beta_c\asymp\delta$. Garban and Kupiainen proved that any subsequential scaling limit of the two-dimensional Ising model perturbed in the $\beta$-direction in a near-critical way is singular with respect to the critical scaling limit \cite{garban2025energyfieldcriticalising}. Their singularity is detected in the quad-crossing topology by observing crossing events at many mesoscopic scales simultaneously. The mechanism is different from the one in the random--field case: the temperature perturbation couples to the energy field, and the singularity is tied to the additional logarithmic fluctuations of the energy field. While the effect of the perturbation at a single scale is small, these contributions accumulate coherently across many boxes and many scales, eventually making the perturbed model distinguishable from the critical one with probability tending to one. Random near-critical temperature perturbations have also been studied recently by Mahfouf and Av\'erous \cite{averous2025nearcriticalrandombondfkpercolation, mahfouf2025nearcriticalrandombond}. For the corresponding random-bond models, they establish uniform crossing estimates, which in particular yield tightness of the associated interfaces.

In summary, the deterministic magnetic-field perturbation is the softest among the perturbative regimes discussed above: even the spin-loop ensemble remains absolutely continuous with respect to its critical counterpart. The deterministic temperature perturbation is singular at the level of the full crossing or interface configuration, through a mechanism related to the logarithmic fluctuations of the energy field. For the random-field perturbation considered in this paper, we prove a stronger form of singularity, namely that the singularity is already visible from the outermost spin loops. This phenomenon originates from the variance accumulation created by the random external field and, while it is also conjectured to hold for the temperature perturbation in \cite{garban2025energyfieldcriticalising}, it remains open in that setting.

\subsection{Outline of the proof of Theorem \ref{theorem-single-interface} and Theorem \ref{theorem-singularity-outermost-loops}}

The starting point of the proof of Theorem \ref{theorem-single-interface} is to observe that, as in the deterministic setting \cite{papon2024interfacescalinglimitcritical}, the Radon-Nikodym derivative $F_{\delta}^{\lambda,h}$ of $\nu_{\delta}^{\pm,\lambda,h}$ with respect to the law $\nu_{\delta}^{\pm}$ of the interface of the critical Ising model has an explicit expression. Indeed, it is easy to see that in fact, up to a normalization constant, $F_{\delta}^{\lambda,h}$ is equal to the conditional expectation of $\exp(\delta^{\frac{7}{8}}\sum_{x}\lambda(x)h_x\sigma_x)$ given the interface under $\mu_{\delta}^{\pm}$. Here, $\mu_{\delta}^{\pm}$ is the critical Ising measure in $D_{\delta}$ with Dobrushin boundary conditions. Given this observation, the proof of Theorem \ref{theorem-single-interface} consists of two steps,  which both exploit the explicit expression of $F_{\delta}^{\lambda,h}$.

The first step is to prove tightness of $(\nu_{\delta}^{\pm,\lambda,h})_{\delta}$, in a suitable topology (see Section \ref{sec_topology_curves}). Since $(\nu_{\delta}^{\pm})_{\delta}$ converges in law to $\nu_{\operatorname{SLE}_3}^{(D,a,b)}$ \cite{cvgIsingSLE}, to establish tightness of $(\nu_{\delta}^{\pm,\lambda,h})_{\delta}$, it suffices to show that $\EE \otimes \mu_{\delta}^{\pm}[F_{\delta}^{\lambda,h}(\gamma_{\delta})^{1+\alpha}]$ is uniformly bounded in $\delta$, for some $\alpha > 0$. This is achieved by using a polynomial chaos expansion of $F_{\delta}^{\lambda,h}$ that is similar in spirit to the one considered in \cite{CSZ16} (for further application, see also \cite{continuum-2d-RFIM,DHX23,HHX25}). However, in our setting, several issues arise that were not present in \cite{CSZ16}. First, the underlying domains are random and a part of their boundaries is an SLE$_{3}$-type curve, which is typically a fractal curve. Second, we will actually have to look at the product of two polynomial chaos expansions. Finally, $F_{\delta}^{\lambda, h}$ contains a term that corresponds to the magnetization of the interface and that cannot be handled using chaos expansion. To show that this term does not blow up in the scaling limit, we will need to control the moments of a quantity that can be seen as the ``discrete" Minkowski content of the interface of the critical Ising model. This requires proving good upper bounds on the probability that the interface of the critical Ising model gets close to $k$ points, for $k \in \mathbb{N}$.

The second step of the proof of Theorem \ref{theorem-single-interface} is to prove a characterization result that uniquely identifies the limit of any converging subsequence of $(\nu_{\delta}^{\pm,\lambda,h})_{\delta}$. This characterization is established by identifying the limit of $F_{\delta}^{\lambda,h}$ as $\delta \to 0$. To do so, as in the first step of the proof, we rely on a polynomial chaos expansion of $F_{\delta}^{\lambda,h}$ and adapt the Lindeberg principle of \cite{CSZ16} to our setting. Again, the main difficulty that is encountered in the implementation of this strategy arises from the randomness and the roughness of the boundaries of the underlying domain. Moreover, to identify the limit of $F_{\delta}^{\lambda,h}$, we will also have to show that the term corresponding to the magnetization of the interface not only stays bounded as $\delta \to 0$ but actually vanishes.

Let us now discuss the proof of Theorem \ref{theorem-singularity-outermost-loops}. We first prove tightness of the law of $(\nu_{\delta}^{\operatorname{oloop},h})_{\delta}$ by using an argument similar to that used to prove tightness of the magnetization field of the near-critical RFIM in \cite{continuum-2d-RFIM}. Then, to detect singularity, we construct the following observable. We cover our domain with a collection of small boxes, inside each of which we draw an annulus. In this way, we obtain a finite collection $(A_i)_i$ of disjoint annuli. To each $A_i$, we associate a random variable $X_i$ which equals to $1$ if there is an outermost loop in $A_i$ surrounding the inner boundary of $A_{i}$, and $0$ otherwise. We then show that the total variation distance between the law of $X=(X_i)_i$ under $\nu_{\mathrm{CLE}_{3}}^{D}$ and the law of $X=(X_i)_i$ under $\nu^{W}$ can be made arbitrarily close to $1$ by choosing the size of the annuli sufficiently small, where $\nu^{W}$ is any subsequential limit of $(\nu_{\delta}^{\operatorname{oloop},h})_{\delta}$. The fact that the annuli must be chosen to be small highlights that the singularity established in Theorem \ref{theorem-singularity-outermost-loops} arises from the contribution of the small loops.

To lower bound the total variation distance in the way claimed above, we first show that we can actually prove this lower bound under $\nu_{\delta}^{+}$ and $\nu_{\delta}^{\operatorname{oloop},h}$ and then take the limit as $\delta \to 0$. This second step relies on a lemma showing ``continuity of the crossing events", which is a classical argument in percolation \cite{planar-percolation,garban2025energyfieldcriticalising}. For the first step, i.e., to detect singularity at the discrete level, the idea is the following. We first observe that under the critical Ising measure $\nu_{\delta}^{+}$, there will typically be more than $n^{\frac{7}{8}+\iota}$ annuli such that $X_i =1$, where $n$ is the total number of annuli and $\iota>0$ is independent of $\delta$. By turning on the random external field in one box after the other, we show that for $\tau \in \{0,1\}^n$, the random external field inside each box changes the probability that $X=\tau$ a little bit, if $\tau_i = 1$. Moreover, thanks to the independent variance structure of the random external field, the fluctuations coming from boxes with $\tau_i=0$, although difficult to control individually, do not systematically cancel the cumulative effect created by the large number of boxes with $\tau_i=1$. Since we typically have at least $n^{\frac{7}{8}+\iota}$ annuli with $\tau_i = 1$ under $\nu_{\delta}^{+}$, summing up the small changes produced by the external field inside each box will result in a significant change in the probability that $X=\tau$ under $\nu_{\delta}^{\operatorname{oloop},h}$. This in turn will be enough to establish a lower bound on the total variation distance between the law of $X$ under $\nu_{\delta}^{+}$ and the law of $X$ under $\nu_{\delta}^{\operatorname{oloop},h}$, which is uniform in $\delta$ (for $\delta$ small enough) and converges to $1$ as the size of the annuli goes to $0$ on a macroscopic scale.

The main difficulty in aggregating the small changes is that the outermost loops are not independent. Section \ref{sec: proof of DHX rare event} is devoted to overcoming this issue. We prove that, conditioning on the values $\{X_j\}_{j<i}$, we can get a similar bound on the $k$-point function as the bound without conditioning. This is achieved by establishing an exploration procedure to get rid of the  event $\{X_j=\tau_j,~\forall j<i\}$, which is neither increasing nor decreasing.

\paragraph{Organization of the paper.} This work is split into two parts for clarity of exposition. In the first part, we give the proof of Theorem \ref{theorem-single-interface} and Theorem \ref{theorem-singularity-outermost-loops}, as well as the proof of the singularity of the nested collection of spin loops in the scaling limit. While the proof of Theorem \ref{theorem-single-interface} is self-contained, the proof of Theorem \ref{theorem-singularity-outermost-loops} relies on results derived entirely in the discrete setting. Since the ideas and techniques involved in proving these results are somewhat of a different nature, the proofs are delayed to the second part of the paper. The second part of the paper also contains the large-deviation-type bound on the discrete Radon-Nikodym derivative of the RFIM interface with respect to the critical interface mentioned after Theorem \ref{theorem-single-interface}. A reader interested only in the scaling limit of the near-critical RFIM interfaces may read only the first part, and a reader more inclined to learn about the interfaces of the near-critical RFIM at the discrete level may find the second part more relevant.

\paragraph{Acknowledgments.} We warmly thank Barbara Dembin, Piet Lammers, and Franco Severo, the organizers of the ``Percolation today" seminar through which we met and started discussing ideas related to this paper. We thank Xin Sun for helpful discussions, and in particular, for the discussion on the outermost loops singularity problem. Fenglin Huang is partially supported by the New Cornerstone Science Foundation through the New Cornerstone
Investigator Program.

\part{Convergence of the laws of the near-critical RFIM interfaces}

\section{Background}

\subsection{Assumptions on the domain and its discrete approximations} \label{sec_assumptions_domain}

We consider a bounded, open, and simply connected subset $D$ of the complex plane $\mathbb{C}$ such that $0 \in D$. We fix two marked boundary points $a, b \in \partial D$ and we assume that both $a$ and $b$ are degenerate prime ends of $D$, see \cite[Section~2.5]{Pommeranke} for a definition. These assumptions on $D$ and the boundary points $a$ and $b$ will in particular allow us to use \cite[Theorem~4.2]{Karrila}. 

Regarding the assumptions on the boundary $\partial D$ of $D$, Theorem \ref{theorem-single-interface}, Theorem \ref{theorem-singularity-outermost-loops}, and Proposition \ref{prop-singularity-nested-loops} will be shown assuming that $\partial D$ is piecewise smooth. This assumption could possibly be weakened to an upper bound on the Minkowski dimension of $\partial D$, but to keep the proof as simple as possible, we will work under this stronger smoothness assumption.

We assume that $(D_{\delta})_{\delta}$ is a sequence of graphs approximating $D$ in the sense that we will now explain. For each $\delta > 0$, $D_{\delta}$ is a finite connected subgraph of the square lattice $\delta \mathbb{Z}^2$, so that every edge of $D_{\delta}$ has length $\delta$. We denote by $V(D_{\delta})$ the set of vertices of $D_{\delta}$ and define the boundary $\partial D_{\delta}$ of $D_{\delta}$ as
\begin{equation*}
    \partial D_{\delta} := \{ w \in \delta \mathbb{Z}^2 \setminus V(D_{\delta}): \text{there exists $v \in V(D_{\delta})$ such that $v \sim w$}\}
\end{equation*}
where $v \sim w$ means that there is an edge of $\delta \mathbb{Z}^2$ connecting $v$ and $w$. With this definition of $\partial D_{\delta}$, it is legitimate to set $\Int(D_{\delta}):=V(D_{\delta})$. We also let $E(D_{\delta})$ denote the set of edges of $\delta\mathbb{Z}^2$ with at least one endpoint in $\Int(D_{\delta})$.

We associate to each $D_{\delta}$ an open and simply connected polygonal domain
$\hat D_{\delta} \subset \mathbb{C}$ by taking the union of all open squares with side length $\delta$ centered at vertices in $V(D_{\delta})$. We assume that for any $\delta > 0$, $0$ belongs to $\hat D_{\delta}$ and that there exists $\bar{R} > 0$ such that for any $\delta > 0$, $\hat D_{\delta} \subset B(0,\bar{R})$ and $D \subset B(0,\bar{R})$. These assumptions are necessary to apply \cite[Theorem~4.2]{Karrila}. The marked boundary points $a$ and $b$ of $\partial D$ are then approximated by sequences $(a_{\delta})_{\delta}$ and $(b_{\delta})_{\delta}$ where, for each $\delta >0$, $a_{\delta}$ and $b_{\delta}$ belong to $\partial \hat D_{\delta}$. 

The sequence $(\hat D_{\delta}; a_{\delta}, b_{\delta})_{\delta}$ is assumed to converge to $(D; a, b)$ in the Carath\'eodory topology. That is,
\begin{itemize}
    \item each inner point $z \in D$ belongs to $\hat D_{\delta}$ for $\delta$ small enough;
    \item for every boundary point $\zeta \in \partial D$, there exists a sequence $(\upzeta_{\delta})_{\delta}$ such that $\upzeta_{\delta} \to \upzeta$ as $\delta \to 0$, where, for each $\delta >0$, $\upzeta_{\delta} \in \partial \hat D_{\delta}$.
\end{itemize}
This can be rephrased in terms of conformal maps. Let $\psi: D \to \mathbb{D}$ be a conformal map such that $\psi(a)=1$ and $\psi(0)=0$. Similarly, for each $\delta >0$, let $\psi_{\delta}: \hat D_{\delta} \to \mathbb{D}$ be a conformal map such that $\psi_{\delta}(a_{\delta})=1$ and $\psi_{\delta}(0)=0$. Then, by \cite[Theorem~1.8]{Pommeranke}, the Carath\'eodory convergence of $(\hat D_{\delta}; a_{\delta}, b_{\delta})_{\delta}$ to $(D; a, b)$ is equivalent to
\begin{align*}
    & \psi_{\delta} \to \psi \quad \text{uniformly on compact subsets of $D$ and}\\
    & \psi_{\delta}^{-1} \to \psi^{-1} \quad \text{uniformly on compact subsets of $\mathbb{D}$}.
\end{align*}
To avoid the existence of ``disappearing fjords", we assume that there exists $\tilde \delta > 0$ such that for any $0 < \delta < \tilde \delta$, $D \subset \cap_{\delta < \tilde \delta} \hat D_{\delta}$.

Furthermore, we suppose that $a_{\delta}$, respectively $b_{\delta}$, is a close approximation of $a$, respectively $b$, as defined by Karrila in \cite[Section~4.3]{Karrila}. Let us recall what this means. To lighten the notations, we identify the prime ends $a_{\delta}$ and $a$ with their corresponding radial limit points. For $r>0$, let $S_r$ be the arc of $\partial B(a,r) \cap D$ disconnecting in $D$ the prime end $a$ from $0$ and that is closest to $a$. In other words, $S_r$ is the last arc from the (possibly countable) collection $\partial B(a,r) \cap D$ of arcs that a path running from $0$ to $a$ inside $D$ must cross. Such an arc exists by \cite[Lemma~A.1]{Karrila} and approximation by radial limits. $a_{\delta}$ is then said to be a close approximation of $a$ if
\begin{itemize}
    \item $a_{\delta} \to a$ as $\delta \to 0$; 
    \item for each $r$ small enough and for all sufficiently (depending on $r$) small $\delta$, the boundary point $a_{\delta}$ of $\hat D_{\delta}$ is connected to the midpoint of $S_r$ inside $\hat D_{\delta} \cap B(a,r)$. 
\end{itemize}
For the proof of Theorem \ref{theorem-single-interface}, we also assume that $(\partial D_{\delta})_{\delta}$ approximates $\partial D$ in the following quantitative way. For $l \in \mathbb{N}$, let us set $N_{l}(\partial D_{\delta}):= \{x \in D_{\delta}: e^l\delta \leq \dist(x,\partial D_{\delta}) \leq e^{l+1}\delta\}$. Then, we require that there exists $C>0$ such that for any $l \in \mathbb{N}$ and any $ \delta >0$,
\begin{equation} \label{condition-approximation-boundary}
    \vert N_{l}(\partial D_{\delta}) \vert \leq C e^{l+1}\delta^{-1}.
\end{equation}
This will prevent $\partial D_{\delta}$ from having a too big Minkowski content, which could be problematic to control the moments of the Radon-Nikodym derivative in the proof of Theorem \ref{theorem-single-interface}.

In what follows, we will need to split the boundary $\partial D_{\delta}$ of $D_{\delta}$ into two sets of vertices, depending on their position with respect to $a_{\delta}$ and $b_{\delta}$. To this end, we denote by $\partial D_{\delta}^{-}$, respectively $\partial D_{\delta}^{+}$, the vertices of $\partial D_{\delta}$ that are traversed when following the vertices of $\partial D_{\delta}$ from $a_{\delta}$ to $b_{\delta}$ counterclockwise, respectively clockwise.

\subsection{The two-dimensional Ising model}

For $\delta >0$, let $G$ be a finite subgraph of $\delta \mathbb{Z}^2$. The Ising model on $G$ is a probability measure on spin configurations $\sigma: V(G) \cup \partial G \to \{-1,1\}$, where as above $V(G)$ denotes the set of interior vertices of $G$ and $\partial G$ is defined as in Section \ref{sec_assumptions_domain}. In this setting, since the boundary $\partial G$ of $G$ is non-empty, boundary conditions can be imposed to specify the measure. Three different choices of boundary conditions will be of interest here. The first one is that of Dobrushin boundary conditions: the boundary of $G$ is split into two arcs, denoted $\partial G^{+}$ and $\partial G^{-}$ and the spins of the vertices of $\partial G^{-}$ are all $-1$ while the spins of the vertices of $\partial G^{+}$ are all $+1$. The probability of an Ising spin configuration $\sigma$ in $G$ with Dobrushin boundary conditions is given by
\begin{equation} \label{def_Ising_pm}
    \mu_{G,\beta}^{\pm}(\sigma)=\frac{1}{\mathcal{Z}_{G,\beta}^{\pm}}\exp\bigg(\beta \sum_{x \sim y} \sigma_x\sigma_y\bigg)\mathbb{I}_{\sigma_{\vert \partial G^{+}}=+1}\mathbb{I}_{\sigma_{\vert \partial G^{-}}=-1}
\end{equation}
where $\beta > 0$ is a parameter called the inverse temperature and where $\mathcal{Z}_{G,\beta}^{\pm}$ is the partition function. Above, the sum is over nearest neighbor vertices $x \sim y$ of $\delta \mathbb{Z}^2$ such that at least one of them belongs to $V(G)$.

The other set of boundary conditions that we will consider is that of $+1$ boundary conditions and $-1$ boundary conditions. The probability $\mu_{G, \beta}^{+}(\sigma)$, respectively $\mu_{G, \beta}^{-}(\sigma)$, of an Ising spin configuration $\sigma$  in $G$ with $+1$, respectively $-1$, boundary conditions is given by replacing $\mathbb{I}_{\sigma_{\vert \partial G^{+}}=+1}\mathbb{I}_{\sigma_{\vert \partial G^{-}}=-1}$ by $\mathbb{I}_{\sigma_{\vert \partial G}=+1}$, respectively $\mathbb{I}_{\sigma_{\vert \partial G}=-1}$ in \eqref{def_Ising_pm}. One can also define the Ising model with free boundary conditions: in this case, the spins of the vertices on $\partial G$ are not prescribed.

The parameter $\beta$ plays an important role in the macroscopic behavior of the Ising model on $G$. Indeed, there exists a critical value $0 < \beta_{c} < \infty$ at which the model undergoes a phase transition. On the square lattice, it has been shown that $\beta_{c} = \frac{1}{2}\log(\sqrt{2}+1)$ \cite{book-Ising}. Moreover, at $\beta = \beta_{c}$, as $\delta \to 0$, the Ising model on the square lattice exhibits conformal invariance properties, and this has been the topic of extensive research in the past twenty years, see \cite{Chelkak-survey, DC-survey} for a survey. 

Here, we will always take $\beta = \beta_{c}$ and for $G$ a finite subgraph of $\delta \mathbb{Z}^2$, we set $\mu_{G}^{\xi}:=\mu_{G, \beta_c}^{\xi}$, $\xi \in \{-,\pm,+\}$, and call the corresponding Ising model the critical Ising model. In the specific case when $G=D_{\delta}$ with $D_{\delta}$ as in Section \ref{sec_assumptions_domain}, we set $\mu_{\delta}^{\xi}:=\mu_{D_{\delta}}^{\xi}$, for $\xi \in \{-,\pm,+\}$. When the Ising model is defined on $\Lambda_{N}:= [-N,N]^2 \cap \mathbb{Z}^2$, the box of side-length $2N$ and meshsize $\delta = 1$, we set $\mu_{\Lambda_{N},0}^{\xi}:=  \mu_{\Lambda_{N}}^{\xi}$, for $\xi \in \{-,\pm,+\}$ (here, the $0$ refers to the absence of magnetic field). Expectations with respect to $\mu_{\Lambda_{N},0}^{\xi}$ will often be denoted by $\langle \cdot \rangle_{\Lambda_{N},0}^{\xi}$.

\subsubsection{Spin correlations at \texorpdfstring{$\beta = \beta_c$}{}} \label{section_spin_correlations}

A crucial input for the construction of the near-critical random field Ising model in two dimensions is the existence of a scaling limit for the $k$-point spin correlations of the critical planar Ising model, when the boundary conditions are either $+1$, $-1$ or Dobrushin. This result is established in \cite{spinCorrelations} and as it will play an important role in the proof of Theorem \ref{theorem-single-interface}, let us recall it. 

\begin{theorem}[\cite{spinCorrelations}] \label{theorem_spin_correlations}
Let $D \subset \mathbb{C}$ be a bounded and simply connected domain and $a, b \in \partial D$ be two prime ends of $\partial D$. Let $(D_{\delta};a_{\delta}, b_{\delta})_{\delta}$ be discrete approximations of $(D;a,b)$ satisfying the assumptions of Section \ref{sec_assumptions_domain}. Then, for any $\eta >0$ and any $k \geq 1$, we have that
\begin{align*}
    &\delta^{-\frac{k}{8}} \mu_{\delta}^{\pm}[\sigma_{x_1}\dots \sigma_{x_{k}}] = C_{\sigma}^k f_{D}^{(\pm, k)}(x_1,\dots ,x_k) + o(1) \\
    &\delta^{-\frac{k}{8}} \mu_{\delta}^{+}[\sigma_{x_1}\dots \sigma_{x_{k}}] = C_{\sigma}^k f_{D}^{(+,k)}(x_1,\dots ,x_k) + o(1) \\
    &\delta^{-\frac{k}{8}} \mu_{\delta}^{-}[\sigma_{x_1}\dots \sigma_{x_{k}}] = C_{\sigma}^k f_{D}^{(-,k)}(x_1,\dots ,x_k) + o(1)
\end{align*}
as $\delta \to 0$, uniformly over all $x_1, \dots, x_k \in D$ at distance at least $\eta$ from $\partial D$ and from each other. Here, $C_{\sigma}>0$ is a constant that corresponds to the fact that we are working on the square lattice. The functions $f_{D}^{(\pm,k)}$, $f_{D}^{(+,k)}$ and $f_{D}^{(-,k)}$ are explicit and satisfy the following conformal covariance property. If $\psi: D \to \tilde D$ is a conformal map, then for any $k \geq 1$ and any $x_1, \dots , x_k \in D$, with $\xi \in \{\pm,-,+\}$,
\begin{align*}
    &f_{D}^{(\xi, k)}(x_1,\dots ,x_k) = \prod_{j=1}^{k} \vert \psi'(x_j) \vert^{1/8} f_{\tilde D}^{(\xi,k)}(\psi(x_1),\dots ,\psi(x_k)).
\end{align*}
\end{theorem}

\begin{remark}
In \cite{spinCorrelations}, Theorem \ref{theorem_spin_correlations} is established under the assumption that $D_{\delta}$ is a subset of the rotated square lattice $\sqrt{2}\delta e^{i\frac{\pi}{4}}\mathbb{Z}^2$ with meshsize $\sqrt{2}\delta$. The difference of meshsize does not change anything while rotating $\delta \mathbb{Z}^2$ only affects the value of the constant $C_{\sigma}$. The functions $(f_{D}^{(\xi,k)})_{k\geq 1}$, $\xi \in \{\pm, -, +\}$, remain the same. Besides, in \cite{spinCorrelations}, Theorem \ref{theorem_spin_correlations} is in fact shown to hold under no assumption on the smoothness of the underlying domain $D$, and therefore is not restricted to discrete approximations $(D_{\delta};a_{\delta},b_{\delta})$ and domains $(D;a,b)$ satisfying the stronger assumptions of Section \ref{sec_assumptions_domain}.
\end{remark}

Spin correlations of the critical Ising model with $+1$ boundary conditions are often easier to work with than those of the critical Ising model with Dobrushin boundary conditions. The following lemma bounds the latter correlations by the former and this bound will enable us to get sufficient control on the moments of $\sum_x \sigma_x$. We refer the reader to \cite[Section~2]{papon2024interfacescalinglimitcritical} for a proof.

\begin{lemma} \label{lemma_EEpm_EEplus}
Let $k \geq 1$. For any $x_1, \dots, x_k \in D_{\delta}$, $\vert \mu_{\delta}^{\pm}[\prod_{j=1}^{k}\sigma_{x_j}] \vert \leq \mu_{\delta}^{+}[\prod_{j=1}^k \sigma_{x_j}]$. 
\end{lemma}

Let us now collect further results on spin correlations that will be needed in the proof of Theorem \ref{theorem-single-interface}. First, the following lemma provides a bound on $k$-point spin correlations of the critical Ising model with $+1$ boundary conditions. We refer the reader to \cite[Section~2]{papon2024interfacescalinglimitcritical} for a proof.

\begin{lemma} \label{lemma-bound-spin-correlations}
There exists a constant $C>0$ such that for any $k \geq 1$ and any $x_1, \dots , x_k \in D_{\delta}$, 
\begin{equation*}
    \mu_{\delta}^{+}[\prod_{j=1}^{k}\sigma_{x_j}] \leq C^k\delta^{\frac{k}{8}}\exp\bigg( \frac{1}{8}\sum_{j=1}^{k} \log \frac{1}{L(x_j) \wedge \dist(x_j, \partial D_{\delta})} \bigg)
\end{equation*}
where for each $j$, $L(x_j) = \frac{1}{2}\min_{p:p \neq j} \vert x_p - x_j \vert$. Above, for $x \in D_{\delta}$, $d_{\delta}(x):=\dist(x, \partial D_{\delta})$ is the Euclidean distance between $x$ and $\partial D_{\delta}$. Besides, by convention, for $k=1$, $L(x)=\infty$.
\end{lemma}

Second, in the proof of Theorem \ref{theorem-single-interface}, we will need some control on the sum over $D_{\delta}^k$ of the quantity on the right-hand side of the estimate of Lemma \ref{lemma-bound-spin-correlations}. The following two statements will therefore be relevant.

\begin{lemma} \label{lemma-sum-prod-L-x-j-delta}
    Let $0 < \alpha < 1$. There exists a constant $C>0$ depending on $D$ only such that for any $k \in \mathbb{N}$,
    \begin{equation*}
        \delta^{2k} \sum_{(x_1,\dots,x_k) \in D_{\delta}^k \setminus B_k(\delta)} \big[\prod_{j=1}^{k} L(x_j)\big]^{-\alpha} \leq C^kk^{\alpha k}
    \end{equation*}
    where we have set $B_k(\delta):= \{(x_1,\dots,x_k) \in D_{\delta}^k: \exists i \neq j, \vert x_i - x_j \vert \leq 4\delta \}$.
\end{lemma}

\begin{proof}
    This follows from combining the arguments of the proof of \cite[Proposition~3.10]{MourratIsing} to handle the sum together with the arguments of the proof of \cite[Lemma~3.10]{Junnila} to obtain that the $k$-dependent constant $C_k$ of \cite[Proposition~3.10]{MourratIsing} can in fact be replaced by $C^k$ for some constant $C>0$.
\end{proof}

\begin{lemma} \label{lemma-sum-dist-to-boundary-power-alpha}
    Let $0<\alpha <1$. There exists a constant $C_{\alpha}>0$ depending only on $D$ and $\alpha$ such that $\delta^{2}\sum_{x \in D_{\delta}} d_{\delta}(x)^{-\alpha} \leq C_{\alpha}$. Here, $d_{\delta}(x)$ is as in Lemma \ref{lemma-bound-spin-correlations}. 
\end{lemma}

\begin{proof}
    Let $0<\alpha < 1$ and $\delta > 0$. Set $D(\delta,\alpha)=\delta^2\sum_{x \in D_{\delta}} d_{\delta}(x)^{-\alpha}$. With $d_{D}:= \diam(D)$, we have that
    \begin{equation} \label{ineq-D-alpha-N-l}
        D(\delta,\alpha) \leq \delta^2 \sum_{l=0}^{\log(\delta^{-1}d_{D})} \vert N_{l}(\partial D_{\delta}) \vert e^{-\alpha l}\delta^{-\alpha}
    \end{equation}
    where for $l \in \mathbb{N}$, we have set $N_{l}(\partial D_{\delta}) := \{ x \in D_{\delta}: e^{l}\delta \leq d_{\delta}(x) < e^{l+1}\delta \}$. By the condition \eqref{condition-approximation-boundary} imposed on $(\partial D_{\delta})_{\delta}$ (see Section \ref{sec_assumptions_domain}), we have that for any $l \in \mathbb{N}$, $\vert N_{l}(\partial D_{\delta}) \vert \leq C e^{l+1}\delta^{-1}$. Plugging this upper bound into the right-hand side of \eqref{ineq-D-alpha-N-l}, we obtain the lemma.
\end{proof}

\subsubsection{The discrete spin-spin interface of the Ising model}

Under $\mu_{\delta}^{\pm}$, in each spin configuration, there is an interface $\gamma_{\delta}$ separating $+1$ and $-1$ spins: this is a discrete path running from $a_{\delta}$ to $b_{\delta}$ such that a vertex adjacent to it has spin $+1$ if it is on its left and $-1$ if it is on its right. Note that this interface is a priori not unique. In what follows, we always choose the leftmost such interface but this choice is irrelevant: any other reasonable way of prescribing how to turn when encountering a face surrounded by four alternating spins would yield the same scaling limit as $\delta \to 0$. As in \cite{cvgIsingSLE}, for technical reasons, we draw $\gamma_{\delta}$ on the auxiliary square-octagon lattice, with octagons corresponding to vertices of $D_{\delta}$. This for example, prevents $\gamma_{\delta}$ from having self-intersection points.

\subsubsection{The discrete spin loops of the Ising model}\label{sec:def-of-loops}

For $G$ a finite subgraph of $\delta \mathbb{Z}^2$, under $\mu_{G}^{+}$, each spin configuration can be mapped to a configuration of loops defined on $G^{*}$, the dual of $G$, such that for each edge $e$ of $G$ crossed by a loop, the endpoints of $e$ have opposite spins. We call these loops the spin loops of the Ising model.

For what comes next, we will need to give a more precise definition of these spin loops. To do so, let us first define what paths and $*$-paths are. A sequence of vertices $v_1, \dots, v_n$ of $G$ is called a path if for any $i$, $v_{i}$ and $v_{i+1}$ share an edge. It is called a $*$-path if for any $i$, $v_{i}$ and $v_{i+1}$ share a face.

Let $G^{*}$ denote the graph dual to $G$. An Ising spin loop $\ell$ is an oriented simple loop on $G^{*}$ (i.e., a closed path in $G^{*}$ such that no edge is used twice) such that any edge in $\ell$ has a $+1$ spin on its left and a $-1$ spin on its right. Note that this definition implies in particular that the loop $\ell$ is oriented clockwise if it has $+1$ spins on its outside, and counterclockwise if it has $-1$ spins on its outside. Moreover, we say that an Ising spin loop $\ell$ is leftmost if it follows a path of $+1$ spins on its left side, and rightmost if it follows a path of $-1$ spins on its right.

Besides, when the Ising model in $G$ has $+1$ boundary conditions, an Ising spin loop is said to be outermost if it is not strictly contained in any other spin loop. Observe that two outermost loops can intersect (even though this almost surely does not happen in the scaling limit when the model is at criticality). Moreover, each Ising spin loop has a naturally defined level: it has level $1$ if it is outermost, level $2k$ with $k \geq 1$ if it is contained inside an Ising spin loop of level $2k-1$ and is not separated by a weak path of $+1$ spins from this loop, and level $2k+1$ with $k \geq 1$ if it is contained inside an Ising loop of level $2k$ and is not separated by a $*$-path of $-1$ spins from this loop.

\subsection{Scaling limit of the spin-spin interface at criticality}

In this section, we discuss the scaling limit of the spin-spin interface of the critical Ising model. To this end, we first give a few reminders on Loewner chains and SLE curves and describe the space of curves and the topology in which the convergence of the Ising spin-spin interface was proved.

\subsubsection{Loewner chains}

Set $\HH:= \{ z \in \mathbb{C}: \Im(z) > 0 \}$ and let $\gamma: [0,\infty) \to \overline{\mathbb{H}}$ be a non-self-crossing curve targeting $\infty$ and such that $\gamma(0)=0$. For $t \geq 0$, let $K_t$ be the hull generated by $\gamma([0,t])$, that is $\HH \setminus K_t$ is the unbounded connected component of $\HH \setminus \gamma([0,t])$. In the case where $\gamma([0,t])$ is non-self-touching, $K_t$ is simply given by $\gamma([0,t])$. For each $t \geq 0$, it is easy to see that there exists a unique conformal $g_t: \HH \setminus K_t \to \HH$ satisfying the normalization $g_t(\infty) = \infty$ and such that $\lim_{z \to \infty} (g_t(z) - z) = 0$. It can then be proved that $g_t$ satisfies the asymptotic
\begin{equation*}
    g_t(z) = z + \frac{a_1(t)}{z} + O(\vert z \vert^{-2}) \quad \text{as } \vert z \vert \to \infty.
\end{equation*}
The coefficient $a_1(t)$ is equal to $\text{hcap}(K_t)$, the half-plane capacity of $K_t$, which, roughly speaking, is a measure of the size of $K_t$ seen from $\infty$. Moreover, one can show that $a_1(0)=0$ and that $t \mapsto a_1(t)$ is continuous and strictly increasing. Therefore, the curve $\gamma$ can be reparametrized in such a way that at each time $t$, $a_1(t) = 2t$. $\gamma$ is then said to be parameterized by half-plane capacity.

In this time-reparametrization and with the normalization of $g_t$ just described, it is known that there exists a unique real-valued function $t \mapsto W_t$, called the driving function, such that the following equation, called the Loewner equation, is satisfied:
\begin{equation} \label{eq_Loewner}
    \partial_t g_t(z) = \frac{2}{g_t(z) - W_t},  \quad g_0(z)=z, \quad \text{for all $z \in \HH \setminus K_t$}.
\end{equation}
Indeed, it can be shown that $g_t$ extends continuously to $\gamma(t)$ and setting $W_t = g_t(\gamma(t))$ yields the above equation, see e.g. \cite[Chapter~4]{book_Lawler} and \cite[Chapter~4]{book_SLE}.

Conversely, given a continuous and real-valued function $t \mapsto W_t$, one can construct a locally growing family of hulls $(K_t)_t$ by solving the equation \eqref{eq_Loewner}. Under additional assumptions on the function $t \mapsto W_t$, the family of hulls obtained using \eqref{eq_Loewner} is generated by a curve, in the sense explained above. 

Schramm-Loewner evolutions, or SLE for short, are random Loewner chains introduced by Schramm \cite{Schramm_SLE}. For $\kappa \geq 0$, SLE$_\kappa$ is the Loewner chain obtained by considering the Loewner equation \eqref{eq_Loewner} with driving function $W_t = \sqrt{\kappa}B_t$, where $(B_t, t \geq 0)$ is a standard one-dimensional Brownian motion. As such, SLE$_{\kappa}$ is defined in $\HH$ but, thanks to the conformal invariance of the Loewner equation, SLE$_{\kappa}$ can be defined in any simply connected domain $\Omega \subset \mathbb{C}$ with two marked boundary points $a, b \in \partial \Omega$ by considering a conformal map $\phi: \Omega \to \HH$ with $\phi(a)=0$ and $\phi(b) = \infty$ and taking the image of SLE$_{\kappa}$ in $\HH$ by $\phi^{-1}$. In particular, SLE$_{\kappa}$ is conformally invariant, and it turns out that this conformal invariance property together with a certain domain Markov property characterizes the family (SLE$_{\kappa}, \kappa \geq 0)$. 

\subsubsection{Topology and convergence on the space of curve} \label{sec_topology_curves}

Following \cite{tightness-curves}, the space of curves that we will consider is a subspace of the space of continuous mappings from $[0,1]$ to $\mathbb{C}$ modulo reparametrization. More precisely, let
\begin{equation*}
    \mathcal{C}':= \left \{
    f \in \mathcal{C}([0,1], \mathbb{C}): 
    \begin{aligned}
    &\text{ either $f$ is not constant on any subinterval of $[0,1]$} \\ &\text{or $f$ is constant on $[0,1]$} 
    \end{aligned}
    \right \}
\end{equation*}
and let $f_1, f_2 \in \mathcal{C}'$ be equivalent if there exists an increasing homeomorphism $\psi: [0,1] \to [0,1]$ with $f_2 = f_1 \circ \psi$. We denote by $[f]$ the equivalence class of $f$ under this equivalence relation and set
\begin{equation*}
    X(\mathbb{C}):= \{ [f]: f \in \mathcal{C}'\}.
\end{equation*}
$X(\mathbb{C})$ is called the space of curves. We turn $X(\mathbb{C})$ into a metric space by equipping it with the metric
\begin{equation*}
    d_X(f,g) := \inf \{ \| f_0 - g_0 \|_{\infty}: \, f_0 \in [f], g_0 \in [g] \}.
\end{equation*}
$(X(\mathbb{C}), d_X)$ is a separable and complete metric space, but it is not compact. Given $\Omega \subsetneq \mathbb{C}$ with $\partial \Omega \neq \emptyset$ and two marked boundary points $a$ and $b$, we define the space of simple curves from $a$ to $b$ in $\Omega$ as
\begin{equation*}
    X_{\text{simple}}(\Omega, a, b) := \{ [f]: f \in \mathcal{C}', f((0,1)) \subset \Omega, f(0)=a, f(1)=b, \, \text{$f$ injective} \}.
\end{equation*}
We then let $X_0(\Omega, a, b)$ be the closure of $X_{\text{simple}}(\Omega, a, b)$ in $X(\mathbb{C})$ with respect to the metric $d_X$. Curves in $X_0(\Omega, a, b)$ run from $a$ to $b$, may touch $\partial \Omega$ elsewhere than at their endpoints, may touch themselves and have multiple points, but they can have no transversal self-crossings. Notice that if $(\PP_n)_n$ is a sequence of probability measures supported on $X_{\text{simple}}(\Omega, a, b)$ that converges weakly to a probability measure $\PP^{*}$, then a priori $\PP^{*}$ is supported on $X_0(\Omega, a, b)$.

Assume that we have a family of random curves $(\gamma_{\delta})_{\delta}$ distributed according to some family $(\PP_{\delta})_{\delta}$ such that for each $\delta > 0$, $\gamma_{\delta}$ is an element of $X_{0}(\hat \Omega_{\delta};a_{\delta},b_{\delta})$, with $((\hat \Omega_{\delta};a_{\delta},b_{\delta}))_{\delta}$ satisfying the assumptions of Section \ref{sec_assumptions_domain}. The curves $(\gamma_{\delta})_{\delta}$ can be made to be supported on the same space $X_{0}(\Omega;a,b)$ by uniformization. More precisely, let $\phi: \Omega \to \mathbb{H}$ be a conformal map such that $\phi(a) = 0$, $\phi(b)=\infty$ and $\Im \phi(0)=1$ and similarly, for $\delta > 0$, let $\phi_{\delta}: \hat \Omega_{\delta} \to \mathbb{H}$ be a conformal map such that $\phi_{\delta}(a_{\delta}) = 0$, $\phi_{\delta}(b_{\delta})=\infty$ and $\Im \phi_{\delta}(0)=1$. Setting $\gamma_{\delta}^{\HH}:= \phi_{\delta}(\gamma_{\delta})$, the family $(\gamma_{\delta}^{\HH})_{\delta}$ is supported on $X_{0}(\HH;0,\infty)$, where $X_{0}(\HH;0,\infty)$ can be understood as obtained by extending the above definition to curves defined on the Riemann sphere. Moreover, for each $\delta > 0$, when parametrized by half-plane capacity, $\gamma_{\delta}^{\HH}$ has an associated driving function $W_{\delta}$. In this setting, there are two different topological spaces in which we may wish to show tightness of $(\gamma_{\delta})_{\delta}$:
\begin{enumerate}[leftmargin=1.5cm]
    \item [{\crtcrossreflabel{(T.1)}[topo_1]}] the space of curves $X_0(\HH,0,\infty)$ equipped with the metric $d_X$;
    \item [{\crtcrossreflabel{(T.2)}[topo_2]}] the metrizable space of continuous functions on $[0, \infty)$ with the topology of uniform convergence on compact subsets of $[0, \infty)$.
\end{enumerate}
Tightness could also be established for the driving functions $(W_{\delta})_{\delta}$, in which case the topological space to consider is
\begin{enumerate}[itemsep=-1ex, leftmargin=1.5cm]
    \item [{\crtcrossreflabel{(T.3)}[topo_3]}] the metrizable space of continuous functions on $[0, \infty)$ with the topology of uniform convergence on compact subsets of $[0, \infty)$.
\end{enumerate}
It turns out that, under appropriate crossing estimates, by \cite[Corollary~1.7]{tightness-curves}, weak convergence in one of the topologies \ref{topo_1}--\ref{topo_3} implies weak convergence in the other two and that the limits agree, in the sense that the limiting random curve is driven by the limiting driving function. In the case of the critical planar Ising model, these crossing estimates are established in \cite[Section~4.2]{tightness-curves}.

Another natural topological space in which to establish tightness of $(\gamma_{\delta})_{\delta}$ is
\begin{enumerate}[leftmargin=1.5cm]
    \item [{\crtcrossreflabel{(T.4)}[topo_4]}] the space of curves $X(\mathbb{C})$ equipped with the metric $d_X$.
\end{enumerate}
Under the assumption that $(\hat \Omega_{\delta}; a_{\delta}, b_{\delta})_{\delta}$ converges in the Carath\'eodory sense to $(\Omega;a,b)$ and that $(a_{\delta})_{\delta}$ and $(b_{\delta})_{\delta}$ are close approximations of the degenerate prime ends $a$ and $b$, \cite[Theorem~4.1]{Karrila} and \cite[Theorem~4.2]{Karrila} show that weak convergence of $(\gamma_{\delta})_{\delta}$ in the topology \ref{topo_4} implies weak convergence in the other three topologies \ref{topo_1}--\ref{topo_3}, see the discussion below \cite[Theorem~4.2]{Karrila}. Moreover, if $\gamma$ denotes the limit in the topology \ref{topo_4} and $\gamma^{\HH}$ denotes the limit in the topology \ref{topo_1}, then $\gamma$ has the same law as $\phi^{-1}(\gamma^{\HH})$. We note that \cite[Theorem~4.2]{Karrila} also guarantees that $\gamma$ is supported on $\overline{\Omega}$.

\subsubsection{The scaling limit of the critical Ising interface} \label{sec-critical-Ising-curve-convergence}

A conformal invariance property of the critical Ising model emerges when looking at the scaling limit of the interface $\gamma_{\delta}$ under $\mu_{\delta}^{\pm}$ as $\delta \to 0$. Indeed, as a consequence of the next theorem, its limiting law is conformally invariant. 

\begin{theorem}[\cite{cvgIsingSLE}] \label{theorem_SLE3}
Assume that $(D;a,b)$ and $(D_{\delta},a_{\delta},b_{\delta})_{\delta}$ are as in Section \ref{sec_assumptions_domain}. Then, as $\delta \to 0$, $(\gamma_{\delta})_{\delta}$ under $(\mu_{\delta}^{\pm})_{\delta}$ converges weakly in the topologies \ref{topo_1}--\ref{topo_4} to SLE$_3$ in $D$ from $a$ to $b$.
\end{theorem}

Theorem \ref{theorem_SLE3} will be the starting point of the proof of Theorem \ref{theorem-single-interface}. Let us note that in \cite{cvgIsingSLE}, Theorem \ref{theorem_SLE3} is established without assuming that $\partial D$ is smooth and therefore, the smoothness assumption on $\partial D$ imposed in Section \ref{sec_assumptions_domain} is not necessary for this result to hold, but is necessary for Theorem \ref{theorem-single-interface}.

\subsection{Scaling limit of the spin loops at criticality}

In this section, we discuss the scaling limit of the spin loops of the Ising model at criticality. To state the result of \cite{cvg-nested-loops-Ising}, which proves convergence of the critical Ising spin loops, we first recall the topological framework in which convergence can be obtained and then give a short reminder on conformal loop ensembles (CLE).

\subsubsection{The space of loops collections and its topology} \label{sec-topology-loops}

An oriented loop $\ell$ is an equivalence class of continuous and injective maps from the unit circle $S^1$ to $\mathbb{C}$, where the equivalence relation is the following: $\ell,\ell': S^1 \to \mathbb{C}$ are equivalent if there exists an orientation-preserving reparametrization $\psi: S^1 \to S^1$ such that for any $t \in S^1$, $\ell(t) = \ell'(\psi(t))$. 
We consider the following metric of the space of oriented loops: if $\ell_1$, $\ell_2$ are two oriented loops,
\begin{equation*}
    d_{\LL}(\ell_1,\ell_2) := \inf_{\psi_1,\psi_2}\inf_{t} \| \ell_1(\psi_1(t)) - \ell_2 (\psi_2(t))\|_{\infty}
\end{equation*}
where the infimum is taken over all orientation-preserving reparametrizations $\psi_1, \psi_2$ of $\ell_1$ and $\ell_2$. We then let $\LL$ be the completion of the set of simple oriented loops with respect to the metric $d_{\LL}$: $\mathcal{L}$ is the set of oriented non-self-crossing loops in the plane. Note that a loop in $\LL$ may have double points or may be reduced to a point.

The space $\XX$ of loop collections is the space of at most countable collections $\{\ell_j\}_j$ of loops in $\LL$ (the empty collection is included in $\XX$) such that
\begin{itemize}
    \item for each $i$, the loop $\ell_i$ is not reduced to a point;
    \item for any $\eps>0$, there are finitely many loops in $\{\ell_j\}_j$ such that $\diam(\ell) \geq \eps$.
\end{itemize}
The space $\XX$ is endowed with the following $\sigma$-algebra. We call a matching of two sets $I$ and $J$ a subset $\pi \subset I \times J$ such that for each $i \in I$, there is at most one $j \in J$ such that $(i,j) \in \pi$, and reciprocally, for each $j \in J$, there is at most one $i \in I$ such that $(i,j) \in \pi$. Given a matching $\pi$, $I^{\pi}$, respectively $J^{\pi}$, denotes the set of unmatched indices in $I$, respectively $J$. The $\sigma$-algebra that we put on $\mathcal{X}$ is then the Borel $\sigma$-algebra associated with the following metric:
\begin{equation*}
    d_{\mathcal{X}}(\{\ell_i\}_{i \in I}, \{\tilde \ell_j\}_{j\in J}):= \inf_{\pi} \max \bigg( \sup_{(i,j) \in \pi} d_{\LL}(\ell_i, \tilde \ell_j), \sup_{i \in I^{\pi}}\diam(\ell_i), \sup_{j \in J^{\pi}} \diam(\tilde \ell_j) \bigg)
\end{equation*}
where the infimum is taken over all matchings $\pi$ of the sets $I$ and $J$.

A loop collection $\{\ell_{j}\}_j$ in $\mathcal{X}$ may be quite arbitrary. We will say that it is non-crossing if for any two loops $\ell_1$ and $\ell_2$ in $\{\ell_{j}\}_j$, one can find two sequences $(\ell_{1,n})_n$ and $(\ell_{2,n})_n$ in $\LL$ such that $\ell_{1,n} \to \ell_1$, $\ell_{2,n} \to \ell_2$ and for any $n \in \mathbb{N}$, $\ell_{1,n}$ and $\ell_{2,n}$ are disjoint. The loop collection $\{\ell_{j}\}_j$ will also be said non-nested if for any $\ell_1, \ell_2 \in \{\ell_{j}\}_j$, the interiors of $\ell_1$ and $\ell_2$ are disjoint.

Let us now recall the following simple but important lemma proved in \cite{cvg-nested-loops-Ising}.

\begin{lemma}[Lemma 5 in \cite{cvg-nested-loops-Ising}]\label{lemma-measurability-in-XX}
    The metric space $(\XX, d_{\XX})$ is complete and separable. Moreover, the following events on $\mathcal{X}$ are measurable for the Borel $\sigma$-algebra: \{the collection $\{\ell_i\}_{j \in J}$ is non-crossing\}, \{the collection $\{\ell_i\}_{j \in J}$ is non-nested\}, \{the loops of $\{\ell_j\}$ are disjoint\} and \{all the loops in $\{\ell_j\}_{j \in J}$ are simple\}.
\end{lemma}

\subsubsection{Conformal loop ensembles} \label{sec-def-CLE}

Conformal loop ensembles (CLE) are a family of probability distributions on countable ensembles of non-nested loops in open and simply connected domains of the complex plane \cite{CLE_She, CLE_Markovian}. This family is indexed by a parameter $\kappa \in (8/3,8)$ and $\operatorname{CLE}_{\kappa}$ is connected to $\operatorname{SLE}_{\kappa}$ via the so-called branching tree construction \cite{CLE_She}. The geometry of the loops of a $\operatorname{CLE}_{\kappa}$ in an open and simply connected domain $D$ depends on the value of $\kappa$: when $\kappa \in (8/3,4]$, these loops are almost surely simple loops that do not intersect each other or $\partial D$; on the contrary, when $\kappa \in (4,8)$, they are almost surely non-simple but non-self-crossing and they may touch (but not cross) $\partial D$ and each other. Another important property of CLE$_\kappa$ is their conformal invariance in law: if $\varphi: D \to \tilde D$ is a conformal map between two open and simply connected domains of $\mathbb{C}$ and $\Gamma$ is a CLE$_{\kappa}$ in $D$, then $\varphi(\Gamma)$ has the law of a CLE$_\kappa$ in $\tilde D$.

When $\kappa \in (8/3,4]$, nested CLE$_{\kappa}$ in a simply connected domain $D \subset \mathbb{C}$ is constructed from a
(non-nested) CLE$_{\kappa}$ by iterating the construction of CLE$_{\kappa}$ in each loop in the following way. The outermost loops of nested CLE$_{\kappa}$ in $D$ have the law of CLE$_{\kappa}$ in $D$ and for each $k \geq 1$, the loops of level $k+1$ are independent samples of CLE$_{\kappa}$ in the interiors of the loops of level $k$. Here, we choose to orient the loops according to their level: loops of even level are oriented counterclockwise, loops of odd level are oriented clockwise.

\subsubsection{The scaling limit of the critical Ising spin loops}\label{subsection-limit-critical-loops}

As in the case of Dobrushin boundary conditions, a conformal invariance property of the critical Ising model emerges when taking the scaling limit of its collection of nested spin loops: it converges to nested CLE$_3$ \cite{cvg-nested-loops-Ising}. A precise version of this convergence result will be stated just below, once we will have introduced the assumptions on the continuum domain $D$ and its approximating sequence $(D_{\delta})_{\delta}$.

The domain $D \subset \mathbb{C}$ is assumed to satisfy the same assumptions as in Section \ref{sec_assumptions_domain}. As for the approximating sequence $(D_{\delta})_{\delta}$, we require that it satisfies the following. Let the polygonal domains $(\hat D_{\delta})_{\delta}$ and the maps $(\psi_{\delta})_{\delta}$ be defined as in Section \ref{sec_assumptions_domain}. For each $\delta > 0$, $\partial \hat D_{\delta}$ is a continuous Jordan curve and therefore, $\psi_{\delta}^{-1}$ has a continuous extension to $\mathbb{D} \cup \partial \mathbb{D}$. Similarly, $\psi^{-1}$ can also be continuously extended to $\mathbb{D} \cup \partial \mathbb{D}$. With a slight abuse of notations, we still denote these extensions by $\psi_{\delta}^{-1}$ and $\psi^{-1}$. In this paper, when looking at the convergence of nested spin loops, we will always assume the following on the maps $(\psi_{\delta}^{-1})_{\delta}$ and $\psi$: as $\delta \to 0$, $\psi_{\delta}^{-1} \to \psi^{-1}$ uniformly on $\partial \mathbb{D}$. Note that by Rad\'o's theorem (see \cite[Theorem~2.11]{Pommeranke}), this implies that as $\delta \to 0$, $\psi_{\delta}^{-1} \to \psi^{-1}$ uniformly in $\mathbb{D} \cup \partial \mathbb{D}$.

\begin{theorem}[Theorem 6 in \cite{cvg-nested-loops-Ising}] \label{theorem-cvg-loops-critical}
    Under the above assumptions on $(D_{\delta})_{\delta}$ and $D$, as $\delta \to 0$, $(\Gamma_{\delta})_{\delta}$ under $(\mu_{\delta}^{+})_{\delta}$ converges in law with respect to the metric $d_{\XX}$ to nested CLE$_3$ in $D$.

    Furthermore, for any $\eps >0$, the following holds with probability tending to $1$ as $\delta \to 0$: for any Ising loop $\ell$ of diameter larger than $\eps$, there exists a leftmost loop $\ell_{L}$ such that $d_{\LL}(\ell, \ell_{L}) \leq \eps$ and such that the connected components of $(\ell \cup \ell_{L}) \setminus (\ell \cap \ell_{L})$ have diameter less than $\eps$.
\end{theorem}

As explained in \cite{cvg-nested-loops-Ising}, the same result holds when one considers the sequence $(\Gamma_{\delta})_{\delta}$ of rightmost Ising spin loops. Moreover, the second part of Theorem \ref{theorem-cvg-loops-critical} shows that as $\delta \to 0$, any Ising spin loop is very close to a leftmost Ising loop. In particular, understanding the leftmost Ising spin loops is enough to understand all Ising spin loops.

\subsection{The two-dimensional random field Ising model} \label{sec-def-RFIM}

In this section, we define the two-dimensional random field Ising model (RFIM) and its near-critical regime. We then recall the convergence result of \cite{CSZ16}, which established convergence in law of the RFIM partition function in the near-critical regime.

Let $G$ be a finite subgraph of $\delta \mathbb{Z}^2$, for some $\delta > 0$. Let $(H_{x})_{x \in \mathbb{Z}^{2}}$ be a collection of independent Gaussian random variables with mean $0$ and variance $1$. We denote their joint law by $\PP$. For each $x \in G$, set $h_{x}:= H_{x/\delta}$. Let also $\lambda: D \to (0,\infty)$ be a continuous and bounded function.

For $\xi \in \{-,\pm,+\}$, the random field Ising model (RFIM) is the probability measure $\mu_{G}^{\xi, \lambda, h}$ on the set of spin configurations $\sigma: V(G) \cup \partial G \to \{-1,1\}$ with
\begin{equation*}
    \mu_{G}^{\xi, \lambda, h}(\sigma):= \frac{1}{\mathcal{Z}_{G}^{\xi,\lambda,h}} \exp \bigg(\sum_{x \in V(G)} \lambda(x)h_x\sigma_x \bigg)\mu_{G}^{\xi}(\sigma)
\end{equation*}
where the normalizing constant $\mathcal{Z}_{G}^{\xi,\lambda,h}$, called the partition function, is given by
\begin{equation*}
    \mathcal{Z}_{G}^{\xi,\lambda,h} := \mu_{\delta}^{\xi}\bigg[ \exp \bigg(\sum_{x \in V(G)} \lambda(x)h_x\sigma_x \bigg) \bigg].
\end{equation*}
When $G=D_{\delta}$, we set $\mathcal{Z}_{\delta}^{\xi,\lambda,h}:= \mathcal{Z}_{G}^{\xi,\delta^{\frac{7}{8}}\lambda,h}$ and $\mu_{\delta}^{\xi,\lambda,h}:= \mu_{G}^{\xi,\delta^{\frac{7}{8}}\lambda,h}$ (see just below for the choice of scaling factor $\delta^{\frac{7}{8}}$).

\subsubsection{Scaling limit of the partition function}

Since many observables of the Ising model under $\mu_{\delta}^{\xi}$, $\xi \in \{-,\pm,+\}$, have an interesting behavior in the scaling limit $\delta \to 0$, one may wonder if observables of the RFIM under $\mu_{\delta}^{\xi,h}$ also do. To answer this question, the first natural step is to look at the scaling limit of the partition function. This limit, as stated below, has been shown to exist provided that the noise variables $(h_{x})_{x}$ are appropriately rescaled and that $\mathcal{Z}_{\delta}^{\xi,\lambda,h}$ is multiplied by the correct renormalization factor.

\begin{theorem}[Theorem 3.14 in \cite{CSZ16}]\label{theorem-cvg-RFIM-partition-function}
    Let $D \subset \mathbb{C}$, $a,b \in \partial D$ and $(D_{\delta};a_{\delta},b_{\delta})_{\delta}$ be as in Section \ref{sec_assumptions_domain}. Let $\lambda: \bar D \to (0,\infty)$ be a bounded and continuous function. Let $(H_{x})_{x \in \mathbb{Z}^2}$ be independent Gaussian random variables with mean $0$ and variance $1$ and set, for $x \in D_{\delta}$, $h_x = H_{x/\delta}$. Let $\xi \in \{-,\pm,+\}$.  Then, the rescaled partition function
    \begin{equation} \label{def-tilde-Z-delta}
        \tilde{\mathcal{Z}}_{\delta}^{\xi,\lambda, h}:= \exp\big(-\frac{\|\lambda\|_{L^{2}}^2}{2}\delta^{-\frac{1}{4}}\big) \mu_{\delta}^{\xi}\bigg[\exp\bigg(\delta^{\frac{7}{8}}\sum_{x \in D_{\delta}}\lambda(x)h_{x}\sigma_{x}\bigg)\bigg]
    \end{equation}
    converges in distribution as $\delta \to 0$ to a random variable $\mathcal{Z}_{D}^{\xi,W,\lambda}$ with Wiener chaos expansion
    \begin{equation} \label{partition-function-chaos-expansion}
        \mathcal{Z}_{D}^{\xi,W,\lambda} = 1+ \sum_{n=1}^{\infty} \frac{C_{\sigma}^n}{n!}\int_{D^n} f_{D}^{(\xi,n)}(z_1,\dots,z_n) \prod_{j=1}^{n} \lambda(z_j)W(dz_j)
    \end{equation}
    where $W(\cdot)$ denotes white noise on $\mathbb{R}^2$ and where the functions $(f_{D}^{(\xi,n)})_n$ are as in Theorem \ref{theorem_spin_correlations}.
\end{theorem}

In \cite{CSZ16}, Theorem \ref{theorem-cvg-RFIM-partition-function} is proved for $\xi = +$. The proof of $\xi \in \{\pm, -\}$ is almost identical: this is because, by Lemma \ref{lemma_EEpm_EEplus}, the spin correlations functions $\mu_{\delta}^{\xi}[\prod_{j=1}^{n}\sigma_{x_j}]$, $n \in \mathbb{N}$, are bounded in absolute value by the same function of $x_1,\dots,x_n$, regardless of the value of $\xi$. This implies, as explained in the proof of \cite[Theorem~3.14]{CSZ16}, that Conditions (ii) and (iii) of \cite[Theorem~2.3]{CSZ16} are satisfied, which yields convergence of $\tilde{\mathcal{Z}}_{\delta}^{\xi,\lambda,h}$ as $\delta \to 0$. Besides, in \cite{CSZ16}, Theorem \ref{theorem-cvg-RFIM-partition-function} is established under weaker conditions on the noise variables $(h_x)_{x}$, see \cite[Equation~(3.1)]{CSZ16}. We stick here to the Gaussian case for simplicity.
 
For the proof of Theorem \ref{theorem-single-interface}, it will turn out to be useful to have in mind the main steps of the proof of Theorem \ref{theorem-cvg-RFIM-partition-function}. This also sheds light on why $\delta^{\frac{7}{8}}$ is the correct rescaling for the noise. Set
\begin{equation} \label{def-theta-delta}
    \theta(\delta):=\exp\big(-\frac{\Vert\lambda\Vert^2_{L^2}}{2}\delta^{-\frac{1}{4}}\big)
\end{equation}
and for simplicity, take $\xi=+1$ and $\lambda \equiv 1$. The starting point is the following high-temperature expansion of $\tilde{\mathcal{Z}}_{\delta}^{+,h}$: for any $\delta > 0$, we have that
\begin{align*}
    \tilde{\mathcal{Z}}_{\delta}^{+,h} = \theta(\delta) \mu_{\delta}^{+}\bigg[\exp\big(\sum_{x \in D_{\delta}} \delta^{\frac{7}{8}}h_{x}\sigma_x\big)\bigg]
    &= \theta(\delta) \mu_{\delta}^{+}\bigg[ \prod_{x \in D_{\delta}} \cosh(\delta^{\frac{7}{8}}h_x)+\sigma_x\sinh(\delta^{\frac{7}{8}}h_{x})\bigg ]\\
    &= \theta(\delta) \sum_{I \subset D_{\delta}} (\prod_{x \in D_{\delta}\setminus I}\cosh(\delta^{\frac{7}{8}}h_{x}))(\prod_{x \in I}\sinh(\delta^{\frac{7}{8}}h_{x})) \mu_{\delta}^{+}\big[ \prod_{x \in I}\sigma_{x}\big ] \\
    &=\theta(\delta)(\prod_{x \in D_{\delta}}\cosh(\delta^{\frac{7}{8}}h_{x})) \sum_{I \subset D_{\delta}} \mu_{\delta}^{+} \big[ \prod_{x \in I}\sigma_{x}\big ] \prod_{x \in I}\tanh(\delta^{\frac{7}{8}}h_{x})
\end{align*}
where the sum is over all subsets $I$ of lattice points in $D_{\delta}$. Setting $\vartheta_{\delta}^2:=\Var[\tanh(\delta^{\frac{7}{8}} h_{x})]$ and $\eta_x = \vartheta_{\delta}^{-1}(\tanh(\delta^{\frac{7}{8}}h_x)-a_{\delta})$ and observing that observing that $\EE[\tanh(\delta^{\frac{7}{8}}h_{x})]=0$, we obtain that
\begin{equation}\label{chaos-expansion-partition-function}
    \tilde{\mathcal{Z}}_{\delta}^{+,h} = \theta(\delta)(\prod_{x \in D_{\delta}}\cosh(\delta^{\frac{7}{8}}h_{x})) \sum_{I \subset D_{\delta}} \mu_{\delta}^{+}\big[ \prod_{x \in I}\sigma_{x} \big] \prod_{x \in I} \vartheta_{\delta}\eta_x.
\end{equation}
It can be shown that the prefactor $\theta(\delta)(\prod_{x \in D_{\delta}}\cosh(\delta^{\frac{7}{8}}h_{x}))$ converges to $1$ as $ \delta \to 0$ in any $L^p(\PP)$ with $p \geq 1$. In \cite{CSZ16}, to prove the convergence of the sum on the right-hand side of \eqref{chaos-expansion-partition-function} to the chaos expansion \eqref{partition-function-chaos-expansion} and thus conclude the proof of the convergence of $\tilde{\mathcal{Z}}_{\delta}^{+,h}$, a new Lindeberg principle for polynomial chaos is established. This new Lindeberg principle can be applied to prove convergence of the above sum because this sum is uniformly bounded in $L^2(\PP)$ as $\delta \to 0$, thanks to the rescaling $\delta^{\frac{7}{8}}$ of the noise variables. We will not give a full proof of this fact here, but to get an intuition of why this should be the case, the reader can consider the contribution $S_{\delta}^{(1)}(h)$ of subsets $I$ containing a single point of $D_{\delta}$. Using a Taylor expansion, one can see that $\EE[\tanh(\delta^{\frac{7}{8}}h_{x})^2] = \delta^{\frac{7}{4}} + O(\delta^{\frac{7}{2}})$. By independence of the $(h_{x})_{x}$, this yields that
\begin{equation*}
    \EE[S_{\delta}^{(1)}(h)^2] \leq C\delta^{\frac{7}{4}}\sum_{x \in D_{\delta}} \mu_{\delta}^{+}[\sigma_x]^2 
\end{equation*}
where $C>0$ is a constant independent of $\delta$. Using the upper bound of Lemma \ref{lemma-bound-spin-correlations} on $\mu_{\delta}^{+}[\sigma_x]$, we obtain that
\begin{equation*}
    \EE[S_{\delta}^{(1)}(h)^2] \leq C\delta^{2} \sum_{x \in D_{\delta}} \dist(x, \partial D_{\delta})^{-\frac{1}{4}}
\end{equation*}
and the right-hand side of the above inequality is uniformly bounded in $\delta$ by Lemma \ref{lemma-sum-dist-to-boundary-power-alpha}. For subsets $I \subset D_{\delta}$ containing more than one point, one can proceed in a similar way, even though the technical details become more tedious. Computations of this type will be performed in the course of the proof of Theorem \ref{theorem-single-interface}.

\section{Scaling limit of the interface with Dobrushin boundary conditions}\label{sec: Scaling limit of the interface with Dobrushin boundary conditions}

In this section, we prove Theorem \ref{theorem-single-interface}. Below, we denote by $\mu_{\delta}^{\pm,\lambda,h}$ the random field Ising measure in $D_{\delta}$ with Dobrushin boundary conditions, noise variables $(h_{x})_{x \in D_{\delta}}$ rescaled by $\delta^{\frac{7}{8}}$ and external field strength $\lambda$. We refer the reader to Section \ref{sec-def-RFIM} for the definition of this measure. We also recall that $\mu_{\delta}^{\pm}$ denotes the critical Ising measure in $D_{\delta}$ with Dobrushin boundary conditions (and no external field).

\begin{proof}[Proof of Theorem \ref{theorem-single-interface}]
    The proof of Theorem \ref{theorem-single-interface} consists of two steps. In Proposition \ref{proposition-moments-RN} below, we first show tightness of the laws of $(\gamma_{\delta})_{\delta}$ under $(\mu_{\delta}^{\pm,\lambda,h})_{\delta}$ in the topology \ref{topo_4}, that is tightness of $(\nu_{\delta}^{(\pm,\lambda,h)})_{\delta}$. Then, in Proposition \ref{proposition-uniqueness-law-curve}, we prove a characterization of the limiting law of any weakly convergent subsequence obtained as a consequence of the tightness of $(\gamma_{\delta})_{\delta}$ under $(\mu_{\delta}^{\pm,\lambda,h})$. This yields convergence in law of $(\gamma_{\delta})_{\delta}$ under $(\mu_{\delta}^{\pm,\lambda,h})_{\delta}$ in the topology \ref{topo_4}. As explained in Section \ref{sec_topology_curves}, this also implies convergence of $(\gamma_{\delta})_{\delta}$ under $(\mu_{\delta}^{\pm,\lambda,h})_{\delta}$ in the topologies \ref{topo_1}--\ref{topo_3} and the limiting laws in each topology agree. For later reference, let us denote by $\nu_{\delta}^{\pm}$ the law of the interface under $\mu_{\delta}^{\pm}$.
    
    The key observation that will be used both to prove tightness and uniqueness in law is that $\PP$-almost surely, for each $\delta > 0$, the law $\nu_{\delta}^{\pm,\lambda,h}$ of $\gamma_{\delta}$ under $\mu_{\delta}^{\pm,\lambda, h}$ is absolutely continuous with respect to that of $\gamma_{\delta}$ under $\mu_{\delta}^{\pm}$ with Radon-Nikodym derivative given by
    \begin{equation*}
        F^{h}(\gamma_{\delta}):= \frac{\der \nu_{\delta}^{\pm,\lambda,h}}{\der \nu_{\delta}^{\pm}}(\gamma_{\delta}) = \frac{1}{\mathcal{Z}_{\delta}^{\pm,\lambda, h}} \mu_{\delta}^{\pm}\bigg[\exp\bigg(\delta^{\frac{7}{8}}\sum_{x \in D_{\delta}} \lambda(x)h_x\sigma_x \bigg) \bigg \vert \gamma_{\delta} \bigg].
    \end{equation*}
    Tightness of the laws $(\nu_{\delta}^{\pm,\lambda,h})_{\delta}$ in the topology \ref{topo_4} will in fact follow from the uniform (in $\delta$) boundedness of $(F^{h}(\gamma_{\delta}))_{\delta}$ in $L^{1+\alpha}$, for any $0<\alpha<1$. This is established in the proposition below.
    
    \begin{proposition}[tightness] \label{proposition-moments-RN}
        Let $0<\alpha <1 $. There exist $\delta_{0} > 0$ and $C>0$ such that for any $0< \delta < \delta_0$,
        \begin{equation*}
            \EE \otimes \mu_{\delta}^{\pm} \big[ F^{h}(\gamma_{\delta})^{1+\alpha} \big] \leq C.
        \end{equation*}
    \end{proposition}

    We postpone the proof of Proposition \ref{proposition-moments-RN} to Section \ref{sub:proof-tightness-interface} and now state our characterization of the limiting law of weakly convergent subsequences.

    \begin{proposition}[characterization of the limiting law]\label{proposition-uniqueness-law-curve}
        There exists a unique law $\nu_{W,\lambda}^{(D,a,b)}$ such that the following holds. $\PP$-almost surely, the Radon-Nikodym derivative of $\nu_{W,\lambda}^{(D,a,b)}$ with respect to $\nu_{SLE_{3}}^{(D,a,b)}$ is given by
        \begin{equation}\label{RN-single-interface-limit}
            \frac{\der \nu_{W,\lambda}^{(D,a,b)}}{\der \nu_{\operatorname{SLE}_{3}}^{(D,a,b)}}(\gamma) := \frac{S_{L}^{W,\lambda}S_{R}^{W,\lambda}}{\mathcal{Z}_{D}^{\pm,W,\lambda}}
        \end{equation}
        with for $j=L,R$,
        \begin{equation*}
            S_{j}^{W,\lambda}:= 1 + \sum_{k \geq 1} \frac{C_{\sigma}^k}{k!} \int_{D_{j}^{k}} f_{D_{j}}^{(\xi_j,k)}(z_1,\dots,z_k) \prod_{p=1}^{k}\lambda(z_p)W(dz_p)
        \end{equation*}
        and for any sequence $(\delta_n)_n$ such that $(\nu_{\delta_n}^{\pm,\lambda,h})_{n}$ converges weakly as $n \to \infty$ in the topology \ref{topo_4}, 
        \begin{equation*}
            (W_{\delta}^{h} ,\tilde{\mathcal{Z}}_{\delta_{n}}^{\pm,h}, \nu_{\delta_n}^{\pm,\lambda,h})_{n} \to (W, \mathcal{Z}_{D}^{\pm,W,\lambda}, \nu_{W,\lambda}^{(D,a,b)})
        \end{equation*}
        weakly as $n \to \infty$. Above, the process $W_{\delta}^{h}$ is defined via, for $\phi \in L^2(\mathbb{R}^2)$, $W_{\delta}^{h}(\phi)=\delta\sum_{x \in D_{\delta}}h_{x}\phi(x)$, and under $\PP$, $W$ is a two-dimensional white noise while $\mathcal{Z}_{D}^{\pm,W,\lambda}$ is the limiting partition function of Theorem \ref{theorem-cvg-RFIM-partition-function}. The domain $D_{L}$, respectively $D_{R}$, is the connected component of $D \setminus \gamma$ on the left, respectively right, of $\gamma$, $\xi_{L}=+$, $\xi_{R}=-$ and the constant $C_{\sigma}$ and the functions $(f_{D_{j}}^{(\xi_j,k)})_{k,j}$ are as in Theorem \ref{theorem_spin_correlations}.
    \end{proposition}

    We will give the proof of Proposition \ref{proposition-uniqueness-law-curve} in Section \ref{sub:proof-characterization-interface}. In view of Proposition \ref{proposition-moments-RN} and Proposition \ref{proposition-uniqueness-law-curve}, the convergence part of Theorem \ref{theorem-single-interface} follows from Prokhorov's theorem while the almost sure absolute continuity of $\nu_{W,\lambda}^{(D,a,b)}$ with respect to $\nu_{\operatorname{SLE}_{3}}^{(D,a,b)}$ follows from Proposition \ref{proposition-moments-RN}. Finally, conformal covariance of $\nu_{W,\lambda}^{(D,a,b)}$ is established in Proposition \ref{proposition-conformal-covariance-curve} below.
\end{proof}

\subsection{Tightness of \texorpdfstring{$(\nu_{\delta}^{\pm,\lambda,h})_{\delta}$}{}: proof of Proposition \ref{proposition-moments-RN}}\label{sub:proof-tightness-interface}

\begin{proof}[Proof of Proposition \ref{proposition-moments-RN}]
    Let $\alpha > 0$ be such that $1+\alpha <2$. Let us first introduce some notations. For $\delta > 0$ fixed, we let $D_{L,\delta}$, respectively $D_{R,\delta}$, be the connected component of $D_{\delta} \setminus \gamma_{\delta}$ on the left, respectively right, of $\gamma_{\delta}$. We also define the following two random renormalization factors:
    \begin{equation} \label{definition-theta-L-R}
        \theta_{L}(\delta) := \exp\big(-\frac{1}{2}\delta^{\frac{7}{4}}\sum_{x \in D_{L,\delta}}\lambda(x)^2 \big), \quad \theta_{R}(\delta) := \exp\big(-\frac{1}{2}\delta^{\frac{7}{4}}\sum_{x \in D_{R,\delta}}\lambda(x)^2\big).
    \end{equation}
    Below, $\mu_{\delta,D_{L}}^{+}$, respectively $\mu_{\delta,D_{R}}^{-}$, will stand for the Ising measure in $D_{L,\delta}$ with $+1$ boundary conditions, respectively in $D_{R,\delta}$ with $-1$ boundary conditions. Finally, we introduce the two following random variables
    \begin{equation}\label{definition-S-L-S-R}
        S_{L}^{h}(\delta):= \mu_{\delta,D_{L}}^{+}\bigg [ \exp\big(\delta^{\frac{7}{8}}\sum_{x \in D_{L,\delta}}\lambda(x)h_x\sigma_x\big) \bigg], \quad S_{R}^{h}(\delta):= \mu_{\delta,D_{R}}^{-} \bigg [ \exp\big(\delta^{\frac{7}{8}}\sum_{x \in D_{R,\delta}}\lambda(x)h_x\sigma_x\big) \bigg].
    \end{equation}
    By definition of $\gamma_{\delta}$ and the Markov property of the Ising model, we have that, $\PP \otimes \mu_{\delta}^{\pm}$-almost surely,
    \begin{equation} \label{decomposition-RN-interface}
        F^{h}(\gamma_{\delta}) = \frac{S_{L}^{h}(\delta)S_R^{h}(\delta)}{\mathcal{Z}_{\delta}^{\pm,\lambda, h}} \times \exp\bigg( \sum_{x \in V_L(\gamma_{\delta})}\delta^{\frac{7}{8}}\lambda(x)h_x - \sum_{x \in V_R(\gamma_{\delta})}\delta^{\frac{7}{8}}\lambda(x)h_x \bigg)
    \end{equation}
    where $V_{L}(\gamma_{\delta})$, respectively $V_{R}(\gamma_{\delta})$, is the set of vertices of $D_{\delta}$ that are adjacent to $\gamma_{\delta}$ on its left, respectively its right. Here, a vertex $x \in D_{\delta}$ is said to be adjacent to $\gamma_{\delta}$ if $\gamma_{\delta}$ traverses an edge of the octagon whose center is $x$. Let us set
    \begin{equation} \label{definition-magnetization-interface}
        M^{h}(\gamma_{\delta}) := \sum_{x \in V_L(\gamma_{\delta})}\delta^{\frac{7}{8}}\lambda(x)h_x - \sum_{x \in V_R(\gamma_{\delta})}\lambda(x)\delta^{\frac{7}{8}}h_x.
    \end{equation}
    Fix some $p > 1$ such that $(1+\alpha)p <2$ and let $q>1$ be such that $p^{-1}+q^{-1}=1$. H\"older's inequality implies that
    \begin{align} \label{Holder-for-RN-curve}
        \EE \otimes \mu_{\delta}^{\pm} \big[ F^{h}(\gamma_{\delta})^{1+\alpha} \big] \leq &\EE \otimes \mu_{\delta}^{\pm} \big[\big( \theta_{L}(\delta)\theta_{R}(\delta) S_{R}^{h}(\delta)S_{L}^{h}(\delta) \big)^{p(1+\alpha)}\big]^{\frac{1}{p}}\times\bigg(\frac{\theta(\delta)}{\theta_D(\delta)}\bigg)^{1+\alpha} \nonumber \\
        &\times \EE \otimes \mu_{\delta}^{\pm} \bigg[\exp(q(1+\alpha)M(\gamma_{\delta}))(\tilde{\mathcal{Z}}_{\delta}^{\pm,\lambda, h})^{-q(1+\alpha)} \bigg(\frac{\theta_{D}(\delta)}{\theta_{R}(\delta)\theta_{L}(\delta)}\bigg)^{q(1+\alpha)} \bigg]^{\frac{1}{q}},
    \end{align}
    where $\theta(\delta)$ is defined as in \eqref{def-theta-delta} and where we have set
    \begin{equation*}
        \theta_D(\delta) := \exp\big( - \frac{\delta^{\frac{7}{4}}}{2} \sum_{x \in D_{\delta}} \lambda(x)^2 \big).
    \end{equation*}
    Observe that $\theta(\delta)/\theta_D(\delta) \to 1$ as $\delta \to 0$: above, we introduced $\theta_{D}(\delta)$ because the ratio $\frac{\theta_{D}(\delta)}{\theta_{R}(\delta)\theta_{L}(\delta)}$ is easier to control that the ratio $\frac{\theta(\delta)}{\theta_{R}(\delta)\theta_{L}(\delta)}$. In view of this and of \eqref{Holder-for-RN-curve}, Proposition \ref{proposition-moments-RN} is a consequence of the following two claims, whose proofs are given just below. In the statement of these claims, the constants depend only on $\lambda$ and $D$.

    \begin{claim} \label{claim-uniform-boundedness-left-right-spin-fields}
        There exist $\delta_1 >0$ and $C_1 > 0$ such that for any $0 < \delta < \delta_1$,
        \begin{equation*}
            \EE \otimes \mu_{\delta}^{\pm} \big[\big( \theta_{L}(\delta)\theta_{R}(\delta) S_{R}^{h}(\delta)S_{L}^{h}(\delta) \big)^{p(1+\alpha)}\big]^{\frac{1}{p}} \leq C_1.
        \end{equation*}
    \end{claim}

    \begin{claim} \label{claim-uniform-boundedness-magnetization-partition-function}
        There exist $\delta_2 >0$ and $C_2 > 0$ just that for any $0 < \delta < \delta_2$,
        \begin{equation*}
            \EE \otimes \mu_{\delta}^{\pm} \bigg[\exp(q(1+\alpha)M(\gamma_{\delta}))(\tilde{\mathcal{Z}}_{\delta}^{\pm,\lambda,h})^{-q(1+\alpha)} \bigg(\frac{\theta_D(\delta)}{\theta_{R}(\delta)\theta_{L}(\delta)}\bigg)^{q(1+\alpha)} \bigg] \leq C_2.
        \end{equation*}
    \end{claim}
\end{proof}

\begin{proof}[Proof of Claim \ref{claim-uniform-boundedness-left-right-spin-fields}]
    We first use the Fubini-Tonelli theorem and write
    \begin{align*}
        &\EE \otimes \mu_{\delta}^{\pm} \big[ (\theta_{R}(\delta)\theta_{L}(\delta)S_{R}^{h}(\delta)S_{L}^{h}(\delta))^{(1+\alpha)p}\big] \\
        &= \mu_{\delta}^{\pm}\big[ \theta_{R}(\delta)^{(1+\alpha)p} \theta_{L}(\theta)^{(1+\alpha)p}\EE\big[ S^{h}_R(\delta)^{(1+\alpha)p}S^{h}_L(\delta)^{(1+\alpha)p} \big]\big] 
    \end{align*}
    We now observe that the random variables $S^{h}_R(\delta)$ and $S^{h}_L(\delta)$ are $\PP$-independent as they are measurable with respect to the random variables $(h_x)_x$ in two disjoint domains. This implies that to control the above inner expectation, we must control $\EE[S^{h}_R(\delta)^{(1+\alpha)p}]\EE[S^{h}_L(\delta)^{(1+\alpha)p}]$. For this, we use the same expansion as in the proof of \cite[Lemma~B.1]{continuum-2d-RFIM}. This yields that $\mu_{\delta}^{\pm}$-almost surely,
    \begin{align}
        &\theta_{R}(\delta)S_{R}^{h}(\delta) = \theta_{R}(\delta)\big(\prod_{x \in D_{R,\delta}} \cosh(\delta^{\frac{7}{8}}\lambda(x)h_x)\big) \sum_{I \subset D_{R,\delta}} \mu_{\delta,D_{R}}^{-}\big[ \prod_{x \in I} \sigma_x\big] \prod_{x\in I}\vartheta_{\delta}(x)\eta_{x} \label{S-R-chaos-expansion}\\
        &\theta_{L}(\delta)S_{L}^{h}(\delta) = \theta_{L}(\delta)\big(\prod_{x \in D_{L,\delta}} \cosh(\delta^{\frac{7}{8}}\lambda(x)h_x)\big) \sum_{I \subset D_{L,\delta}} \mu_{\delta,D_{L}}^{+}\big[ \prod_{x \in I} \sigma_x\big] \prod_{x\in I}\vartheta_{\delta}(x)\eta_{x} \label{S-L-chaos-expansion}
    \end{align}
    where the sums are over subsets of $D_{R,\delta}$ and $D_{L,\delta}$ and where we have set
    \begin{align} \label{def-sigma-mu-eta}
        &\vartheta_{\delta}(x)^2 := \Var[\tanh(\delta^{\frac{7}{8}}\lambda(x) h_{x})] = \delta^{\frac{7}{4}}\lambda(x)^2 + O(\delta^{\frac{7}{2}}), \nonumber \\
        &\eta_x := \vartheta_{\delta}(x)^{-1}\tanh(\delta^{\frac{7}{8}}\lambda(x)h_{x}).
    \end{align} 
    Note in particular that the random variables $(\eta_x)_x$ are $\PP$-independent, have mean $0$ and variance $1$. Setting
    \begin{align} \label{def-Psi-R-Psi-L}
        &\Psi_{\delta,L}(\eta):= \sum_{I \subset D_{L,\delta}} \mu_{\delta,D_{L}}^{+}\big[ \prod_{x \in I} \sigma_{x} \big] \prod_{x \in I} \vartheta_{\delta}(x)\eta_x \\
        &\Psi_{\delta,R}(\eta):= \sum_{J \subset D_{R,\delta}} \mu_{\delta,D_{R}}^{-} \big [ \prod_{x \in J} \sigma_{x} ] \prod_{x \in J} \vartheta_{\delta}(x)\eta_x, \nonumber
    \end{align}
    and choosing $r, \tilde r>1$ such that $(1+\alpha)rp = 2$ and $ r^{-1} + \tilde r^{-1} = 1$, we obtain by H\"older's inequality that
    \begin{align*}
        &\mu_{\delta}^{\pm} \otimes \EE \big[ \big( \theta_{R}(\delta)\theta_{L}(\delta) S_{R}^{h}(\delta) S_{L}^{h}(\delta) \big)^{(1+\alpha)p}\big] \\
        &\leq \mu_{\delta}^{\pm}[\EE[\Psi_{\delta,L}(\eta)^2]\EE[\Psi_{\delta,R}(\eta)^2]]^{\frac{1}{r}}
        \mu_{\delta}^{\pm}\otimes \EE\big[\big(\theta_{R}(\delta)\theta_{L}(\delta)\prod_{x \in D_{L,\delta} \cup D_{R,\delta}}\cosh(\delta^{\frac{7}{8}}\lambda(x)h_x) \big)^{(1+\alpha)p\tilde r}\big]^{\frac{1}{\tilde r}}.
    \end{align*}
    The second expectation on the above right-hand side is controlled thanks to the following claim, whose proof is postponed to the end.

    \begin{claim} \label{claim-Lq-norm-renomalization-factor}
        As $\delta \to 0$, $\theta_{R}(\delta)\prod_{x \in D_{R,\delta}}\cosh(\delta^{\frac{7}{8}}\lambda(x)h_x)$ converges to $1$ in $\PP \otimes \mu_{\delta}^{\pm}$-probability. Moreover, this random variable is uniformly bounded in $L^q$, for any $q \geq 1$. A similar statement holds with $R$ replaced by $L$.
    \end{claim}
    
    With Claim \ref{claim-Lq-norm-renomalization-factor} in hands, it now remains control $\mu_{\delta}^{\pm}[\EE[\Psi_{\delta,L}(\eta)^2]\EE[\Psi_{\delta,R}(\eta)^2]]$. We write
    \begin{align} \label{dvt-variance-expectation-psi}
        \mu_{\delta}^{\pm}[\EE[\Psi_{\delta,L}(\eta)^2]\EE[\Psi_{\delta,R}(\eta)^2]] = &\mu_{\delta}^{\pm}\big[\EE[\Psi_{\delta,L}(\eta)]^2\EE[\Psi_{\delta,R}(\eta)]^2\big] + \mu_{\delta}^{\pm}\big[\EE[\Psi_{\delta,L}(\eta)]^2\Var(\Psi_{\delta,R}(\eta))] \nonumber \\
        &+ \mu_{\delta}^{\pm}\big[\Var(\Psi_{\delta,L}(\eta))\EE[\Psi_{\delta,R}(\eta)]^2\big] \nonumber \\
        &+ \mu_{\delta}^{\pm}\big[\Var(\Psi_{\delta,L}(\eta))\Var(\Psi_{\delta,R}(\eta))\big].
    \end{align}
    Let us start by considering the first term on the right-hand side of \eqref{dvt-variance-expectation-psi}. It is easy to see that by $\PP$-independence of the $(h_x)_x$ (note that in the high-temperature expansion of the Ising model, points of $D_{L,\delta}$ in the sets $I \subset D_{L,\delta}$ cannot appear more than once in $I$ and $I$ is allowed to be empty), we have that $\mu_{\delta}^{\pm}$-almost surely,
    \begin{align} \label{expectation-Psi-L}
        \EE[\Psi_{\delta,L}(\eta)] = \sum_{I \subset D_{L,\delta}} \mu_{\delta,D_{L}}^{+}[\prod_{x \in I}\sigma_x] \EE[\prod_{x \in I}\vartheta_{\delta}(x)\eta_x]=1.
    \end{align}
    This yields that for any $\delta > 0$,
    \begin{equation} \label{prod-of-expectations-is-1}
        \mu_{\delta}^{\pm}\big[\EE[\Psi_{\delta,L}(\eta)]^2\EE[\Psi_{\delta,R}(\eta)]^2\big] =1.
    \end{equation}
    Let us now bound the product of variances in \eqref{dvt-variance-expectation-psi}. Using the same expansion and bounds as in the proof of \cite[Lemma~B.1]{continuum-2d-RFIM}, we obtain that for any $\eps >0$,
    \begin{align*}
        &\mu_{\delta}^{\pm} \big[ \Var(\Psi_{\delta,L}(\eta))\Var(\Psi_{\delta,R}(\eta))\big]  \\
        &\leq \mu_{\delta}^{\pm}\bigg[
        \big(\sum_{I \subset D_{R,\delta}} (1+\eps)^{\vert I \vert}[\prod_{x \in I}\vartheta_{\delta}(x)^2]\mu_{\delta,D_{R}}^{-}\big[\prod_{x \in I}\sigma_x \big]^2\big)\big(\sum_{J \subset D_{L,\delta}} (1+\eps)^{\vert J \vert}[\prod_{x \in J}\vartheta_{\delta}(x)^2]\mu_{\delta,D_{L}}^{+}\big[ \prod_{x \in J}\sigma_x \big]^2\big) \bigg].
    \end{align*}
    Let us denote by $\mathcal{E}_{\delta}$ the $\mu_{\delta}^{\pm}$-expectation on the right-hand side above. To upper bound $\mathcal{E}_{\delta}$, we can expand the sums, use Fubini-Tonelli theorem and the boundedness of $\lambda$ to obtain that
    \begin{equation*}
        \mathcal{E}_{\delta} \leq \sum_{k,p \geq 0} \frac{(1+\eps)^{k+p}M_{\lambda}^{k+p}\vartheta_{\delta}^{2(k+p)}}{p!k!} \mu_{\delta}^{\pm}\bigg[ \big(\sum_{x_1,\dots,x_p \in D_{R,\delta}}\mu_{\delta, D_{R}}^{-}[\prod_{n=1}^{p}\sigma_{x_{n}}]^2\big) \big(\sum_{x_1,\dots,x_k \in D_{L,\delta}}\mu_{\delta,D_{L}}^{+}[\prod_{j=1}^{k} \sigma_{x_{j}}]^2\big) \bigg]
    \end{equation*}
    where $\vartheta_{\delta}^{2}=\delta^{\frac{7}{4}}$ and $M_{\lambda}$ is such that for any $x \in D$, $\lambda(x)^2 \leq M_{\lambda}$. We now must bound the $\mu_{\delta}^{\pm}$-expectations above. These bounds are provided by the following claim, whose proof is postponed to the end.

    \begin{claim} \label{claim-bound-jk-expectations}
        There exists a constant $C>0$ such that for any $k,p \geq 0$ and any $\delta > 0$,
        \begin{equation*}
            \vartheta_{\delta}^{2k+2p} \mu_{\delta}^{\pm}\bigg[ \big(\sum_{x_1,\dots,x_p \in D_{R,\delta}}\mu_{\delta, D_R}^{-}[\prod_{n=1}^{p}\sigma_{x_{n}}]^2\big) \big(\sum_{x_1,\dots,x_k \in D_{L,\delta}}\mu_{\delta,D_{L}}^{+}[\prod_{n=1}^{k}\sigma_{x_{n}}]^2\big) \bigg]  \leq C^{k+p} k^{\frac{k}{2}}p^{\frac{p}{2}}.
        \end{equation*}
    \end{claim}
    Claim \ref{claim-bound-jk-expectations} implies that there exists $K_2>0$ such that for any $\delta > 0$,
    \begin{equation*}
        \mu_{\delta}^{\pm}\big[\Var(\Psi_{\delta,L}(\eta))\Var(\Psi_{\delta,R}(\eta))\big] \leq K_2.
    \end{equation*}
    In view of this upper bound and of \eqref{dvt-variance-expectation-psi} and \eqref{prod-of-expectations-is-1}, to conclude the proof of Claim \ref{claim-uniform-boundedness-left-right-spin-fields}, it remains to upper bound uniformly in $\delta >0$ the terms 
    \begin{equation*}
        \mu_{\delta}^{\pm}[\EE[\Psi_{\delta,L}(\eta)]^2\Var(\Psi_{\delta,R}(\eta))] \quad \operatorname{and} \quad \mu_{\delta}^{\pm}[\EE[\Psi_{\delta,R}(\eta)]^2\Var(\Psi_{\delta,L}(\eta))].
    \end{equation*}
    The equality \eqref{expectation-Psi-L} yields that
    \begin{equation*}
        \mu_{\delta}^{\pm}[\EE[\Psi_{\delta,L}(\eta)]^2\Var(\Psi_{\delta,R}(\eta))] \leq \mu_{\delta}^{\pm}[\Var(\Psi_{\delta,L}(\eta))].
    \end{equation*}
    The expectation $\mu_{\delta}^{\pm}[\Var(\Psi_{\delta,L}(\eta))]$ can be bounded uniformly in $\delta$ using similar arguments as those used to bound $\mu_{\delta}^{\pm}[\Var(\Psi_{\delta,L}(\eta))\Var(\Psi_{\delta,R}(\eta))]$. We leave the details to the reader. The term $\mu_{\delta}^{\pm}[\EE[\Psi_{\delta,R}(\eta)]^2\Var(\Psi_{\delta,L}(\eta))]$ is treated similarly.

    Putting everything together, we have obtained that there exists $K>0$ such that for any $\delta > 0$
    \begin{equation*}
        \EE \otimes \mu_{\delta}^{\pm} \big[ (\theta_{R}(\delta)\theta_{L}(\delta)S_{R}^{h}(\delta)S_{L}^{h}(\delta))^{(1+\alpha)p}\big] \leq K,    \end{equation*}
    which completes the proof of Claim \ref{claim-uniform-boundedness-left-right-spin-fields}.
\end{proof}

We now prove Claim \ref{claim-Lq-norm-renomalization-factor} and Claim \ref{claim-bound-jk-expectations}.

\begin{proof}[Proof of Claim \ref{claim-Lq-norm-renomalization-factor}]
    We establish the claim for $N_{R}(\delta):=\theta_{R}(\delta)\prod_{x \in D_{R,\delta}}\cosh(\delta^{\frac{7}{8}}\lambda(x)h_x)$ since the proof for $N_{L}(\delta)$ is the same, where we have defined $N_{L}(\delta)$ as $N_{R}(\delta)$ but with $R$ replaced by $L$.
    
    Let first show that $N_{R}(\delta)$ converges to $1$ in $\PP \otimes \mu_{\delta}^{\pm}$-probability as $\delta \to 0$. Observe that for any $x \in D_{\delta}$,
    \begin{equation*}
        \EE[\log \cosh(\delta^{\frac{7}{8}}\lambda(x)h_x)] = \frac{1}{2}\delta^{\frac{7}{4}}\lambda(x)^2 +O(\delta^{\frac{7}{2}}).
    \end{equation*}
    Therefore, we have that $\mu_{\delta}^{\pm}$-almost surely,
    \begin{align*}
        &\sum_{x \in D_{R,\delta}} \EE[\log \cosh(\delta^{\frac{7}{8}}\lambda(x)h_x)] \nonumber \\
        &= \sum_{x \in D_{R,\delta}} (\frac{\delta^{\frac{7}{4}}}{2}\lambda(x)^2+O(\delta^{\frac{7}{2}})) = -\log \theta_{R}(\delta) + \sum_{x \in D_{R,\delta}} O(\delta^{\frac{7}{2}}).
    \end{align*}
    This yields that $\PP\otimes \mu_{\delta}^{\pm}$-almost surely,
    \begin{align}
        N_{R}(\delta) = \exp\bigg( \sum_{x \in D_{R,\delta}} \big[\log \cosh(\delta^{\frac{7}{8}}\lambda(x)h_x) - \EE[\log \cosh (\delta^{\frac{7}{8}}\lambda(x)h_x)] + O(\delta^{\frac{7}{2}})\big] \bigg).
    \end{align}
    Setting $x \in D_{\delta}$, $W_{x}^{\delta}:= \log \cosh(\delta^{\frac{7}{8}}\lambda(x)h_x) -\EE[\log \cosh(\delta^{\frac{7}{8}}\lambda(x)h_x)]$, we thus obtain that $\PP \otimes \mu_{\delta}^{\pm}$-almost surely,
    \begin{equation*}
        \exp\big( \sum_{x \in D_{R,\delta}} W_{x}^{\delta} - O(\delta^{\frac{3}{2}}\vol(D)) \big) \leq N_{R}(\delta) \leq \exp\big( \sum_{x \in D_{R,\delta}} W_{x}^{\delta} + O(\delta^{\frac{3}{2}}\vol(D) \big).
    \end{equation*}
    Moreover, by Chebychev's inequality and $\PP$-independence of the $(W_{x}^{\delta})_{x \in D_{\delta}}$,
    \begin{align*}
        \PP \otimes \mu_{\delta}^{\pm}\big[ \vert \sum_{x \in D_{R,\delta}} W_{x}^{\delta} \vert \geq \eps \big] &\leq \eps^{-2} \EE \otimes \mu_{\delta}^{\pm}\big[\big(\sum_{x \in D_{R,\delta}} W_{x}^{\delta} \big)^2\big]\\
        &= \eps^{-2}\mu_{\delta}^{\pm}\big[ \sum_{x \in D_{R,\delta}} \Var(W_{x}^{\delta})\big] \\
        &\leq \eps^{-2}\mu_{\delta}^{\pm}\big[\sum_{x \in D_{R,\delta}} O(\delta^{\frac{7}{2}})\big] \leq \eps^{-2}O(\vol(D)\delta^{\frac{3}{2}}).
    \end{align*}
    This shows that $\sum_{x \in D_{R,\delta}} W_{x}^{\delta}$ converges to $0$ as $\delta \to 0$ in $\PP \otimes \mu_{\delta}^{\pm}$-probability, which implies that $N_{R}(\delta)$ converges to $1$ in $\PP \otimes \mu_{\delta}^{\pm}$-probability as $\delta \to 0$.

    Finally, to prove that $N_{\delta}(R)$ is uniformly bounded in $L^q$ for any $q \geq 1$, we can argue as in the last part of the proof of \cite[Lemma~3.5]{continuum-2d-RFIM}. This completes the proof of Claim \ref{claim-Lq-norm-renomalization-factor}. 
\end{proof}

\begin{proof}[Proof of Claim \ref{claim-bound-jk-expectations}]
    To prove the claim, we are going to use Lemma \ref{lemma-bound-spin-correlations} to upper bound $\mu_{L,\delta}^{+}[\prod_{j=1}^{k} \sigma_{x_{j}}]$ and $\mu_{R,\delta}^{-}[\prod_{j=1}^{k} \sigma_{x_{j}}]$. However, one can see that the upper bound on these quantities provided by Lemma \ref{lemma-bound-spin-correlations} deteriorates as $\min_{i \neq j} \vert x_i - x_j \vert$ gets smaller.  As this produces a divergence that must be treated carefully, let us introduce the following sets. Let $k \in \mathbb{N}$. For $\delta > 0$, we define $B_{k}(\delta):= \{(x_1,\dots,x_k) \in D_{\delta}^{k}: \exists i \neq j, \vert x_i - x_j \vert \leq 4\delta \}$ and set
    \begin{equation} \label{def-good-subsets}
        \tilde D_{\delta}^{k} := D_{\delta}^k \setminus B_{k}(\delta), \quad \tilde D_{Q,\delta}^k := D_{Q,\delta}^{k} \setminus B_k(\delta),
    \end{equation}
    for $Q=L,R$. $\tilde{D}_{\delta}^{k}$, $\tilde D_{Q,\delta}^k$, $Q=L,R$, are the ``good" sets on which points are sufficiently far from each other. Let us also define the following random variables
    \begin{align*}
        &F_{L}(k,\delta) = \vartheta_{\delta}^{2k}\sum_{(x_1,\dots,x_k) \in \tilde D_{L,\delta}^k} \mu_{L,\delta}^{+} \big[\prod_{j=1}^{k} \sigma_{x_{j}} \big]^2 \quad \operatorname{and} \\
        &B_{L}(k,\delta) = \vartheta_{\delta}^{2k}\sum_{(x_1,\dots,x_k) \in B_k(\delta) \cap D_{L,\delta}^k} \mu_{L,\delta}^{+} \big[\prod_{j=1}^{k} \sigma_{x_{j}} \big]^2.
    \end{align*}
    $F_{R}(k,\delta)$ and $B_{R}(k,\delta)$ are defined similarly but with $L$ replaced by $R$ and $+$ replaced by $-$. The first step of the proof of Claim \ref{claim-bound-jk-expectations} consists in showing that there exists $C>0$ such that for any $k, p \in \mathbb{N}$ and any $\delta >0$,
    \begin{equation} \label{upper-bound-F-L-F-R}
        \mu_{\delta}^{\pm}[F_{L}(k,\delta)F_{R}(p,\delta)] \leq C^{k+p}k^{\frac{k}{2}}p^{\frac{p}{2}}.
    \end{equation}
    We will then deal with the contributions arising from $B_{L}(k,\delta)$ and $B_{R}(p,\delta)$ by proving that there exists $\tilde C > 0$ such that for any $k,p \in \mathbb{N}$ and any $\delta > 0$,
    \begin{align} \label{bounds-B-R-B-L-F-R-F-L}
        &\mu_{\delta}^{\pm}[F_{L}(k,\delta)B_{R}(p,\delta)] \leq \tilde C^{k+p}\delta^{\frac{7}{4}} k^{\frac{k}{2}}p^{\frac{p}{2}}, \quad \mu_{\delta}^{\pm}[F_{R}(p,\delta)B_{L}(k,\delta)] \leq \tilde C^{k+p}\delta^{\frac{7}{4}} k^{\frac{k}{2}}p^{\frac{p}{2}}, \nonumber \\
        &\mu_{\delta}^{\pm}[B_{R}(p,\delta)B_{L}(k,\delta)] \leq \tilde C^{k+p}\delta^{\frac{7}{4}} k^{\frac{k}{2}}p^{\frac{p}{2}}.
    \end{align}
    Claim \ref{claim-bound-jk-expectations} will then follow from combining \eqref{upper-bound-F-L-F-R} and \eqref{bounds-B-R-B-L-F-R-F-L}.
    
    Let us start by showing \eqref{upper-bound-F-L-F-R}. The first step is to use Lemma \ref{lemma-bound-spin-correlations} together with Cauchy-Schwarz inequality with respect  to $\mu_{\delta}^{\pm}$ to obtain that
    \begin{equation*}
        \mu_{\delta}^{\pm}[F_{L}(k,\delta)F_{R}(p,\delta)] \leq D_{L}(k,\delta)D_{R}(p,\delta)
    \end{equation*}
    where we have set
    \begin{align*}
        D_{L}(k,\delta) = \delta^{2k}\sum_{(x_1,\dots,x_k) \in \tilde D_{\delta}^k} \big[\prod_{j=1}^{k} L(x_j)\big]^{-\frac{1}{4}} \mu_{\delta}^{\pm}\big[\mathbb{I}_{(x_1,\dots,x_k) \in D_{L,\delta}^k} \big[\prod_{j=1}^{k}d_{\delta,\infty}(x_j)\big]^{-\frac{1}{2}} \big]^{\frac{1}{2}}
    \end{align*}
    and where $D_{R}(p,\delta)$ is defined analogously but with $L$ replaced by $R$. Above, for $Q \in \{L,R\}$ and $x \in D_{Q,\delta}$, we have defined $d_{\delta,\infty}(x_j):= \dist(x,\partial D_{Q,\delta})$, where $\dist$ is the Euclidean distance. The upper bound \eqref{upper-bound-F-L-F-R} will follow immediately once we have shown that there exists $C>0$ such that for any $k, p \in \mathbb{N}$ and any $\delta >0$,
    \begin{equation} \label{upper-bound-D-L-D-R}
        D_{L}(k,\delta) \leq C^kk^{\frac{k}{2}}, \quad \operatorname{and} \quad D_{R}(p,\delta) \leq C^pp^{\frac{p}{2}}.
    \end{equation}
    Since the arguments are the same for both $D_{L}(k,\delta)$ and $D_{R}(p,\delta)$, we are going to show \eqref{upper-bound-D-L-D-R} for $D_{L}(k,\delta)$ only (\eqref{upper-bound-D-L-D-R} for $D_{R}(p,\delta)$ actually follows from symmetry). To upper bound $D_{L}(k,\delta)$, we introduce, for $0 \leq r \leq k$, the sets
    \begin{equation*}
        \mathcal{L}(r,k) := \{(x_1,\cdots,x_k) \in D_{\delta}^k: \forall j \leq r, L(x_j) \leq d_{\delta}(x_j), \quad \forall j \geq r+1, L(x_j) > d_{\delta}(x_j) \}
    \end{equation*}
    where for $x \in D_{\delta}$, $d_{\delta}(x):= \dist(x,\partial D_{\delta})$, $\dist$ being the Euclidean distance. Note that if $(x_1,\dots,x_k)$ belongs to $\mathcal{L}(r,k)$, then in the upper bound of Lemma \ref{lemma-bound-spin-correlations} on $k$-spin correlations, we can replace $L(x_j) \wedge d_{\delta}(x_j)$ by $L(x_j)$ if $j \leq r$ and by $d_{\delta}(x_j)$ if $j \geq r+1$. This is why introducing the sets $\mathcal{L}(r,k)$ will turn out to be useful. Using these sets, we can write that
    \begin{align} \label{decomposition-D-L-k-delta}
        D_{L}(k,\delta) &= \sum_{r=0}^{k} \binom{k}{r} \delta^{2k}\sum_{(x_1,\dots,x_k) \in \tilde D_{\delta}^k} \mathbb{I}_{\mathcal{L}(r,k)} \big[\prod_{j=1}^{k} L(x_j)\big]^{-\frac{1}{4}} \mu_{\delta}^{\pm}\big[\mathbb{I}_{(x_1,\cdots,x_k)\in D_{L,\delta}^k} \big[\prod_{j=1}^{k}d_{\delta,\infty}(x_j)\big]^{-\frac{1}{2}} \big]^{\frac{1}{2}}\nonumber \\
        &=: \sum_{r=0}^{k} \binom{k}{r} I(r,k;\delta)
    \end{align}
    where we recall the definition \eqref{def-good-subsets} of $\tilde D_{\delta}^{k}$. Above, the binomial coefficient $\binom{k}{r}$ accounts for the number of ways to choose the coordinates of the points for which $L(x_j) \leq d_{\delta}(x_j)$. Let us start with the term $r=0$ in \eqref{decomposition-D-L-k-delta}. For $(x_1,\dots,x_k) \in \tilde D_{\delta}^{k} \cap \mathcal{L}(0,k)$, we have that
    \begin{align*}
        &\mu_{\delta}^{\pm}\big[\mathbb{I}_{(x_1,\dots,x_k) \in D_{L,\delta}^k} \big[\prod_{j=1}^{k}d_{\delta,\infty}(x_j)\big]^{-\frac{1}{2}} \big]\\
        &\leq \sum_{l_1=0}^{\log(\delta^{-1}d_{\delta}(x_1))} \dots \sum_{l_k=0}^{\log(\delta^{-1}d_{\delta}(x_k))} \prod_{j=1}^{k} (e^{l_j}\delta)^{-\frac{1}{2}} \mu_{\delta}^{\pm} \bigg[ \bigcap_{j=1}^{k} \{ x_j \in N_{l_{j}}(\gamma_{\delta}) \}\bigg]
    \end{align*}
    where for $l \in \mathbb{N}$, we have set
    \begin{equation*}
        N_{l_{j}}(\gamma_{\delta}) := \{ x \in D_{\delta}: e^{l}\delta \leq \dist(x,\gamma_{\delta})  \leq e^{l+1}\delta \}.
    \end{equation*}
    By Lemma \ref{lemma-k-pts-estimate-interface-discrete}, we obtain that there exist $C, \delta_0 > 0$ such that for any $0 < \delta < \delta_0$ and any $k \in \mathbb{N}$,
    \begin{align*}
        &\mu_{\delta}^{\pm}\big[\mathbb{I}_{(x_1,\dots,x_k) \in D_{L,\delta}^k} \big[\prod_{j=1}^{k}d_{\delta,\infty}(x_j)\big]^{-\frac{1}{2}} \big]\\
        &\leq \sum_{l_1=0}^{\log(\delta^{-1}d_{\delta}(x_1))} \dots \sum_{l_k=0}^{\log(\delta^{-1}d_{\delta}(x_k))} \prod_{j=1}^{k} (e^{l_j}\delta)^{-\frac{1}{2}} \bigg(\prod_{j=1}^{k}\frac{e^{l+1}\delta}{d_{\delta}(x_j)}\bigg)^{\frac{5}{8}-\rho}\\
        &\leq C^k \prod_{j=1}^{k} d_{\delta}(x_j)^{-\frac{1}{2}}.
    \end{align*}
    Plugging this upper bound in the definition of $I(0,k;\delta)$ and using Lemma \ref{lemma-sum-prod-L-x-j-delta} and Lemma \ref{lemma-sum-dist-to-boundary-power-alpha} to upper bound the resulting sum, we find that for some $\tilde C>0$,
    \begin{equation} \label{upper-bound-I-0-k-delta}
        I(0,k;\delta) \leq \tilde C^k k^{\frac{k}{4}}.
    \end{equation}
    Let us now turn to the terms with $1 \leq r \leq k$ in \eqref{decomposition-D-L-k-delta}. For fixed $(x_{1},\dots,x_{k}) \in \tilde D_{\delta}^{k} \cap \mathcal{L}(r,k)$, we define the event, for $0 \leq m \leq r$,
    \begin{equation*}
        \mathcal{G}(r,m) := \{\forall 1 \leq j \leq m, d_{\delta,\infty}(x_{j}) > L(x_j), \quad \forall m+1 \leq j \leq r, d_{\delta,\infty}(x_j) \leq L(x_j) \}.
    \end{equation*}
    The events $(\mathcal{G}(r,m))_m$ will be useful to estimate the $\mu_{\delta}^{\pm}$-expectation of $\prod_j d_{\delta,\infty}(x_j)$ appearing in \eqref{decomposition-D-L-k-delta}. Indeed, we will use Lemma \ref{lemma-k-pts-estimate-interface-discrete} to upper bound this expectation and introducing the events $(\mathcal{G}(r,m))_m$ will allow us to upper bound the random variable $d_{\delta,\infty}(x)^{-1/2}$ by $L(x_j)^{-1/2}$ for points that are close to each other and for which the estimate provided by Lemma \ref{lemma-k-pts-estimate-interface-discrete} is thus not good. We also let $M(x_j)=L(x_j)$ if $m+1 \leq j \leq r$ and $M(x_j) = d_{\delta}(x_j)$ if $j \geq r+1$. For each $1 \leq r \leq k$, we then have that
    \begin{align} \label{decomposition-I-r-k-delta}
        I(r,k;\delta) = \sum_{m=0}^{r} \binom{r}{m} \delta^{2k}\sum_{(x_1,\dots,x_k) \in \tilde D_{\delta}^k} &\mathbb{I}_{\mathcal{L}(r,k)} \big[\prod_{j=1}^{k} L(x_j)\big]^{-\frac{1}{4}} \nonumber \\
        &\mu_{\delta}^{\pm}\big[\mathbb{I}_{(x_1,\dots,x_k) \in D_{L,\delta}^k} \mathbb{I}_{\mathcal{G}(r,m)}\big[\prod_{j=1}^{k}d_{\delta,\infty}(x_j)\big]^{-\frac{1}{2}} \big]^{\frac{1}{2}}.
    \end{align}
    Moreover, for each $1 \leq r \leq k$, $0 \leq m \leq r$ and $(x_1,\dots,x_k) \in \tilde D_{\delta}^k \cap \mathcal{L}(r,k)$, we have that, using Lemma \ref{lemma-k-pts-estimate-interface-discrete} and a reasoning similar to the one below \eqref{decomposition-D-L-k-delta},
    \begin{align*}
         &\mu_{\delta}^{\pm}\big[\mathbb{I}_{(x_1,\dots,x_k) \in D_{L,\delta}^k} \mathbb{I}_{\mathcal{G}(r,m)}\big[\prod_{j=1}^{k}d_{\delta,\infty}(x_j)\big]^{-\frac{1}{2}}\big] \\
         &\leq \big[\prod_{j=1}^{m} L(x_j)\big]^{-\frac{1}{2}}  \mu_{\delta}^{\pm}\big[\mathbb{I}_{(x_1,\dots,x_k) \in D_{L,\delta}^k} \mathbb{I}_{\mathcal{G}(r,m)}\big[\prod_{j=m+1}^{k}d_{\delta,\infty}(x_j)\big]^{-\frac{1}{2}}\big] \\
         &\leq  \big[\prod_{j=1}^{m} L(x_j)\big]^{-\frac{1}{2}} \sum_{l_{m+1}=0}^{\log(\delta^{-1}M(x_{m+1}))} \dots \sum_{l_{k}=0}^{\log(\delta^{-1}M(x_k))} \begin{aligned}[t]&\big[\prod_{j=1}^{k} e^ l\delta]^{-\frac{1}{2}} \bigg(\prod_{j=m+1}^{r}\frac{e^{(l_j+1)}\delta}{L(x_j)}\bigg)^{\frac{5}{8}-\rho}\\
         &\times \bigg(\prod_{j=r+1}^{k}\frac{e^{(l_j+1)}\delta}{d_{\delta}(x_j)}\bigg)^{\frac{5}{8}-\rho}\end{aligned}\\
         &\leq C^{k-m}\big[\prod_{j=1}^{r}L(x_j)\big]^{-\frac{1}{2}}\big[\prod_{j=r+1}^{k}d_{\delta}(x_j) \big]^{-\frac{1}{2}}.
    \end{align*}
    Using this to upper bound the right-hand side of \eqref{decomposition-I-r-k-delta}, we obtain that for each $1 \leq r \leq k$,
    \begin{align*}
        I(r,k;\delta) &\leq \sum_{m=0}^{r} \binom{r}{m} \delta^{2k}\sum_{(x_1,\dots,x_k) \in \tilde D_{\delta}^k} \mathbb{I}_{\mathcal{L}(r,k)} [\prod_{j=1}^{k} L(x_j)]^{-\frac{1}{4}} [\prod_{j=1}^{r} L(x_j)]^{-\frac{1}{4}} [\prod_{j=r+1}^{k} d_{\delta}(x_j)]^{-\frac{1}{4}}\\
        &\leq C^k\sum_{m=0}^{r} \binom{r}{m} \delta^{2k}\sum_{(x_1,\dots,x_k) \in \tilde D_{\delta}^k} \mathbb{I}_{\mathcal{L}(r,k)} [\prod_{j=1}^{k} L(x_j)]^{-\frac{1}{2}}[\prod_{j=r+1}^{k} d_{\delta}(x_j)]^{-\frac{1}{4}}
    \end{align*}
    where the last inequality follows from rescaling $D_{\delta}$ by $\diam(D)$. To upper bound the sums appearing on the right-hand side above, one can use H\"older's inequality with $\alpha =\frac{3}{2}$ for the term $[\prod_{j=1}^{k} L(x_j)]^{-\frac{1}{2}}$ and $\tilde \alpha=3$ for the term $[\prod_{j=r+1} d_{\delta}(x_j)]^{-\frac{1}{4}}$ and then apply Lemma \ref{lemma-sum-prod-L-x-j-delta} and Lemma \ref{lemma-sum-dist-to-boundary-power-alpha}. We then obtain that for each $1 \leq r \leq k$,
    \begin{equation*}
        I(r,k;\delta) \leq \tilde C^k k^{\frac{k}{2}}.
    \end{equation*}
    Going back to the decomposition \eqref{decomposition-D-L-k-delta}, this upper bound together with the inequality \eqref{upper-bound-I-0-k-delta} for the $r=0$ term shows that the inequality \eqref{upper-bound-D-L-D-R} indeed holds, which, as explained above, implies the inequality \eqref{upper-bound-F-L-F-R}.

    To complete the proof of Claim \ref{claim-bound-jk-expectations}, it remains to prove the inequalities \eqref{bounds-B-R-B-L-F-R-F-L}. Let us give the details for the term $\mu_{\delta}^{\pm}[F_{L}(k,\delta)B_{R}(p,\delta)]$, as the upper bounds on the other terms follow from similar arguments. Let $k,p \in \mathbb{N}$. We have that
    \begin{align*}
        &\mu_{\delta}^{\pm}\big[F_{L}(k,\delta)B_{R}(p,\delta)\big]\\
        &= \vartheta_{\delta}^{2(k+p)} \sum_{\substack{(x_{p+1},\dots,x_{p+k}) \in \tilde D_{\delta}^{k} \\ (x_1,\dots,x_p) \in B_{p}(\delta)}} \mu_{\delta}\big[ \mathbb{I}_{\substack{(x_1,\dots,x_p) \in D_{R,\delta}\\ (x_{p+1},\dots,x_{k+p}) \in D_{L,\delta}}} \mu_{L,\delta}^{+}[\prod_{j=p+1}^{k+p}\sigma_{x_{j}}]^2 \mu_{R,\delta}^{-}[\prod_{j=1}^{p}\sigma_{x_{j}}]^2 \big].
    \end{align*}
    For $2 \leq r \leq p$, we define the following subset of $B_{p}(\delta)$:
    \begin{equation*}
        \mathcal{C}(r,p) := \{(x_1,\dots,x_p) \in D_{\delta}^p: \forall 1 \leq j \leq r, L(x_j) \leq 4\delta,\quad \forall j \geq r+1, L(x_j) > 4\delta \}.
    \end{equation*}
    The sets $\mathcal{C}(r,p)$ will be useful to isolate the ``bad", i.e., potentially diverging, contributions of points close to each other in the upper bound on $k$-point spin correlations provided by Lemma \ref{lemma-bound-spin-correlations}. We then have that
    \begin{align} \label{decomposition-F-L-B-R}
        \mu_{\delta}^{\pm}\big[F_{L}(k,\delta)B_{R}(p,\delta)\big]
        =\sum_{r=2}^{p} \binom{p}{r} J(r,p,k;\delta)
    \end{align}
    where we have set
    \begin{align} \label{def-J-r-p-k}
        &J(r,p,k;\delta) \nonumber \\
        &:= \vartheta_{\delta}^{2(k+p)} \sum_{\substack{(x_{p+1},\dots,x_{p+k}) \in \tilde D_{\delta}^{k} \\ (x_1,\dots,x_p) \in B_{p}(\delta)}} \mathbb{I}_{(x_1,\dots,x_p)\in \mathcal{C}(r,p)} \mu_{\delta}^{\pm}\big[ \mathbb{I}_{\substack{(x_1,\dots,x_p) \in D_{R,\delta}\\ (x_{p+1},\dots,x_{k+p}) \in D_{L,\delta}}} \mu_{L,\delta}^{+}[\prod_{j=p+1}^{k+p}\sigma_{x_{j}}]^2 \mu_{R,\delta}^{-}[\prod_{j=1}^{p}\sigma_{x_{j}}]^2 \big].
    \end{align}
    For $(x_{1},\dots,x_{r}) \in \mathcal{C}(r,p)$, denote by $(x_{i_{1}},\dots, x_{i_{m}}) \subset (x_1,\dots,x_r)$ a choice of points such that for any $n \neq m$, $\vert x_{i_{n}} - x_{i_{m}} \vert > 4\delta$. We denote by $\mathcal{M}$ this set of coordinates. In the special case $r=2$, simply choose one of the two points that are distance less than $4\delta$ from each other. Note also that $m \leq r/2$. For $(x_1,\dots,x_{k+p}) \in \mathcal{C}(r,p) \times \tilde D_{\delta}^k$, Lemma \ref{lemma-bound-spin-correlations} yields that
    \begin{align}
        &\mu_{\delta}^{\pm}\big[ \mathbb{I}_{\substack{(x_1,\dots,x_p) \in D_{R,\delta} \nonumber \\ (x_{p+1},\dots,x_{k+p}) \in D_{L,\delta}}} \mu_{L,\delta}^{+}[\prod_{j=p+1}^{k+p}\sigma_{x_{j}}]^2 \mu_{R,\delta}^{-}[\prod_{j=1}^{p}\sigma_{x_{j}}]^2 \big] \nonumber \\
        &\leq \delta^{\frac{(k+p-r+m)}{4}}\big[\prod_{j=p+1}^{k+p} L(x_j;X[p+1,k+p])\big]^{-\frac{1}{4}} \nonumber \\
        &\times \big[\prod_{j \in \mathcal{M} \cup \{r+1,\dots, p \}} L(x_j;X[(r+1,p)\cup \mathcal{M}])\big]^{-\frac{1}{4}} \nonumber \\
        &\times \mu_{\delta}^{\pm}\big[\mathbb{I}_{\substack{(x_{r+1},\dots,x_p) \in D_{L,\delta}^k \\ (x_{r+1}, \dots, x_{p})}} \big[\prod_{j\in \mathcal{M} \cup \{r+1,\dots,k+p\}} d_{\delta,\infty}(x_j)\big]^{-\frac{1}{4}}\big]. \label{upper-bound-for-J-r-p}
    \end{align}
    Above, we have set $X[p+1,k+p]=(x_{p+1},\dots,x_{k+p})$ and $X[(r+1,p) \cup \mathcal{M}]$ is the set of points $(x_j)_j$ with index $j$ in the set $[r+1,p] \cup \mathcal{M}$. For a subset $P$ of indices in $[1,\dots,k+p]$, the quantity $L(x_j; X[P])$ is then defined as
    \begin{equation*}
        L(x_j; X[P]) := \frac{1}{4} \min_{p \in P, p \neq j} \vert x_p - x_j \vert.
    \end{equation*}
    Plugging the upper bound \eqref{upper-bound-for-J-r-p} in the definition \eqref{def-J-r-p-k} of $J(r,p,k;\delta)$ and using the same arguments as above to control the sum, we find that for each $2 \leq r \leq p$,
    \begin{align*}
        J(r,p,k;\delta) \leq \sum_{m=1}^{\lfloor \frac{r}{2} \rfloor} \binom{r}{m} \delta^{2m} \delta^{\frac{7(r-m)}{4}}C^{k+p-r+m} k^{\frac{k}{2}} (p-r+m)^{\frac{(p-r+m)}{2}}.
    \end{align*}
    Using the decomposition \eqref{decomposition-F-L-B-R}, this shows that
    \begin{equation*}
        \mu_{\delta}^{\pm}[F_{L}(k,\delta)B_{R}(p,\delta)] \leq \tilde C^{k+p}\delta^{\frac{7}{4}} k^{\frac{k}{2}}p^{\frac{p}{2}}.
    \end{equation*}
    As explained above, this inequality completes the proof of Claim \ref{claim-bound-jk-expectations}.
\end{proof}

\begin{proof}[Proof of Claim \ref{claim-uniform-boundedness-magnetization-partition-function}]
    We use H\"older's inequality twice to obtain that
    \begin{align*}
        &\EE \otimes \mu_{\delta}^{\pm} \bigg[ \exp\big(q(1+\alpha)M^{h}(\gamma_{\delta})\big)\big(\tilde{\mathcal{Z}}_{\delta}^{\pm, h}\big)^{-q(1+\alpha)}\bigg(\frac{\theta_{D}(\delta)}{\theta_L(\delta)\theta_{R}(\delta)}\bigg)^{q(1+\alpha)} \bigg]\\
        &\leq \EE \otimes \mu_{\delta}^{\pm} \bigg[ \exp\big(\alpha_1 M^{h}(\gamma_{\delta})\big)\bigg]^{\frac{1}{\alpha_1}} \EE \otimes \mu_{\delta}^{\pm} \bigg[ \big(\tilde{\mathcal{Z}}_{\delta}^{\pm, \lambda, h}\big)^{-\alpha_2}\bigg]^{\frac{1}{\alpha_2}} \EE \otimes \mu_{\delta}^{\pm} \bigg[ \bigg(\frac{\theta_{D}(\delta)}{\theta_L(\delta)\theta_{R}(\delta)}\bigg)^{\alpha_3}\bigg]^{\frac{1}{\alpha_3}} 
    \end{align*}
    for some $\alpha_1, \alpha_2, \alpha_3 > 0$. \cite[Corollary~C.2]{continuum-2d-RFIM} shows that the second expectation on the right-hand side above is uniformly bounded in $\delta$. On the other hand, Claim \ref{claim-cvg-magnetization-interface} below implies that the first expectation on the right-hand side above is uniformly bounded. To conclude the proof of the claim, it suffices to observe that $\mu_{\delta}^{\pm}$-almost surely,
    \begin{equation*}
        \bigg(\frac{\theta_{D}(\delta)}{\theta_L(\delta)\theta_{R}(\delta)}\bigg)^{\alpha_3} = \exp \big(-\frac{\alpha_3\delta^{\frac{7}{4}}}{2}\sum_{x \in V(\gamma_{\delta})} \lambda(x)^2 \big),
    \end{equation*}
    which implies that the $\PP \otimes \mu_{\delta}^{\pm}$-expectation of $(\theta_D(\delta)/\theta_{L}(\delta)\theta_{R}(\delta))^{\alpha_3}$ is bounded by $1$, uniformly in $\delta$. This completes the proof of Claim \ref{claim-uniform-boundedness-magnetization-partition-function}.
\end{proof}

\begin{claim} \label{claim-cvg-magnetization-interface}
    Let $q \in (-\infty,0) \cup (0,+\infty)$. Then, as $\delta \to 0$, $\EE \otimes \mu_{\delta}^{\pm} \big[ \exp(qM^{h}(\gamma_{\delta})) \big] \to 1$.
\end{claim}

\begin{proof}[Proof of Claim \ref{claim-cvg-magnetization-interface}]
    Observe first that by symmetry of the law of $(h_x)_{x}$ under $\PP$, we can assume that $q \in (0,\infty)$. Thus, let $q \in (0,\infty)$. By definition of $\gamma_{\delta}$, $\mu_{\delta}^{\pm}$-almost surely, $V_{L}(\gamma_{\delta}) \cap V_{R}(\gamma_{\delta}) = \emptyset$. Therefore, since the random variables $(h_{x})_{x}$ are $\PP$-independent, we have that
    \begin{equation*}
        \EE \otimes \mu_{\delta}^{\pm} \big[  \exp(qM^{h}(\gamma_{\delta})) \big] = \mu_{\delta}^{\pm}\big[ \exp(\frac{q^2}{2}\delta^{\frac{7}{4}}\sum_{x \in V(\gamma_{\delta})} \lambda(x)^2) \big]
    \end{equation*}
    where we used Fubini-Tonelli theorem and where $V(\gamma_{\delta}) = V_{L}(\gamma_{\delta}) \cup V_{R}(\gamma_{\delta})$ denotes the set of vertices that are neighboring $\gamma_{\delta}$. Letting $0 < M_{\lambda}<\infty$ be such that for any $x \in D$, $\lambda(x)^2 \leq M_{\lambda}$, we obtain the upper bound
    \begin{equation*}
        \EE \otimes \mu_{\delta}^{\pm} \big[  \exp(qM^{h}(\gamma_{\delta})) \big] \leq \mu_{\delta}^{\pm}\big[ \exp(\frac{q^2M_{\lambda}}{2}\delta^{\frac{7}{4}}\vert V(\gamma_{\delta}) \vert ) \big].
    \end{equation*}
    Expanding the exponential on the right-hand side above and using the Fubini-Tonelli theorem, we obtain that
    \begin{equation*}
        \EE \otimes \mu_{\delta}^{\pm} \big[  \exp(qM^{h}(\gamma_{\delta})) \big] \leq \sum_{k \geq 0} \frac{2^{-k}q^{2k}M_{\lambda}^{k}\delta^{\frac{7k}{4}}}{k!} \mu_{\delta}^{\pm}[\vert V(\gamma_{\delta}) \vert^k]. 
    \end{equation*}
    To show that the sum over $k\ge1$ on the right-hand side converges to $0$ as $\delta \to 0$, we are going to use the dominated convergence theorem. To do so, we will prove the following two claims.

    \begin{claim} \label{claim-cvg-to-0-k-moment-V-gamma}
        Let $k \in \mathbb{N}\setminus \{0\}$. Then, $\lim_{\delta \to 0} \delta^{\frac{7k}{4}}\mu_{\delta}^{\pm}[\vert V(\gamma_{\delta}) \vert^k] = 0$.
    \end{claim}

    \begin{claim} \label{claim-bound-EE-k-moment-V-gamma}
        Let $0<\rho \ll 1$. There exist $\delta_{0} > 0$ and a constant $C>0$ depending on $D$ only such that for any $k \geq 1$ and any $0 < \delta <\delta_{0}$, $\delta^{\frac{7k}{4}}\mu_{\delta}^{\pm}[\vert V(\gamma_{\delta}) \vert^k] \leq C^kk^{(\frac{5}{8}-\rho)k}$.
    \end{claim}

    The proof of these claims is given just below. Assuming them for a moment, observe that the sequence $(C^kk^{(\frac{5}{8}-\rho)k})_{k \geq 0}$ of upper bounds given by Claim \ref{claim-bound-EE-k-moment-V-gamma} is summable in $k$ once divided by ($k!$). Therefore, as explained above, Claim \ref{claim-cvg-magnetization-interface} now follows from the dominated convergence theorem.
\end{proof}

\begin{proof}[Proof of Claim \ref{claim-cvg-to-0-k-moment-V-gamma}]
    Let $k \in \mathbb{N} \setminus \{0\}$ and $0 < \rho \ll 1$. We define the following two subsets of $D_{\delta}^{k}$:
    \begin{align*}
        &B_{k}(4\delta):= \{(x_1,\dots,x_{k}) \in D_{\delta}^{k}: \exists i \neq j, \vert x_i - x_j \vert \leq 4\delta\}\\
        &C_{k}(\delta) := \{(x_1,\dots,x_k) \in D_{\delta}^{k}: \exists j, \dist(x_j,\partial D_{\delta}) \leq c\delta\}.
    \end{align*}
    where $c>1$. With these definitions, we have that
    \begin{align}
        \mu_{\delta}^{\pm}[\vert V(\gamma_{\delta}) \vert^k]
        &= \sum_{(x_1,\dots,x_k) \in D_{\delta}^k} \mu_{\delta}^{\pm}\bigg[ \bigcap_{j=1}^{k} \{ x_j \in V(\gamma_{\delta}) \}  \bigg] \nonumber \\
        &\leq \sum_{(x_1,\dots,x_k) \in D_{\delta}^k} \mu_{\delta}^{\pm}\bigg[ \bigcap_{j=1}^{k} \{ d(x_j, \gamma_{\delta}) \leq \delta \}  \bigg] \nonumber\\
        &= \sum_{(x_1,\dots,x_k) \in D_{\delta}^k \setminus (B_{k}(4\delta) \cup C_{k}(\delta))} \mu_{\delta}^{\pm}\bigg[ \bigcap_{j=1}^{k} \{ d(x_j, \gamma_{\delta}) \leq \delta \}  \bigg] \nonumber \\
        &+ \sum_{(x_1,\dots,x_k) \in B_{k}(4\delta) \cup C_{k}(\delta)} \mu_{\delta}^{\pm}\bigg[ \bigcap_{j=1}^{k} \{ d(x_j, \gamma_{\delta}) \leq \delta \}  \bigg] \nonumber\\
        &=: I_{1}(k,\delta) + I_{2}(k,\delta). \label{ineq-k-moment--I-1-I-2}
    \end{align}
    Let us first show that
    \begin{equation} \label{cvg-to-0-I-1}
        \lim_{\delta \to 0} \delta^{\frac{7k}{4}}I_1(k,\delta) = 0.
    \end{equation}
    To shorten notations, we set $G(k,\delta) :=  D_{\delta}^k \setminus (B_{k}(4\delta) \cup C_{k}(\delta))$. Using Lemma \ref{lemma-k-pts-estimate-interface-discrete} to upper bound the probability that $\gamma_{\delta}$ visits the $\delta$-neighborhood of $x_{1}, \dots, x_k$, we obtain that there exist $C, \delta_0 >0$ such that for any $0 < \delta < \delta_0$ and any $k$,
    \begin{align} \label{ineq-I-1-I-tilde-1}
        I_1(k,\delta) &\leq \delta^{(\frac{5}{8}-\rho)k}\sum_{(x_1,\dots,x_k) \in G(k,\delta)} \big[ \prod_{j=1}^{k} L(x_j) \wedge d_{\delta}(x_j) \big]^{-\frac{5}{8}+\rho} =:\delta^{(\frac{5}{8}-\rho)k} \tilde{I}_{1}(k,\delta).
    \end{align}
    To upper bound $\tilde{I}_{1}(k,\delta)$, by symmetry of the summand, we may fix the coordinates of points such that $L(x_j) \wedge d_{\delta}(x_j) = d_{\delta}(x_j)$. For $0 \leq r \leq k$, we set
    \begin{align*}
       G(r,k;\delta):= \Big\{&(x_1,\dots,x_k) \in G(k,\delta): \forall 1 \leq j \leq r,\\ &L(x_j) \wedge d_{\delta}(x_j) = d_{\delta}(x_j), \forall j \geq r+1,  L(x_j) \wedge d_{\delta}(x_j) = L(x_j) \Big\}. 
    \end{align*}
    We then have that
    \begin{align*}
        \delta^{(\frac{5}{8}-\rho)k}\tilde{I}_{1}(k,\delta) &= \delta^{(\frac{5}{8}-\rho)k}\sum_{r=0}^{k} \binom{k}{r} \sum_{(x_1,\dots,x_k) \in G(r,k;\delta)} \big[\prod_{j=1}^{r}d_{\delta}(x_j)\big]^{-\frac{5}{8}+\rho} \prod_{j=r+1}^{k}L(x_j)\big]^{-\frac{5}{8}+\rho}\\
        &\leq \delta^{(\frac{5}{8}-\rho)k}\sum_{r=0}^{k} \binom{k}{r} \sum_{(x_1,\dots,x_r) \in G(r,\delta)} \big[\prod_{j=1}^{r}d_{\delta}(x_j)\big]^{-\frac{5}{8}+\rho} f_{\delta}(x_1,\dots,x_r)
    \end{align*}
    where we have set
    \begin{equation*}
        f_{\delta}(x_1,\dots,x_r) := \sum_{(x_{r+1},\dots,x_{k})} \mathbb{I}_{G(k,\delta)}(x_1,\dots,x_k) \prod_{j=r+1}^{k} L(x_j; x_1,\dots,x_k)^{-\frac{5}{8}+\rho}
    \end{equation*}
    and where we wrote $L(x_j; x_1,\dots,x_k)$ instead of just $L(x_j)$ to stress that in the definition of $L$, the minimum is taken over points in $(x_1,\dots,x_k)$ and not only over points in $(x_{r+1},\dots,x_k)$.
    To upper bound $f_{\delta}(x_1,\dots,x_r)$, we can use the same arguments as in \cite[Proposition~3.10]{MourratIsing} and \cite[Lemma~3.10]{Junnila}. We leave the details to the reader and simply note that the important point to notice is that in the course of the proof, one must count the number of a certain type of graphs and that this number is upper bounded by the number of such graphs appearing when summing $\prod_{j=1}^{k}L(x_j)^{-\frac{5}{8}+\rho}$ over $G(k,\delta)$ instead of $\prod_{j=r+1}^{k}L(x_j)^{-\frac{5}{8}+\rho}$ over $(x_{r+1},\dots,x_k)$ with $(x_1,\dots,x_r)$ fixed. This observation allows us to upper bound $f_{\delta}(x_1,\dots,x_r)$ uniformly in $(x_1,\dots,x_r)$ and we obtain that for any $(x_1,\dots,x_r)$,
    \begin{equation*}
        f_{\delta}(x_1,\dots,x_r) \leq \delta^{-2(k-r)}c^kk^{(\frac{5}{8}-\rho)k}
    \end{equation*}
    for $c>0$ a constant depending on $D$ only. This in turn yields that
    \begin{align*}
        \delta^{(\frac{5}{8}-\rho)k}\tilde{I}_{1}(k,\delta) \leq \delta^{(\frac{5}{8}-\rho)k} c^kk^{(\frac{5}{8}-\rho)k} \sum_{r=0}^{k} \binom{k}{r}  \delta^{-2(k-r)}\sum_{(x_1,\dots,x_r) \in G(r,\delta)} \big[\prod_{j=1}^{r}d_{\delta}(x_j)\big]^{-\frac{5}{8}+\rho}.
    \end{align*}
    Using Lemma \ref{lemma-sum-dist-to-boundary-power-alpha} to upper bound the sum on the right-hand side above, we get that for some constant $C>0$ depending on $D$ only,
    \begin{equation} \label{final-bound-I-1-k-delta}
         \delta^{(\frac{5}{8}-\rho)k}\tilde{I}_{1}(k,\delta) \leq \delta^{(\frac{5}{8}-\rho)k} \delta^{-2k}c^kk^{(\frac{5}{8}-\rho)k} \sum_{r=0}^{k} \binom{k}{r} C^r \leq \delta^{(\frac{5}{8}-\rho)k}\delta^{-2k}\tilde{C}^kk^{(\frac{5}{8}-\rho)k}.
    \end{equation}
    Plugging this upper bound into \eqref{ineq-I-1-I-tilde-1} then shows that the convergence \eqref{cvg-to-0-I-1} indeed holds. Going back to the inequality \eqref{ineq-k-moment--I-1-I-2}, we thus see that to conclude the proof of the claim, it remains to prove that
    \begin{equation} \label{cvg-to-0-I-2}
        \lim_{\delta \to 0} \delta^{\frac{7k}{4}} I_2(k,\delta) = 0.
    \end{equation}
    To do so, we write
    \begin{align}
        I_2(k,\delta) &= \sum_{(x_1,\dots,x_k) \in B_{k}(4\delta) \setminus C_k(\delta) } \PP\bigg[ \bigcap_{j=1}^{k} \{ d(x_j, \gamma_{\delta}) \leq \delta \} \bigg] \nonumber \\
        &+ \sum_{(x_1,\dots,x_k) \in B_{k}(4\delta) \cap C_k(\delta) } \PP\bigg[ \bigcap_{j=1}^{k} \{ d(x_j, \gamma_{\delta}) \leq \delta \} \bigg] \nonumber \\
        &=: I_2^{B}(k,\delta) + I_2^{C}(k,\delta). \label{decomposition-I-2}
    \end{align}
    Let us first establish an upper bound on $I_2^{B}(k,\delta)$. We decompose $I_2^{B}(k,\delta)$ by first choosing the number and the coordinates of the points that are such that $L(x_j) \leq 4\delta$. For $2 \leq m \leq k$, we set
    \begin{equation*}
        B(m,k;\delta) := \{ (x_1,\dots,x_k) \in D_{\delta}^k: \forall 1 \leq j \leq m, L(x_j) \leq 4\delta, \forall j \geq m+1, L(x_j) > 4\delta \}.
    \end{equation*}
    Then, if $(x_1,\dots,x_k) \in B(m,k;\delta)$, we count the number $b$ of squares of side-length $4\delta$ that are necessary to cover these points. In each of these squares, we choose a ``representative'' for the square among the points in $(x_1,\dots,x_m)$ that belong to this square. Observe that we can choose the representatives $(x_{i_1},\dots,x_{i_b})$ of the $b$ squares in such a way that that for any $p \neq n$, $\vert x_{i_{n}} - x_{i_{p}} \vert > 4\delta$. Let us denote by $\mathcal{B}(m)$ the set of coordinates of the representatives of the $b$ squares and define $\mathcal{S}(m,k):= \mathcal{B}(m) \cup \{m+1,\dots,k\}$. With this decomposition, we obtain that
    \begin{align}
        I_2^{B}(k,\delta) &= \sum_{m=2}^{k} \binom{k}{m} \sum_{b=\lfloor \frac{m}{16} \rfloor +1}^{\lfloor \frac{m}{2} \rfloor} \sum_{(x_1,\dots,x_k)\in D_{\delta} \setminus C_{k}(\delta)}\mathbb{I}_{B(m,k;\delta)}  \mathbb{I}_{\vert \mathcal{B}(m)\vert=b} \PP\bigg[ \bigcap_{j=1}^{k} \{ d(x_j, \gamma_{\delta}) \leq \delta \} \bigg]\nonumber \\
        &\leq \sum_{m=2}^{k} \binom{k}{m} \sum_{b=\lfloor \frac{m}{16} \rfloor +1}^{\lfloor \frac{m}{2} \rfloor} \sum_{(x_1,\dots,x_k)\in D_{\delta} \setminus C_{k}(\delta)} \mathbb{I}_{B(m,k;\delta)} \mathbb{I}_{\vert \mathcal{B}(m)\vert=b} \PP\bigg[ \bigcap_{j \in \mathcal{S}(m,k)} \{ d(x_j, \gamma_{\delta}) \leq \delta \} \bigg] \nonumber \\
        &\leq \sum_{m=2}^{k} \binom{k}{m} \sum_{b=\lfloor \frac{m}{16} \rfloor +1}^{\lfloor \frac{m}{2} \rfloor} \delta^{(\frac{5}{8}-\rho)(k-m+b)}I_{2}^{B}(b,m,k,\delta) \label{ineq-decomposition-I-2-B}
    \end{align}
    where we have set
    \begin{equation*}
        I_{2}^{B}(b,m,k,\delta) := \sum_{(x_1,\dots,x_k)\in D_{\delta} \setminus C_{k}(\delta)} \mathbb{I}_{B(m,k;\delta)} \mathbb{I}_{\vert \mathcal{B}(m)\vert =b} \big[ \prod_{j \in \mathcal{S}(m,k)} L(x_j; \mathcal{S}(m,k)) \wedge d_{\delta}(x_j)\big]^{-\frac{5}{8}+\rho} 
    \end{equation*}
    and where the last inequality follows from Lemma \ref{lemma-k-pts-estimate-interface-discrete}. Above, for $j \in \mathcal{S}(m,k)$, we have also defined
    \begin{equation*}
        L(x_j; \mathcal{S}(m,k)) = \frac{1}{4} \min_{p \in \mathcal{S}(m,k), p \neq j } \vert x_p - x_j \vert.
    \end{equation*}
    To upper bound $ I_{2}^{B}(b,m,k,\delta)$, we use H\"older's inequality with $\alpha, \tilde \alpha >1$ such that $\alpha(\frac{5}{8}-\rho)<1$, $\tilde \alpha(\frac{5}{8}-\rho)<2$ and $\alpha^{-1}+\tilde \alpha^{-1} = 1$. This yields that
    \begin{align*}
        I_{2}^{B}(b,m,k,\delta) \leq &\bigg( \sum_{(x_1,\dots,x_k)\in D_{\delta} \setminus C_{k}(\delta)} \mathbb{I}_{B(m,k;\delta)} \mathbb{I}_{\vert \mathcal{B}(m)\vert =b}  \bigg)^{\frac{1}{\tilde \alpha}} \\
        &\times \bigg( \delta^{-2(k-m+b)} \sum_{(x_{1},\dots,x_{k-m+b}) \in G(k-m+b,\delta)} \big[ \prod_{j=1}^{k-m+b}L(x_j) \wedge d_{\delta}(x_j) \big]^{-\alpha(\frac{5}{8}-\rho)} \bigg)^{\frac{1}{\alpha}}
    \end{align*}
    where we recall that $G(k-m+b,\delta) =  D_{\delta}^{k-m+b} \setminus (B_{k-m+b}(4\delta) \cup C_{k-m+b}(\delta))$. Since $\alpha(\frac{5}{8}-\rho)<1$, the sum on the right-hand side above can be upper bounded using the same arguments as those which led to the upper bound \eqref{final-bound-I-1-k-delta}. We thus obtain that
    \begin{align*}
        I_{2}^{B}(b,m,k,\delta) \leq \delta^{-\frac{2}{\tilde \alpha}(k-m+b)} \delta^{-\frac{2}{\alpha}k} C^{k-m+b}(k-m+b)^{(\frac{5}{8}-\rho)(k-m+b)}
    \end{align*}
    Multiplying both sides by $\delta^{(\frac{5}{8}-\rho)(k-m+b)}$ and observing that $\frac{2}{\tilde \alpha}(m-b) - (\frac{5}{8}-\rho)(m-b)>0$ by our choice of $\tilde \alpha$ and since $m-b \geq m/2 \geq 1$, we thus see that
    \begin{equation*}
        I_{2}^{B}(b,m,k,\delta) \leq \delta^{(\frac{5}{8}-\rho)k}\delta^{-2k}C^{k-m+b}(k-m+b)^{(\frac{5}{8}-\rho)(k-m+b)}.
    \end{equation*}
    Plugging this upper bound into \eqref{ineq-decomposition-I-2-B} then yields that
    \begin{equation} \label{final-bound-I-2-B}
        I_2^{B}(k,\delta) \leq  \tilde C^{k}\delta^{(\frac{5}{8}-\rho)k}\delta^{-2k} C^k k^{(\frac{5}{8}-\rho)k}.
    \end{equation}
    To conclude the proof of \eqref{cvg-to-0-I-2}, in view of \eqref{decomposition-I-2}, it now remains to find an appropriate upper bound on $I_2^{C}(k,\delta)$. To do so, we first fix the coordinates of the points in $C_{k}(\delta)$ and write
    \begin{align*}
        I_2^{C}(k,\delta) &= \sum_{r=1}^{k} \binom{k}{r} \sum_{\substack{(x_1,\dots,x_r) \in C_1(\delta)^r \\ (x_{r+1}, \dots,x_k) \in D_{\delta}^k \setminus C_k(\delta)}} \PP\bigg[ \bigcap_{j=1}^{k} \{ d(x_j, \gamma_{\delta}) \leq \delta \} \bigg]\\
        &\leq \sum_{r=1}^{k} \binom{k}{r} \Area(C_{1}(\delta)^r) \sum_{(x_{r+1},\dots,x_{k}) \in D_{\delta}^{k-r} \setminus C_{k-r}(\delta)} \PP \bigg[ \bigcap_{j=r+1}^{k} \{ d(x_j, \gamma_{\delta}) \leq \delta \} \bigg].
    \end{align*}
    Decomposing $D_{\delta}^{k-r} \setminus C_{k-r}(\delta)$ as $D_{\delta}^{k-r} \setminus (B_{k-r}(\delta) \cup C_{k-r}(\delta)) \cup B_{k-r}(\delta) \setminus C_{k-r}(\delta)$, we see that the inner sum above can upper bounding using the upper bounds \eqref{final-bound-I-1-k-delta} and \eqref{final-bound-I-2-B} that we obtained in the first part of the proof. This yields that
    \begin{align*}
        I_2^{C}(k,\delta) \leq \sum_{r=1}^{k} \binom{k}{r} \delta^{-r} \delta^{-2(k-r)} \delta^{(\frac{5}{8}-\rho)(k-r)}C^{k-r}(k-r)^{(\frac{5}{8}-\rho)(k-r)} \leq c^k \delta^{-2k} \delta^{(\frac{5}{8}-\rho)k}\delta^{(\frac{3}{8}+\rho)r} k^{(\frac{5}{8}-\rho)k}.
    \end{align*}
    Together with the decomposition \eqref{decomposition-I-2} and the upper bound \eqref{final-bound-I-2-B}, this implies that
    \begin{equation} \label{final-bound-I-2-k-delta}
        I_2(k,\delta) \leq \tilde c^k \delta^{-2k} \delta^{(\frac{5}{8}-\rho)k}k^{(\frac{5}{8}-\rho)k}
    \end{equation}
    which concludes the proof of the convergence \eqref{cvg-to-0-I-2} and thus completes the proof of the claim.
\end{proof}

\begin{proof}[Proof of Claim \ref{claim-bound-EE-k-moment-V-gamma}]
    This directly follows from the upper bound on $\delta^{\frac{7k}{4}}\EE[\vert V(\gamma_{\delta}) \vert^k]$ obtained in the course of the proof of Claim \ref{claim-cvg-to-0-k-moment-V-gamma}, see in particular \eqref{final-bound-I-1-k-delta} and \eqref{final-bound-I-2-k-delta}.
\end{proof}

\subsection{Characterization of the limiting law: proof of Proposition \ref{proposition-uniqueness-law-curve}} \label{sub:proof-characterization-interface}

\begin{proof}[Proof of Proposition \ref{proposition-uniqueness-law-curve}]
    To prove this proposition, we will also need to consider convergence of the random variables $(W_{\delta}^{h}(\phi):=\delta\sum_{x \in D_{\delta}}h_{x}\phi(x))_{\phi}$ for functions $\phi: \mathbb{R}^2 \to \mathbb{R}$ that are square-integrable. We refer the reader to \cite[Section~2.1]{CSZ16} for details on the topology in which convergence takes place and background material on white noise. 

    Recall the decomposition \eqref{decomposition-RN-interface} of the Radon-Nikodym derivative $F^h(\gamma_{\delta})$ of $\nu_{\delta}^{\pm,\lambda,h}$ with respect to $\nu_{\delta}^{\pm}$ and the definitions of $\theta_{L}(\delta), \theta_{R}(\delta), S_{L}^{h}(\delta)$ and $S_{R}^{h}(\delta)$ given in \eqref{definition-theta-L-R} and \eqref{definition-S-L-S-R}. In view of this decomposition, the first step to prove Proposition \ref{proposition-uniqueness-law-curve} is to prove the following lemma.
    
    \begin{lemma}\label{lemma-joint-convergence}
        $(W_{\delta}^{h}, \tilde{\mathcal{Z}}_{\delta}^{\pm,\lambda,h}, \theta_L(\delta)S_{L}^{h}(\delta), \theta_{R}(\delta)S_{R}^{h}(\delta))_{\delta}$ under $\PP \otimes \mu_{\delta}^{\pm}$ jointly converge in finite dimensional distribution to $(W,\mathcal{Z}_{D}^{\pm,W,\lambda},S_{L}^{W}, S_{R}^{W})$ under $\PP \otimes \nu_{\operatorname{SLE}_{3}}^{(D,a,b)}$. Here, $W$ is a two-dimensional white noise, $\mathcal{Z}_{D}^{\pm,W,\lambda}$ is the limiting partition function of Theorem \ref{theorem-cvg-RFIM-partition-function} and $S_{L}^{W}$ and $S_{R}^{W}$ are defined as
        \begin{align*}
            &S_{L}^{W}:= \sum_{k \geq 0} \frac{C_{\sigma}^k}{k!} \int_{D_{L}^k} f_{D_{L}}^{(+,k)}(z_1,\dots,z_k) \prod_{n=1}^{k} \lambda(z_j)W(dz_j), \\
            &S_{R}^{W}:= \sum_{k \geq 0} \frac{C_{\sigma}^k}{k!} \int_{D_{R}^k} f_{D_{R}}^{(-,k)}(z_1,\dots,z_k) \prod_{n=1}^{k}\lambda(z_j) W(dz_j).
        \end{align*}
        Above, $D_{L}$, respectively $D_{R}$, is the connected component on the left, respectively right, of $\gamma$. 
    \end{lemma}
    
    Lemma \ref{lemma-joint-convergence} will be shown just below. Assuming this result, the second step of the proof of Proposition \ref{proposition-uniqueness-law-curve} is to take into account the contribution of $\exp(M^{h}(\gamma_{\delta}))$ and of $(\theta_{L}(\delta)\theta_{R}(\delta))^{-1}\theta_{D}(\delta)$ to the limit of $F^h(\gamma_{\delta})$. However, this contribution in fact vanishes in the limit, as implied by Claim \ref{claim-cvg-magnetization-interface} and its proof. Indeed, this claim shows that the Laplace transform of $M^{h}(\gamma_{\delta})$ under $\PP \otimes \mu_{\delta}^{\pm}$ converges to $1$, which implies that $M^{h}(\gamma_{\delta})$ converges in law to $0$, and therefore that, since $x \to \exp(x)$ is continuous, $\exp(M^{h}(\gamma_{\delta}))$ converges in law to $1$. A similar reasoning, using again (the proof of) Claim \ref{claim-cvg-magnetization-interface}, applies to $(\theta_{L}(\delta)\theta_{R}(\delta))^{-1}\theta_{D}(\delta)$ once one notices that this random variable is in fact equal to $\exp(-\frac{\delta^{\frac{7}{4}}}{2}\sum_{x \in V(\gamma_{\delta})}\lambda(x)^2)$. This concludes the proof of Proposition \ref{proposition-uniqueness-law-curve}.
\end{proof}

Let us now prove Lemma \ref{lemma-joint-convergence}. As the strategy of the proof follows quite closely that of \cite[Proposition~3.1]{continuum-2d-RFIM}, we will be somewhat brief.

\begin{proof}[Proof of Lemma \ref{lemma-joint-convergence}]
    The starting point of the proof is to observe that $\theta_L(\delta)S_{L}^{h}(\delta)$, $\theta_{R}(\delta)S_{R}^{h}(\delta)$ and $\tilde{\mathcal{Z}}_{\delta}^{\pm,\lambda,h}$ can be expressed in terms of a polynomial chaos decomposition, see \eqref{S-L-chaos-expansion}, \eqref{S-R-chaos-expansion} and \eqref{chaos-expansion-partition-function}. Moreover, recall that by Claim \ref{claim-Lq-norm-renomalization-factor}, the prefactors $\theta_j(\delta)\prod_{x \in D_{j,\delta}} \cosh(\delta^{\frac{7}{8}}\lambda(x)h_x)$, $j=L,R$, converge to $1$ in probability and are uniformly bounded in $L^p$, for any $p \geq 1$. Similarly, as shown in \cite[Lemma~3.5]{continuum-2d-RFIM}, the prefactor $\theta(\delta)\prod_{x \in D_{\delta}} \cosh(\delta^{\frac{7}{8}}\lambda(x)h_x)$ appearing in the chaos expansion \eqref{chaos-expansion-partition-function} of $\tilde{\mathcal{Z}}_{\delta}^{\pm,\lambda, h}$ converges to $1$ in any $L^p$, $p \geq 1$. Therefore, it suffices to show the lemma with $\theta_{L}(\delta)S_{L}^{h}(\delta)$ and $\theta_{R}(\delta)S_{R}^{h}(\delta)$ replaced by $\Psi_{\delta,L}(\eta)$ and $\Psi_{\delta,R}(\eta)$ as defined in \eqref{def-Psi-R-Psi-L} and with $\tilde{\mathcal{Z}}_{\delta}^{\pm,\lambda,h}$ replaced by
    \begin{align*}
        \Psi_{\delta}^{\pm}(\eta) := \sum_{I \subset D_{\delta}} \mu_{\delta}^{\pm}[\prod_{x \in I}\sigma_{x}] \prod_{x \in I} \vartheta_{\delta}(x)\eta_x
    \end{align*}
    where $\eta_x$ and $\vartheta_{\delta}(x)$ are as in \eqref{def-sigma-mu-eta}. Then, as in \cite{continuum-2d-RFIM}, we first truncate the chaos expansions $\Psi_{\delta}(\eta)$, $\Psi_{\delta,L}(\eta)$ and $\Psi_{\delta,R}(\eta)$ to order $\ell$, that is we restrict the sums in the definitions of $\Psi_{\delta}^{\pm}(\eta)$, $\Psi_{\delta,L}(\eta)$ and $\Psi_{\delta,R}(\eta)$ to subsets $I$ of size at most $\ell$. Denote by $\Psi_{\delta,L}^{\leq \ell}(\eta)$, $\Psi_{\delta,R}^{\leq \ell}(\eta)$ and $\Psi_{\delta}^{\pm,\leq \ell}(\eta)$ these truncated expansions. The first step is to show that
    \begin{align}
        &\lim_{\ell \to \infty} \lim_{\delta \to 0} \EE\otimes\mu_{\delta}^{\pm}\big[\vert \Psi_{\delta,j}^{\leq \ell}(\eta) - \Psi_{\delta,j}(\delta)\vert^2\big] = 0, \quad j=L,R \label{cvg-truncated-S-j} \\
        &\lim_{\ell \to \infty} \lim_{\delta \to 0} \EE\otimes\mu_{\delta}^{\pm}\big[\vert \Psi_{\delta}^{\pm,\leq \ell}(\eta) -  \Psi_{\delta}^{\pm}(\eta) \vert^2\big] =0 . \label{cvg-truncated-partition-function}
    \end{align}
    The second step is to linearize the noise: namely, in the definitions of $\Psi_{\delta,L}^{\leq \ell}(\eta)$, $\Psi_{\delta,R}^{\leq \ell}(\eta)$ and $\Psi_{\delta}^{\pm,\leq \ell}(\eta)$, we replace $\tanh(\delta^{\frac{7}{8}}\lambda(x)h_x)$ with $\delta^{\frac{7}{8}}\lambda(x)h_x$, $\vartheta_{\delta}(x)^2$ with $\delta^{\frac{7}{4}}\lambda(x)^2$ 
(in the definition \eqref{def-sigma-mu-eta} of $\eta$). Denote by $\Xi_{\delta,L}^{\leq \ell}(h)$, $\Xi_{\delta,R}^{\leq \ell}(h)$ and $\Xi_{\delta,D}^{\pm,\leq \ell}(h)$ the random variables thus obtained. The second step then consists in establishing that
    \begin{align}
        &\lim_{\ell \to \infty} \lim_{\delta \to 0} \EE \otimes \mu_{\delta}^{\pm}\big[ \vert \Psi_{\delta,j}^{\leq \ell}(\eta)  - \Xi_{\delta,j}^{\leq \ell}(h) \vert^2\big] = 0, \quad j=L,R \label{cvg-linearized-S-j} \\
        &\lim_{\ell \to \infty} \lim_{\delta \to 0} \EE \otimes \mu_{\delta}^{\pm}\big[ \vert \Psi_{\delta}^{\pm,\leq \ell}(\eta)  - \Xi_{\delta,D}^{\pm,\leq \ell}(h) \vert^2]= 0. \label{cvg-linearized-partition-function}
    \end{align}
    In the third step, we replace the noise variables in $\Xi_{\delta,L}^{\leq \ell}(h)$, $\Xi_{\delta,R}^{\leq \ell}(h)$ and $\Xi_{\delta,D}^{\pm, \leq \ell}(h)$ with their white noise approximation: each random variable $h_x$ is replaced by $w_{\delta}(x):= 2\delta^{-1}\int_{S_{\delta}(x)}W(dz)$ where $S_{\delta}(x)$ is the box centered at $x$ with side length $\delta$ and $W$ the limiting white noise. Denote by $\Theta_{\delta,L}^{W,\leq \ell}$, $\Theta_{\delta,R}^{W,\leq \ell}$ and $\Theta_{\delta, D}^{\pm,W,\leq \ell}$ the random variables thus obtained. Let us also define, for $\phi \in L^2(\mathbb{R}^2)$, $W_{\delta}(\phi):= \delta \sum_{x \in \delta\mathbb{Z}^2} w_{\delta}(x)\phi(x)$. We then show that for any $\ell, m \in \mathbb{N}$, any $\phi_{1}, \dots, \phi_{m} \in L^{2}(\mathbb{R}^2)$ as $\delta \to 0$,
    \begin{equation} \label{cvg-white-noise-approximation}
        ((W_{\delta}(\phi_j))_{j=1}^{m}, \Theta_{\delta, D}^{\pm,W,\leq \ell}, \Theta_{\delta, L}^{W,\leq \ell}, \Theta_{\delta,R}^{W,\leq \ell}) \to ((W(\phi_j))_{j=1}^{m}, \mathcal{Z}_{D}^{\pm,W, \leq \ell}, \mathcal{Z}_{L}^{+,W, \leq \ell}, \mathcal{Z}_{R}^{-,W, \leq \ell})
    \end{equation}
    where the convergence is in $L^2(\PP \otimes \nu_{\operatorname{SLE}_{3}}^{(D,a,b)})$.

    The last step of the proof is to prove that for any $\ell,j \in \mathbb{N}$, any $\phi_1,\dots,\phi_m \in L^2(\mathbb{R}^2)$ and any bounded function $g: \mathbb{R}^{m+3} \to \mathbb{R}$ whose first three derivatives are bounded,
    \begin{align}
        \lim_{\delta \to 0} \bigg \vert &\EE\otimes \mu_{\delta}^{\pm}[g((W^{h}_{\delta}(\phi_j))_{j=1}^{m}, \Xi_{\delta,L}^{\leq \ell}(h), \Xi_{\delta,R}^{\leq \ell}(h), \Xi_{\delta,D}^{\pm, \leq \ell}(h))] \nonumber \\
        &- \EE\otimes \mu_{\delta}^{\pm}[g((W_{\delta}(\phi_j))_{j=1}^{m}, \Theta_{\delta,L}^{W,\leq \ell}, \Theta_{\delta,R}^{W,\leq \ell}, \Theta_{\delta,D}^{\pm,W,\leq \ell})] \bigg \vert = 0. \label{cvg-theta-vs-Xi}
    \end{align}
    Putting everything together, this concludes the proof of Proposition \ref{proposition-uniqueness-law-curve}. The convergence results \eqref{cvg-truncated-S-j}--\eqref{cvg-theta-vs-Xi} are established just below.
\end{proof}

\begin{proof}[Proof of \eqref{cvg-truncated-S-j} and \eqref{cvg-truncated-partition-function}]
    Let us prove \eqref{cvg-truncated-S-j} for $j=L$, as the convergences \eqref{cvg-truncated-S-j} for $j=R$ and \eqref{cvg-truncated-partition-function} follow from the same arguments. Let $\ell \in \mathbb{N}$. We have that
    \begin{align*}
        \Psi_{\delta,L}(\eta) - \Psi_{\delta,L}^{\leq \ell}(\eta) = \sum_{I \subset D_{L,\delta}, \vert I \vert \geq \ell+1}\mu_{\delta, D_{L}}^{+}[\prod_{x \in I} \sigma_x ] \prod_{x \in I} \vartheta_{\delta}(x)\eta_x =: \Psi_{\delta,L}^{> \ell}(\eta).
    \end{align*}
    We now write $\EE \otimes \mu_{\delta}^{\pm}[\vert \Psi_{\delta,L}^{> \ell}(\eta) \vert^2] = \mu_{\delta}^{\pm}[\EE[\Psi_{\delta,L}^{> \ell}(\eta)]^2]+ \mu_{\delta}^{\pm}(\Var(\Psi_{\delta,L}^{> \ell}(\eta)))$. The term $\mu_{\delta}^{\pm}[\EE[\Psi_{\delta,L}^{> \ell}(\eta)]^2]$ converges to $0$ as $\delta \to 0$ and then $\ell \to \infty$ since only the empty interval contributes to this expectation, see the discussion around \eqref{prod-of-expectations-is-1}. On the other hand, using the same bound as in the proof of \cite[Lemma~B.1]{continuum-2d-RFIM}, we have that for any $\eps > 0$,
    \begin{align*}
        \mu_{\delta}^{\pm}[\Var(\Psi_{\delta,L}^{> \ell}(\eta))]
        \leq \mu_{\delta}^{\pm}\bigg[\sum_{I \subset D_{L,\delta}, \vert I \vert > \ell} (1+\eps)^{\vert I \vert} [\prod_{x \in I}\vartheta_{\delta}(x)^{2}] \mu_{\delta,D_{L}}^{+}\big[\prod_{x \in I} \sigma_x \big]^2 \bigg ].
    \end{align*}
    To prove that $\mu_{\delta}^{\pm}[\Var(\Psi_{\delta,L}^{> \ell}(\eta))]$ converges to $0$ as $\delta \to 0$ and then $\ell \to \infty$, it thus suffices to show that this is the case of the $\mu_{\delta}^{\pm}$-expectation of the sum on the above right-hand side. But in fact, the expectation of each term in this sum can be bounded using the same arguments as those used to prove Claim \ref{claim-bound-jk-expectations} (taking $j=0$ everywhere). This implies that the expectation of this sum can be upper bounded uniformly in $\delta$ by the rest of order $> \ell$ of an absolutely convergent series. This in turn yields that $\mu_{\delta}^{\pm}[\Var(\Psi_{\delta,L}^{> \ell}(\eta))]$ converges to $0$ as $\delta \to 0$ and then $\ell \to \infty$ and concludes the proof of \eqref{cvg-truncated-S-j}.
\end{proof}

\begin{proof}[Proof of \eqref{cvg-linearized-S-j} and \eqref{cvg-linearized-partition-function}]
    Let us prove \eqref{cvg-linearized-S-j} for $j=L$, as the convergences \eqref{cvg-linearized-S-j} for $j=R$ and \eqref{cvg-linearized-partition-function} follow from the same arguments. Thanks to the bounds of Claim \ref{claim-bound-jk-expectations} on the $\mu_{\delta}^{\pm}$-expectation of $\sum_{I: \vert I \vert = k}\vartheta_{\delta}^{2k}\mu_{\delta,L}^{+}[\prod_{x \in I} \sigma_x]^2$ (where $\mu_{\delta,L}^{+}[\prod_{x \in I} \sigma_x]$ is the analogue of $\psi_{\delta}(I)$ in the notation of \cite{continuum-2d-RFIM}) and thanks to Theorem \ref{theorem_spin_correlations} applied to $\mu_{\delta,L}^{+}[\prod_{x \in I} \sigma_x]$, this in fact follows from the same arguments as the ones used to prove \cite[Lemma~3.3]{continuum-2d-RFIM}. Details are left to the reader.
\end{proof}

\begin{proof}[Proof of \eqref{cvg-white-noise-approximation}]
    Let $\ell \in \mathbb{N}$ and let us first show that $\EE \otimes \mu_{\delta}^{\pm}\big[\vert \Theta_{L}^{W,\leq \ell}(\delta) - \mathcal{Z}_{L}^{+,W,\leq \ell} \vert^2 \big] \to 0$ as $\delta \to 0$. This in fact follows from the same arguments as those used in Step 1 of the proof of \cite[Theorem~2.3]{CSZ16}. Indeed, applying the same decomposition as in \cite[Equation~(5.4)]{CSZ16} and using the bounds of Claim \ref{claim-bound-jk-expectations} on the $\mu_{\delta}^{\pm}$-expectations to upper bound the terms that we get from this decomposition, we obtain that $\EE \otimes \mu_{\delta}^{\pm}\big[\vert \Theta_{L}^{W,\leq \ell}(\delta) - \mathcal{Z}_{L}^{+,W,\leq \ell} \vert^2 \big]$ indeed converges to $0$ as $\delta \to 0$. A similar reasoning can be used to show that $\EE \otimes \mu_{\delta}^{\pm}\big[\vert \Theta_{R}^{W,\leq \ell}(\delta) - \mathcal{Z}_{R}^{-,W,\leq \ell} \vert^2 \big]$ and $\EE \otimes \mu_{\delta}^{\pm}\big[\vert \Theta_{D}^{W,\leq \ell}(\delta) - \mathcal{Z}_{D}^{\pm,W,\leq \ell} \vert^2 \big]$ converge to $0$ as $\delta \to 0$. This proves \eqref{cvg-white-noise-approximation}.
\end{proof}

\begin{proof}[Proof of \eqref{cvg-theta-vs-Xi}]
    Let $\ell, m \in \mathbb{N}$, $\phi_1, \dots, \phi_m \in L^2(\mathbb{R}^2)$, and let $g: \mathbb{R}^{m+3} \to \mathbb{R}$ be a bounded function whose first three derivatives are bounded. To show \eqref{cvg-theta-vs-Xi}, we first apply the multivariate Lindeberg's principle of \cite[Lemma~3.6]{continuum-2d-RFIM}, see also Step 3 of the proof of \cite[Theorem~2.3]{CSZ16}, to obtain that for any $\eps > 0$, $\mu_{\delta}^{\pm}$-almost surely,
    \begin{align}\label{upper-bound-via-Lindeberg}
        &\vert \EE[\begin{aligned}[t]&g((W^{h}_{\delta}(\phi_j))_{j=1}^{m}, \Xi_{L}^{h,\leq \ell}(\delta), \Xi_{R}^{h,\leq \ell}(\delta), \Xi_{D}^{\pm, h,\leq \ell}(\delta))] \\
        &- \EE[g((W^{h}_{\delta}(\phi_j))_{j=1}^{m}, \Theta_{L}^{h,\leq \ell}(\delta), \Theta_{R}^{h,\leq \ell}(\delta), \Theta_{D}^{\pm, h,\leq \ell}(\delta))] \vert\end{aligned} \nonumber \\
        &\leq C_{g,\ell}\sum_{j=1}^{m}\Var(W_{\delta}(\phi_j))\max_{x \in D_{\delta}} \big(\operatorname{Inf}_{x}[W_{\delta}(\phi_j)])^{\frac{1}{2}} + \tilde{C}_{g,\ell}C_{D}^{(\eps,\ell)}(\Psi;\delta)\big(\max_{x \in D_{\delta}}\operatorname{Inf}_{x}[\Psi_{\delta}^{(\eps),\leq \ell}]\big)^{\frac{1}{2}} \nonumber \\
        &+ C_{g,\ell}'C_{L}^{(\eps,\ell)}(\Psi;\delta)\big(\max_{x \in D_{L,\delta}}\operatorname{Inf}_{x}[\Psi_{\delta,L}^{(\eps),\leq \ell}])^{\frac{1}{2}}+C_{g,\ell}'' C_{R}^{(\eps,\ell)}(\Psi;\delta)\big(\max_{x \in D_{R,\delta}}\operatorname{Inf}_{x}[\Psi_{\delta,R}^{(\eps),\leq \ell}])^{\frac{1}{2}}
    \end{align}
    where $C_{g,\ell}, \tilde{C}_{g,\ell}, C_{g,\ell}', C_{g,\ell}'' > 0$ are constant depending on $g$ and $\ell$ only. Above, we have set
    \begin{align*}
        & \operatorname{Inf}_{x}[W_{\delta}(\phi_j)] := \delta \phi_j(x)\\
        &C_{D}^{(\eps,\ell)}(\Psi;\delta):= \sum_{I \subset D_{\delta}, \vert I \vert \leq \ell} (1+\eps)^{\vert I \vert} \vartheta_{\delta}^{2\vert I \vert}\mu_{\delta}^{\pm}[\prod_{x \in I} \sigma_x]^2 \\
        &\operatorname{Inf}_{x \in D_{\delta}}[\Psi_{\delta}^{(\eps),\leq \ell}] := \sum_{I \subset D_{\delta}: \vert I \vert \leq \ell, x \in I} (1+\eps)^{\vert I \vert} \vartheta_{\delta}^{2\vert I \vert}\mu_{\delta}^{\pm}[\prod_{z \in I} \sigma_{z}]^2
    \end{align*}
    and for $j=R,L$ and $x \in D_{j,\delta}$, we have set
    \begin{align*}
        &C_{j}^{(\eps,\ell)}(\Psi;\delta) = \sum_{I \subset D_{j,\delta}, \vert I \vert \leq \ell} (1+\eps)^{\vert I \vert} [\prod_{x \in I}\vartheta_{\delta}(x)^{2}]\mu_{\delta,j}^{\xi_{j}}[\prod_{x \in I} \sigma_x]^2 \\
        &\operatorname{Inf}_{x}[\Psi_{\delta,j}^{(\eps),\leq \ell}] := \sum_{I \subset D_{j,\delta}: \vert I \vert \leq \ell. x \in I} (1+\eps)^{\vert I \vert} [\prod_{z \in I}\vartheta_{\delta}(z)^{2}]\mu_{\delta,j}^{\xi_{j}}[\prod_{z \in I} \sigma_{z}]^2.
    \end{align*}
    It is easy to see that for each $j=1,\dots,m$, $\Var(W_{\delta}(\phi_j))$ is uniformly bounded in $\delta$ while $\operatorname{Inf}_{x}[W_{\delta}(\phi_j)] \to 0$ as $\delta \to 0$. Therefore, the first summand on the right-hand side of \eqref{upper-bound-via-Lindeberg} converges to $0$ as $\delta \to 0$. Moreover, the fact that $C_{D}^{(\eps,\ell)}(\Psi;\delta)\big(\max_{x} \operatorname{Inf}_{x \in D_{\delta}}[\Psi_{\delta}^{(\eps),\leq \ell}]\big)^{\frac{1}{2}}$ also vanishes as $\delta \to 0$ is proved in Step 2 of the proof of \cite[Theorem~3.14]{CSZ16}. Thus, to conclude the proof of \eqref{cvg-theta-vs-Xi}, it thus suffices to show that the $\mu_{\delta}^{\pm}$-expectations of the two last summands on the right-hand side of \eqref{upper-bound-via-Lindeberg} converge to $0$ as $\delta \to 0$. Let us prove this for the first one, as the proof for the second one is identical. We first use the Cauchy-Schwarz inequality to get that
    \begin{equation*}
        \mu_{\delta}^{\pm}[C_{L}^{(\eps,\ell)}(\Psi;\delta)\big(\max_{x \in D_{L,\delta}}\operatorname{Inf}_{x}[\Psi_{\delta,L}^{(\eps),\leq \ell}])^{\frac{1}{2}}]^2 \leq \mu_{\delta}^{\pm}[C_{L}^{(\eps,\ell)}(\Psi;\delta)^2]\mu_{\delta}^{\pm}[\max_{x \in D_{L,\delta}} \operatorname{Inf}_{x}\Psi_{\delta,L}^{(\eps),\leq \ell}].
    \end{equation*}
    The term $\mu_{\delta}^{\pm}[C_{L}^{(\eps,\ell)}(\Psi;\delta)^2]$ can be bounded uniformly in $\delta$: when expanding the square as the product of two sums, the $\mu_{\delta}^{\pm}$-expectation of each summand can be bounded using arguments similar to those used to prove Claim \ref{claim-bound-jk-expectations}. In fact, replacing ``R" with ``L" in the statement of Claim \ref{claim-bound-jk-expectations}, one can see that the corresponding inequality can be proved by repeating the proof of Claim \ref{claim-bound-jk-expectations} by replacing ``R" with ``L" everywhere. It thus now remains to show that $\mu_{\delta}^{\pm}[\max_{x \in D_{L,\delta}} \operatorname{Inf}_{x} [\Psi_{\delta,L}^{(\eps),\leq \ell}]]$ converges to $0$ as $\delta \to 0$. To do so, we are going to use the dominated convergence theorem. Indeed, it is easy to see that almost surely, $\max_{x \in D_{L,\delta}} \operatorname{Inf}_{x} [\Psi_{\delta,L}^{(\eps),\leq \ell}] \to 0$ as $\delta \to 0$. On the other hand, since for any $x \in D_{L,\delta}$, $ \operatorname{Inf}_{x} [\Psi_{\delta,L}^{(\eps),\leq \ell}] \geq 0$, we have that $\mu_{\delta}^{\pm}$-almost surely,
    \begin{align*}
        \max_{x \in D_{L,\delta}} \operatorname{Inf}_{x} [\Psi_{\delta,L}^{(\eps),\leq \ell}] &\leq \sum_{x \in D_{L,\delta}} \sum_{I \subset D_{L,\delta}: \vert I \vert \leq \ell, x \in I} (1+\eps)^{\vert I \vert} [\prod_{z \in I}\vartheta_{\delta}(z)^{2}]\mu_{\delta,L}^{+}[\prod_{z \in I} \sigma_{z}]^2\\
        &= \ell \sum_{I \subset D_{L,\delta}: \vert I \vert \leq \ell} (1+\eps)^{\vert I \vert} [\prod_{z \in I}\vartheta_{\delta}(z)^{2}]\mu_{\delta,L}^{+}[\prod_{z \in I} \sigma_{z}]^2.
    \end{align*}
    Claim \ref{claim-bound-jk-expectations} (applied with no ``right" terms) implies that the $\mu_{\delta}^{\pm}$-expectation of the random variable on the right-hand side above is uniformly bounded in $\delta$. Therefore, the dominated convergence theorem yields that $\mu_{\delta}^{\pm}[\max_{x \in D_{L,\delta}} \operatorname{Inf}_{x} [\Psi_{\delta,L}^{(\eps),\leq \ell}]] \to 0$ as $\delta \to 0$. As explained before, this concludes the proof of \eqref{cvg-theta-vs-Xi}.
\end{proof}

\subsection{Conformal covariance of \texorpdfstring{$\nu_{W,\lambda}^{(D,a,b)}$}{}} \label{sec-conformal-covariance}

In this section, we prove conformal covariance in law of the limiting measure $\nu_{W,\lambda}^{(D,a,b)}$ as claimed in Theorem \ref{theorem-single-interface} and made precise in the proposition just below.

\begin{proposition}\label{proposition-conformal-covariance-curve}
    Let $D, \tilde D \subset \mathbb{C}$ be two bounded, open an simply connected domains with smooth boundary and let $a,b \in \partial D$. Let $\phi: D \to \tilde D$ be a conformal map and set $\tilde a = \phi(a)$, $\tilde b = \phi(b)$. Then, $\nu_{W,\lambda}^{(D,a,b)}$ has the same law as $\nu_{W, \tilde \lambda}^{(\tilde D, \tilde a, \tilde b)}$ where for $z \in \tilde D$, $\tilde{\lambda}(z)= \vert (\phi^{-1})'(z) \vert^{\frac{7}{8}} \lambda(\phi^{-1}(z))$.
\end{proposition}

\begin{proof}
    Proposition \ref{proposition-conformal-covariance-curve} is proved by using the explicit expression \eqref{RN-single-interface-limit} for the Radon-Nikodym derivative $F_{W,\lambda}^{(D,a,b)}$ of $\nu_{W,\lambda}^{(D,a,b)}$ with respect to $\nu_{\operatorname{SLE}_{3}}^{(D,a,b)}$. Indeed, since SLE$_3$ is conformally invariant in law, it suffices to prove that $F_{W,\lambda}^{(D,a,b)}$ transforms as claimed in Proposition \ref{proposition-conformal-covariance-curve}. To do so, we first observe that by \cite[Corollary~3.16]{CSZ16}, $\mathcal{Z}_{D}^{\pm,W,\lambda}$ has the same law as $\mathcal{Z}_{\tilde{D}}^{\pm,W,\tilde \lambda}$ where $\tilde \lambda$ is as in the statement of the proposition. Next, let us look at how the term $S_{L}^{W,\lambda}$ transforms. If $W$ denotes white noise on $\mathbb{R}^2$, then the processes $(\int_{D_{L}}f(z)W(dz))_{f \in L^2(D)}$ and $(\int_{\tilde D_{L}}f(\phi^{-1}(z))\vert (\phi^{-1})'(z) \vert W(dz))_{f \in L^2(D)}$ have the same law, since there are both centered Gaussian processes with the same covariance. Here, we have set $\tilde D_{L}:=\phi(D_{L})$. This extends to an equality in law for multiple integrals, that is, for any $k \in \mathbb{N}$,
    \begin{equation*}
        \int_{D_{L}^k} f(z_{1},\dots,z_{k})\prod_{j=1}^{k}W(dz_j) \overset{(d)}{=}\int_{\tilde{D}_{L}^k} f(\phi^{-1}(z_{1}),\dots,\phi^{-1}(z_{k}))\prod_{j=1}^{k}\vert (\phi^{-1})'(z_j) \vert W(dz_j).
    \end{equation*}
    It follows from this and the conformal covariance of the functions $(f_{L}^{(+,k)})_{k}$ in Theorem \ref{theorem_spin_correlations} that $S_{L}^{W,\lambda}$ has the same law as
    \begin{equation*}
        1 + \sum_{k \geq 0} \frac{C_{\sigma}^k}{k!}\int_{\tilde{D}_{L}} [\prod_{j=1}^{k}\vert (\phi^{-1})'(z_j) \vert^{-\frac{1}{8}}] f_{\tilde{D}_{L}}^{(+,k)}(z_1,\dots,z_k) \prod_{j=1}^{k} \vert (\phi^{-1})'(z_j) \vert \lambda(\phi^{-1}(z_j)) W(dz_j).
    \end{equation*}
    A similar equality in law holds for $S_{R}^{W,\lambda}$. Since by conformal covariance of SLE$_3$, $\tilde{D}_{L}$ and $\tilde{D}_{R}:=\phi(D_{R})$ have the same law as $D_{L}$ and $D_{R}$, the proof of Proposition \ref{proposition-conformal-covariance-curve} is complete.
\end{proof}

\section{Singularity of the outermost spin loops in the scaling limit}

In this section, we prove Theorem \ref{theorem-singularity-outermost-loops}, using inputs for the lattice near-critical RFIM established in Section \ref{sec: singularity for outermost cle}.

\begin{proof}[Proof of Theorem \ref{theorem-singularity-outermost-loops}]
    The first step of the proof is to show the tightness of
    \begin{equation*}
        (\delta^{-1}\sum_{x \in \Lambda_{\delta}}h_x\mathbb{I}_{S_{\delta}(x)}, \nu_{\delta}^{\oloop,h})_{\delta}.
    \end{equation*}
    The topology in which tightness of $\delta^{-1}\sum_{x \in \Lambda_{\delta}}h_x\mathbb{I}_{S_{\delta}(x)}$ holds is as described at the beginning of Proposition \ref{proposition-uniqueness-law-curve} while tightness of $(\nu_{\delta}^{\oloop})_{\delta}$ is going to be proved in the metric space $(\mathcal{X},d_{\mathcal{X}})$ described in Section \ref{sec-topology-loops}. Tightness of $\delta^{-1}\sum_{x \in \Lambda_{\delta}}h_x\mathbb{I}_{S_{\delta}(x)}$ in fact follows from standard arguments, and we refer the reader to \cite[Section~2.1]{CSZ16} for details. As for tightness of $(\nu_{\delta}^{\operatorname{oloop},h})_{\delta}$, it is established by the following lemma.

    \begin{lemma}\label{lemma-tightness-outermost-loops}
         $(\nu_{\delta}^{\operatorname{oloop},h})_{\delta}$ is tight in $(\mathcal{X}, d_{\mathcal{X}})$.
    \end{lemma}

    Lemma \ref{lemma-tightness-outermost-loops} will be shown just below. With this lemma in hand, it now remains to prove the singularity of any convergent subsequence of $(\nu_{\delta}^{\operatorname{oloop},h})_{\delta}$. In order to do so, for $\iota > 0$, define $M = \lfloor \iota L \rfloor$ and partition $\Lambda_{L}$ into $2M \times 2M$ boxes $B_{1}, \dots, B_{n}$. Denote by $u_{i}$ the center of the box $B_{i}$, so that $B_{i} = \Lambda_{M}(u_i)$ and define the annulus $A_{M}(u_i):= \Lambda_{M/2}(u_i) \setminus \Lambda_{M/4}(u_i)$. For $i=1,\dots,n$, set
    \begin{equation*}
        X_{i} := \mathbb{I}_{\{\text{there exists an outermost loop in $A_{M}(u_i)$}\}}.
    \end{equation*}
    For $\tau \in \{0,1\}^n$, we then define the event $\mathcal{E}_{\tau}:=\{(X_1,\dots, X_n) = \tau\}$. To prove singularity of the law of any of the subsequential limits of $(\nu_{\delta}^{\operatorname{oloop,h}})_{\delta}$ with respect to the law $\nu_{\operatorname{CLE}_3}^{\Lambda_{L}}$ of CLE$_3$ in $\Lambda$, we will show the following proposition. 

    \begin{proposition}\label{prop-observable-singularity-continuum}
        Let $(\nu_{\delta_{k}}^{\operatorname{oloop},h})_{k}$ be a convergent subsequence of $(\nu_{\delta}^{\operatorname{oloop},h})_{\delta}$ and denote by $\nu^{W}$ its limit. Then, there exists $c > 0$ such that for any $\iota > 0$,
        \begin{equation*}
            \frac{1}{2}\EE \bigg[ \sum_{\tau \in \{0,1\}^n} \vert \nu_{\operatorname{CLE}_{3}}^{\Lambda}(\mathcal{E}_{\tau}) - \nu^{W}(\mathcal{E}_{\tau}) \vert \bigg] \geq 1 - c^{-1}\iota^{c}.
        \end{equation*}
    \end{proposition}
    
    Proposition \ref{prop-observable-singularity-continuum} will be proved just below. Observe that $\frac{1}{2}\sum_{\tau \in \{0,1\}^n} \vert \nu_{\operatorname{CLE}_{3}}^{\Lambda}(\mathcal{E}_{\tau}) - \nu^{W}(\mathcal{E}_{\tau}) \vert$ is in fact the total variation distance between the law of $(X_1,\dots,X_n)$ under $\nu_{\operatorname{CLE}_3}^{\Lambda}$ and the law of $(X_1,\dots,X_n)$ under $\nu^{W}$. Thus, singularity of the law of any of the subsequential limits of $(\nu_{\delta}^{\operatorname{oloop,h}})_{\delta}$ with respect to $\nu_{\operatorname{CLE}_3}^{\Lambda}$ follows directly from Proposition \ref{prop-observable-singularity-continuum}. This completes the proof of Theorem \ref{theorem-singularity-outermost-loops}.
\end{proof}

Let us now prove Lemma \ref{lemma-tightness-outermost-loops} and Proposition \ref{prop-observable-singularity-continuum}.

\begin{proof}[Proof of Lemma \ref{lemma-tightness-outermost-loops}]
     The proof is similar to what is done in \cite[Section~4]{continuum-2d-RFIM}. As random measures probability measures on the complete and separable metric space $(\mathcal{X},d_{\mathcal{X}})$ (see Lemma \ref{lemma-measurability-in-XX}), tightness of $(\nu_{\delta}^{\oloop,h})_{\delta}$ will be established once we show that for any $\eps > 0$, there exists a compact set $K \subset \mathcal{X}$ such that
     \begin{equation} \label{condition-tightness}
        \sup_{\delta \in(0,1]} \EE[\nu_{\delta}^{\oloop,h}(K^c)] \leq \eps.
     \end{equation}
     This follows from the tightness criterion for random probability measures of \cite[Theorem~4.10]{Kallenberg-rdm-measures}. So let us prove \eqref{condition-tightness}. Fix $\delta > 0$ and observe that $\nu_{\delta}^{\operatorname{oloop},h}$ has the following Radon-Nikodym derivative with respect to $\nu_{\delta}^{\operatorname{oloop}}$, the law of the outermost spin loops under $\mu_{\delta}^{+}$:
    \begin{equation*}
        \frac{\der \nu_{\delta}^{\operatorname{oloop},h}}{\der \nu_{\delta}^{\operatorname{oloop}}}(\Gamma_{\delta}) = \frac{1}{\mathcal{Z}_{\delta}^{+,h}}\mu_{\delta}^{+}\bigg[\exp\bigg(\delta^{\frac{7}{8}}\sum_{x \in \Lambda_{\delta}} h_x\sigma_x\bigg) \big \vert \Gamma_{\delta} \bigg].
    \end{equation*}
    Above, $\Gamma_{\delta}$ denotes the collection of outermost spin loops. To prove \eqref{condition-tightness}, we first observe that for any compact set $K \subset \mathcal{X}$ and any $\eta > 0$,
    \begin{equation*}
        \EE[\mu_{\delta}^{+,h}[K^{c}]] \leq \PP[\tilde{\ZZ}_{\delta}^{+,h} \leq \eta] + \EE[\mathbb{I}_{\tilde{\ZZ}_{\delta}^{+,h} > \eta}\mu_{\delta}^{+,h}[K^c]]
    \end{equation*}
    where $\tilde{\ZZ}_{\delta}^{+,h}$ is defined as in \eqref{def-tilde-Z-delta} (with $\lambda \equiv 1$ here). Since the limit in law of $\tilde{\ZZ}_{\delta}^{+,h}$ is $\PP$-positive almost surely by \cite[Lemma~2.2]{continuum-2d-RFIM}, for any $\eps >0$, we can choose $\eta >0$ such that
    \begin{equation*}
        \limsup_{\delta \to 0} \PP[\tilde{\ZZ}_{\delta}^{+,h} \leq \eta] \leq \frac{1}{2}\eps.
    \end{equation*}
    To establish \eqref{condition-tightness}, it is therefore enough to upper bound  $\EE[\mu_{\delta}^{+,h}[K^c]\mathbb{I}_{\tilde{\ZZ}_{\delta}^{+,h} > \eta}]$. To do so, we note that for any compact set $K \subset \mathcal{X}$,
    \begin{align*}
        \EE[\mu_{\delta}^{+,h}[K^c]\mathbb{I}_{\tilde{\ZZ}_{\delta}^{+,h} > \eta}] &= \EE\bigg[\frac{\theta(\delta)}{\tilde{\ZZ}_{\delta}^{+,h}}\mathbb{I}_{\tilde{ \ZZ}_{\delta}^{+,h} > \eta} \mu_{\delta}^{+}\bigg[\mathbb{I}_{K^c}\mu_{\delta}^{+}\bigg[ \exp\bigg(\delta^{\frac{7}{8}}\sum_{x \in \Lambda_{\delta}} h_x\sigma_x\bigg) \vert \Gamma_{\delta} \bigg] \bigg] \bigg]\\
        &= \EE\bigg[\frac{\theta(\delta)}{\tilde{\ZZ}_{\delta}^{+,h}}\mathbb{I}_{\tilde{\ZZ}_{\delta}^{+,h} > \eta} \mu_{\delta}^{+}\bigg[\mu_{\delta}^{+}\bigg[\mathbb{I}_{K^c}\exp\bigg(\delta^{\frac{7}{8}}\sum_{x \in \Lambda_{\delta}} h_x\sigma_x\bigg) \vert \Gamma_{\delta} \bigg] \bigg] \bigg]\\
        &\leq \frac{1}{\eta}\mu_{\delta}^{+}\bigg[ \mathbb{I}_{K^c}\theta(\delta)\EE\bigg[ \exp\bigg(\delta^{\frac{7}{8}}\sum_{x \in \Lambda_{\delta}}h_x\sigma_x\bigg) \bigg] \bigg]
    \end{align*}
    where the second equality follows from the measurability of $K^c$ with respect to $\sigma(\Gamma_{\delta})$. Now, by \cite[Lemma~3.5]{continuum-2d-RFIM}, $\theta(\delta)\EE[\exp(\delta^{\frac{7}{8}}\sum h_x\sigma_x)]$ converges to $1$ as $\delta \to 0$, uniformly in $(\sigma_x)_x$. Therefore, we obtain that
    \begin{equation*}
         \EE[\mu_{\delta}^{+,h}[K^c]\mathbb{I}_{\tilde{\ZZ}_{\delta}^{+,h} > \eta}] \leq \frac{2}{\eta}\mu_{\delta}^{+}[K^c].
    \end{equation*}
    A direct consequence of \cite{cvg-nested-loops-Ising} (see Theorem \ref{theorem-cvg-loops-critical} in Section \ref{sec-topology-loops}) is that the law of $(\Gamma_{\delta})_{\delta}$ under $(\mu_{\delta}^{+})_{\delta}$ is tight in $(\mathcal{X}, d_{\mathcal{X}})$. Hence, for any $\eps > 0$, there exists a compact set $K$ such that for any $\delta > 0$, $\mu_{\delta}^{+}[K^c]\leq \frac{1}{4}\eta\eps$. This establishes \eqref{condition-tightness} and concludes the proof of Lemma \ref{lemma-tightness-outermost-loops}. 
\end{proof}

\begin{proof}[Proof of Proposition \ref{prop-observable-singularity-continuum}]
    To prove Proposition \ref{prop-observable-singularity-continuum}, let us first show that
    \begin{align} \label{observable-as-limit-discrete-observable}
        \frac{1}{2}\EE \bigg[ \sum_{\tau \in \{0,1\}^n} \vert \nu_{\operatorname{CLE}_{3}}^{\Lambda}(\mathcal{E}_{\tau}) - \nu^{W}(\mathcal{E}_{\tau}) \vert \bigg] = \lim_{k \to \infty} \frac{1}{2}\EE \bigg[ \sum_{\tau \in \{0,1\}^n} \vert \nu_{\delta_{k}}^{\operatorname{oloop}}(\mathcal{E}_{\tau}^{\delta_{k}}) - \nu_{\delta_{k}}^{\operatorname{oloop},h}(\mathcal{E}_{\tau}^{\delta_{k}}) \vert \bigg]
    \end{align}
    where $\mathcal{E}_{\tau}^{\delta_{k}}$ is defined analogously to $\mathcal{E}_{\tau}$, but for the collection of outermost loops $\Gamma_{\delta_{k}}$ in $D_{\delta_{k}}$ and with respect to the annuli $(A_{M,\delta_{k}}(u_i))_i:= (A_{M}(u_i) \cap \delta_{k} \mathbb{Z}^2)_i$. The equality \eqref{observable-as-limit-discrete-observable} is in fact a consequence of the following lemmas.

    \begin{lemma}\label{lemma-continuity-crossing-events}
        For $Q$ a topological rectangle, denote by $\mathcal{C}^{\pm}(Q)$ the event that there is a $\pm$-interface joining two opposite sides of $Q$, without hitting $\partial Q$ otherwise. Then, for any $Q \subset \Lambda$,
        \begin{equation*}
            \lim_{k \to \infty} \nu_{\delta_{k}}^{\operatorname{oloop}}[\mathcal{C}^{\pm}(Q)] = \nu_{\operatorname{CLE}_3}^{\Lambda}[\mathcal{C}^{\pm}(Q)].
        \end{equation*}
    \end{lemma}

    \begin{lemma}\label{lemma-abs-continuity-annealed}
        Let $(\nu_{\delta_k}^{\operatorname{oloop},h})_k$ be a convergent subsequence of $(\nu_{\delta}^{\operatorname{oloop},h})_{\delta}$ and denote by $\nu^{W}$ its limit. Then the annealed measure $\EE[\nu^{W}[\cdot]]$ on $(\mathcal{X},d_{\mathcal{X}})$ is absolutely continuous with respect to $\nu_{\operatorname{CLE}_3}^{\Lambda}$.
    \end{lemma}

    The proofs of Lemma \ref{lemma-continuity-crossing-events} and Lemma \ref{lemma-abs-continuity-annealed} are given just below. Lemma \ref{lemma-continuity-crossing-events} shows continuity of the crossing events. This in particular implies that $\nu_{\delta_{k}}^{\oloop}(\mathcal{E}_{\tau}^{\delta_{k}})$ converges to $\nu_{\operatorname{CLE}_3}^{\Lambda}(\mathcal{E}_{\tau})$ as $k \to \infty$. Indeed, for each $i$, one can divide $A_{M}^{\delta_k}(u_i)$ into eight boxes $(B_{M}^{\delta_{k},j}(u_i))_j$ such that for a loop in $A_m^{\delta_k}(u_i)$ to touch $\partial A_m^{\delta_k}(u_i)$ without crossing $\partial A_m^{\delta_k}(u_i)$, there must be a box $B_m^{\delta_k,j}(u_i)$ such that one of the two events denoted $\mathcal{C}_{\operatorname{out}}^{\xi}(B_m^{\delta_k, j}(u_i))\setminus \mathcal{C}_{\operatorname{in}}^{\xi}(B_m^{\delta_k, j}(u_i))$, $\xi = \pm$, in the proof of Lemma \ref{lemma-continuity-crossing-events} occurs. But Lemma \ref{lemma-continuity-crossing-events} shows that these events have $\nu_{\delta_k}^{\oloop}$-probability converging to $0$ as $k \to \infty$. Therefore, using the absolutely continuity of $\PP[\nu^{W}[\cdot]]$ with respect to $\nu_{\operatorname{CLE}_3}^{\Lambda}$ from Lemma \ref{lemma-abs-continuity-annealed}, we obtain that for any $\tau \in \{0,1\}^n$,
    \begin{align*}
        \lim_{k \to \infty} &\EE\big[\vert \nu_{\delta_{k}}^{\operatorname{oloop}}[(X_1^{\delta_k},\dots,X_n^{\delta_k}) = \tau \big] -  \nu_{\delta_{k}}^{\operatorname{oloop},h}[(X_1^{\delta_k},\dots,X_n^{\delta_k}) = \tau]\vert \big] \\
        =& \EE[\vert \nu_{\operatorname{CLE_3}}^{\Lambda}[(X_1,\dots,X_n) = \tau] - \nu^{W}[(X_1,\dots,X_n) = \tau] \vert].
    \end{align*}
    Since the sum over $\tau \in \{0,1\}^n$ in \eqref{observable-as-limit-discrete-observable} is a finite sum, the limit \eqref{observable-as-limit-discrete-observable} follows from the above equality.

    With the equality \eqref{observable-as-limit-discrete-observable} at hand, it now suffices to prove that there exists $c > 0$ and $\delta_0$ such that for any $0 < \delta < \delta_0$ and any $\iota > 0$,
    \begin{equation*}
        \frac{1}{2}\EE \bigg[ \sum_{\tau \in \{0,1\}^n} \vert \nu_{\delta_{k}}^{\operatorname{oloop}}(\mathcal{E}_{\tau}^{\delta_{k}}) - \nu_{\delta_{k}}^{\operatorname{oloop},h}(\mathcal{E}_{\tau}^{\delta_{k}}) \vert \bigg] \geq 1-c^{-1}\iota^{c}.
    \end{equation*}
    This inequality is established by Theorem \ref{thm: cle outermost singularity} in Section \ref{sec: singularity for outermost cle} in Part \ref{part-2}. There, $N$ plays the role of $\delta_{k}^{-1}$ and the underlying domain $\Lambda_{\delta}$ is denoted by $\Lambda_{N}$ of side-length $2N$. Besides, in Theorem \ref{thm: cle outermost singularity}, the laws of the loops under $\mu_{\Lambda_{N},0}^{+}$ and under $\mu_{\Lambda_{N},\eps h}^{+}$ are the same as under $\nu_{\delta}^{\oloop}$ and $\nu_{\delta}^{\operatorname{oloop},h}$. This completes the proof of Theorem \ref{theorem-singularity-outermost-loops}.
\end{proof}

Let us now prove Lemma \ref{lemma-continuity-crossing-events} and Lemma \ref{lemma-abs-continuity-annealed}.

\begin{proof}[Proof of Lemma \ref{lemma-continuity-crossing-events}]
We only discuss the event \(\mcc C^+(Q)\), since the proof for \(\mcc C^-(Q)\) is identical.

Let \(\mathcal X\) be the space of loop collections equipped with the metric
\(d_{\mathcal X}\) defined in \cite{cvg-nested-loops-Ising}. By the convergence theorem of \cite{cvg-nested-loops-Ising}
 for critical Ising interfaces, along the subsequence
\((\delta_k)\) under consideration we have
\[
  \nu_{\delta_k}^{\mathrm{oloop}}
  \Longrightarrow
  \nu_{\mathrm{CLE}_3}^{\Lambda}
  \qquad\text{weakly on }(\mathcal X,d_{\mathcal X}).
\]
It is therefore enough to prove that \(\mcc C^+(Q)\) is a
\(\nu_{\mathrm{CLE}_3}^{\Lambda}\)-continuity event, namely
\[
  \nu_{\mathrm{CLE}_3}^{\Lambda}\bigl(\partial \mcc C^+(Q)\bigr)=0.
\]
Indeed, the desired convergence then follows immediately from the Portmanteau
theorem.

We now prove the continuity statement. For \(\eta>0\), define an inner approximation \(\mcc C^+_{\mathrm{in},\eta}(Q)\) of \(\mcc C^+(Q)\) as the event that there exists an oriented loop \(\gamma\) and a connected component
\(\alpha\) of \(\gamma\cap Q\) which joins the two prescribed opposite sides of \(Q\), has the correct orientation, and satisfies the following \(\eta\)-robustness conditions: the two endpoints of \(\alpha\) lie at a distance of at least \(\eta\) from the four corners of \(Q\), and \(\alpha\), away from small \(\eta\)-neighbourhoods of its endpoints, stays at a distance of at least
\(\eta\) from \(\partial Q\). Similarly, define an outer approximation \(\mcc C^+_{\mathrm{out},\eta}(Q)\) as the event that there exists an oriented loop whose trace has an \(\eta\)-neighbourhood containing a path in \(Q\) joining the two prescribed opposite sides with the correct orientation.

Then
\[
  \mcc C^+_{\mathrm{in},\eta}(Q)
  \subset \mcc C^+(Q)
  \subset \mcc C^+_{\mathrm{out},\eta}(Q).
\]
Moreover, \(\mcc C^+_{\mathrm{in},\eta}(Q)\) is open in \((\mathcal X,d_{\mathcal X})\).
Indeed, an \(\eta\)-robust crossing remains a crossing under sufficiently small
uniform perturbations of the corresponding loop. Likewise,
\(\mcc C^+_{\mathrm{out},\eta}(Q)\) is closed, since an \(\eta\)-approximate
crossing is preserved under limits of loop collections in \(d_{\mathcal X}\).

Hence, by weak convergence,
\[
  \liminf_{k\to\infty}
  \nu_{\delta_k}^{\mathrm{oloop}}[\mcc C^+(Q)]
  \ge
  \nu_{\mathrm{CLE}_3}^{\Lambda}[\mcc C^+_{\mathrm{in},\eta}(Q)]
\]
and
\[
  \limsup_{k\to\infty}
  \nu_{\delta_k}^{\mathrm{oloop}}[\mcc C^+(Q)]
  \le
  \nu_{\mathrm{CLE}_3}^{\Lambda}[\mcc C^+_{\mathrm{out},\eta}(Q)].
\]
It remains to let \(\eta\downarrow0\).

We claim that
\[
  \lim_{\eta\downarrow0}
  \nu_{\mathrm{CLE}_3}^{\Lambda}
  \bigl[
    \mcc C^+_{\mathrm{out},\eta}(Q)
    \setminus \mcc C^+_{\mathrm{in},\eta}(Q)
  \bigr]
  =0.
\]
Indeed, if a loop collection belongs to
\(\mcc C^+_{\mathrm{out},\eta}(Q)\setminus \mcc C^+_{\mathrm{in},\eta}(Q)\), then some
macroscopic CLE\(_3\) loop produces a crossing of \(Q\) which is not stable
under an \(O(\eta)\)-perturbation. This can happen only if, at scale \(\eta\),
one of the following degenerate configurations occurs near \(\partial Q\):
the loop hits an \(O(\eta)\)-neighbourhood of a corner of \(Q\); the loop has a
near-tangential contact with one of the sides of \(Q\); or two distinct pieces of
the loop approach each other and the boundary of \(Q\) within distance
\(O(\eta)\) in a way that changes the crossing topology.

Each of these alternatives implies a multi-arm event from scale \(\eta\) to a
macroscopic scale. By the standard arm estimates for critical Ising interfaces,
equivalently for their CLE\(_3\) scaling limit as obtained in
Benoist--Hongler, the probability of such a degenerate event tends to zero as
\(\eta\downarrow0\). Therefore
\[
  \nu_{\mathrm{CLE}_3}^{\Lambda}\bigl(\partial \mcc C^+(Q)\bigr)=0.
\]

Letting \(\eta\downarrow0\) in the preceding liminf and limsup estimates gives
\[
  \lim_{k\to\infty}
  \nu_{\delta_k}^{\mathrm{oloop}}[\mcc C^+(Q)]
  =
  \nu_{\mathrm{CLE}_3}^{\Lambda}[\mcc C^+(Q)].
\]
This proves the claim for \(\mcc C^+(Q)\). The proof for \(\mcc C^-(Q)\) is identical.
\end{proof}

\begin{proof}[Proof of Lemma \ref{lemma-abs-continuity-annealed}]
    The proof is similar to that of the absolutely continuity part of \cite[Theorem~1.2]{continuum-2d-RFIM}. Indeed, recall that for fixed $\delta_k > 0$, we have that
    \begin{equation*}
        \frac{\der \nu_{\delta_{k}}^{\oloop,h}}{\der \nu_{\delta_{k}}^{\oloop}}(\Gamma_{\delta_{k}}) = \frac{\theta(\delta_{k})}{\theta(\delta_{k})\mathcal{Z}_{\delta_{k}}^{+,h}}\mu_{\delta_{k}}^{+}\bigg[ \exp\bigg( \delta^{\frac{7}{8}}\sum_{x \in \Lambda_{\delta_{k}}} h_x\sigma_x \bigg) \vert \Gamma_{\delta_{k}}\bigg].
    \end{equation*}
    Let $A$ be an event in the Borel $\sigma$-algebra of $(\mathcal{X}, d_{\mathcal{X}})$. We have that
    \begin{align*}
        \EE\big[ \theta(\delta_k)\mathcal{Z}_{\delta_{k}}^{+,h} \nu_{\delta_{k}}^{\oloop,h}[A] \big] &= \theta(\delta_{k}) \EE\bigg[ \mu_{\delta_{k}}^{+}\bigg[ \mathbb{I}_{A} \mu_{\delta_{k}}^{+}\bigg[ \exp\bigg( \delta_{k}^{\frac{7}{8}}\sum_{x \in \Lambda_{\delta_{k}}}h_x\sigma_x \bigg) \vert \Gamma_{\delta_{k}}\bigg] \bigg] \bigg]\\
        &= \theta(\delta_{k}) \EE\bigg[ \mu_{\delta_{k}}^{+}\bigg[ \mathbb{I}_{A} \exp\bigg( \delta_{k}^{\frac{7}{8}}\sum_{x \in \Lambda_{\delta_{k}}}h_x\sigma_x \bigg) \bigg]\bigg] \\
        &=\theta(\delta_{k}) \mu_{\delta}^{+}\bigg[ \mathbb{I}_{A} \EE\big[ \exp\big( \delta_{k}^{\frac{7}{8}}\sum_{x \in \Lambda_{\delta_{k}}}h_x\sigma_x \big) \big] \bigg]\\
        &=\theta(\delta_{k}) \mu_{\delta}^{+}\big[ \mathbb{I}_{A} \exp(\frac{\delta_k^{\frac{7}{4}}}{2} \vert \Lambda_{\delta_{k}} \vert) \big]
    \end{align*}
    which proves that
    \begin{equation*}
        \lim_{\delta_{k} \to 0} \EE\big[ \theta(\delta_k)\mathcal{Z}_{\delta_{k}}^{+,h} \nu_{\delta_{k}}^{\oloop,h}[A] \big] = \nu_{\operatorname{CLE}_3}^{\Lambda}[A].
    \end{equation*}
    On the other hand, by weak convergence of $\nu_{\delta_{k}}^{\oloop,h}$ to $\nu^{W}$ and of $\theta(\delta_{k})\mathcal{Z}_{\delta_{k}}^{+,h}$ to $\mathcal{Z}_{\Lambda}^{+,W}$, we have that
    \begin{equation*}
        \lim_{\delta_{k} \to 0} \EE\big[ \theta(\delta_k)\mathcal{Z}_{\delta_{k}}^{+,h} \nu_{\delta_{k}}^{\oloop,h}[A] \big] = \EE[\mathcal{Z}_{\Lambda}^{+,W}\nu^{W}[A]].
    \end{equation*}
    Since $\mathcal{Z}_{\Lambda}^{+,W} > 0$ $\PP$-almost surely by \cite[Lemma~2.2]{continuum-2d-RFIM}, for any measurable event $F$ with $\nu_{\operatorname{CLE}_3}^{\Lambda}[F]=0$, we can find a set $E_F$ with $\PP[W \in E_F] = 1$ such that $\nu^{W}[F]=0$ for all $W \in E_F$. Therefore, $\EE[\nu^{W}[F]] = 0$, and thus $\EE[\nu^{W}[\cdot]]$ is absolutely continuous with respect to $\nu_{\operatorname{CLE}_3}^{\Lambda}$, which completes the proof of Lemma \ref{lemma-abs-continuity-annealed}. Let us remark as in \cite{continuum-2d-RFIM} that this does not contradict the claim that almost surely, $\nu^{W}$ is singular with respect to $\nu_{\operatorname{CLE}_3}^{\Lambda}$ because $\cap E_{F}$ will not be a measurable event with positive probability.
\end{proof}

\section{Singularity of the nested spin loops in the scaling limit} \label{sec-singularity-nested-loops}

In this section, we prove tightness of the laws of the nested collection of spin loops under $(\mu_{\delta}^{+,h})_{\delta}$ and establish that any subsequential limit thus obtained is $\PP$-almost surely singular with respect to nested CLE$_3$.

\begin{proposition} \label{prop-singularity-nested-loops}
    Let $D$ be a bounded, open and simply connected domain with smooth boundary and assume that $(D_{\delta})_{\delta}$ is a sequence of connected subsets of $\delta \mathbb{Z}^2$ satisfying the assumptions of Section \ref{subsection-limit-critical-loops}. Let $H:=(H_x)_{x \in \mathbb{Z}^2}$ be i.i.d Gaussian random variables with mean $0$ and variance $1$. Denote the law of $H$ by $\PP$. Define $h:=(h_x)_{x \in D_{\delta}}$ via $h_{x}=H_{x/\delta}$. Let $\nu_{\delta}^{\operatorname{loop},h}$ be the law of the collection of nested loops separating $+1$ and $-1$ spins in the random-field Ising model on $D_{\delta}$ with external field $\delta^{\frac{7}{8}}h_{x}$ and $+1$ boundary conditions. Then, the joint law of $(\delta^{-1}\sum_{x \in D_{\delta}}h_{x}\mathbb{I}_{S_{\delta}(x)},\nu_{\delta}^{\operatorname{loop},h})_{\delta}$ is tight as $\delta \to 0$ and any subsequential limit of $(\nu_{\delta}^{\operatorname{loop},h})_{\delta}$ is $\PP$--almost surely singular with respect to the law of nested CLE$_3$ in $D$.
\end{proposition}

Above, tightness of the laws $(\nu_{\delta}^{\operatorname{loop},h})_{\delta}$ is proved in the topology on the space of loops $(\mathcal{X},d_{\mathcal{X}})$ described in Section \ref{sec-topology-loops}. Tightness of $(\delta^{-1}\sum_{x \in D_{\delta}}h_{x}\mathbb{I}_{S_{\delta}(x)})_{\delta}$ is established in the same topology as the one described in the proof of Lemma \ref{lemma-joint-convergence}.

To prove Proposition \ref{prop-singularity-nested-loops}, we start with the following lemma, which establishes tightness of $(\nu_{\delta}^{\operatorname{loop},h})_{\delta}$. For later reference, we denote by $\nu_{\delta}^{\operatorname{loop}}$ the law of $\Gamma_{\delta}$ under $\mu_{\delta}^{+}$.

\begin{lemma}\label{lemma-tightness-nested-loops}
    $(\nu_{\delta}^{\operatorname{loop},h})_{\delta}$ is tight in $(\mathcal{X},d_{\mathcal{X}})$.
\end{lemma}

\begin{proof}[Proof of Lemma \ref{lemma-tightness-nested-loops}]
    The proof is similar to the proof of Lemma \ref{lemma-tightness-outermost-loops} and details are left to the reader.
\end{proof}

With Lemma \ref{lemma-tightness-nested-loops} in hand, let us now turn to the proof of Proposition \ref{prop-singularity-nested-loops}.

\begin{proof}[Proof of Proposition \ref{prop-singularity-nested-loops}]
    Tightness of $(\Gamma_{\delta})_{\delta}$ under $(\mu_{\delta}^{+,h})_{\delta}$ in the topology described in Section \ref{sec-topology-loops} is established by Lemma \ref{lemma-tightness-nested-loops}. On the other hand, tightness of $(\delta^{-1}\sum_{x \in D_{\delta}}h_{x}\mathbb{I}_{S_{\delta}(x)})_{\delta}$ in the topology described in Lemma \ref{lemma-joint-convergence} follows from standard arguments, see \cite[Section~2.1]{CSZ16}. We immediately get from this that $(\delta^{-1}\sum_{x \in D_{\delta}}h_{x}\mathbb{I}_{S_{\delta}(x)})_{\delta}, \nu_{\delta}^{\operatorname{loop},h})_{\delta}$ is tight. To prove Proposition \ref{prop-singularity-nested-loops}, it thus only remains to prove the singularity of any subsequential limit. Let $(\nu_{\delta_k}^{\operatorname{loop},h})_{k}$ be a convergent subsequence. Below, for each $k$, we abbreviate $\Gamma_{\delta_k}$ by $\Gamma_{k}$ and $\mu_{\delta_{k}}^{+}$, respectively $\mu_{\delta_{k}}^{+,h}$, by $\mu_{k}^{+}$, respectively $\mu_{k}^{+,h}$. We first observe that $\mu_{\delta}^{+}$-almost surely, for any $x \in D_k$,
    \begin{equation*}
        \sigma_{x} = (-1)^{\vert \Gamma_{k}(x) \vert}
    \end{equation*}
    where $\Gamma_{k}(x)$ denotes the set of loops in $\Gamma_k$ that surround $x$. This equality also holds $\PP \otimes \mu_{\delta}^{+,h}$-almost surely. Once this observation has been made, singularity of the limiting law of $(\nu_{\delta_{k}}^{\operatorname{loop},h})_k$ with respect to nested CLE$_3$ in $D$ can be established using a similar strategy as that used to show that the law of the magnetization field of the $2d$ near-critical RFIM is singular with respect to that of the magnetization field of the $2d$ critical Ising model established in \cite{continuum-2d-RFIM}. Given this similarity, let us now only briefly explain how the proof can be adapted, leaving the details to the reader. Below, we keep the same notations as in the proof of \cite[Theorem~1.2]{continuum-2d-RFIM}.

    The families of random variables $(\Phi_{i,j}^{k,N})_{i,j}$ and $(\Phi_{i,j}^{k,N,m})_{i,j}$ are defined as in \cite{continuum-2d-RFIM} but by expressing them using the random variables $((-1)^{\vert \Gamma_{k}(x) \vert})_{x \in D_k}$. That is, for each $k, N, m \in \mathbb{N}$, we set
    \begin{align*}
        \Phi_{i,j}^{k,N} := \sum_{x \in B_{i,j}^{N}} \delta_{k}^{\frac{15}{8}} (-1)^{\vert \Gamma_{k}(x) \vert}, \quad \Phi_{i,j}^{k,N,m} := 2^{-m}\lfloor 2^m\Phi_{i,j}^{k,N} \rfloor
    \end{align*}
    where $B_{i,j}^{N} := (\frac{i-1}{N}, \frac{i}{N}) \times (\frac{j-1}{N}, \frac{j}{N})$. The families of random variables $(W_{i,j}^{k,N})_{i,j}$ and $(W_{i,j}^{k,N,m})_{i,j}$ are defined as in \cite{continuum-2d-RFIM} and the $\sigma$-algebras $\mathcal{F}_{N}^{k}$ and $\mathcal{F}_{N,m}^{k}$ are constructed using the families $(\Phi_{i,j}^{k,N}, W_{i,j}^{k,N,m})_{i,j}$ and $(\Phi_{i,j}^{k,N,m}, W_{i,j}^{k,N,m})_{i,j}$, respectively. We then have that
    \begin{align*}
        Q_{N}^{k} := \EE \nu_{\delta_{k}}^{\operatorname{loop}} \bigg[ \frac{\der \nu_{\delta_k}^{\operatorname{loop},h}}{\der \nu_{\delta_k}^{\operatorname{loop}}}(\Gamma_{k}) \vert \mathcal{F}_{N}^{k}\bigg] &= \EE \mu_{k}^{+} \bigg[ \frac{\der \mu_{k}^{+,h}}{\der \mu_{k}^{+}}(\Gamma_k)\vert \mathcal{F}_{N}^{k}\bigg]\\
        &= \EE \mu_{k}^{+} \bigg[ \frac{\theta(\delta_{k})}{\tilde{\ZZ}_{\delta_{k}}^{+,h}} \mu_{k}^{+}\bigg[ \exp\bigg(\delta_{k}^{\frac{7}{8}}\sum_{x \in D_k} h_x\sigma_x \bigg) \vert \Gamma_{k} \vert \bigg] \vert \mathcal{F}_{N}^{k} \bigg] \\
        &= \EE \mu_{k}^{+} \bigg[ \frac{\theta(\delta_{k})}{\tilde{\ZZ}_{\delta_{k}}^{+,h}} \mu_{k}^{+}\bigg[ \exp\bigg(\delta_{k}^{\frac{7}{8}}\sum_{x \in D_k} h_x(-1)^{\vert \Gamma_k(x) \vert} \bigg) \vert \Gamma_{k} \vert \bigg] \vert \mathcal{F}_{N}^{k} \bigg]\\
        &= \EE \mu_{k}^{+}\bigg[ \frac{\theta(\delta_{k})}{\tilde{\ZZ}_{\delta_{k}}^{+,h}} \exp\bigg(\delta_{k}^{\frac{7}{8}}\sum_{x \in D_k} h_x(-1)^{\vert \Gamma_k(x) \vert} \bigg) \vert \mathcal{F}_{N}^{k} \bigg].
    \end{align*}
    Similarly,
    \begin{align*}
        Q_{N,m}^{k} &:= \EE \nu_{\delta_{k}}^{\operatorname{loop}} \bigg[ \frac{\der \nu_{\delta_k}^{\operatorname{loop},h}}{\der \nu_{\delta_k}^{\operatorname{loop}}}(\Gamma_{k}) \vert \mathcal{F}_{N,m}^{k}\bigg]\\
        &=\EE \mu_{k}^{+} \bigg[ \frac{\der \mu_{k}^{+,h}}{\der \mu_{k}^{+}} (\Gamma_k)\vert \mathcal{F}_{N,m}^{k}\bigg] = \EE \mu_{k}^{+}\bigg[ \frac{\theta(\delta_{k})}{\tilde{\ZZ}_{\delta}^{+,h}} \exp\bigg(\delta^{\frac{7}{8}}\sum_{x \in D_k} h_x(-1)^{\vert \Gamma_k(x)\vert} \bigg) \vert \mathcal{F}_{N,m}^{k} \bigg].
    \end{align*}
    As in \cite{continuum-2d-RFIM}, to prove singularity of the limiting law of $(\Gamma_{k})_{k}$ under $(\mu_{k}^{+,h})_k$ with respect to nested CLE$_3$ in $D$, it suffices to show that
    \begin{equation*}
        \lim_{N \to \infty} \lim_{m \to \infty} \lim_{k \to \infty} \EE \mu_{k}^{+}[(Q_{N,m}^{k})^{\frac{1}{2}}] = 0.
    \end{equation*}
    This can be shown exactly as \cite[Lemma~5.2]{continuum-2d-RFIM}, by simply replacing the random variables $(\sigma_x)_x$ by $((-1)^{\vert \Gamma_k(x) \vert})_x$ everywhere. We leave the details to the reader.
\end{proof}

\part{Results on the lattice near-critical RFIM interfaces} \label{part-2}
From now on, we will focus on the (discrete) RFIM on the two-dimensional integer lattice $\mbb Z^2$.
\section{Background on the discrete RFIM}
\subsection{Notations}
For $\Gamma\subset \mathbb{Z}^2$, we use the notation $\partial\Gamma$ to denote its interior boundary. That is, $\partial \Gamma=\{u\in\Gamma\mid u\sim v \text{ for some } v\in \Gamma^c \}$.  In addition, we define $\partial_{\mathrm{ext}}\Gamma$ to denote its exterior boundary. That is, $\partial_{\text{ext}} \Gamma=\{u\notin\Gamma\mid u\sim v \text{ for some } v\in \Gamma \}$. For any configuration $\omega\in\mathbb{R}^{\Gamma}$ and any region $\Gamma'\subset \Gamma$, we use the notation $\omega|_{\Gamma'}\in \mbb R^{\Gamma'}$ to denote its restriction on $\Gamma'$, i.e., $\omega|_{\Gamma'}(u)=\omega_u$ for any $u\in \Gamma'$. 

We will use the notation $\mu_{G,h}^{A,\xi}$ to denote the Ising measure on a graph $G=(V,E)$ with boundary condition $\xi$ on $A\subset V$ and with the external field $h$. This is a natural extension of \eqref{def_Ising_pm}. We define the random field Ising model Hamiltonian with external field $\eps h$ to be \begin{equation*}
H_{G, \epsilon h} (\sigma)= - \left(\sum_{\substack{u,v \in G,\\\, u\sim v}} \sigma_u \sigma_v  + \sum_{u \in G} \epsilon h_u \sigma_u\right) \mbox{ for } \sigma \in \{-1, 1\}^{V}\,.
\end{equation*}Then the partition function is defined as \begin{equation*}
    \mathcal{Z}^{\xi,A}_{\beta,G,\eps h}=\sum_{\sigma\in\{-1,1\}^V}\exp\left(-\beta H_{G,\eps h}(\sigma)\right)\mathbb{I}_{\sigma|_A=\xi}
\end{equation*} and the measure is defined as \begin{equation*}
    \mu^{\xi,A}_{\beta,G,\eps h}(\sigma)=\frac{1}{\mathcal{Z}^{\xi,A}_{\beta,G,\eps h}}\exp\left(-\beta H_{G,\eps h}(\sigma)\right)\mathbb{I}_{\sigma|_A=\xi}.
\end{equation*}
\begin{remark}
    In the above definition, the external field term is multiplied by $\beta$. This convention is slightly different from the one used in Part I, where the external field is not multiplied by $\beta$. Since $\beta$ is fixed throughout the paper, the two conventions differ only by a deterministic rescaling of the external field, and all the results remain unchanged after adjusting the corresponding constants.
\end{remark}
We may drop $G$ and $A$ in the notation when they are clear from the context. For any subset $S\subset V$ and any external field $h$, denote by $h_{S} = \sum_{v\in S} h_v$.

For any two vertices $x,y\in\mathbb Z^2$, let $\dist(x,y)$ denote the $\ell^\infty$-distance between them. In addition, for two sets of vertices $A$ and $B$, define
\[
\dist(A,B):=\min_{x\in A,\;y\in B}\dist(x,y).
\]
We write $\dist(x,B)$ for $\dist(\{x\},B)$.
For two vertices $x,y\in\mathbb Z^2$, we say that they are neighboring if $\{x,y\}\in E(\mathbb Z^2)$, and $*$-neighboring if $\dist(x,y)=1$.
For two distinct vertices $x,y\in\mathbb Z^2$, a sequence of distinct vertices
$
(x_0=x,x_1,\ldots,x_n=y)
$
is called a ($*$-)path from $x$ to $y$ if $x_i$ and $x_{i+1}$ are ($*$-)neighbors for every $i=0,\ldots,n-1$.
Furthermore, for a set $A\subset\mathbb Z^2$, we say that $x$ and $y$ are ($*$-)connected in $A$ if there exists a ($*$-)path from $x$ to $y$ entirely contained in $A$.

A sequence of distinct vertices
$u_1,u_2,\ldots,u_n,u_{n+1}=u_1$
is called a circuit if
$
\{u_i,u_{i+1}\}\in E(\mathbb Z^2)
$
for every $i=1,\ldots,n$. Similarly, it is called a $*$-circuit if
$
\dist(u_i,u_{i+1})=1
$
for every $i=1,\ldots,n$. Note that every circuit is also a $*$-circuit.

For a $*$-circuit $(u_1,\ldots,u_n)$, define the region enclosed by it as
\[
\Gamma:=\Bigl\{x:\text{ every path from $x$ to $\infty$ intersects the circuit }\Bigr\}.
\]

For $N\geq 1$, let $\lamn=[-N,N]^2 \cap \mathbb Z^2$ be the box of side length $2N+1$ centered at the origin $o$. More generally, for any $u\in \mathbb Z^2$ we use the notation $\Lambda_N(u)$ to denote the box of side length $2N+1$ centered at the origin $u$, and we refer to it as the $N$-box centered at $u$. Throughout the paper, unless otherwise specified, boundary conditions for the Ising model on $\Lambda_N$ are imposed on $\partial\Lambda_N$.

\subsection{The FK-Ising model}
In this subsection, we briefly review the FK-Ising model with an external field, introduced in \cite{ES88}. Let $G=(V,E)$ be a finite graph and let $\{h_x:x\in V\}$ denote an external field. We consider an edge configuration $\omega\in\Omega=\{0,1\}^E$, where $0$ indicates that an edge is \textbf{closed} and $1$ indicates that it is \textbf{open}. 

For $\omega\in\Omega$ and distinct vertices $x,y\in V$, we say that $x$ is \textbf{connected} to $y$ in $\omega$, denoted by $x\longleftrightarrow y$, if there exists a sequence of vertices
$
x_0=x,x_1,\cdots,x_n=y
$
such that $\{x_i,x_{i+1}\}\in E$ and $\{x_i,x_{i+1}\}$ is open in $\omega$ for every $0\le i\le n-1$. Such a sequence is called an \textbf{open path}. Given $\omega\in\Omega$, the graph $G$ is partitioned into a disjoint collection of connected components, which we refer to as \textbf{open clusters}.

Let $\mathfrak C$ denote the collection of the open clusters in $\omega$.
Then the FK-Ising model on $G$ with parameter $p \equiv 1-\exp(-2\beta)$ is a probability measure on  $\Omega$ given by (recalling our notational convention $h_S = \sum_{v\in S} h_v$)

\begin{equation}\label{eq-def-FK-external-field}
\phi^{\xi}_{p,G,h}(\omega)=\frac{1}{\cZ^{\xi,\phi}_{p,G,h}}  \prod_{e\in E}\hspace{0.5em}p^{\omega(e)}(1-p)^{1-\omega(e)}\prod_{\mathcal C\in \mathfrak C} 2 \cosh(\beta h_{\mathcal C})\,,
\end{equation}
where $\cZ^{\xi,\phi}_{p,G,h}$ is the normalizing constant (partition function) of $\phi^{\xi}_{p,G,h}$ and $\xi$ denotes the boundary condition.

In \cite{Onsager44} (see also \cite{DHN11,spinCorrelations,GW20}), it was shown that for {$d=2$ and} $ p = p_c=1-\exp(-2 \beta_c)$, there exists a constant $c>0$ such that for any boundary condition $\xi$,
\begin{equation}\label{eq: FK one arm event exponent}
c^{-1}n^{-\frac{1}{8}}\le \phi_{\beta_c,\Lambda_n,0}^{\xi}(o\longleftrightarrow\partial\Lambda_n)\leq cn^{-\frac{1}{8}}.
\end{equation}

\subsection{The extended Ising model}\label{sec: extended ising}

In this subsection, we introduce an extension of the Ising model in which spin variables are assigned not only to vertices but also to edges, taking values in $\{-1,0,1\}$. This construction is reminiscent of the extension of the Gaussian free field to metric graphs, which has led to important applications. The idea was initiated in \cite{Tit16}, introduced for the Ising model in \cite{AHP20}, and has since found further applications; see also \cite{DHX23}. 

For a graph $G=(V,E)$, we use $\bar{\sigma}$ to denote an extended configuration on $G$, satisfying
\begin{enumerate}[label=(\arabic*)]
\item for every vertex $v\in V(G)$, $\bar{\sigma}_v\in\{-1,+1\}$;
\item for every edge $e\in E(G)$, $\bar{\sigma}_e\in\{-1,0,+1\}$;
\item\label{hard constraint}
if $v$ is an endpoint of $e$ (denoted by $v\in e$), then
$|\bar{\sigma}_v-\bar{\sigma}_e|\le 1.$

\end{enumerate}

 For any external field $h$, let $\bar{\mu}_{\beta,G,h}$ be the unique probability measure on the set of extended configurations $\bar{\sigma}$ with
\begin{equation}\label{eq:P_Lambda_tau_def}
  \bar{\mu}_{\beta,G,h}(\bar{\sigma}) := \frac{1}{\bar{\cZ}_{\beta,G,h}} \prod_{ \substack{e \in {E}(G) \\ v \in e}}  W(\bar{\sigma}_v,\bar{\sigma}_{e}) \cdot  e^{ U(\bar\sigma)}
\end{equation}
where
$U(\bar\sigma) = \beta\sum_{v \in {V}(G)}  h_v \sigma_v$ and $W$ is defined as follows.
For each $a \in\{-1,1\}$ and $ b\in\{-1,0,1\}$,
\begin{equation}\label{eq:Wdef}
    W(a,b):= \lambda \Big( \delta _{a,b}  \ + \ t \delta_{b,0}  \Big)\ = \  \lambda  \cdot \begin{cases}
    1& b=a,\\
    t& b =0,\\
    0&b =-a,
  \end{cases}
\end{equation}
and $(t, \lambda)$ is related to the inverse temperature $\beta$ by
\begin{equation*} \label{t_beta}
  t=\Big(\exp(2\beta)-1\Big)^{-\frac{1}{2}} \,,  \qquad \lambda =  \Big(2 \sinh(\beta)\Big)^{\frac{1}{2}} \,.
\end{equation*}

For $A\subset\bar G$, we write $\bar{\mu}_{\beta,G,h}^{A,+}$ ($\bar{\mu}_{\beta,G,h}^{A,-}$) to denote the extended Ising measure with the plus (minus) boundary condition on $A$. More generally, for any boundary condition $\xi\in \{-1,1\}^{A \cap V(G)}\times \{-1,0,1\}^{A \cap E(G)}$, we denote  the extended Ising measure with the boundary condition $\xi$ on $A$ by $\bar \mu^{A, \xi}_{\beta,G,h}$. For notational clarity, we might omit $A$ from the superscript in the case that the boundary condition is posed on $\partial \bar{G}$.
When considering the extended Ising model on $\bar G$ where $G$ is a subgraph of $\mathbb{Z}^2$, we define the boundary of $\Gamma\subset\bar G$ by $$\partial\Gamma=\{e={\{x,y\}}\in \Gamma:x\notin \Gamma\text{ or }y\notin \Gamma\}\cup\{x\in \Gamma:\exists\, {\{x,y\}}\in \bar G\setminus\Gamma\}\,.$$

Basic properties such as the domain Markov property (DMP) have been proved in \cite{AHP20}, and it is straightforward to verify that the extended Ising model satisfies the FKG inequality.

The extended Ising model is indeed an extension of the Ising model: its restriction to vertex spins follows the law of the Ising model (see \cite[Section 2]{AHP20}). In addition, it also extends the FK-Ising model: its restriction to edge spins follows the law of the FK-Ising model (see \cite[Section 2]{DHX23}). Therefore, every realization of the Ising model may be viewed as the projection of an extended Ising configuration onto the vertex spins. Throughout this paper, we will work with the usual Ising model whenever possible and only lift configurations to the extended setting when edge variables become relevant.

\subsection{Basic properties of the model(s)}

Here we record some of the standard and well-known properties that will be used repeatedly (see e.g., \cite{FV17}, \cite{D19} for details and proofs).

\begin{enumerate}
    \item \textbf{FKG inequality.} This was introduced in \cite{FKG71}  and named after the three authors.
    If $\mathtt A,\mathtt B$ are both increasing (or decreasing) events, then 
$$P(\mathtt A\cap \mathtt B)\geq P(\mathtt A)\times P(\mathtt B).$$
The relevant notions of increasing (or decreasing) events can be found in \cite[(1.8)]{D19}.    
    \item \textbf{Comparison of boundary conditions (CBC)}.
    If two boundary conditions $\xi_1\leq \xi_2$, then for any increasing event $\mathtt A$ we have 
    $$P^{\xi_1}(\mathtt A)\leq P^{\xi_2}(\mathtt A).$$
    Here $\leq$ stands for a partial order in the set of boundary conditions: for spin configurations, we say $\xi_1\leq \xi_2$ if every plus spin in $\xi_1$ is also plus in $\xi_2;$ for bond configurations, we say $\xi_1\leq \xi_2$ if every open edge in $\xi_1$ is also open in $\xi_2$.
    \item \textbf{Domain Markov property (DMP)}. For two domains $\Gamma_1\subset\Gamma_2$, given the configuration $\xi$ on $\Gamma_1$, the influence on the measure in $\Gamma_2\setminus\Gamma_1$ 
    behaves like a boundary condition:
    $$P_{\Gamma_2}(\cdot\mid \xi)=P_{\Gamma_2\setminus\Gamma_1}^{\xi|_{\partial\Gamma_1}}(\cdot),$$
where $\xi|_{\partial\Gamma_1}$ is the restriction of $\xi$ to $\partial\Gamma_1$ and it denotes for the boundary condition on $\partial\Gamma_1$.

       \item  {\bf Russo-Seymour-Welsh bound (RSW).}  The RSW bound provides robust crossing probability estimates that characterize criticality of the planar FK-Ising model at $p_c$.

        To be precise, let $\phi_{p_c}$ denote the critical planar FK-Ising model. Then the following holds \cite[Theorem 1]{DHN11}(see also  \cite{Tassion16, KT23} for other remarkable progress on the RSW theory):
        let $0<\rho_1<\rho_2,$ then there exist constants $0<c_-<c_+<1$ only relies on $\rho_1,\rho_2$ such that for any rectangle $R=[a,b]\times[c,d]\subset \mbb Z^2$ with $\rho_1|a-b|\leq|c-d|\leq \rho_2|a-b|$ and any boundary condition $\xi,$ we have 
        \begin{equation}
            c_-\leq \phi_{p_c}^{\xi}\left(\mcc H(R)\right)\leq c_+,\quad  c_-\leq \phi_{p_c}^{\xi}\left( \mcc V(R)\right)\leq c_+. \label{eq:RSW}
        \end{equation}
        Here, $\mcc H(R)$ denotes the event that there exists an open path in  $\omega$ crossing $R$ horizontally, and $\mcc V(R)$ stands for vertical crossing.

Moreover, by the Edwards-Sokal coupling, such RSW bounds could be translated into (extended) Ising measure with plus or  boundary condition. For general boundary conditions, however, one has to be slightly more careful. The FK crossing event is an event of the edge configuration, and the estimate \eqref{eq:RSW} is uniform in the boundary condition even if the rectangle is adjacent to the boundary. In contrast, a spin crossing event can be strongly affected by the prescribed boundary spins. Therefore, for general Ising boundary conditions, the corresponding RSW statement should be understood as a bulk estimate.

More precisely, fix $0<\rho_1<\rho_2$ and $\rho>0$. Let $\Omega\subset \mbb Z^2$ be a finite domain and let $R=[a,b]\times[c,d]\subset \Omega$ be a rectangle satisfying \[ \rho_1 |a-b|\le |c-d|\le \rho_2 |a-b|  \quad\mbox{ and }  \quad \operatorname{dist}(R,\partial\Omega) \ge \rho \operatorname{diam}(R). \] Then there exist constants $ 0<c_-=c_-(\rho_1,\rho_2,\rho) <c_+=c_+(\rho_1,\rho_2,\rho)<1 $ such that, uniformly over all boundary conditions $\eta\in\{-1,1\}^{\partial\Omega}$, \[ c_- \le \mu_{\Omega,\beta_c}^{\eta}\bigl(\mcc H^+(R)\bigr) \le c_+, \qquad c_- \le \mu_{\Omega,\beta_c}^{\eta}\bigl(\mcc V^+(R)\bigr) \le c_+. \] The same estimates hold with $\mcc H^+(R)$ and $\mcc V^+(R)$ replaced by $\mcc H^-(R)$ and $\mcc V^-(R)$. Here $\mcc H^+(R)$ denotes the event that there exists a nearest-neighbor path of $+$ spins in $R$ connecting the left side of $R$ to the right side of $R$, and $\mcc V^+(R)$ is defined analogously for vertical crossings.

The point of the distance assumption is that one can use the mixing property (Lemma~\ref{lem: spatial mixing of Ising model}) to compare the boundary condition $\eta$ with $+$ (or $-$) boundary condition.
        
\item {\bf Planar duality.} \label{item:planar-duality} We review the dual theory for the FK-Ising model initiated in \cite{KW42}, and our presentation follows that in {\cite[Section 2.6]{BD12}}. Define $(\mathbb{Z}^2)^{\diamond}=\mathbb{Z}^2+(\frac12,\frac12)$ to be the dual of $\mathbb{Z}^2$, which can be viewed as a translation of $\mathbb{Z}^2$ in $\mathbb{R}^2$. We see that every vertex in $(\mathbb{Z}^2)^{\diamond}$ is the center of a unit square in $\mathbb{Z}^2$ and vice versa. In other words, every edge $e\in\mathbb{Z}^2$ intersects with a unique edge $e^{\diamond}\in(\mathbb{Z}^2)^{\diamond}$. For a configuration $\omega\in\{0,1\}^{E(\mathbb{Z}^2)}$, we define its dual configuration $\omega^{\diamond}$ by setting
$$\omega^{\diamond}(e^{\diamond})=1-\omega(e),~~~\mbox{for~all~}e\in E(\mathbb{Z}^2).$$
Importantly, the measure of $\omega^{\diamond}$ is also an FK-Ising measure with a dual boundary condition and dual parameter $p^{\diamond}$ (see {\cite[Section 2.6]{BD12}} for more details). By {the first formula in} \cite{Onsager44} (see also \cite[(2.4)]{BD12} for modern illustration), we have
$$p_c^{\diamond}=p_c=\dfrac{\sqrt{2}}{1+\sqrt{2}}.$$
Now for a rectangle $R=[a,b]\times[c,d]$, let ${\mcc H}^{\diamond}(R)$ denote the event that there exists an open dual path crossing R horizontally. This is a slight abuse of notation since the dual path lives in $[a-\frac12,b+\frac12]\times[c+\frac12,d-\frac12]\subset(\mathbb{Z}^2)^{\diamond}$. Similarly, we define the event of vertical dual crossing $\mathcal{V^{\diamond}(R)}$. By duality, we have $\mathcal{V(R)}$ happens if and only if $\mathcal{H^{\diamond}(R)}$ fails, and $\mathcal{H(R)}$ happens if and only if $\mathcal{V^{\diamond}(R)}$ fails. Therefore, with the same notation in (iv), we have \begin{equation}
            c_-\leq \phi_{p_c}^{\xi}\left(\mcc H^{\diamond}(R)\right)\leq c_+,\quad  c_-\leq \phi_{p_c}^{\xi}\left( \mcc V^{\diamond}(R)\right)\leq c_+. \label{eq:dual-RSW}
        \end{equation}
    

\end{enumerate}

In the above (i.e., (1)-(4)) $P$ could stand for the (extended) Ising measure or the FK-Ising measure with or without the external field and with or without the boundary condition.
We will write FKG, CBC, DMP, and RSW in the following for convenience.

\section{Mutual absolute continuity of the single interface in the scaling limit}\label{sec: Mutual absolute continuity of the single interface in the scaling limit}
In this section, we consider the Radon-Nikodym derivative for the interface of the Ising model under the Dobrushin boundary condition with and without disorder. We first provide a careful description of the Dobrushin boundary condition on $\partial\Lambda_N$ for later convenience.

\begin{definition}\label{def: plus minus interface}
    Let $$\mathtt{Dob}^+:=\partial\lamn\cap \{z=(z^{(1)},z^{(2)})\in\mbb Z^d, z^{(2)}\geq 0\},\quad \mathtt{Dob}^-:=\partial\lamn\cap \{z=(z^{(1)},z^{(2)})\in\mbb Z^d, z^{(2)}<0\}$$
    denote the upper and bottom half parts of the boundary of the box $\lamn.$ Then the Dobrushin boundary condition (denoted $\pm$) is the boundary condition that being plus on $\mathtt{Dob}^+$ and minus on $\mathtt{Dob}^-.$ Now given any configuration $\sigma\in\{-1,1\}^{\Lambda_N}$ satisfying Dobrushin boundary condition, we define the upmost plus--minus boundary interface $\gamma=\gamma(\sigma)$ as follows.

Let
$
\mathtt{Plus}(\sigma)=\{z\in\Lambda_N:\sigma_z=1\}
$
denote the set of plus spins. By definition, we have
$
\mathtt{Dob}^+\subset \mathtt{Plus}(\sigma).
$
Since $\mathtt{Dob}^+$ is connected in $\mathbb Z^2$, there exists a unique connected component of $\mathtt{Plus}(\sigma)$ containing $\mathtt{Dob}^+$, which we denote by $\mathtt{ConPlus}(\sigma)$.

Define
\[
\Lambda^+(\sigma)
=
\big\{
z\in\Lambda_N:
\text{every path from }z
\text{ to }\partial\Lambda_{N+1}
\text{ intersects }
\mathtt{ConPlus}(\sigma)
\big\},
\]
and let
$
\Lambda^-(\sigma)=\Lambda_N\setminus\Lambda^+(\sigma).
$
Intuitively, $\Lambda^+(\sigma)$ is obtained by filling in all holes of $\mathtt{ConPlus}(\sigma)$.

We then define
\[
\gamma^+(\sigma)
=
\{x\in\Lambda^+(\sigma):
\exists y\in\Lambda^-(\sigma)
\text{ with }
x\sim_* y
\}
\setminus \mathtt{Dob}^+,
\]
and
$\gamma^-(\sigma)=\partial_{\mathrm{ext}}\Lambda^+(\sigma)\setminus \mathtt{Dob}^-.$
Here $x\sim_* y$ means that $x$ and $y$ are $*$-neighbors. By planar duality, $\gamma^+$ forms a connected set of plus spins, while $\gamma^-$ forms a $*$-connected set of minus spins. Hence
$
\gamma=(\gamma^+,\gamma^-)
$
defines a plus--minus interface. Moreover, $\gamma$ is the upmost interface, since every $*$-connected minus path from $\mathtt{Dob}^-$ to $\mathtt{Dob}^-$ must intersect $\gamma^+$.

Finally, let $\Sigma_\gamma$ denote the collection of configurations satisfying $\gamma(\sigma)=\gamma$.
It is easy to see that 
$\Sigma_\gamma=\{\sigma\in\{-1,1\}^{\lamn}: \sigma_{\mid \mathtt{Dob}^+\cup\gamma^+}=1, \sigma_{\mid \mathtt{Dob}^-\cup\gamma^-}=-1\}.$
For convenience, write
$
\mu^\pm_{\Lambda_N,\epsilon h}(\gamma)
=
\mu^\pm_{\Lambda_N,\epsilon h}(\Sigma_\gamma),
$
and similarly
$
\mu^\pm_{\Lambda_N,0}(\gamma)
=
\mu^\pm_{\Lambda_N,0}(\Sigma_\gamma),
$
where the superscript $\pm$ indicates Dobrushin boundary conditions.
  
\end{definition}

Henceforth, we work at the critical temperature $\beta=\beta_c$ and omit $\beta$ from the subscripts. Without loss of generality, we consider the critical disorder $\eps=N^{-\frac{7}{8}}$. The case $\eps=cN^{-\frac{7}{8}}$ with any fixed $c\in(0,\infty)$
follows similarly. 
The following is the main theorem of this section.
\begin{theorem}\label{thm: sle continuity critical main}
    There exists a constant $c>0$ such that for any $\iota>0$ the following holds.\begin{align}
        \PP\bigg(\mu^{\pm}_{\Lambda_N,0}\Big(e^{-\iota^{-4}}&\le \frac{\mu^{\pm}_{\Lambda_N,\eps h}(\gamma)} {\mu^{\pm}_{\Lambda_N,0}(\gamma)}\le e^{\iota^{-4}}\Big) \ge 1-e^{-\iota^{-0.1}}\bigg)\ge  1-c\exp(-c^{-1}\iota^{-0.1}).\label{eq: absolute continuity of sle with respect to the pure measure}\\
        \PP\bigg(\mu^{\pm}_{\Lambda_N,\eps h}\Big(e^{-\iota^{-4}}&\le \frac{\mu^{\pm}_{\Lambda_N,\eps h}(\gamma)} {\mu^{\pm}_{\Lambda_N,0}(\gamma)}\le e^{\iota^{-4}}\Big) \ge1-e^{-\iota^{-0.1}}\bigg)\ge  1-c\exp(-c^{-1}\iota^{0.1}).\label{eq: absolute continuity of sle with respect to the external field measure}
    \end{align}
\end{theorem}

\begin{remark}
   In the following, we always assume that $\iota$ is sufficiently small. In particular, inequalities such as
$3\iota^{-0.1}<\iota^{-1}$
will be used without further comment. The underlying reason is that it suffices to prove Theorem~\ref{thm: sle continuity critical main} for $\iota<\iota_0$. Indeed, once this is established, one may choose the constant $c$ sufficiently large so that
$1-c\exp(-c^{-1}\iota^{-0.1})\le 1$
for all $\iota\ge \iota_0$, and the desired bounds then follow trivially from the non-negativity of probabilities.
\end{remark}

\begin{remark}\label{remark-mutual-abs-continuity}
We now explain the role of Theorem~\ref{thm: sle continuity critical main}.
First, it gives a quantitative form of mutual absolute continuity at the discrete level between the critical and RFIM interface measures, after removing an exceptional set whose probability tends to zero in the external field. This statement is not a consequence of the absolute continuity of the scaling limits and is therefore of independent interest. Second, it provides the missing input for proving the reverse absolute continuity of the limiting RFIM interface law with respect to \(\mathrm{SLE}_3\); combined with the opposite direction obtained earlier, this yields mutual absolute continuity in the scaling limit.

Indeed, for every \(\iota>0\), with probability tending to one in the external
field, there exists an event which is typical under both
\(\mu_{\lamn,0}^{\pm}\) and \(\mu_{\lamn,\eps h}^{\pm}\), and on which the
Radon--Nikodym derivative is bounded from below and from above by deterministic
constants depending only on \(\iota\). Consequently, up to an error
\(\eta_\iota\) with \(\eta_\iota\to0\) as \(\iota\downarrow0\), any event which
has positive mass under the critical interface measure also has positive mass
under the RFIM interface measure.

Passing to the scaling limit, and using
Proposition~\ref{proposition-uniqueness-law-curve} together with the convergence
of the critical interface to \(\mathrm{SLE}_3\), this comparison prevents an
event of positive \(\mathrm{SLE}_3\)-mass from having zero mass under the
limiting RFIM interface law. This gives the reverse absolute continuity, and
therefore mutual absolute continuity.
\end{remark}
\begin{remark}
    Theorem \ref{thm: sle continuity critical main} holds for the FK-Ising model with and without disorder, with changes in the exponents of $\iota.$ See Remarks~\ref{rmk: sle RN derivative critical} and \ref{rmk: large deviation for sle interface dimension} for further details.
\end{remark}
Similar to the proof in Section~\ref{sec: Scaling limit of the interface with Dobrushin boundary conditions}, the estimation of the ratio $\frac{\mu^\pm_{\Lambda_N,\eps h}(\gamma)}{\mu^\pm_{\Lambda_N,0}(\gamma)}$ essentially boils down to controlling the partition functions on different domains. 

For any interface $\gamma=(\gamma^+,\gamma^-)$, recall the definition of $\Lambda^+$ and $\Lambda^-$.
We denote by $\mathcal{Z}_{h}^{\xi}(\Lambda)$ the partition function of Ising model on region $\Lambda$ with external field $h$ and boundary condition $\xi$. 
Note that $\Sigma_{\gamma}$ is the set of configurations such that all the spins on $\gamma^+$ are plus and all the spins on $\gamma^-$ are minus.
 Thus we calculate as follows:\begin{equation}\label{eq: RN derivative expansion to partition function}
  \frac{\mu^\pm_{\Lambda_N,\eps h}(\gamma)}{\mu^\pm_{\Lambda_N,0}(\gamma)}=\frac{\cZ^{\pm}_{0}(\Lambda_N)}{\cZ^{\pm}_{\eps N}(\Lambda_N)}\cdot \frac{\cZ^{-}_{\eps h}(\Lambda^-)}{\cZ^{-}_{0}(\Lambda^-)}\cdot \frac{\cZ^{+}_{\eps h}(\Lambda^+)}{\cZ^{+}_{0}(\Lambda^+)}=\frac{\mathsf Z^+(\gamma)\cdot \mathsf Z^-(\gamma)}{\mathsf Z}
\end{equation} where \begin{align}
    &\mathsf Z=\frac{\cZ^{\pm}_{\eps h}(\Lambda_N)}{\cZ^{\pm}_{0}(\Lambda_N)}\cdot \prod_{v\in \lamn}\frac{1}{\cosh(\eps\beta_c h_v)}, \mbox{ and }\nonumber\\ &\mathsf Z^+(\gamma)=\frac{\cZ^{+}_{\eps h}(\Lambda^+)}{\cZ^{+}_{0}(\Lambda^+)}\cdot \prod_{v\in \Lambda^+}\frac{1}{\cosh(\eps\beta_c h_v)},\quad ~\mathsf Z^-(\gamma)=\frac{\cZ^{-}_{\eps h}(\Lambda^-)}{\cZ^{-}_{0}(\Lambda^-)}\cdot \prod_{v\in \Lambda^-}\frac{1}{\cosh(\eps\beta_c h_v)}.\nonumber
\end{align}
Note that $\mathsf Z$ is a $\PP$-random variable independent of $\gamma$. Moreover, for any fixed $\gamma$, $\mathsf Z^+(\gamma)$ and $\mathsf Z^-(\gamma)$ are also $\PP$-random variables.

  The choice of $\mathsf Z$ is quite natural in the following sense: We expand the ratio between the partition functions with and without disorder and obtain
\begin{align}
       \frac{\cZ^{\pm}_{\eps h}(\lamn)}{\cZ^{\pm}_{0}(\lamn)}
       &= \left\langle ~\exp(\sum_{v\in \lamn}\eps \beta_c h_v\sigma_v) ~ \right\rangle_{\lamn,0}^{\pm}\nonumber= \left\langle ~\prod_{v\in \lamn}[\cosh(\eps \beta_c h_v)+\sigma_v\sinh(\eps \beta_c h_v)]~\right\rangle_{\lamn,0}^{\pm}\nonumber\\ &= \left[\prod_{v\in \lamn} \cosh(\eps \beta_c h_v)\right] \left\langle ~\prod_{v\in \lamn}[1+\sigma_v\tanh(\eps \beta_c h_v)]~\right\rangle_{\lamn,0}^{\pm}.\label{eq: partition function expansion 1}
\end{align}Thus we get that \begin{equation}
    \mathsf Z=\left\langle ~\prod_{v\in \lamn}[1+\sigma_v\tanh(\eps \beta_c h_v)]~\right\rangle_{\lamn,0}^{\pm}.\label{eq: partition function expansion 2}\nonumber
\end{equation}
Similarly, we have  \begin{align}
    \mathsf Z^+(\gamma)
       &= \langle ~\prod_{v\in  \Lambda^+}[1+ \sigma_v\tanh(\eps \beta_c h_v)]~\rangle_{\Lambda^+,0}^{+}~~\mbox{and}~~
       \mathsf Z^-(\gamma)
       =\langle ~\prod_{v\in  \Lambda^-}[1+ \sigma_v\tanh(\eps \beta_c h_v)]\rangle_{\Lambda^-,0}^{-}\label{eq: partition function expansion 3}
\end{align}

Since we will frequently deal with negative powers of distances, throughout this section and Section \ref{sec: singularity for outermost cle}, we use the convention that $\dist(A,B)=1$ whenever $A\cap B\neq\emptyset$.
For any $\gamma$, let \begin{equation}
    S(\gamma)=\sum_{x\in \lamn}\dist(x,\gamma\cup \partial\lamn)^{-\frac{1}{4}}.\label{eq: definition of S gamma}
\end{equation} Here and in the following, we abuse the notation that $\gamma$ also denotes the union of the two interfaces $\gamma^+\cup\gamma^-.$ Roughly speaking, $S(\gamma)$ may be viewed as a discrete analogue of the $\frac74$-dimensional Minkowski content of $\gamma$.

Throughout the rest of this paper, let $\kappa>0$ be a sufficiently small constant. Although one may ultimately take $\kappa=10^{-5}$, we retain the notation $\kappa$ for generality.
For notation convenience, we may write $\mathsf Z^+, \mathsf Z^-$ instead of $ \mathsf Z^+(\gamma), \mathsf Z^-(\gamma)$ when $\gamma $ is clear in the context. We have the following theorem, which controls the typical value of $\mathsf Z$. Moreover, for any fixed $\gamma$, it also controls the typical values of $\mathsf Z^+(\gamma)$ and $\mathsf Z^-(\gamma)$ in terms of $S(\gamma)$.

\begin{proposition}\label{prop: sle RN derivative critical}
    Fix $\beta=\beta_c$ as the inverse temperature of the two-dimensional Ising model. There exists a constant $c_1>0$ such that for any small constant $\iota\in(0,1)$ and any interface $\gamma$, 
    \begin{align}
        \EE\max\{\mathsf Z^+(\gamma),\mathsf Z^-(\gamma)\} \le \exp\Big( c_1(\frac{S(\gamma)}{N^{\frac{7}{4}}})^{1+3\kappa}\cdot \iota^{-2}\Big).\label{eq: partition function ratio expectation bound in sle}
    \end{align}
    In addition, we also have \begin{align}
        &\PP\Big(\max\{|\log\big(\mathsf Z^+(\gamma)\big)|,|\log(\mathsf Z^-(\gamma)\big)|\}\ge  c_1(\frac{S(\gamma)}{N^{\frac{7}{4}}})^{1+3\kappa}\cdot \iota^{-2}\Big)\le \exp\Big(- c_1^{-1}(\frac{S(\gamma)}{N^{\frac{7}{4}}})^{\frac{21\kappa}{32}}\cdot \iota^{-\frac{1}{4}}\Big).\label{eq: partition function ratio rare event bound in sle}\\
        &\PP\Big(|\log(\mathsf Z)|\ge
        c_1\iota^{-2})\le~ \exp\Big(- c_1^{-1} \iota^{-\frac{1}{4}}\Big).\label{eq: partition function ratio rare event bound in sle whole region}
    \end{align}
\end{proposition}
\begin{remark}\label{rmk: sle RN derivative critical}
    Proposition~\ref{prop: sle RN derivative critical} holds also for the FK-Ising model with the same proof. 
\end{remark}
For notation clarity, we write $\mathsf r=\mathsf r(\gamma)=\frac{N^{\frac{7}{8}}}{\sqrt{S(\gamma)}}$. Note that 
$$ S(\gamma)=\sum_{x\in\Lambda_N}\mathrm{dist}(x,\gamma\cup \partial\lamn)^{-\frac{1}{4}}\geq\sum_{x\in\Lambda_N}\mathrm{dist}(x,\partial\lamn)^{-\frac{1}{4}}\geq \sum_{L=1}^N L^{-\frac{1}{4}}\cdot 8(N-L-1)\geq N^{\frac{7}{4}},$$
therefore we get that $\mathsf r\le 1.$ As hinted in \cite[Lemma 3.5]{DHX23}, \(\mathsf Z\) is close to \(1\) when
\(\eps \ll N^{-7/8}\). Under the assumption \(\eps=N^{-7/8}\), we
therefore perform a similar expansion after partitioning $\lamn$ into
boxes of side length at most \(M\), where \(\eps \ll M^{-7/8}\). We begin
by defining the partition procedure.

\begin{lemma}\label{lem: partition procedure in sle}
There exist a constant $c_2>0$ and a  partition $\fB_1,\cdots,\fB_n$ of $\lamn$ such that the following holds:
\begin{enumerate}
    \item For every \(1\le i\le n\), we have
\begin{equation}
\sum_{x\in \fB_i}\dist(x,\gamma\cup \partial\fB_i)^{-1/4}
\le c_2M^{7/4}.
\label{eq: property for the partition}
\end{equation}
\item The number of boxes satisfies \begin{equation}
n\le \frac{c_2}{\iota^2 \mathsf r^{2+\frac{21}{4}\kappa}}.
\label{eq: partition procedure box number bound}
\end{equation}
\end{enumerate}  
\end{lemma}

\begin{proof}
 Set
$M=\iota \mathsf  r^{3\kappa}N$ and for a box $\fC$, write
$
S_\gamma(\fC):=\sum_{x\in\fC}\dist(x,\gamma)^{-1/4}.
$
We introduce the following partition procedure. 
    We first partite $\lamn$ into $M$-boxes. For each box $\fB$, we have two possibilities \begin{itemize}
        \item If $\sum_{x\in \fB}\dist(x,\gamma)^{-\frac{1}{4}}\ge M^{\frac{7}{4}}$, then we divide $\fB$ into two small boxes $\fC_1$ and $\fC_2$ such that \begin{equation}
            \max\Big\{S_\gamma(\fC_1),S_\gamma(\fC_2)\Big\}\le (\frac{1}{2}+\kappa)\cdot S_\gamma(\fB).\label{eq: definition of the partition procedure in sle}
        \end{equation}
        \item If $S_\gamma(\fB)< M^{\frac{7}{4}}$, then we turn to the next box.
    \end{itemize} We next explain how we achieve \eqref{eq: definition of the partition procedure in sle} in each step.
    Suppose that a current box $\fB$ satisfies
$
S_\gamma(\fB)\ge M^{7/4}.
$
Assume without loss of generality that
\[
\fB=[1,a]\times[1,b],
\qquad 1\le a\le b\le M.
\]
For \(0\le c\le b\), set
\[
f_1(c)=S_\gamma(\fC_1(c)),
\qquad
f_2(c)=S_\gamma(\fC_2(c)).
\]
When \(c\) increases by one, only one horizontal slice is moved from \(\fC_2\) to \(\fC_1\). Since \(\dist(\cdot,\gamma)\ge1\), the contribution of such a slice is at most \(M\). Hence
\[
|f_1(c+1)-f_1(c)|\le M,
\qquad
|f_2(c+1)-f_2(c)|\le M.
\]
As \(f_1(0)=0\) and \(f_2(b)=0\), the minimizing choice of \(d\) gives
\[
|f_1(d)-f_2(d)|\le 2M.
\]
Therefore,
\[
\max\{S_\gamma(\fC_1(d)),S_\gamma(\fC_2(d))\}
\le
\frac12S_\gamma(\fB)+M.
\]
Since \(S_\gamma(\fB)\ge M^{7/4}\), and \(M\) is sufficiently large, we have \(M\le \kappa S_\gamma(\fB)\). Thus
\begin{equation}
\max\{S_\gamma(\fC_1(d)),S_\gamma(\fC_2(d))\}
\le
\Big(\frac12+\kappa\Big)S_\gamma(\fB).
\label{eq: balanced subdivision}
\end{equation}
Similarly,
\begin{equation}
\min\{S_\gamma(\fC_1(d)),S_\gamma(\fC_2(d))\}
\ge
\Big(\frac12-\kappa\Big)M^{7/4}.
\label{eq: lower bound for child box}
\end{equation}

The procedure terminates after finitely many steps, since each subdivision strictly decreases the longer side length of the box being subdivided, and all side lengths are positive integers.

Let \(\mathfrak B\) be the collection of final boxes which are not original \(M\)-boxes. By \eqref{eq: lower bound for child box}, every box in \(\mathfrak B\) has \(S_\gamma\)-mass at least \((\frac12-\kappa)M^{7/4}\). Since the boxes are disjoint,
\[
|\mathfrak B|
\le
\frac{S(\gamma)}{(\frac12-\kappa)M^{7/4}}
\le
C_1\iota^{-7/4}\mathsf r^{-2-\frac{21}{4}\kappa}.
\]
The number of original \(M\)-boxes is at most
$
\Big(\frac NM\Big)^2
=
\iota^{-2}\mathsf r^{-6\kappa}.
$
Therefore,
\[
n
\le
|\mathfrak B|+\Big(\frac NM\Big)^2
\le
C_1\iota^{-7/4}\mathsf r^{-2-\frac{21}{4}\kappa}
+
\iota^{-2}\mathsf r^{-6\kappa}
\le
\frac{C_2}{\iota^2 \mathsf r^{2+\frac{21}{4}\kappa}}.
\]
It remains to prove \eqref{eq: property for the partition}. By the stopping rule, every final box \(\fB_i\) satisfies
$
S_\gamma(\fB_i)<M^{7/4}.
$
Moreover, since \(\fB_i\) is contained in an \(M\)-box,
\[
\sum_{x\in\fB_i}\dist(x,\partial\fB_i)^{-1/4}
\le
C_3M^{7/4}.
\]
Using
\[
\dist(x,\gamma\cup\partial\fB_i)^{-1/4}
\le
\dist(x,\gamma)^{-1/4}
+
\dist(x,\partial\fB_i)^{-1/4},
\]
we obtain
\[
\sum_{x\in\fB_i}
\dist(x,\gamma\cup\partial\fB_i)^{-1/4}
\le
(1+C_3)M^{7/4}.
\]
Choosing $c_2$ large enough proves \eqref{eq: property for the partition}.
\end{proof}

For notation convenience, we write $\fB_i^+=\fB_i\cap \Lambda^+$, $\fB_i^-=\fB_i\cap \Lambda^-$. 
We want to control $\mathsf Z,\mathsf Z^+$ and $\mathsf Z^-$ by introducing the external field incrementally, one box at a time.  To this end, we define $H^i$ by
\begin{equation*}
    H^i_u=\left\{\begin{aligned}
        h_u,~~~&u\in \cup_{j=1}^i \fB_j\\
        0,~~~&u\in \cup_{j=i+1}^{n} \fB_j
\end{aligned}\right.
\end{equation*} The next lemma shows that adding  a field to $H^i$ will not change the ratios too much. Let   \begin{equation}\label{eq: definition of ai in sle continuity}
    \begin{aligned}
        \sfa_i&=\frac{\cZ^{\pm}_{\eps H^{i}}(\lamn)}{\cZ^{\pm}_{\eps H^{i-1}}(\lamn)}\cdot \frac{1}{\prod_{x\in\fB_i}\cosh(\eps \beta_ch_x)}-1,~~
    \sfa_i^+=\frac{\cZ^{+}_{\eps H^{i}}(\Lambda^+)}{\cZ^{+}_{\eps H^{i-1}}(\Lambda^+)}\cdot \frac{1}{\prod_{x\in\fB_i^+}\cosh(\eps\beta_c h_x)}-1,\\ \mbox{and  }\sfa_i^-&=\frac{\cZ^-_{\eps H^{i}}(\Lambda^-)}{\cZ^{-}_{\eps H^{i-1}}(\Lambda^-)}\cdot \frac{1}{\prod_{x\in\fB_i^-}\cosh(\eps\beta_c h_x)}-1.
    \end{aligned}
\end{equation}
\begin{lemma}\label{lem: adding external field in sle continuity}
Let $\mathcal{F}_i$ be the
$\sigma$-algebra generated by $H^i$.
There exists a constant $c_3>0$  
such that for any $1\le i\le n$, we have \begin{equation}
    \PP\Big(\max\{|\sfa_i|,|\sfa_i^+|,|\sfa_i^-|\}\ge \frac{1}{2}\mid \mathcal{F}_{i-1}\Big)\le c_3^{-1}\exp\Big(-c_3\big(\eps M^{\frac{7}{8}}\big)^{-\frac{1}{2}}\Big).\label{eq: adding external field in sle continuity good external}
\end{equation}
In addition, we have \begin{equation}
    \EE\Big(\max\{|\sfa_i|,|\sfa_i^+|,|\sfa_i^-|\}\mid \mathcal{F}_{i-1}\Big)\le 1.\label{eq: adding external field in sle continuity expectation}
\end{equation}
\end{lemma}
Before proving Lemma~\ref{lem: adding external field in sle continuity}, we start with a Lemma similar to \cite[Lemma 3.5]{DHX23}. 
\begin{lemma}
    \label{lem: upper-bound for the sum of squares of k point function}
    Recall the convention that $\dist(A,B)=1$ if $A\cap B\neq\emptyset$.
    For any subset $I,S \subset\Lambda_N$, let \begin{equation}\label{eq: definition of F for k point function}
        F(S, I)=\prod_{x\in I}\frac{1}{\big[\dist(x,S\cup (I\setminus\{x\}))\big]^{\frac{1}{{8}}}}.
    \end{equation}
For any \(s\in(0,1)\), there exists a constant \(c_4=c_4(s)>0\) such that the following holds. Let \(t>0\), and let \(S,\Lambda\subset\Lambda_N\) satisfy
$
\partial\Lambda\subset S\subset\Lambda,
$
and assume moreover that
\begin{equation}\label{eq: boundary dimension assumption}
\sum_{x\in\Lambda}\frac{1}{\dist(x,S)^s}
\le
tN^{2-s}.
\end{equation}
Then
\begin{align}
\sum_{x_i\in\Lambda,\,1\le i\le k}
F(S,\{x_1,x_2,\cdots,x_k\})^{8s}
\le
(c_4t)^kN^{(2-s)k}(k!)^s.
\label{eq: upper-bound for the sum of squares of k point function 3}
\end{align}
    
     As a consequence, we have
    \begin{align}
        \sum_{\substack{I\subset \Lambda,|I|=k} } F(S, I)^{8s}&\le (c_4t)^{k}N^{(2-s)k}\cdot(k!)^{-1+s}\label{eq: upper-bound for the sum of squares of k point function2}.
    \end{align} 
\end{lemma}
\begin{proof}
Equation \eqref{eq: upper-bound for the sum of squares of k point function2} follows immediately from \eqref{eq: upper-bound for the sum of squares of k point function 3}. Hence, it suffices to establish \eqref{eq: upper-bound for the sum of squares of k point function 3}. The argument is similar to those in \cite[Lemma 8.3]{CSZ16} and \cite[Lemma 3.2]{DHX23}.

We proceed by induction on \(k\). The case \(k=1\) follows directly from \eqref{eq: boundary dimension assumption}. Assume that \eqref{eq: upper-bound for the sum of squares of k point function 3} holds for \(k=j\); we now prove it for \(k=j+1\). Throughout the proof, we adopt the notation from \cite[Lemma 3.2]{DHX23}.

    Fix $J=\{x_1,x_2,\cdots,x_j\}$, let $$\Omega_0=\{y\in\Lambda\setminus J:\dist(y,\partial\Lambda)=\dist(y,J\cup S)\},$$
    $$\Omega_i=\{y\in(\Lambda\setminus J)\setminus(\cup_{r=1}^{i-1}\Omega_r) :\dist(y,x_i)=\dist(y,J\cup S)\}, \text{ for }i=1,2,\cdots,j.$$ Then $\Omega_0,\cdots,\Omega_j$ is a partition of $\Lambda\setminus J$. First we get from \cite[(3.8)]{DHX23} and \eqref{eq: boundary dimension assumption} that \begin{equation}\label{eq: induction from FS 1}
       \sum_{x\in \Omega_0}F(S, J\cup\{x\})^{8s}\le 2^{3/2} \sum_{x\in \Omega_0}\frac{F(S, J)^{8s}}{(\dist(x,S))^{s}}\le C_1tN^{2-s}F(S, J)^{8s}. 
    \end{equation} Furthermore, for $1\le i\le j$ and $x\in \Omega_i$, we get from \cite[(3.9)]{DHX23} that \begin{equation}\label{eq: k point function on i region}
        \sum_{x\in \Omega_i}F(S, J\cup\{x\})^{8s}\le  C_2 N^{s}\cdot |\Omega_i|^{1-s}F(S, J)^{8s}.
    \end{equation}
    With \eqref{eq: induction from FS 1} and \eqref{eq: k point function on i region}, the rest of the proof is the same as \cite[Lemma 3.2]{DHX23}. 
\end{proof}
The next lemma will be devoted to control the $k$-point function $\langle\sigma^I\rangle^+_{\Lambda^+,\eps H^{i-1}}$ (recalling that  $\sigma^I=\prod_{x\in I}\sigma_x$).
\begin{lemma}\label{lem: DHX input for specific rare event in sle}
    There exists a constant $c_5>0$ such that for any $1\le i\le n$, any external field $h\in\mbb R^{\Lambda_N}$ and any set $I\subset \fB_i$, we have  \begin{align}
       |\langle ~\sigma^I~\rangle_{\Lambda^+,\hat h^i}^{+}|&\le c_5^{|I|} F(\partial\fB_i^+,I),\label{eq: DHX input for specific rare event 2}
    \end{align}
    where $\hat h^i_x=h_x\cdot \mbf 1_{x\not\in\fB_i^+}$ is the external field vanishing on $\fB_i^+.$ The same results also hold for $\fB_i^-$ and $\fB_i.$
\end{lemma}
\begin{proof}
 The proof of Lemma \ref{lem: DHX input for specific rare event in sle} follows from \cite[Lemma 4.16]{DHX23}. In the sequel (see Lemma \ref{lem: DHX input for specific rare event}), we will establish a significantly stronger version that controls the spin average under certain events that are neither increasing nor decreasing. This extension goes beyond the results of \cite[Lemma 4.16]{DHX23} and requires a self-contained argument. Therefore, we omit the proof of Lemma \ref{lem: DHX input for specific rare event in sle} here.
\end{proof}
For notation clarity, let $\langle\cdot\rangle^+_{i-1}$ be the expectation operator with respect to $\mu_{\Lambda^+,\eps H^{i-1}}^+$.  

\begin{corollary}\label{cor: DHX input for specific rare event in sle}
There exists a constant $c_6>0$ such that 
    for any integers $1\le i\le n$ and $k\ge 1$, we have \begin{align*}
        \PP\Big(\big|\sum_{I\subset \fB_i^+,|I|=k}(\langle\sigma^I \rangle^+_{i-1}\prod_{x\in I}\tanh(\eps\beta_c h_x)\big|\ge (\eps M^{\frac{7}{8}})^{k/2}\Big)&\le c_6^{-1}\exp\Big(-c_6\sqrt{\eps^{-1}M^{-\frac{7}{8}}}\times(k!)^{\frac{3}{8k}}\Big).
    \end{align*}Furthermore, if $\iota<0.1$, then for any integer $r\ge 1$, we have \begin{align*}
        \PP\Big(\big|\sum_{k=r}^{\infty}\sum_{I\subset \fB_i^+,|I|=k}(\langle\sigma^I \rangle^+_{i-1}\prod_{x\in I}\tanh(\eps\beta_c h_x)\big|\ge 2(\eps M^{\frac{7}{8}})^{r/2}\Big)&\le c_6^{-1}\exp\Big(-c_6\sqrt{\eps^{-1}M^{-\frac{7}{8}}}\times(r!)^{\frac{3}{8r}}\Big).
    \end{align*} The same results hold for $\fB_i^-$ and $\fB_i^+$.
\end{corollary}
\begin{proof}
    The proof for $\fB_i$ follows directly from combining Lemma~\ref{lem: DHX input for specific rare event in sle} and \cite[Lemma B.1]{DHX23}, so we omit it here. Combining with \eqref{eq: property for the partition}, we extend the result to $\fB_i^-$ and $\fB_i^+.$
\end{proof}
We are now ready to prove Lemma~\ref{lem: adding external field in sle continuity}.
\begin{proof}[Proof of Lemma~\ref{lem: adding external field in sle continuity}]
We first prove the case for $\sfa_i^+$. 
    Similar to \eqref{eq: partition function expansion 1}, we apply chaos expansion to $\cZ^{+}_{\eps H^{i}}(\Lambda^+)$ with respect to $\cZ^{+}_{\eps H^{i-1}}(\Lambda^+)$ and get \begin{equation}
        \sfa_i^+=\sum_{k=1}^{\infty}\sum_{I\subset\fB_i^+,|I|=k}\langle\sigma^I\rangle^+_{i-1}\prod_{x\in I}\tanh(\eps\beta_c h_x).\label{eq: chaos expansion for ai in sle}
    \end{equation} 
    Combining Lemmas \ref{lem: upper-bound for the sum of squares of k point function} and \ref{lem: DHX input for specific rare event in sle} with \eqref{eq: property for the partition}, we get  that \begin{align}
        \sum_{|I|=k,I\subset\fB_i^+}\big(\langle\sigma^I\rangle^+_{i-1}\big)^2\le C_{1}^{k} \sum_{|I|=k,I\subset\fB_i^+}F(\partial\fB_i^+,I)^2\le C_{2}^{k}M^{\frac{7k}{4}}.\label{eq: second moment in adding external field}
    \end{align}Combined with \cite[Lemma B.1]{DHX23}, we get for any $k\ge 1$ that \begin{align}
        \PP\Big(\big|\sum_{|I|=k,I\subset\fB_i^+}\langle\sigma^I\rangle^+_{i-1}\prod_{x\in I}\tanh(\eps\beta_c h_x)\big|\ge \frac12(\eps M^{\frac{7}{8}})^{\frac{k}{2}} \Big)\le C_{3}^{-1}\exp(-C_{3} \big(\eps M^{\frac{7}{8}}\big)^{-\frac{1}{2}}\cdot (k!)^{\frac{3}{8k}}).\nonumber
    \end{align}Taking union bound over $1\le k\le |\fB_i^+|$ and noting that $\sum_{k=1}^{\infty}\frac12 (\eps M^{\frac{7}{8}})^{\frac{k}{2}}\le \frac{1}{2}$, we get that \begin{align}
        &\PP\Big(\big|\sum_{k=1}^{\infty}\sum_{|I|=k,I\subset\fB_i^+}\langle\sigma^I\rangle^+_{i-1}\prod_{x\in I}\tanh(\eps\beta_c h_x)\big|\ge \frac{1}{2} \Big)\nonumber\\\le~&\sum_{k=1}^{\infty} \PP\Big(\big|\sum_{|I|=k,I\subset\fB_i^+}\langle\sigma^I\rangle^+_{i-1}\prod_{x\in I}\tanh(\eps\beta_c h_x)\big|\ge \frac12 (\eps M^{\frac{7}{8}})^{\frac{k}{2}} \Big)\nonumber\\\le~&  \sum_{k=1}^{\infty}C_{3}^{-1}\exp\Big(-C_{3} \big(\eps M^{\frac{7}{8}}\big)^{-\frac{1}{2}}\cdot (k!)^{\frac{3}{8k}}\Big)\le C_{4}^{-1}\exp\Big(-C_{4}\big(\eps M^{\frac{7}{8}}\big)^{-\frac{1}{2}}\Big).\nonumber
    \end{align}Thus we finish the proof of \eqref{eq: adding external field in sle continuity good external} for $\sfa_i^+$.  Next, we turn to the proof of \eqref{eq: adding external field in sle continuity expectation}. By the Cauchy-Schwarz inequality and \eqref{eq: second moment in adding external field},  we get that \begin{align}
        &\EE \Big|\sum_{|I|=k,I\subset\fB_i^+}\langle\sigma^I\rangle^+_{i-1}\prod_{x\in I}\tanh(\eps\beta_c h_x)\Big|\le\sqrt{\EE \Big(\sum_{|I|=k,I\subset\fB_i^+}\langle\sigma^I\rangle^+_{i-1}\prod_{x\in I}\tanh(\eps\beta_c h_x)\Big)^2}\nonumber\\=&
\sqrt{\sum_{|I|=|J|=k,I,J\subset\fB_i^+}\langle\sigma^I\rangle^+_{i-1}\langle\sigma^J\rangle^+_{i-1}\prod_{x\in I\cap J}\EE \left(\tanh(\eps\beta_c h_x)\right)^2\prod_{x\in I\Delta J}\EE \tanh(\eps\beta_c h_x)}
        \nonumber\\ \le&\sqrt{\sum_{|I|=k,I\subset\fB_i^+}\left(\langle\sigma^I\rangle^+_{i-1}\right)^2\cdot \left(\eps^2\beta_c^2\right)^k}\stackrel{\eqref{eq: second moment in adding external field}}{\le} C_{5}\eps^{k}M^{\frac{7k}{8}},\nonumber
    \end{align}
   where the second inequality uses the facts that the expectation is not 0 only if $I\Delta J=\emptyset$ and $(\tanh u)^2\leq u^2$.
    Summing over $1\le k\le |\fB_i^+|$, we get that \begin{equation}
        \EE |\sfa_i^+|\le \sum_{k=1}^{\infty}
        \EE \Big|\sum_{|I|=k,I\subset\fB_i^+}\langle\sigma^I\rangle^+_{i-1}\prod_{x\in I}\tanh(\eps\beta_c h_x)\Big|\le \sum_{k=1}^{\infty} C_{5}\eps^{k}M^{\frac{7k}{8}}\le \frac{1}{3}\nonumber
    \end{equation}where we chose $\iota>0$ to be small enough in the last inequality.

    The proof for $\sfa_i^-$ is the same, so we omit further details. The proof for $\sfa_i$ follows from the special case $\gamma=\emptyset$.
\end{proof}
Now we are ready to prove Proposition~\ref{prop: sle RN derivative critical}.
\begin{proof}[Proof of Proposition~\ref{prop: sle RN derivative critical}]
Without loss of generality, we may assume $\iota>0$ is small enough.
We first prove \eqref{eq: partition function ratio expectation bound in sle}. Recalling the definition of $\sfa_i^+, \sfa_i^-$ from \eqref{eq: definition of ai in sle continuity} and
applying \eqref{eq: adding external field in sle continuity expectation} recursively, we get that \begin{equation}
   \EE(\max\{\mathsf Z^+,\mathsf Z^-\}) =\EE\Big(\max\{\prod_{i=1}^n(1+\sfa_i^+),\prod_{i=1}^n(1+\sfa_i^-)\}\Big)\le 2^{n}.\nonumber
\end{equation}
Combined with \eqref{eq: partition procedure box number bound}, we obtain
\[
2^n
\leq
\exp\big(c_2\ln 2\cdot \mathsf r^{-2-\frac{21}{4}\kappa}\cdot\iota^{-2}\big)
\leq
\exp\Big(c_2\ln 2\cdot
\Big(\frac{S(\gamma)}{N^{7/4}}\Big)^{1+3\kappa}
\cdot\iota^{-2}\Big),
\]
where the second inequality follows from the relation
$\mathsf r=\frac{N^{8/7}}{\sqrt{S(\gamma)}}\le 1.$
This completes the proof of
\eqref{eq: partition function ratio expectation bound in sle}. Next, we turn to the proof of \eqref{eq: partition function ratio rare event bound in sle}. 
Let $\cH_i$ denote the collection of external fields such that 
$$\max\{|\sfa_i|,|\sfa_i^+|,|\sfa_i^-|\}\le \frac{1}{2}.$$ In addition, let $\cH=\cap_{i=1}^n\cH_i$. Then by \eqref{eq: adding external field in sle continuity good external} we get that \begin{equation}
        \PP(\cH)\ge 1- n\cdot c_3^{-1}\exp\Big(-c_3\big(\eps M^{\frac{7}{8}}\big)^{-\frac{1}{2}}\Big)\ge 1- C_1^{-1}\exp\Big(-C_1\big(\frac{S(\gamma)}{N^{\frac{7}{4}}}\big)^{\frac{21\kappa}{32}}\iota^{-\frac{1}{4}}\Big)\label{eq: good external field given interface in sle continuity}
    \end{equation}
    where the last inequality follows from Lemma~\ref{lem: partition procedure in sle} and the fact that $M=\iota \mathsf r^{3\kappa}N,$ $\mathsf r=\frac{\sqrt{S(\gamma)}}{N^{\frac{7}{8}}}$.
    For any $h\in \cH$, since $|\log(1+a)|\leq 2|a|$ for all $|a|\leq \frac{1}{2},$
    we have \begin{align}
        &\max\{|\log(\mathsf Z)|,|\log(\mathsf Z^+)|,|\log(\mathsf Z^-)|\}\nonumber\\\le~&\max\{\sum_{i=1}^n |\log(1+\sfa_i)|,\sum_{i=1}^n|\log(1+\sfa_i^+)|,\sum_{i=1}^n\log|(1+\sfa_i^-)|\}\nonumber\\\le~&  2\sum_{i=1}^n\big(\max\{|\sfa_i|,|\sfa_i^+|,|\sfa_i^-|\}\big) \le n.\nonumber\label{eq: sle RN derivative convert to ai lower bound}
    \end{align} Combined with \eqref{eq: good external field given interface in sle continuity} and Lemma~\ref{lem: partition procedure in sle}, we complete the proof of \eqref{eq: partition function ratio rare event bound in sle}. The proof for \eqref{eq: partition function ratio rare event bound in sle whole region} follows from the special case $\gamma=\emptyset.$
\end{proof}
By Proposition~\ref{prop: sle RN derivative critical}, we can control $\frac{\mathsf Z^+\cdot\mathsf Z^-}{\mathsf Z}$ if $S(\gamma)$ is close to $N^\frac{7}{4}$. Thus we need to control the  probability that $S(\gamma)$ is much larger $N^\frac{7}{4}$ under $\mu^{\pm}_{\Lambda_N,0}$. 

\begin{lemma}\label{lem: large deviation for sle interface dimension}
For any $\kappa_1\in (0,0.1)$, there exists  a constant $c_7=c_7(\kappa_1)>0$ such that for any $\lambda>0$ we have \begin{equation}
        \mu^{\pm}_{\Lambda_N,0 }\Big(S(\gamma)\ge N^{\frac{7}{4}}\lambda\Big)\le c_7\exp(-c_7^{-1}\lambda^{4-\kappa_1}).\label{eq: large deviation for sle interface dimension}
    \end{equation}
\end{lemma}
We postpone the proof of Lemma~\ref{lem: large deviation for sle interface dimension} to Section~\ref{sect: ldp for sle interface without disorder}.
Now we are ready to prove the mutual absolute continuity of the interface under the critical and the near-critical RFIM measure, i.e., Theorem \ref{thm: sle continuity critical main}. We first show that $\mu^{\pm}_{\Lambda_N,\eps h}$ is absolute continuous to $\mu^{\pm}_{\Lambda_N,0 }$.
\begin{definition} 
Let $\cA_0(\iota)$ denote the collection of pairs $(h,\gamma)$ such that \begin{equation}
        e^{-\iota^{-4}}\le \frac{\mu^{\pm}_{\Lambda_N,\eps h}(\gamma)} {\mu^{\pm}_{\Lambda_N,0}(\gamma)}\le e^{\iota^{-4}}.\nonumber
    \end{equation}Let $\cH_0(\iota)$ denote the collection of external fields $h$ such that \begin{equation}
        \mu^{\pm}_{\Lambda_N,0}(\cA_0(\iota)\mid h):=\mu^{\pm}_{\Lambda_N,0}(\left\{\gamma:(h,\gamma)\in\cA_0(\iota)\right\}
        )\ge 1-e^{-\iota^{-0.1}}.\nonumber    
\end{equation}
\end{definition}
\begin{proof}[Proof of Theorem~\ref{thm: sle continuity critical main}, Inequality \eqref{eq: absolute continuity of sle with respect to the pure measure}]
    We first control the probability of $\cA_0(\iota)$. Note that $$\frac{\mu^{\pm}_{\Lambda_N,\eps h}(\gamma)}{\mu^{\pm}_{\Lambda_N,0}(\gamma)}=\frac{\mathsf Z^+\cdot \mathsf Z^-}{\mathsf Z}.$$ Thus, if $(h,\gamma)$ satisfies $\max\{|\log(\mathsf Z)|,|\log(\mathsf Z^+)|,|\log(\mathsf Z^-)|\}\le \iota^{-3.1}$, then $(h,\gamma)\in \cA_0(\iota)$.
    
    For any $\gamma$ such that $\frac{S(\gamma)}{N^{\frac{7}{4}}}\le \iota^{-1}$, we get from Proposition~\ref{prop: sle RN derivative critical} that \begin{align}
        \PP\Big[\max\{|\log(\mathsf Z)|,|\log(\mathsf Z^+)|,|\log(\mathsf Z^-)|\}\ge \iota^{-3.1}\Big]\le \exp\left(-C_1\iota^{-\frac{1}{4}}\right).\label{eq: good external field in sle pure rn derivative}
    \end{align}
     Combining with Lemma~\ref{lem: large deviation for sle interface dimension}, we get that \begin{equation}
        \PP\otimes\mu^{\pm}_{\Lambda_N,0}\big(\cA_0(\iota)^c\big)\le \mu^{\pm}_{\Lambda_N,0}\bigg(\frac{S(\gamma)}{N^{\frac{7}{4}}}> \iota^{-1}\bigg)+\exp\left(-C_1\iota^{-\frac{1}{4}}\right)\le \exp\left(-C_2\iota^{-\frac{1}{4}}\right).\label{eq: product measure bound of rn derivative in sle pure}
    \end{equation} Note that for any $h\notin \cH_0(\iota)$, we have \begin{equation}
        \mu^{\pm}_{\Lambda_N,0}(\cA_0(\iota)^c\mid h)\ge \exp\left(-\iota^{-0.1}\right).\nonumber    
\end{equation}Applying Markov inequality to \eqref{eq: product measure bound of rn derivative in sle pure}, we get that \begin{align}
    \PP\big(\cH_0(\iota)\big)\ge 1-\frac{\exp\left(-C_2\iota^{-\frac{1}{4}}\right)}{\exp\left(-\iota^{-0.1}\right)}\ge 1-\exp\left(-C_3\iota^{-\frac{1}{4}}\right).\nonumber
\end{align}Thus, we complete the proof of \eqref{eq: absolute continuity of sle with respect to the pure measure}.
\end{proof}
To prove the inequality \eqref{eq: absolute continuity of sle with respect to the external field measure} of Theorem \ref{thm: sle continuity critical main}, we need to control the Radon-Nikodym derivative under the measure $\mu^{\pm}_{\Lambda_N,\eps h}$. In particular, we need to show that interfaces which are atypical under $\mu^{\pm}_{\Lambda_N,0}$ remain atypical under $\mu^{\pm}_{\Lambda_N,\eps h}$. Parallel to Lemma~\ref{lem: large deviation for sle interface dimension}, our goal is to establish a large deviation estimate for $S(\gamma)$ under the measure $\mu^{\pm}_{\Lambda_N,\eps h}$. However, such an estimate cannot hold uniformly for arbitrary external fields $h$. This motivates the following definition.
\begin{definition}
    For any constant $\iota>0$
and any integer $k\ge \lfloor-\log_2\iota\rfloor$, let $\cH^{\mathrm{aty}}_k(\iota)$ denote the collection of external fields such that \begin{align}
    \sum_{\gamma:\frac{S(\gamma)}{N^{7/4}}\in (2^{k-1},2^k]}\mu^{\pm}_{\Lambda_N,\eps h}(\gamma)\le e^{-2^{k}}.\nonumber
\end{align}
\end{definition}
\begin{lemma}\label{lem: large deviation for sle interface dimension with disorder}
    There exists a constant $c_8>0$ such that $$\PP\big(\bigcap_{k=\lfloor-\log_2\iota\rfloor}^{\infty}\cH^{\mathrm{aty}}_k(\iota)\big)\ge 1-c_8\exp(-c_8^{-1}\iota^{-\frac{1}{4}}).$$
\end{lemma}
\begin{proof}
The idea is quite straightforward. Lemma~\ref{lem: large deviation for sle interface dimension} already provides a large deviation estimate for atypical interfaces under the measure $\mu^{\pm}_{\Lambda_N,0}$. Therefore, to obtain the corresponding estimate under $\mu^{\pm}_{\Lambda_N,\eps h}$, it suffices to control the Radon--Nikodym derivative
$\frac{\mu^{\pm}_{\Lambda_N,\eps h}(\gamma)}
{\mu^{\pm}_{\Lambda_N,0}(\gamma)}$
for interfaces satisfying
$N^{-7/4}S(\gamma)\in (2^{k-1},2^k].$
 Recall that $\frac{\mu^{\pm}_{\Lambda_N,\eps h}(\gamma)}{\mu^{\pm}_{\Lambda_N,0}(\gamma)}=\frac{\mathsf Z^+(\gamma)\cdot \mathsf Z^-(\gamma)}{\mathsf Z}$ and that $\mathsf Z$ is independent of $\gamma.$
 
 We begin by establishing a lower bound for $\mathsf Z$. 
 Let $\cH^*(\iota)$  denote the collection of external field such that $\mathsf Z\ge \exp(-c_1\iota^{-2}).$ 
 For any $\gamma$ such that ${S(\gamma)}{N^{-7/4}}\in (2^{k-1},2^k]$, we compute from Proposition~\ref{prop: sle RN derivative critical} and the fact that $\mathsf Z^+(\gamma) $ and $\mathsf Z^-(\gamma)$ are independent under $\PP$ that \begin{align}
        \EE \Big[\frac{\mu^{\pm}_{\Lambda_N,\eps h}(\gamma)}{\mu^{\pm}_{\Lambda_N,0}(\gamma)}\1_{\cH^*(\iota)}\Big]&\le \EE [\mathsf Z^+(\gamma)]\cdot \EE [\mathsf Z^-(\gamma)]\cdot \exp(c_1\iota^{-2})\nonumber\\\le~& \exp\Big( 2c_1(\frac{S(\gamma)}{N^{\frac{7}{4}}})^{1+3\kappa}\cdot \iota^{-2}\Big)\cdot \exp(c_1\iota^{-2})\le \exp\Big( C_12^{(1+3\kappa)k}\cdot \iota^{-2}\Big).\nonumber
    \end{align} For $k\ge \lfloor-\log_2\iota\rfloor$, summing over $\gamma$ and choosing $\kappa_1=\kappa$ in Lemma~\ref{lem: large deviation for sle interface dimension}, we obtain that \begin{align}
        \sum_{\gamma:\frac{S(\gamma)}{N^{7/4}}\in (2^{k-1},2^k]}\EE \Big[\mu^{\pm}_{\Lambda_N,\eps h}(\gamma)\1_{\cH^*(\iota)}\Big] &= \sum_{\gamma:\frac{S(\gamma)}{N^{7/4}}\in (2^{k-1},2^k]}\EE \Big[\frac{\mu^{\pm}_{\Lambda_N,\eps h}(\gamma)}{\mu^{\pm}_{\Lambda_N,0}(\gamma)}\1_{\cH^*(\iota)}\Big]\cdot \mu^{\pm}_{\Lambda_N,0}(\gamma)\nonumber\\
        \le~&  \sum_{\gamma:\frac{S(\gamma)}{N^{7/4}}\in (2^{k-1},2^k]}\exp\Big( C_12^{(1+3\kappa)k}\cdot \iota^{-2}\Big)\cdot \mu^{\pm}_{\Lambda_N,0}(\gamma)\nonumber\\
        \le~& \exp\Big( C_12^{(1+3\kappa)k}\cdot \iota^{-2}\Big)\cdot C_2\exp\big(-C_2^{-1}2^{(4-\kappa)(k-1)}\big)\nonumber\\
        \le~& C_3\exp\big(-C_3^{-1}2^{(4-\kappa)k} \big).\nonumber
    \end{align}
    Applying the Markov inequality then yields that \begin{align}
        \PP\big(\cH_k^{\mathrm{aty}}(\iota)^c\cap\cH^*(\iota)\big) &\le \sum_{\gamma:\frac{S(\gamma)}{N^{7/4}}\in (2^{k-1},2^k]}\EE \Big[\mu^{\pm}_{\Lambda_N,\eps h}(\gamma)\1_{\cH^*(\iota)}\Big]\cdot e^{2^{k}}\nonumber\\\le~&  C_3\exp\big(-C_3^{-1}2^{(4-\kappa)k}\big)\cdot e^{2^{k}}\le C_4\exp\big(-C_4^{-1}2^{(4-\kappa)k}\big).\label{eq: good external for the numerator in interface dimension with external field}
    \end{align}
    Summing \eqref{eq: good external for the numerator in interface dimension with external field} over $k\ge \lfloor-\log_2\iota\rfloor$ and combining with Proposition~\ref{prop: sle RN derivative critical}, we get that \begin{align}
        \PP\big(\bigcap_{k=\lfloor-\log_2\iota\rfloor}^{\infty}\cH^{\mathrm{aty}}_k(\iota)\big)&\ge 1-\PP(\cH^*(\iota)^c)-\sum_{k=\iota^{-1}}^{\infty}\PP\big(\cH_k^{\mathrm{aty}}(\iota)^c\cap\cH^*(\iota)\big)\nonumber\\&\ge 1-\exp(-C_5\iota^{-\frac{1}{4}})-\sum_{k=\lfloor-\log_2\iota\rfloor}^{\infty}C_4\exp\big(-C_4^{-1}2^{(4-\kappa)k}\big)\nonumber\\&\ge 1-\exp(-C_6\iota^{-\frac{1}{4}}).\label{eq: good external field for interface with external field}
    \end{align}Thus we finish the proof of Lemma~\ref{lem: large deviation for sle interface dimension with disorder}.
\end{proof}
\begin{proof}[Proof of Theorem~\ref{thm: sle continuity critical main}, Inequality \eqref{eq: absolute continuity of sle with respect to the external field measure}]
With Lemma~\ref{lem: large deviation for sle interface dimension with disorder} in place of Lemma \ref{lem: large deviation for sle interface dimension}, the proof is the same as that of \eqref{eq: absolute continuity of sle with respect to the pure measure}.
\end{proof}
\subsection{Large deviation for the interface without disorder}\label{sect: ldp for sle interface without disorder}
In this section, we prove Lemma~\ref{lem: large deviation for sle interface dimension}. Recall the definition of $S(\gamma)$ from \eqref{eq: definition of S gamma}.
In order to prove the large deviation bound for $S(\gamma)$, we first give an upper bound for $S(\gamma)$ in terms of the number of boxes on each scale that $\gamma$ intersects.
Without loss of generality, we assume that $\log_2 N\in \mathbb Z.$
For any $1\leq k\leq \log_2N$, we partition $\lamn$ into $\frac{N}{2^k}$ boxes.  Let \(T_k\) denote the number of \(\frac{N}{2^k}\)-boxes intersecting \(\gamma\). For notation convenience, we write $T_k=0$ if $2^k>N.$ Since every vertex \(x\) with \(\dist(x,\gamma)\le \frac{N}{2^k}\) belongs to some \(\frac{N}{2^k}\)-box intersecting \(\gamma\), we obtain
\begin{equation}
\big|\{x\in\lamn:\tfrac{N}{2^{k+1}}\le \dist(x,\gamma)\le\tfrac{N}{2^k}\}\big|
\le
T_k\cdot \Big(\frac{N}{2^k}\Big)^2.
\label{eq: interface distance vertex bound}
\end{equation}
Therefore, \begin{align}
S(\gamma)&\le\sum_{x\in\Lambda_N}\dist(x,\partial\Lambda_N)^{-\frac{1}{4}}+\sum_{x\in\Lambda_N}\dist(x,\gamma)^{-\frac{1}{4}}\nonumber\\&\le 32N^{\frac{7}{4}}+ \sum_{k=1}^{\infty}\big|\{x:\frac{N}{2^{k+1}}\le \dist(x,\gamma)\le\frac{N}{2^k}\}\big| \cdot  (\frac{N}{2^{k+1}})^{-\frac{1}{4}}\le 32N^{\frac{7}{4}}+\sum_{k=1}^{\infty} 2T_k\cdot  (\frac{N}{2^{k}})^{\frac{7}{4}}.\label{eq: S gamma upper bound using Tk}
\end{align}Next, we give large deviation bounds for $T_k$. 
\begin{lemma}\label{lem: large deviation for Tk in sle}
    For any $\kappa_2\in (0,0.1)$, there exists  a constant $c_9=c_9(\kappa_2)>0$ such that for any $\lambda>0$, we have \begin{equation}
        \mu^{\pm}_{\lamn,0}(T_k\ge \lambda\cdot 2^{\frac{7k}{4}-\kappa_2 k})\le c_9\exp\big(-c_9^{-1}\lambda^{\frac{8}{5}} 2^{(\frac{3-16\kappa_2}{5})k} \big).\label{eq: large deviation bound for Tk in sle interface}
    \end{equation}
\end{lemma}
\begin{proof}
     Let $\mathfrak{D}^k$ denote the collection of $\frac{N}{2^k}$ boxes. For any $\fD\in\mathfrak{D}^k$, let $\mathcal{E}_{\fD}$ denote the event that the interface intersects $\fD$. 
    Note that for any integer $d\ge 0$, if $\mathcal{E}_{\fD}$ occurs, then there exists a plus crossing and a minus crossing in the annulus $d\fD\setminus \fD$. Here $d \fD$ denotes the box with the same center as $\fD$ and has side length $\frac{dN}{2^{k-1}}$ (note that an $\frac{N}{2^k}$-box is a translate of $[-\frac{N}{2^k}, \frac{N}{2^k}]^{ 2}$, and thus have side length $\frac{N}{2^{k-1}}$). By \cite{armExponentsIsing}, we get that for any configuration $\xi$ on $\lamn\setminus d\fD$, \begin{equation}
        \mu^\xi_{d\fD, 0}(\mathcal{E}_\fD)\le C_1 (d+1)^{-\frac{5}{8}+\kappa_2}.\label{eq: interface two arm bound}
    \end{equation} Here we use the convention that $\mu^\xi_{d\fD, 0}(\mathcal{E}_\fD)=1$ if $d=0$.
    Next, we compute the moments of $T_k$. For any integer $r> 0$, we compute that \begin{equation}
    \langle (T_k)^r\rangle_{\lamn,0}^{\pm}=\sum_{\fD_i\in\mathfrak{D}^k,1\le i\le r}\mu^{ \pm}_{\lamn, 0}(\cap_{i=1}^r\mathcal{E}_{\fD_i}).\label{eq: Tk moment expansion}
    \end{equation} Fix a sequence of boxes $\fD_1,\cdots, \fD_{r}\in \mathfrak D^k$. For any $1\le i\le r$, let $$d_i=\Big\lfloor\frac{\min\left\{\min_{j\neq i}\left\{\dist(\fD_i,\fD_j)\right\},\dist(\fD_i,\partial\lamn)\right\}}{\frac{N}{2^{k-1}}}\Big\rfloor$$ be the box distance of $\fD_i$ to the boundary and other box in $\{\fD_1,\fD_2,\cdots,\fD_r\}$. Since \(d_i\) may vanish, we use \(d_i+1\) throughout. Then we get from DMP and \eqref{eq: interface two arm bound} that  \begin{equation}
        \mu^\pm_{\lamn, 0}(\cap_{i=1}^r\mathcal{E}_{\fD_i})\le \prod_{i=1}^{r}C_1 (d_i+1)^{-\frac{5}{8}+\kappa_2}.\nonumber
    \end{equation}Plugging into \eqref{eq: Tk moment expansion} we get that\begin{align}
         \langle (T_k)^r\rangle_{\lamn,0}^{\pm}\le C_1^r \sum_{\fD_i\in\mathfrak{D}^k,1\le i\le r}\prod_{i=1}^{r}C_1 (d_i+1)^{-\frac{5}{8}+\kappa_2}.\label{eq: moment bound for Tk in interface 1}
    \end{align}
    Note that \(\mathfrak D^k\) can be viewed as a coarse-grained graph that is isomorphic to \(\Lambda_{2^k}\). Moreover,
    $$\sum_{\mathsf D\in \mathfrak D_k}(\Big\lfloor\frac{\dist(\fD,\partial\lamn)}{\frac{N}{2^{k-1}}}\Big\rfloor+1)^{-\frac{5}{8}+\kappa_2}\leq 64\times (2^k)^{\frac{11}{8}+\kappa_2}.$$
     Applying  Lemma~\ref{lem: upper-bound for the sum of squares of k point function}   with $s=\frac{5}{8}-\kappa_2, \Lambda=\mathfrak D_k, S=\partial\mathfrak D_k, t=64$ and combining with \eqref{eq: moment bound for Tk in interface 1}, we get that \begin{align}
         \langle (T_k)^r\rangle_{\lamn,0}^{\pm}\le C_2^r 2^{(\frac{11}{8}+\kappa_2)rk}\cdot (r!)^{\frac{5}{8}-\kappa_2}.\nonumber
    \end{align} Combined with the Holder's inequality, it yields that for any integer $m\in\{1,2,3,4\}$, \begin{align}
        \langle (T_k)^{r+\frac{m}{5}}\rangle_{\lamn,0}^{\pm}&\le \Big( \big(\langle (T_k)^{r}\rangle_{\lamn,0}^{\pm}\big)^{5-m}\cdot  \big(\langle (T_k)^{r+1}\rangle_{\lamn,0}^{\pm}\big)^{m}\Big)^{\frac{1}{5}}\nonumber\\&\le C_2^{r+1} 2^{(\frac{11}{8}+\kappa_2)(r+\frac{m}{5})k}\cdot \big((r+1)!\big)^{\frac{5}{8}-\kappa_2}.\label{eq: moment bound for Tk in interface}
    \end{align}
    
    Now we are ready to compute the exponential moment of $(tT_k)^{\frac{8}{5}}$ with $t\le C\cdot 2^{-(\frac{11}{8}+\kappa_2)k}$ for some small enough $C>0$. By \eqref{eq: moment bound for Tk in interface}, we get that \begin{align}
        \langle  \exp\big((tT_k)^{\frac{8}{5}}\big)\rangle_{\lamn,0}^{\pm}=\sum_{r=0}^{\infty}\frac{t^{\frac{8r}{5}}}{r!}\cdot \langle  (T_k)^{\frac{8r}{5}} \rangle_{\lamn,0}^{\pm}\le \sum_{r=0}^{\infty}\frac{t^{\frac{8r}{5}}}{r!}\cdot C_2^{\lceil\frac{8r}{5}\rceil} 2^{(\frac{11}{8}+\kappa_2)\lceil\frac{8r}{5}\rceil k}\cdot\big( \lceil\frac{8r}{5}\rceil!\big)^{\frac{5}{8}-\kappa_2}.\nonumber
    \end{align}Combined with the fact that $\lceil\frac{8r}{5}\rceil!\le (\frac{8r}{5})^{\frac{8r}{5}}$ and $r!\ge (\frac{r}{e})^{r}$, it yields that \begin{align}
        \langle  \exp\big((tT_k)^{\frac{8}{5}}\big)\rangle_{\lamn,0}^{\pm}&\le \sum_{r=0}^{\infty}\frac{t^{\frac{8r}{5}}}{(\frac{r}{e})^{r}}\cdot C_2^{\lceil\frac{8r}{5}\rceil} 2^{(\frac{11}{8}+\kappa_2)\lceil\frac{8r}{5}\rceil k}\cdot (\frac{8r}{5})^{(1-\frac{8\kappa_2}{ 5})r}.\nonumber\\&\le \sum_{r=0}^{\infty} C_3^{\lceil\frac{8r}{5}\rceil} (t2^{(\frac{11}{8}+\kappa_2)k})^{\lceil\frac{8r}{5}\rceil} \le C_4\label{eq: super exponential moment for Tk in interface}
    \end{align}where we used the fact that $t\le C2^{-(\frac{11}{8}+\kappa_2)k}$ and chose $C>0$ to be small enough in the last inequality. Now we are ready to prove \eqref{eq: large deviation bound for Tk in sle interface}. Let $t=C\cdot 2^{-(\frac{11}{8}+\kappa_2)k}$ and apply the Markov inequality, we get that \begin{align}
        \mu^{\pm}_{\lamn,0}(T_k\ge \lambda\cdot 2^{\frac{7k}{4}-\kappa_2 k})&\le \langle  \exp\big((tT_k)^{\frac{8}{5}}\big)\rangle_{\lamn,0}^{\pm}\cdot \exp\big(-(t\cdot \lambda\cdot 2^{\frac{7k}{4}-\kappa_2 k})^{\frac{8}{5}} \big)\nonumber\\&= C_5\exp\big(-C_6 \lambda^{\frac{8}{5}} 2^{\frac{3-16\kappa_2}{5}k} \big).
    \end{align}Thus we complete the proof of Lemma~\ref{lem: large deviation for Tk in sle}.    
\end{proof}
\begin{proof}[Proof of Lemma~\ref{lem: large deviation for sle interface dimension}]
By choosing $c_7>0$ large enough, we can, without loss of generality, assume $\lambda>0$ is large enough.
Let $\kappa_2=\frac{5\kappa_1}{64}$
    and let $k_0=\lfloor4\log_2(\lambda)\rfloor-100>0$. From the trivial bound that $T_k\leq (2^{k})^2$ we get that \begin{equation}
        \sum_{k=1}^{k_0}T_k\cdot (\frac{N}{2^k})^{\frac{7}{4}}\le \sum_{k=1}^{k_0} N^{\frac{7}{4}}\cdot 2^{\frac{k}{4}}\le N^{\frac{7}{4}}2^{\frac{k_0}{4}}\cdot \frac{1}{1-2^{-\frac{1}{4}}}\le N^{\frac{7}{4}}\frac{\lambda}{4}.\nonumber
    \end{equation}

    Combining with \eqref{eq: S gamma upper bound using Tk} and the fact that $$\sum_{k=k_0+1} 2^{\frac{7k}{4}-\kappa_2 k+1}\cdot (\frac{N}{2^k})^{\frac{7}{4}}\le N^{\frac{7}{4}}\cdot \frac{1}{2C_1}$$ for some $C_1>0$, we get that \begin{equation}
        \mu^{\pm}_{\Lambda_N,0 }\Big(S(\gamma)\ge N^{\frac{7}{4}}\lambda\Big)\le \sum_{k=k_0+1}^{\infty}\mu^{\pm}_{\lamn,0}(T_k\ge C_1\lambda\cdot 2^{\frac{7k}{4}-\kappa_2 k}).\label{eq: upper bound for small levels in large devitaion of interface}
    \end{equation}
Applying Lemma~\ref{lem: large deviation for Tk in sle} to \eqref{eq: upper bound for small levels in large devitaion of interface}, we get that \begin{align}
   \mu^{\pm}_{\Lambda_N,0 }\Big(S(\gamma)\ge N^{\frac{7}{4}}\lambda\Big)&\le  \sum_{k=k_0+1}^{\infty}C_2\exp\big(-C_3 \lambda^{\frac{8}{5}} 2^{\frac{3-16\kappa_2}{5}k} \big)\le C_4\exp\big(-C_3 \lambda^{\frac{8}{5}} 2^{\frac{3-16\kappa_2}{5}k_0} \big)\nonumber\\&\le C_4\exp\big(-C_5 \lambda^{4-\frac{64\kappa_2}{5}})\nonumber
\end{align}where we use the fact that $k_0=\lfloor4\log_2(\lambda)\rfloor-100$ and $\kappa_2=\frac{5\kappa_1}{64}$ in the last inequality.
\end{proof}
\begin{remark}\label{rmk: large deviation for sle interface dimension}
The same proof of Lemma~\ref{lem: large deviation for sle interface dimension} yields a corresponding bound for the interface in the critical FK–Ising model. One simply replaces \eqref{eq: interface two arm bound}  with the following estimate from
\cite{Wu18FK}: \begin{equation}
        \mu^\xi_{d\fD, 0}(\mathcal{E}_\fD)\le C_1 (d+1)^{-\frac{1}{3}+\kappa_2}.\nonumber
    \end{equation} All other steps carry over verbatim, with only minor modifications to the constants.
\end{remark}
\section{Singularity for the outermost loops in the discrete settings}\label{sec: singularity for outermost cle}

In this section, we establish the singularity of the laws of the outermost loops in the Ising model with and without disorder in the discrete setting.
To this end, we introduce a statistic that captures the discrepancy between the two measures. Fix $\iota>0$ and let $M=\lfloor \iota N\rfloor$. We partition $\Lambda_N$ into $M$-boxes $\fB_1,\ldots,\fB_n$ (for simplicity, we assume that $N$ is divisible by $M$). Let $u_i$ denote the center of $\fB_i$, so that $\fB_i=\Lambda_M(u_i)$.
Analogously to Definition~\ref{def: plus minus interface}, one may define a plus-minus circuit and it is essentially the same as the definition in Section~\ref{sec:def-of-loops}. We then let $X_i$ be the indicator of the event that there exists an outermost plus-minus circuit in the annulus
$\Lambda_{M/2}(u_i)\setminus \Lambda_{M/4}(u_i).$
Throughout this section, we study the joint distribution of
$(X_1,\ldots,X_n).$
For any $\tau\in\{0,1\}^n$, write
\[
\mathcal E_\tau=\{\sigma: X_i(\sigma)=\tau_i \text{ for all } i=1,\ldots,n\} \quad \mbox{and}\quad  |\tau|:=\sum_{i=1}^n \tau_i.
\]
In this section, we work at the critical inverse temperature $\beta=\beta_c$, and thus we omit $\beta$ from the subscripts.
Without loss of generality, we restrict our attention to the critical scaling
$\eps=N^{-7/8}.$ The case $\eps=cN^{-7/8}$ for any fixed $c\in(0,\infty)$ can be treated in exactly the same way. The following theorem is the main result of this section.

\begin{theorem}\label{thm: cle outermost singularity}
  For any $\iota>0$, there exists a constant $N_0=N_0(\iota)>0$ and an absolute constant $c>0$ 
     such that for any integer $N>N_0$ and disorder strength $\eps=N^{-\frac{7}{8}}$, we have \begin{equation}\label{eq: cle singularity}
        \EE\left[ \frac{1}{2}\sum_{\tau\in\{0,1\}^n}\Big|\mu^+_{\lamn,0}(\mathcal{E}_\tau)-\mu^+_{\lamn,\eps h}(\mathcal{E}_\tau)\Big|\right]\ge 1-c^{-1}\iota^{c}.
    \end{equation}
\end{theorem}
\begin{remark}
    The term $\frac{1}{2}\sum_{\tau\in\{0,1\}^n}\Big|\mu^+_{\lamn,0}(\mathcal{E}_\tau)-\mu^+_{\lamn,\eps h}(\mathcal{E}_\tau)\Big|$ can be viewed as the total variation distance between the joint distribution of $(X_1,\cdots,X_n)$ under the two measures $\mu^+_{\lamn,0}$ and $\mu^+_{\lamn,\eps h}$. Note that $X_i$ captures the information from the  outermost loops in $\lamn$. Thus, Theorem~\ref{thm: cle outermost singularity} implies the singularity between the outermost loops under $\mu^{+}_{\Lambda_{N}, 0}$ and the outermost loops under $\mu^+_{\Lambda_{N}, \eps h}$. 
\end{remark}

Our main intuition for Theorem \ref{thm: cle outermost singularity} is that if $\tau_i=1$, then the external field on the box $\fB_i$ changes the probability of the event $\mathcal{E}_\tau$ a little bit. The contribution of the external field on each box $\fB_i$ in total makes a big change in the probability of  $\mathcal{E}_\tau$. To start with, we show that under $\mu^+_{\lamn,0}$,  $|\tau|=\sum_{i=1}^n\tau_i$ is typically large.

\begin{lemma}\label{lem: outermost loops number lower bound 1}
    For any $\alpha\in(\frac{7}{144},\frac{1}{4})$, there exists a constant $c_1=c_1(\alpha)>0$ such that \begin{equation}
        \sum_{|\tau|\ge c_1 n^{\frac{15}{16}-\alpha}}\mu^+_{\lamn,0}(\mathcal{E}_\tau)\ge 1-c_1^{-1}\exp(-c_1n^{\frac{144\alpha-7}{112}}).
    \end{equation}
\end{lemma}
\begin{remark}
    We remark here that neither the quantity $\frac{15}{16}$ or $\frac{15}{16}-\frac{7}{144}$ is expected to be optimal. Nevertheless, since
$\frac{15}{16}-\frac{7}{144}>\frac{7}{8},$
the above estimate already suffices for the purposes of this paper. We further conjecture that the optimal exponent is $\frac{187}{192}$. More precisely, for every $\kappa_3>0$,$$\sum_{|\tau|\ge c_1 n^{\frac{187}{192}-\kappa_3}}\mu^+_{\lamn,0}(\mathcal{E}_\tau)\to 1
\qquad\text{as }N\to\infty.$$ 
\end{remark}
\begin{proof}[Proof of Lemma~\ref{lem: outermost loops number lower bound 1}]
Recall from Section~\ref{sec: extended ising} that, under the Edwards--Sokal coupling, the measure $\mu^{+}_{\Lambda_N,0}$ corresponds to the wired FK measure $\phi^{\w}_{\Lambda_N,0}$. Therefore, throughout this proof, whenever we work with FK configurations, we shall freely identify $\mu^+_{\Lambda_N,0}$ with $\phi^{\w}_{\Lambda_N,0}$.

The proof proceeds in two steps. We first use the FK measure to construct, in many annuli $\Lambda_{M/2}(u_i)\setminus\Lambda_{M/4}(u_i)$, an FK circuit which is connected to $\partial\Lambda_N$. Such a circuit will later serve as the plus circuit. Then, inside each region enclosed by such a circuit, we use the DMP, CBC, and RSW estimates to construct, with uniformly positive probability, an inner FK circuit which could be assigned the minus spin. This will produce a plus-minus circuit in $\Lambda_{M/2}(u_i)\setminus\Lambda_{M/4}(u_i)$.

For any $M$-box $\Lambda_M(u_i)$, let $X_i^1$ be the indicator of the event that there exists an FK circuit in the annulus
$\Lambda_{M/2}(u_i)\setminus\Lambda_{5M/12}(u_i)$ which is connected to $\partial\Lambda_N$. For any $\tau\in\{0,1\}^n$, define
\[
    \widetilde{\mathcal E}_\tau=\{\omega:X_i^1(\omega)=\tau_i,\ 1\le i\le n\},
\]
where $\omega$ denotes the FK edge configuration under $\phi^{\w}_{\Lambda_N,0}$.
We shall use the following claim, whose proof is postponed until the end of the proof of the lemma.

\begin{claim}\label{claim: outermost loops number lower bound}
For any $\alpha>\frac{7}{144}$, there exists a constant $C_1=C_1(\alpha)>0$ such that
    \begin{equation}
    \sum_{|\tau|\ge C_1n^{\frac{15}{16}-\alpha}}
    \mu^+_{\Lambda_N,0}(\widetilde{\mathcal E}_\tau)
    =
    \sum_{|\tau|\ge C_1n^{\frac{15}{16}-\alpha}}
    \phi^{\w}_{\Lambda_N,0}(\widetilde{\mathcal E}_\tau)
    \ge 1-\exp\big(-C_1n^{\frac{144\alpha-7}{112}}\big).
    \label{eq: outermost loops number lower bound}
\end{equation}
\end{claim}

By Claim~\ref{claim: outermost loops number lower bound}, with high probability there are at least $C_1n^{\frac{15}{16}-\alpha}$ indices $i$ for which $X_i^1=1$. We next show that, for each such index, the annulus
$\Lambda_{M/2}(u_i)\setminus\Lambda_{M/4}(u_i)$
contains a plus-minus circuit with probability bounded below uniformly.

Indeed, explore the FK configuration on
$\Lambda_N\setminus\big(\cup_{i=1}^n\Lambda_{M/2}(u_i)\big)$, and then explore each annulus
$\Lambda_{M/2}(u_i)\setminus\Lambda_{5M/12}(u_i)$ until the values of $X_i^1$ are determined. If $X_i^1=1$, let $\Lambda_i$ be the region enclosed by the outermost FK circuit in
$\Lambda_{M/2}(u_i)\setminus\Lambda_{5M/12}(u_i)$ which is connected to $\partial\Lambda_N$. Since the outer boundary condition is plus, the boundary of $\Lambda_i$ carries the plus boundary condition. It remains to construct, inside $\Lambda_i$, a minus circuit surrounding $\Lambda_{M/4}(u_i)$.

Let $X_i^2$ be the indicator of the event that there exists a dual FK circuit in
$\Lambda_{5M/12}(u_i)\setminus\Lambda_{M/3}(u_i)$, and let $X_i^3$ be the indicator of the event that there exists an FK circuit in
$\Lambda_{M/3}(u_i)\setminus\Lambda_{M/4}(u_i)$. Let
$\Xi^2\subset\{0,1\}^{E(\Lambda_{5M/12}(u_i)\setminus\Lambda_{M/3}(u_i))}$
denote the set of configurations on
$\Lambda_{5M/12}(u_i)\setminus\Lambda_{M/3}(u_i)$ for which $X_i^2=1$. Then
\begin{equation}
     \sum_{\omega\in \Xi^2}\phi^{\w}_{\Lambda_i,0}(\omega)
     \ge
     \phi^{\w}_{\Lambda_{5M/12}(u_i)\setminus\Lambda_{M/3}(u_i),0}(X_i^2=1).
     \label{eq: X2 lower bound in small region}
\end{equation}
Combining the domain Markov property, comparison between boundary conditions, and RSW estimates, we obtain that
 \begin{align}
    \phi_{\Lambda_i,0}^{\w}(X_i^2X_i^3=1)
    &\stackrel{\mathrm{DMP}}{\ge}
    \sum_{\omega\in \Xi^2}\phi^{\w}_{\Lambda_i,0}(\omega)\,
    \phi^{\w}_{\Lambda_i,0}(X_i^3=1\mid \omega)
    \nonumber\\
    &\stackrel{\mathrm{CBC}}{\ge}
    \sum_{\omega\in \Xi^2}\phi^{\w}_{\Lambda_i,0}(\omega)\,
    \phi^{\f}_{\Lambda_{M/3}(u_i)\setminus\Lambda_{M/4}(u_i),0}(X_i^3=1)
    \nonumber\\
    &\stackrel{\eqref{eq: X2 lower bound in small region}}{\ge}
    \phi^{\w}_{\Lambda_{5M/12}(u_i)\setminus\Lambda_{M/3}(u_i),0}(X_i^2=1)\,
    \phi^{\f}_{\Lambda_{M/3}(u_i)\setminus\Lambda_{M/4}(u_i),0}(X_i^3=1)
    \nonumber\\
    &\stackrel{\mathrm{RSW}}{\ge} C_2 .
    \label{eq: FK circuit connected to boundary in small region}
\end{align}

On the event $\{X_i^2X_i^3=1\}$, let $\gamma$ be the outermost FK circuit in
$\Lambda_{M/3}(u_i)\setminus\Lambda_{M/4}(u_i)$, and let $X_i^4$ be the event that $\gamma$ is assigned the minus spin. Since $X_i^2=1$, the circuit $\gamma$ is not connected to $\partial\Lambda_i$. Hence, under the Edwards--Sokal coupling,
\[
    \phi_{\Lambda_i,0}^{\w}(X_i^4=1\mid X_i^2X_i^3=1)=\frac12 .
\]
Together with \eqref{eq: FK circuit connected to boundary in small region}, this gives
\begin{equation}
    \phi_{\Lambda_i,0}^{\w}(X_i^2X_i^3X_i^4=1)\ge C_3 .
    \label{eq: FK circuit connected to boundary in small region 1}
\end{equation}

If $X_i^1X_i^2X_i^3X_i^4=1$, then there exists a plus circuit in
$\Lambda_{M/2}(u_i)\setminus\Lambda_{M/4}(u_i)$, which is connected to $\partial\Lambda_N$, and there also exists a minus circuit in the same annulus. Hence, an outermost plus-minus circuit exists in
$\Lambda_{M/2}(u_i)\setminus\Lambda_{M/4}(u_i)$, and therefore $X_i=1$.

Consequently, by the DMP and \eqref{eq: FK circuit connected to boundary in small region 1}, for any $I\subset\{1,\ldots,n\}$, conditioned on $X_i^1=1$ for all $i\in I$, the family $\{X_i\}_{i\in I}$ stochastically dominates a Bernoulli process with density $C_3$. Choose $C_4\le C_1C_3/10$. Combining Hoeffding's inequality with Claim~\ref{claim: outermost loops number lower bound}, we get
\begin{align}
    \sum_{|\tau|< C_4n^{\frac{15}{16}-\alpha}}
    \phi^{\w}_{\Lambda_N,0}(\mathcal E_\tau)
    &\le
    \sum_{|\tau|< C_1n^{\frac{15}{16}-\alpha}}
    \phi^{\w}_{\Lambda_N,0}(\widetilde{\mathcal E}_\tau)
    +2\exp\big(-C_5n^{\frac{15}{16}-\alpha}\big)
    \nonumber\\
    &\le C_6^{-1}\exp\big(-C_6n^{\frac{144\alpha-7}{112}}\big),
\end{align}
where the last inequality also used the fact that 
$\frac{15}{16}-\alpha>\frac{144\alpha-7}{112}$ for $\alpha<\frac{1}{4}.$
This completes the proof of Lemma~\ref{lem: outermost loops number lower bound 1}.
\end{proof}

\begin{proof}[Proof of Claim~\ref{claim: outermost loops number lower bound}]
        We choose $\tilde{N}=n^{-\alpha^\prime} N$ with $\alpha^\prime\in (0,1)$ to be defined later. Without loss of generality, we can assume $N\in\tilde N\mathbb Z.$ Partition $\Lambda_N$ into boxes of side length $\widetilde N$. We call a box in this partition a boundary $\widetilde N$-box if it intersects $\partial\Lambda_N$.
 We want to lower-bound the number of $M$-boxes $\fB_i$ in a boundary $\tilde{N}$-box such that $X_i^1=1.$    For a boundary $\tilde{N}$-box $\fD=\Lambda_{\tilde{N}}(v)$, let $I_{\fD}$ denote the collection of indices $i$ such that $\Lambda_{M}(u_i)\subset \Lambda_{\tilde{N}/2}(v)$.  For any $i\in I_{\fD}$ let $\eta_{\fD}$ be the boundary condition on $\partial \fD$ where it is wired on $\partial \fD\cap\partial\lamn$ and is free on $\partial \fD\setminus\partial\lamn$. In addition, for any $i\in I_{\fD}$, let $\tilde{X}^1_i$ denote the event that there exists an FK circuit in $\Lambda_{M/2}(u_i)\setminus\Lambda_{5M/12}(u_i)$ and let $\tilde{X}^2_i$ denote the event that $\partial\Lambda_{5M/12}(u_i)$ is connected to the boundary $\partial \fD\cap\partial\lamn$.  Note that $\tilde{X}^1_i\cdot\tilde{X}^2_i\le X_i^1$, thus we get that \begin{equation}
      \sum_{i\in I_{\fD}}X_i^1\ge \sum_{i\in I_{\fD}}\tilde{X}^1_i\cdot\tilde{X}^2_i.\label{eq: FK circuit connected to boundary 1}
    \end{equation} By FKG and RSW, we get that \begin{equation}
        \phi_{\fD,0}^{\eta_{\fD}}(\tilde{X}^1_i\cdot\tilde{X}^2_i=1)\ge \phi_{\fD,0}^{\eta_{\fD}}(\tilde{X}^1_i=1)\cdot \phi_{\fD,0}^{\eta_{\fD}}(\tilde{X}^2_i=1)\ge C_2(\frac{\tilde{N}}{M})^{-\frac{1}{8}}.\nonumber\label{eq: FK circuit connected to boundary 2}
    \end{equation}Summing over $i\in I_{\fD}$, we get that \begin{equation}
       \langle\sum_{i\in I_{\fD}}\tilde{X}^1_i\cdot\tilde{X}^2_i\rangle_{\fD,0}^{\eta_{\fD}}\ge \frac{C_2}{4}(\frac{\tilde{N}}{M})^{\frac{15}{8}}.\label{eq: FK circuit connected to boundary final}
    \end{equation} In addition, by CBC and \eqref{eq: FK one arm event exponent}, we get that \begin{equation}
         \phi_{\fD,0}^{\eta_{\fD}}(\tilde{X}^1_i\cdot\tilde{X}^2_i=1)\le \phi_{\fD,0}^{\eta_{\fD}}(\tilde{X}^2_i=1)\le \phi^{\w}_{\Lambda_{\tilde{N}/2}}(\partial\Lambda_{\tilde{N}/2}\leftrightarrow\partial \Lambda_{M})\le C_3(\frac{\tilde{N}}{M})^{-\frac{1}{8}}.\nonumber
    \end{equation} Summing over $ i,j\in I_{\fD}$, we compute that \begin{equation}
          \Big\langle\big(\sum_{i\in I_{\fD}}\tilde{X}^1_i\cdot\tilde{X}^2_i\big)^2\Big\rangle_{\fD,0}^{\eta_{\fD}}=\sum_{i,j\in I_{\fD}}\phi_{\fD,0}^{\eta_{\fD}}(\tilde{X}^1_i\cdot\tilde{X}^2_i\cdot \tilde{X}^1_j\cdot\tilde{X}^2_j=1)\le C_3(\frac{\tilde{N}}{M})^{\frac{31}{8}}.\label{eq: FK circuit connected to boundary 3}
    \end{equation}
    Let $\Phi$ be the product measure of $\phi_{\fD,0}^{\eta_{\fD}}$ over all the boundary $\tilde{N}$-box.  By \eqref{eq: FK circuit connected to boundary final}, \eqref{eq: FK circuit connected to boundary 3} and Bernstein inequality, we get that \begin{align}
\Phi\Big(\sum_{i=1}^n\tilde{X}^1_i\cdot\tilde{X}^2_i\le \frac{C_2}{2}\cdot(\frac{\tilde{N}}{M})^{\frac{15}{8}}\cdot\frac{N}{\tilde{N}}\Big)&\le \exp\Big(- \frac{C_4(\frac{\tilde{N}}{M})^{\frac{15}{4}}\cdot(\frac{N}{\tilde{N}})^2}{C_5(\frac{\tilde{N}}{M})^{\frac{31}{8}}\cdot \frac{N}{\tilde{N}}+C_6(\frac{\tilde{N}}{M})^{\frac{15}{8}}\cdot\frac{N}{\tilde{N}}\cdot (\frac{\tilde{N}}{M})^{2}}\Big)\nonumber\\&\le\exp\Big(-C_8 n^{\frac{18\alpha^\prime-1}{16}}\Big).\nonumber
    \end{align} Combined with the fact that $\phi_{\lamn,0}^{\w}$ stochastically dominates $\Phi$ and $\tilde{X}^1_i\cdot\tilde{X}^2_i$ is increasing, it yields that \begin{equation}
\phi_{\lamn,0}^{\w}\Big(\sum_{i=1}^n\tilde{X}^1_i\cdot\tilde{X}^2_i\le \frac{C_2}{2}\cdot(\frac{\tilde{N}}{M})^{\frac{15}{8}}\cdot\frac{N}{\tilde{N}}\Big)\le \exp\Big(-C_3 n^{\frac{18\alpha^\prime-1}{16}}\Big).\label{eq: good circuit boundary box number deviation}
    \end{equation}Note that $(\frac{\tilde{N}}{M})^{\frac{15}{8}}\cdot\frac{N}{\tilde{N}}=n^{\frac{15-14\alpha^\prime}{16}}$, thus we can choose $\alpha^\prime=\frac{8\alpha}{7}$.
\end{proof}
From now on, we fix $\alpha=\frac{7}{144}+\kappa$ where we recall that $\kappa=10^{-5}.$
We call a configuration $\tau\in\{0,1\}^n$ good if $|\tau|=\sum_{i=1}^n\tau_i\ge c_1n^{\frac{15}{16}-\alpha}$ where $c_1=c_1(\alpha)$ is the constant determined in Lemma~\ref{lem: outermost loops number lower bound 1}. By Lemma~\ref{lem: outermost loops number lower bound 1}, we get that good configurations have overwhelming probability.
 From now on, we will fix a good configuration $\tau$. As in Section~\ref{sec: Mutual absolute continuity of the single interface in the scaling limit}, we aim to analyze the effect of the external field on the ratio  $\frac{\mu^+_{\lamn,\eps h}(\mathcal{E}_\tau)}{\mu^+_{\lamn,0}(\mathcal{E}_\tau)}$ box by box. So we still define $H^i$ by
\begin{equation*}
    H^i_u=\left\{\begin{aligned}
        h_u,~~~&u\in \cup_{j=1}^i \fB_j\,,\\
        0,~~~&u\in \cup_{j=i+1}^{n} \fB_j,
\end{aligned}\right.
\end{equation*} 
although the definition of $\fB_j$ is different from Section~\ref{sec: Mutual absolute continuity of the single interface in the scaling limit}.
We also use the convention that $H^0\equiv 0.$

 The next proposition shows that introducing the external field in \(\mathsf B_i\) produces fluctuations in the ratio $\frac{\mu^+_{\lamn,\eps h}(\mathcal{E}_\tau)}{\mu^+_{\lamn,0}(\mathcal{E}_\tau)}$. 

\begin{proposition}\label{prop: adding external field fluctuation} Fix a good configuration $\tau\in\{0,1\}^n$.
    Let $\sfa_i=\frac{\mu^+_{\lamn,\eps H^i}(\mathcal{E}_\tau)}{\mu^+_{\lamn,\eps H^{i-1}}(\mathcal{E}_\tau)}-1$. For any $\kappa_1>0$, there exist absolute constants $c_2,c_3>0$, $c_4=c_4(\kappa_{1})>0$ and three sequences of random variables $\sfb_1,\cdots,\sfb_{n}$, $\sfc_1,\cdots,\sfc_{n}$ and $\sfd_1,\cdots,\sfd_{n}$ such that the following holds. We remark that $\sfa_i$, $\sfb_i$, $\sfc_i$, and $\sfd_i$ depend on the choice of $\tau$. When this dependency needs to be emphasized, we write them as $\sfa_i(\tau)$, $\sfb_i(\tau)$, $\sfc_i(\tau)$, and $\sfd_i(\tau)$.
    \begin{enumerate}[label=(\alph*)]
    \item There exists a filtration $\mathcal{F}_i$ such that for any $1\le i\le n$, $\sfa_i,\sfb_i,\sfc_i$ and $\sfd_i$ are measurable with respect to $\mathcal{F}_i$.
        \item For any $1\le i\le n$, conditioned on $\mathcal{F}_{i-1}$, $\sfb_i$ is a mean $0$ Gaussian variable with variance $\sigma_i^2$ where $ \sigma_i^2\le c_2\eps^2M^{\frac{7}{4}}$. In addition, if $\tau_i=1$, then we have $\sigma_i^2\ge c_3\eps^2M^{\frac{7}{4}}.$ \label{item: bi}
        \item For any $1\le i\le n$, conditioned on $\mathcal{F}_{i-1}$, $\sfc_i$ has mean $0$. In addition, we have $\EE \big(\sum_{i=1}^{n}\sfc_i\big)^2\\=\sum_{i=1}^{n}\big(\sfc_i\big)^2\le c_2\eps^4M^{\frac{7}{2}}\cdot n$.\label{item: ci}
        \item  $\EE \sum_{i=1}^{n}\sfd_i^2\le c_4\eps^4M^{\frac{7}{2}}\cdot \iota^{-\frac{7}{4}-\kappa_{1}}$. In addition, we have $\PP\Big(|\sum_{i=1}^{n}\sfd_i| \ge \sqrt{ \iota^{-2\kappa_{1}}\cdot \sum_{i=1}^{n}\sigma_i^2}\Big\}\Big) \\\le c_4\iota^{\kappa_{1}}.$\label{item: di}
        \item For any $1\le i\le n$, we have $\PP\Big(|\sfa_i-\sfb_i-\sfc_i-\sfd_i|\ge c_2(\eps M^{\frac{7}{8}})^{2.7}\Big)\le c_2^{-1}\exp(-c_2\iota^{-\frac{7}{80}})$.\label{item: ai-bi-ci}
    \end{enumerate}
\end{proposition}
Before proving Proposition~\ref{prop: adding external field fluctuation}, we begin with the following key lemma, which provides an upper bound for the conditional $k$-point function $\langle\sigma^I\mid\mathcal{E}_\tau\rangle^+_{\lamn,\eps H^{i-1}}$ (recalling that  $\sigma^I=\prod_{x\in I}\sigma_x$).
\begin{lemma}\label{lem: DHX input for specific rare event}
    Recall that $u_i$ is the center of $\fB_i$ and let $\bpartial\fB_i=\partial\fB_i\cup\partial\Lambda_{M/2}(u_i)\cup \partial\Lambda_{M/4}(u_i)$. There exists a constant $c_5>0$ such that for any $1\le i\le n$, any external field $h\in\mbb R^{\Lambda_N}$ and any set $I\subset \fB_i$, we have  \begin{align}
       |\langle ~\sigma^I\mid\mathcal{E}_\tau~\rangle_{\Lambda_N,\eps \hat{h}^i}^{+}|&\le c_5^{|I|} F(\bpartial\fB_i,I),\label{eq: DHX input for specific rare event}
    \end{align}
    where $\hat h^i_x=h_x\cdot \mbf 1_{x\not\in\fB_i}$ is the external field obtained by removing the field inside $\fB_i.$
\end{lemma}

The proof of Lemma~\ref{lem: DHX input for specific rare event} is postponed to Section~\ref{sec: proof of DHX rare event}. We continue to use the notation $\langle\cdot\rangle_{i-1}$ for the expectation with respect to $\mu_{\Lambda_N,\eps H^{i-1}}^+$. Replacing Lemma~\ref{lem: DHX input for specific rare event in sle} by Lemma~\ref{lem: DHX input for specific rare event}, we obtain the following analogue of Corollary~\ref{cor: DHX input for specific rare event in sle}. The only difference is that we keep an additional parameter $\kappa_4$ here, whereas in Corollary~\ref{cor: DHX input for specific rare event in sle} it was enough to fix $\kappa_4=1/2$.
\begin{corollary}\label{cor: DHX input for specific rare event}
For any $\kappa_4\in (0,1)$, there exists a constant $c_6=c_6(\kappa_4)>0$ such that 
    for any $1\le i\le n$ and any $k\ge 1$, we have \begin{align*}
        &\PP\Big(\Big|\sum_{I\subset \fB_i,|I|=k}(\langle\sigma^I \mid \mathcal{E}_\tau\rangle_{i-1}\prod_{x\in I}\tanh(\eps\beta_c h_x)\Big|\ge (\eps M^{\frac{7}{8}})^{(1-\kappa_4)k}\Big)\\\le~& c_6^{-1}\exp\Big(-c_6(\eps^{-1}M^{-\frac{7}{8}})^{\kappa_4}\times(k!)^{\frac{3}{8k}}\Big).
    \end{align*}Furthermore, if $\iota<0.1$, then for any $r\ge 1$, we have \begin{align*}
        &\PP\Big(\Big|\sum_{k=r}^{\infty}\sum_{I\subset \fB_i,|I|=k}(\langle\sigma^I \mid \mathcal{E}_\tau\rangle_{i-1}\prod_{x\in I}\tanh(\eps\beta_c h_x)\Big|\ge 2(\eps M^{\frac{7}{8}})^{(1-\kappa_4)r}\Big)\\\le~& c_6^{-1}\exp\Big(-c_6(\eps^{-1}M^{-\frac{7}{8}})^{\kappa_4}\times(r!)^{\frac{3}{8r}}\Big).
    \end{align*}
\end{corollary}
\begin{proof}
    The proof follows directly from combining Lemma~\ref{lem: DHX input for specific rare event} and \cite[Lemma B.1]{DHX23}, so we omit it here.
\end{proof}
\begin{proof}[Proof of Proposition~\ref{prop: adding external field fluctuation}]
Let $\mathcal{F}_i$ be the $\sigma$-field generated by the  external fields on $\cup_{j=1}^i \fB_j$.
    Similar to \eqref{eq: partition function expansion 1}, we first perform the chaos expansion with respect to $\mu_{\lamn,\eps H^{i-1}}^+$: \begin{equation}
        \sfa_i=\frac{\sum_{k=1}^{\infty}\sum_{I\subset \fB_i,|I|=k}\big(\langle\sigma^I \1_{\mathcal{E}_\tau}\rangle_{i-1}-\langle\sigma^I \rangle_{i-1}\cdot \langle \1_{\mathcal{E}_\tau}\rangle_{i-1}\big)\cdot\prod_{x\in I}\tanh(\eps\beta_c h_x)}{\mu^+_{\lamn,\eps H^{i-1}}(\mathcal{E}_\tau)\cdot[1+ \sum_{k=1}^{\infty}\sum_{I\subset \fB_i,|I|=k}\langle\sigma^I \rangle_{i-1}\cdot\prod_{x\in I}\tanh(\eps\beta_c h_x)]}.\label{eq: chaos expansion for cle singularity}
    \end{equation}
    We next decompose the right-hand side of \eqref{eq: chaos expansion for cle singularity} into several terms. Writing $H=\EE\tanh^2(\eps\beta_c h_x)$, we define \begin{align}
        \mathsf{b}_i&=\sum_{x\in \fB_i}\big(\langle\sigma_x \mid {\mathcal{E}_\tau}\rangle_{i-1}-\langle\sigma_x \rangle_{i-1}\big)\cdot \eps\beta_c h_x,\label{eq: def of bi in adding external field}\\
        \mathsf{c}_{i,1}&=\sum_{\{x, y\}\subset \fB_i}\big(\langle\sigma_x\sigma_y \mid {\mathcal{E}_\tau}\rangle_{i-1}-\langle\sigma_x\sigma_y  \rangle_{i-1}\big)\cdot \tanh(\eps\beta_c h_x)\tanh(\eps\beta_c h_y),\label{eq: def of ci in adding external field 1}\\
        \mathsf{c}_{i,2}&=\sum_{x, y\in \fB_i}\big(\langle\sigma_x \mid {\mathcal{E}_\tau}\rangle_{i-1}-\langle\sigma_x  \rangle_{i-1}\big)\cdot\langle\sigma_y  \rangle_{i-1}\cdot \big(-\tanh(\eps\beta_c h_x)\tanh(\eps\beta_c h_y)+H\cdot \1_{x=y}\big) ,\label{eq: def of ci in adding external field 2}\\
        \mathsf{d}_i&=-\sum_{x\in \fB_i}\big(\langle\sigma_x \mid {\mathcal{E}_\tau}\rangle_{i-1}-\langle\sigma_x  \rangle_{i-1}\big)\cdot\langle\sigma_x  \rangle_{i-1}\cdot H .\label{eq: def of di in adding external field}
    \end{align}
We will prove \ref{item: bi}, \ref{item: ci}, \ref{item: di} and \ref{item: ai-bi-ci} for $\sfb_i$, $\sfc_i=\sfc_{i,1}+\sfc_{i,2}$ and $\sfd_i$ in the following Lemmas.
\end{proof}
\begin{lemma}\label{lem: property for bi in adding external field}
    Under the assumptions in Proposition~\ref{prop: adding external field fluctuation}, the variable $\sfb_i$ (recalling \eqref{eq: def of bi in adding external field}) satisfies condition \ref{item: bi}. More precisely,
    conditioned on $\mathcal{F}_{i-1}$, it is a centered Gaussian variable with variance $\sigma_i^2$ where $ \sigma_i^2\le c_2\eps^2M^{\frac{7}{4}}$ for some constant $c_2>0$. In addition, if $\tau_i=1$, then we have $\sigma_i^2\ge c_3\eps^2M^{\frac{7}{4}}$ for some constant $c_3>0.$
\end{lemma}
\begin{proof}
    Recall that $u_i$ is the center of the $M$-box $\fB_i$.
Conditioned on $\mathcal{F}_{i-1}$, it is clear that $\sfb_i$ is a mean 0 Gaussian variable, so it suffices to control its variance \begin{equation}
    \EE \sfb_i^2=\eps^2\beta_c^2\sum_{x\in \fB_i}\big(\langle\sigma_x \mid{\mathcal{E}_\tau}\rangle_{i-1}-\langle\sigma_x \rangle_{i-1}\big)^2.\label{eq: second moment of bi chaos expansion}
\end{equation}The upper bound of the right-hand side of \eqref{eq: second moment of bi chaos expansion} follows directly from combining Lemmas~\ref{lem: DHX input for specific rare event} and \ref{lem: DHX input for specific rare event in sle} with Lemma \ref{lem: upper-bound for the sum of squares of k point function}. 
Therefore, it suffices to derive the lower bound for the right-hand side of \eqref{eq: second moment of bi chaos expansion}. From now on, we assume that $\tau_i=1$. By DMP and CBC, we get that \begin{equation}
    |\langle\sigma_x \rangle_{i-1}|\le \langle\sigma_x\rangle_{\Lambda_{M}(u_i),0}^+ .\label{eq: chaos expansion first moment pure bound}
\end{equation} Next we want to lower-bound $|\langle\sigma_x \mid{\mathcal{E}_\tau}\rangle_{i-1}|$. The rough idea is that, when $\tau_i=1$, the existence of a plus-minus interface in $\Lambda_{M/2}(u_i)\setminus\Lambda_{M/4}(u_i)$ leads to a substantial effect on the spin average in $\Lambda_{M/4}(u_i)$. More precisely,  
let $\gamma^i$ be a plus-minus circuit in $\fB_i$ and let $\Gamma^i$ be the region enclosed by $\gamma^i$. Let $\sigma^i$ be a configuration on $\lamn\setminus\Gamma^i$. We say a configuration $\sigma\in\{-1,1\}^{\lamn}$ is an extension of $\sigma^i$ if $\sigma_x=\sigma^i_x,~\forall x\in \lamn\setminus\Gamma^i$. Next, we show that the events $X_j$ are measurable under $\sigma^i$.
Let $\mathcal{C}$ be the boundary plus cluster in $\sigma$ and let  $\mathcal{C}^i$ be the boundary plus cluster in $\sigma^i$. Note that if $\gamma^i$ is an outermost interface, then for any extension $\sigma$ of $\sigma^i$, we have $\mathcal{C}|_{\lamn\setminus\Gamma^i}=\mathcal{C}^i$ and thus $X_j(\sigma)=X_j(\sigma^i)$ for any $1\le j\le n$. 
Thus we define a pair $(\gamma^i,\sigma^i)$ to be valid if $\gamma^i$ is an outermost interface and for any $1\le j\le n$, $X_j(\sigma^i)=\tau_j$.  
Let $\Sigma_{\tau,i}$ denote the collection of valid pairs $(\gamma^i,\sigma^i)$. Note that $H^{i-1}_x=0$ for $x\in \fB_i$,  then we calculate for any $x\in \Lambda_{M/4}(u_i)$ \begin{equation}
    \langle\sigma_x \mid{\mathcal{E}_\tau}\rangle_{i-1}=\sum_{(\gamma^i,\sigma^i)\in \Sigma_{\tau,i}}\frac{\mu^+_{\lamn,\eps H^{i-1}}(\sigma^i)}{\mu^+_{\lamn,\eps H^{i-1}}(\mathcal{E}_\tau)}\cdot \langle\sigma_x \rangle^-_{\Gamma^i,0}.\label{eq: DMP for outermost loop bi}
\end{equation} Recall the definition of $X_i$, we get from $X_i=1$ that $\Gamma^i\subset\Lambda_{M/2}(u_i)$. Thus by CBC, we get that $-\langle\sigma_x \rangle^-_{\Gamma^i,0}\ge \langle\sigma_x \rangle^+_{\Lambda_{M/2}(u_i),0}\ge 0$. Plugging into \eqref{eq: DMP for outermost loop bi}, we get that for any $x\in \Lambda_{M/4}(u_i)$ \begin{equation}
    |\langle\sigma_x \mid{\mathcal{E}_\tau}\rangle_{i-1}|\ge \sum_{(\gamma^i,\sigma^i)\in \Sigma_{\tau,i}}\frac{\mu^+_{\lamn,\eps H^{i-1}}(\sigma^i)}{\mu^+_{\lamn,\eps H^{i-1}}(\mathcal{E}_\tau)}\cdot \langle\sigma_x \rangle^+_{\Lambda_{M/2}(u_i),0}=\langle\sigma_x \rangle^+_{\Lambda_{M/2}(u_i),0}\label{eq: chaos expansion first moment conditional bound}
\end{equation}where we used the equality $\sum_{(\gamma^i,\sigma^i)\in \Sigma_{\tau,i}}\mu^+_{\lamn,\eps H^{i-1}}(\sigma^i)=\mu^+_{\lamn,\eps H^{i-1}}(\mathcal{E}_\tau)$ in the last equation. Combining \eqref{eq: chaos expansion first moment pure bound} with \eqref{eq: chaos expansion first moment conditional bound}, we get that for any $x\in \Lambda_{M/4}(u_i)$, \begin{equation}
    \big|\langle\sigma_x \mid{\mathcal{E}_\tau}\rangle_{i-1}-\langle\sigma_x \rangle_{i-1}\big|\ge\big|\langle\sigma_x \mid{\mathcal{E}_\tau}\rangle_{i-1}\big|-\big|\langle\sigma_x \rangle_{i-1}\big|\ge \langle\sigma_x \rangle^+_{\Lambda_{M/2}(u_i),0}-\langle\sigma_x \rangle^+_{\Lambda_{M}(u_i),0}\label{eq: one arm increment decay 1}
\end{equation} Recalling the Edwards--Sokal coupling, we obtain that \begin{equation}
    \langle\sigma_x \rangle^+_{\Lambda_{M/2}(u_i),0}-\langle\sigma_x \rangle^+_{\Lambda_{M}(u_i),0}=\phi^\w_{\Lambda_{M/2}(u_i),0}\big(x\leftrightarrow\partial\Lambda_{M/2}(u_i)\big)-\phi^\w_{\Lambda_{M}(u_i),0}\big(x\leftrightarrow\partial\Lambda_{M}(u_i)\big).\label{eq: one arm increment decay 2}
\end{equation}For any region $\Lambda_{M}(u_i)\setminus\Lambda_{M/2}(u_i)\subset\Lambda$, let $\phi^{\w,\w}_{\Lambda,0}$ denote the measure with wired boundary condition on $\partial\Lambda_{M}(u_i)$ and $\partial\Lambda_{M/2}(u_i)$. By CBC, we get that \begin{align}
    &\phi^\w_{\Lambda_{M}(u_i),0}\big(x\leftrightarrow\partial\Lambda_{M}(u_i)\big)\le \phi^{\w,\w}_{\Lambda_{M}(u_i),0}\big(x\leftrightarrow\partial\Lambda_{M}(u_i)\big)\nonumber\\ \le~& \phi^\w_{\Lambda_{M/2}(u_i),0}\big(x\leftrightarrow\partial\Lambda_{M/2}(u_i)\big)\cdot \phi^{\w,\w}_{\Lambda_{M}(u_i)\setminus\Lambda_{M/2}(u_i),0}(\partial\Lambda_{M}(u_i)\leftrightarrow\partial\Lambda_{M/2}(u_i))\nonumber\\ \le~& (1-C_1)\phi^\w_{\Lambda_{M/2}(u_i),0}\big(x\leftrightarrow\partial\Lambda_{M/2}(u_i)\big)\label{eq: one arm increment decay 3}
\end{align} for some $C_1\in(0,1)$, where we use the RSW estimate in the last inequality. Combining \eqref{eq: one arm increment decay 1}, \eqref{eq: one arm increment decay 2} and \eqref{eq: one arm increment decay 3} with \eqref{eq: FK one arm event exponent} yields that \begin{equation}
    \big|\langle\sigma_x \mid{\mathcal{E}_\tau}\rangle_{i-1}-\langle\sigma_x \rangle_{i-1}\big|\ge C_1 \phi^\w_{\Lambda_{M/2}(u_i),0}\big(x\leftrightarrow\partial\Lambda_{M/2}(u_i)\big)\stackrel{\eqref{eq: FK one arm event exponent}}{\ge }C_2 M^{-\frac{1}{8}}. \label{eq: one arm increment decay final}
\end{equation}Plugging \eqref{eq: one arm increment decay final} into \eqref{eq: second moment of bi chaos expansion}, we get the desired lower bound for $\EE b_i^2$. 
\end{proof}
\begin{lemma}\label{lem: property for ci in adding external field}
    Under the assumptions in Proposition~\ref{prop: adding external field fluctuation}, the variable $\sfc_i=\sfc_{i,1}+\sfc_{i,2}$ (recalling \eqref{eq: def of ci in adding external field 1} and \eqref{eq: def of ci in adding external field 2}) satisfies condition \ref{item: ci}, i.e., $$\EE \big(\sum_{i=1}^{n}\sfc_i\big)^2= \EE \sum_{i=1}^{n}\big(\sfc_i\big)^2\le c_3\eps^4M^{\frac{7}{2}}\cdot n.$$
\end{lemma}
\begin{proof}
Note that conditioned on $\mathcal{F}_{j}$ for $j<i$, both $\sfc_{i,1}$ and $\sfc_{i,2}$ have mean $0$. Then we have that, for any $j<i,$  
$$\EE(\sfc_i\sfc_j)=\EE(\EE(\sfc_i\sfc_j\mid \mcc F_j))=\EE(\EE(\sfc_i\mid \mcc F_j)\sfc_j)=0.$$ As a consequence,
$\EE \big(\sum_{i=1}^{n}\sfc_i\big)^2= \EE \sum_{i=1}^{n}\big(\sfc_i\big)^2.$
    
    We first control the second moment of $\sum_{i=1}^{n}\sfc_{i,1}$.  Note that $H=\EE\tanh^2(\eps\beta_c h_x)\le C_1\eps^2$ and thus  \begin{align}
\EE\big(\sum_{i=1}^{n}\sfc_{i,1}\big)^2=\sum_{i=1}^{n}\EE \sfc_{i,1}^2\le C_1^{ 2}\eps^4\sum_{i=1}^{n}\sum_{\{x, y\}\subset \fB_i}\Big(\langle\sigma_x\sigma_y\mid \mathcal{E}_\tau\rangle_{i-1}-\langle\sigma_x\sigma_y  \rangle_{i-1}\rangle_{i-1}\Big)^2.\label{eq: c11 upper bound expansion}
\end{align} Applying Lemmas~\ref{lem: DHX input for specific rare event} and \ref{lem: DHX input for specific rare event in sle} with Lemma \ref{lem: upper-bound for the sum of squares of k point function}, we get that \begin{equation}
    \sum_{\{x, y\}\subset \fB_i}\Big(\langle\sigma_x\sigma_y\mid \mathcal{E}_\tau\rangle_{i-1}-\langle\sigma_x\sigma_y  \rangle_{i-1}\Big)^2\le  \sum_{\{x, y\}\subset \fB_i} 4c_5^2 F(\bpartial\fB_i,\{x,y\})^2\le C_2M^{\frac{7}{2}}.\label{eq: c11 upper bound for each box}
\end{equation} Summing \eqref{eq: c11 upper bound for each box} over $1\le i\le  n$ and combining with \eqref{eq: c11 upper bound expansion}, we obtain that \begin{align}
    \EE\big(\sum_{i=1}^{n}\sfc_{i,1}\big)^2\le \sum_{i=1}^{n}C_1^2C_2\eps^4 M^{\frac{7}{2}}=C_1^2C_2\eps^4 M^{\frac{7}{2}}n .\label{eq: c11 upper bound final}
\end{align}
We then turn to the upper bound on $\EE\big(\sum_{i=1}^{n}\sfc_{i,2}\big)^2$. We calculate \begin{align}
\EE\big(\sum_{i=1}^{n}\sfc_{i,2}\big)^2=\sum_{i=1}^{n}\EE (\sfc_{i,2})^2\le C_1^{ 2}\eps^4\sum_{i=1}^{n}\sum_{x\neq y\in \fB_i}\Big(\big(\langle\sigma_x\mid \mathcal{E}_\tau\rangle_{i-1}-\langle\sigma_x  \rangle_{i-1}\big)\cdot\langle\sigma_y  \rangle_{i-1}\Big)^2.\label{eq: c12 upper bound expansion}
\end{align} Applying Lemmas~\ref{lem: DHX input for specific rare event} and \ref{lem: DHX input for specific rare event in sle} with \ref{lem: upper-bound for the sum of squares of k point function}, we get that \begin{align}
    &\sum_{x\neq y\in \fB_i}\Big(\big(\langle\sigma_x\mid \mathcal{E}_\tau\rangle_{i-1}-\langle\sigma_x  \rangle_{i-1}\big)\cdot\langle\sigma_y  \rangle_{i-1}\Big)^2\nonumber\\\le~&  \sum_{x\neq y\in \fB_i} 4c_5^2 F(\bpartial\fB_i,\{x\})^2\cdot F(\partial\fB_i,\{y\})^{ 2}\le C_3M^{\frac{7}{2}}.\label{eq: c12 upper bound for each box}
\end{align} Summing \eqref{eq: c12 upper bound for each box} over $1\le i\le n$ and combining with \eqref{eq: c12 upper bound expansion} yields that \begin{align}
    \EE\big(\sum_{i=1}^{n}\sfc_{i,2}\big)^2\le \sum_{i=1}^{n}C_1^2C_3\eps^4 M^{\frac{7}{2}}=C_1^2C_3\eps^4 M^{\frac{7}{2}}n .\label{eq: c12 upper bound final}
\end{align}Combining \eqref{eq: c11 upper bound final} and \eqref{eq: c12 upper bound final} with the fact that
$\EE(\sfc_i^2)\leq 2\EE(\sfc_{i,1}^2
)+2\EE(\sfc_{i,2}^2
)$, we complete the proof of Lemma~\ref{lem: property for ci in adding external field}.
\end{proof}
\begin{lemma}\label{lem: property for di in adding external field}
    Under the assumptions in Proposition~\ref{prop: adding external field fluctuation}, $d_i$ (recalling \eqref{eq: def of di in adding external field}) satisfies condition \ref{item: di}, i.e., \begin{align}
        \EE \sum_{i=1}^{n}\sfd_i^2&\le c_4\eps^4M^{\frac{7}{2}}\cdot \iota^{-\frac{7}{4}-\kappa_{1}}.\label{eq: property for di in adding external field 1}
    \end{align}In addition, we have \begin{equation}
        \PP\Big(|\sum_{i=1}^{n}\sfd_i|\ge \sqrt{ \iota^{-2\kappa_{1}}\cdot \sum_{i=1}^{n}\sigma_i^2}\Big)\le c_4\iota^{\kappa_{1}}.\label{eq: property for di in adding external field 2}
    \end{equation}
\end{lemma}
\begin{proof}
    Note that  $H=\EE\tanh^2(\eps\beta_c h_x)\le C_0\eps^2.$ Thus we have  $$|\sfd_i|\le C_0\eps^2\cdot\sum_{x\in \fB_i}\Big|(\langle\sigma_x\mid \mathcal{E}_\tau\rangle_{i-1}-\langle\sigma_x  \rangle_{i-1})\cdot\langle\sigma_x  \rangle_{i-1}\Big|.$$ Combined with Lemma~\ref{lem: DHX input for specific rare event} and the Cauchy-Schwarz inequality, it yields that \begin{equation}
         \sfd_i^2\le C_0^2\eps^4|\fB_i|\sum_{x\in \fB_i} \big(\langle\sigma_x  \rangle_{i-1} \big)^2\cdot F(\bpartial\fB_i,\{x\})^2.\label{eq: di upper bound step 1}
     \end{equation}
    To control the right-hand side of \eqref{eq: di upper bound step 1},
    we shall use the following claim, whose proof is deferred until after the proof of the lemma.
    \begin{claim}\label{claim: expectation upper bound with external field in cle singular}
        For any $\kappa_{5}\in(0,\frac{1}{8})$, there exists a constant $C_1=C_1(\kappa_{5})>0$ such that \begin{align}
        \EE\big|\langle\sigma_x  \rangle_{i-1} \big|^2&\le C_1\big(\dist(\fB_i,\partial\lamn)\vee M\big)^{-\frac{1}{4}+\kappa_{5}}M^{-\kappa_{5}}+C_1\dist(x,\partial\lamn)^{-\frac{1}{4}}.\label{eq: expectation upper bound with external field in cle singular}
    \end{align}
    \end{claim}
    Plugging \eqref{eq: expectation upper bound with external field in cle singular} into \eqref{eq: di upper bound step 1}, we get that \begin{align}
    &\EE \sfd_i^2\le C_0^2C_{1}\eps^4M^2\sum_{x\in \fB_i} \Big(\big(\dist(\fB_i,\partial\lamn)\vee M\big)^{-\frac{1}{4}+\kappa_{5}}M^{-\kappa_5}+\dist(x,\partial\lamn)^{-\frac{1}{4}}\Big)\cdot F(\bpartial\fB_i,\{x\})^2\nonumber\\ \le~& C_{2}\eps^4M^{\frac{15}{4}-\kappa_5}\big(\dist(\fB_i,\partial\lamn)\vee M\big)^{-\frac{1}{4}+\kappa_{5}}+C_2\eps^4M^2\sum_{x\in \fB_i}\dist(x,\partial\lamn)^{-\frac{1}{4}}\cdot F(\bpartial\fB_i,\{x\})^2\label{eq: di upper bound step 2}
\end{align} where the last inequality follows from Lemma~\ref{lem: upper-bound for the sum of squares of k point function}.
Note that for any integer $r\geq 0,$ the number of boxes $\fB_i$ such that $\dist(\fB_i,\partial\lamn)=rM$ is upper-bounded by $\frac{4N}{M}$ and that $\big(\dist(\fB_i,\partial\lamn)\vee M\big)^{-\frac{1}{4}+\kappa_5}$ decreases in $\dist(\fB_i,\partial\lamn)$. Thus, summing over $1\le i\le n$, we obtain that \begin{align}
    \sum_{i=1}^{n} \big(\dist(\fB_i,\partial\lamn)\vee M\big)^{-\frac{1}{4}+\kappa_{5}}&\le (M)^{-\frac{1}{4}+\kappa_{5}}\cdot \frac{4N}{M}+\sum_{r=1}^{N/M}(rM)^{-\frac{1}{4}+\kappa_{5}}\cdot \frac{4N}{M}\nonumber \\
    &\le C_3(\frac{N}{M})^{\frac{3}{4}+\kappa_{5}}\cdot M^{-\frac{1}{4}+\kappa_{5}}\cdot \frac{N}{M}.\label{eq: di upper bound final for first moment summation}
\end{align}To treat the second term in \eqref{eq: di upper bound step 2}, we apply the Cauchy-Schwarz inequality and Lemma~\ref{lem: upper-bound for the sum of squares of k point function} to get that \begin{align}
    \sum_{i=1}^n\sum_{x\in \fB_i}\dist(x,\partial\lamn)^{-\frac{1}{4}}\cdot F(\bpartial\fB_i,\{x\})^2&\le \sqrt{\sum_{i=1}^n\sum_{x\in \fB_i}\dist(x,\partial\lamn)^{-\frac{1}{2}}\cdot \sum_{i=1}^n\sum_{x\in \fB_i} F(\bpartial\fB_i,\{x\})^4}\nonumber\\
    &\le \sqrt{C_4N^{\frac{3}{2}}}\cdot \sqrt{C_4M^{\frac{3}{2}}\cdot n}=C_4M^{\frac{3}{2}}\cdot\iota^{-\frac{7}{4}}.\label{eq: di upper bound final for first moment summation 2}
\end{align}

Plugging \eqref{eq: di upper bound final for first moment summation} and \eqref{eq: di upper bound final for first moment summation 2} into \eqref{eq: di upper bound step 2} and noting that $M=\iota N$, we finish the proof of \eqref{eq: property for di in adding external field 1}. 
Next, we prove \eqref{eq: property for di in adding external field 2}. Applying the Cauchy-Schwarz inequality, we obtain that \begin{equation}
    |\sum_{i=1}^n \sfd_i|^2\le H^2\sum_{i=1}^n\sum_{x\in \fB_i}\big(\langle\sigma_x \mid{\mathcal{E}_\tau}\rangle_{i-1}-\langle\sigma_x \rangle_{i-1}\big)^2\cdot \sum_{i=1}^n\sum_{x\in \fB_i}\big(\langle\sigma_x \rangle_{i-1}\big)^2.\nonumber
\end{equation} Combined with the fact that  $\sigma_i^2=\eps^2\beta_c^2\sum_{x\in \fB_i}\big(\langle\sigma_x \mid{\mathcal{E}_\tau}\rangle_{i-1}-\langle\sigma_x \rangle_{i-1}\big)^2$ and $H\le C_{ 0}\eps ^2$, it yields that  \begin{align}
    |\sum_{i=1}^n \sfd_i|^2&\le C_{ 5}^2\eps^2\sum_{i=1}^n\sigma_i^2\cdot \sum_{i=1}^n\sum_{x\in \fB_i}\big(\langle\sigma_x \rangle_{i-1}\big)^2.\label{eq: sum of di upper bound using variance sum}
\end{align}Summing \eqref{eq: expectation upper bound with external field in cle singular} over $x\in \fB_i$ and $1\le i\le n$, we get from \eqref{eq: di upper bound final for first moment summation} that \begin{align}
    \EE\sum_{i=1}^n\sum_{x\in \fB_i}\big(\langle\sigma_x \rangle_{i-1}\big)^2&~\le\sum_{i=1}^n C_6\big(\dist(\fB_i,\partial\lamn)\vee M\big)^{-\frac{1}{4}+\kappa_{5}}M^{2-\kappa_{5}}+C_1\sum_{i=1}^n\sum_{x\in\fB_i}\dist(x,\partial\lamn)^{-\frac{1}{4}}\nonumber\\&\stackrel{\eqref{eq: di upper bound final for first moment summation}}{\le}C_6C_3(\frac{N}{M})^{\frac{3}{4}+\kappa_{5}}\cdot M^{-\frac{1}{4}+\kappa_{5}}\cdot \frac{N}{M}\cdot M^{2-\kappa_5}+C_7N^{\frac{7}{4}}\nonumber\\&\le C_8N^{\frac{7}{4}+\kappa_{5}}M^{-\kappa_{5}}.\nonumber
\end{align} Combined with the Markov inequality and the fact that $M=\iota N$, it yields that \begin{equation}
    \PP\Big(\sum_{i=1}^n\sum_{x\in \fB_i}\big(\langle\sigma_x \rangle_{i-1}\big)^2\ge N^{\frac{7}{4}}\iota^{-2\kappa_{5}})\le C_9\iota^{\kappa_{5}}.\label{eq: spin expectation square summation with external field}
\end{equation}Combining \eqref{eq: sum of di upper bound using variance sum} and \eqref{eq: spin expectation square summation with external field}, we finish the proof of \eqref{eq: property for di in adding external field 2}.
\end{proof}
\begin{proof}[Proof of Claim~\ref{claim: expectation upper bound with external field in cle singular}]
    Let $\hat{N}=\big(\dist(\fB_i,\partial\lamn)\vee M\big)^{1-4\kappa_{5}}M^{4\kappa_{5}}$ and let $\hat{\fB}_i$ be the box with length $\hat{N}$ and whose center is the same as that of $\fB_i$. Let $\fD=(\cup_{j=1}^{i-1} \fB_j)\cap\hat{\fB}_i$ and let $\hat{H}^{i-1}$ be the external field such that for any vertex $z\in\lamn$, $\hat{H}^{i-1}_z=H^{i-1}_z\cdot \1_{z\notin \fD}$. We can apply chaos expansion to $\langle\sigma_x  \rangle_{i-1}$ with respect to the measure $\mu^+_{\lamn,\eps \hat{H}^{i-1}}$.  \begin{align}
        \langle\sigma_x  \rangle_{i-1}=\frac{\langle\sigma_x  \rangle^+_{\lamn,\eps \hat{H}^{i-1}}+\sum_{k=1}^{\infty}\sum_{|I|=k,I\subset\fD}\langle\sigma^I\sigma_x\rangle^+_{\lamn,\eps \hat{H}^{i-1}}\cdot\prod_{z\in I}\tanh(\eps\beta_c h_z)}{1+\sum_{k=1}^{\infty}\sum_{|I|=k,I\subset\fD}\langle\sigma^I\rangle^+_{\lamn,\eps \hat{H}^{i-1}}\cdot\prod_{z\in I}\tanh(\eps\beta_c h_z)}.\label{eq: chaos expansion for expectation in cle singular}
    \end{align}
    Let \begin{align*}
\Phi_k(h)&=\sum\limits_{\substack{|I|=k,\\I\subset\fD}}\langle\sigma^I\sigma_x\rangle^+_{\lamn,\eps \hat{H}^{i-1}}\cdot\prod_{z\in I}\tanh(\eps\beta_c h_z), \mbox{and } \Psi_k(h)=\sum\limits_{\substack{|I|=k,\\I\subset\fD}}\langle\sigma^I\rangle^+_{\lamn,\eps \hat{H}^{i-1}}\cdot\prod_{z\in I}\tanh(\eps\beta_c h_z) .       
    \end{align*}
Then we get from a similar result as in Corollary \ref{cor: DHX input for specific rare event in sle} with $\mathsf D$ replacing $\fB_i$ (see also \cite[(3.15) and (3.16)]{DHX23} for detailed calculations)  that \begin{equation}\label{eq: upper bound on k point function to be big}
        \begin{aligned}
            \PP\Big(|\Phi_k(h)|>(\eps \hat{N}^{\frac{7}{8}})^{k/2}\cdot \langle\sigma_x  \rangle^+_{\lamn,\eps \hat{H}^{i-1}}\Big)&\le C_2^{-1}\exp\Big(-C_2\sqrt{\eps^{-1} \hat{N}^{-\frac{7}{8}}}\times(k!)^{\frac{3}{8k}}\Big),\\ \mbox{ and }\PP\Big(|\Psi_k(h)|>(\eps \hat{N}^{\frac{7}{8}})^{k/2}\Big)&\le C_2^{-1}\exp\Big(-C_2\sqrt{\eps^{-1} \hat{N}^{-\frac{7}{8}}}\times(k!)^{\frac{3}{8k}}\Big).
        \end{aligned}
    \end{equation}Note that $\sum_{k=1}^{\infty}(\eps \hat{N}^{\frac{7}{8}})^{k/2}\le C_3\sqrt{\eps \hat{N}^{\frac{7}{8}}}$ for some constant $C_3>0$. Moreover, we may assume that $C_3<2$ by choosing $\iota>0$ small enough. Thus, we conclude by \eqref{eq: upper bound on k point function to be big} that 
\begin{align}
    &\PP(|\sum_{k=1}^{\infty}\Psi_k(h)| > 2\sqrt{\eps \hat{N}^{\frac{7}{8}}})\le \sum_{k=1}^{\infty}\PP(|\Psi_k(h)|>(\eps \hat{N}^{\frac{7}{8}})^{k/2})\nonumber\\ \stackrel{\eqref{eq: upper bound on k point function to be big}}{\le}& \sum_{k=1}^{\infty}C_2^{-1}\exp\Big(-C_2\sqrt{\eps^{-1} \hat{N}^{-\frac{7}{8}}}\times(k!)^{\frac{3}{8k}}\Big)\le C_4^{-1}\exp(-C_4\sqrt{\eps^{-1} \hat{N}^{-\frac{7}{8}}}).\label{eq: upper bound on Psi_k to be big}
\end{align} Similarly, we can get the upper bound for $\sum_{k=1}^{\infty}\Phi_k(h)$ as follows: 
\begin{align}
    &\PP(|\sum_{k=1}^{\infty}\Phi_k(h)| > 2\sqrt{\eps \hat{N}^{\frac{7}{8}}}\cdot \langle\sigma_x  \rangle^+_{\lamn,\eps \hat{H}^{i-1}})\le \sum_{k=1}^{\infty}\PP(|\Phi_k(h)|>(\eps \hat{N}^{\frac{7}{8}})^{k/2}\cdot \langle\sigma_x  \rangle^+_{\lamn,\eps \hat{H}^{i-1}})\nonumber\\ \stackrel{\eqref{eq: upper bound on k point function to be big}}{\le}& \sum_{k=1}^{\infty}C_2^{-1}\exp\Big(-C_2\sqrt{\eps^{-1} \hat{N}^{-\frac{7}{8}}}\times(k!)^{\frac{3}{8k}}\Big)\le C_4^{-1}\exp(-C_4\sqrt{\eps^{-1} \hat{N}^{-\frac{7}{8}}}).\label{eq: upper bound on Phi_k to be big}
\end{align} 
Let $\cH(x)$ denote the collection of external fields on $\fD$ such that $|\sum_{k=1}^{\infty}\Psi_k(h)| \le 2\sqrt{\eps \hat{N}^{\frac{7}{8}}}$ and $|\sum_{k=1}^{\infty}\Phi_k(h)| \le 2\sqrt{\eps \hat{N}^{\frac{7}{8}}}\cdot \langle\sigma_x  \rangle^+_{\lamn,\eps \hat{H}^{i-1}}$. For any $h\in \cH(x)$, we get from \eqref{eq: chaos expansion for expectation in cle singular} that \begin{equation}
    |\langle\sigma_x  \rangle_{i-1}-\langle\sigma_x  \rangle^+_{\lamn,\eps \hat{H}^{i-1}}|\le C_5\sqrt{\eps \hat{N}^{\frac{7}{8}}} \cdot |\langle\sigma_x  \rangle^+_{\lamn,\eps \hat{H}^{i-1}}|.\label{eq: expectation upper bound in cle singular}
\end{equation}  By \eqref{eq: upper bound on Psi_k to be big} and \eqref{eq: upper bound on Phi_k to be big}, we obtain that \begin{equation}
    \PP\big(\cH(x)\big)\ge 1-2 C_3^{-1}\exp(-C_3\sqrt{\eps^{-1} \hat{N}^{-\frac{7}{8}}}).\label{eq: good external field in  the upper bound of di}
\end{equation} Combined with \eqref{eq: expectation upper bound in cle singular}, it yields that \begin{align}
    \EE|\langle\sigma_x  \rangle_{i-1}|^2\le 2 C_4^{-1}\exp(-C_4\sqrt{\eps^{-1} \hat{N}^{-\frac{7}{8}}})+(1+C_5\sqrt{\eps \hat{N}^{\frac{7}{8}}})\cdot |\langle\sigma_x  \rangle^+_{\lamn,\eps \hat{H}^{i-1}}|^2.\label{eq: expectation upper bound in cle singular final}
\end{align}Recall that $\hat{N}=\big(\dist(\fB_i,\partial\lamn)\vee M\big)^{1-4\kappa_5}M^{4\kappa_5}$. Combining with DMP and \eqref{eq: FK one arm event exponent}, we get that \begin{align}
    |\langle\sigma_x  \rangle^+_{\lamn,\eps \hat{H}^{i-1}}|&\le c~\dist(x,\partial\hat{\fB}_i)^{-\frac{1}{8}}\le c\big(\hat{N}-M+\dist(x,\partial\lamn)\big)^{-\frac{1}{8}}\le C_5\hat{N}^{-\frac{1}{8}}+C_5\dist(x,\partial\lamn)^{-\frac{1}{8}}\nonumber\\&\le C_6\big(\dist(\fB_i,\partial\lamn)\vee M\big)^{-\frac{1-4\kappa_{5}}{8}}M^{-\frac{\kappa_{5}}{2}}+C_6\dist(x,\partial\lamn)^{-\frac{1}{8}}. \nonumber
\end{align} Plugging into \eqref{eq: expectation upper bound in cle singular final} and noting that $\eps \hat{N}^{\frac{7}{8}}\le (\frac{N}{M})^{-\frac{7}{2}\kappa_{5}}$, we finish the proof of \eqref{eq: expectation upper bound with external field in cle singular}.
\end{proof}
\begin{lemma}\label{lem: property for error term in adding external field}
    Under the assumptions in Proposition~\ref{prop: adding external field fluctuation}, the random variables $\sfb_i$, $\sfc_i=\sfc_{i,1}+\sfc_{i,2}$ and $\sfd_i$ (recalling \eqref{eq: def of bi in adding external field} to \eqref{eq: def of di in adding external field}) satisfy condition \ref{item: ai-bi-ci}, i.e., $$\PP\Big(|\sfa_i-\sfb_i-\sfc_i-\sfd_i|\ge c_3(\eps M^{\frac{7}{8}})^{2.7}\Big)\le c_3^{-1}\exp(-c_3\iota^{-\frac{7}{80}}).$$ 
\end{lemma}
\begin{proof}

For notational clarity, let  \begin{align*}
        A^*_{i,r}&=\sum_{k=r}^{\infty}\sum\limits_{I\subset \fB_i,|I|=k}\big(\langle\sigma^I \mid {\mathcal{E}_\tau}\rangle_{i-1}-\langle\sigma^I \rangle_{i-1}\big)\cdot\prod_{x\in I}\tanh(\eps\beta_c h_x);\\ A_{i,r}&=\sum_{k=r}^{\infty}\sum\limits_{I\subset \fB_i,|I|=k}\langle\sigma^I \rangle_{i-1}\cdot\prod_{x\in I}\tanh(\eps\beta_c h_x).
    \end{align*}
Then \eqref{eq: chaos expansion for cle singularity} becomes
$\sfa_i=\frac{A^*_{i,1}}{1+A_{i,1}}.$ Define
\begin{align*}
    \sfe_{i,1}
    &:=
    \sum_{x\in \fB_i}
    \big(
        \langle\sigma_x \mid \mathcal{E}_\tau\rangle_{i-1}
        -\langle\sigma_x\rangle_{i-1}
    \big)
    \left[
        \tanh(\eps\beta_c h_x)
        -\eps\beta_c h_x
    \right],\qquad 
    \sfe_{i,2}
    :=
    \frac{A^*_{i,1}}{1+A_{i,1}}(A_{i,1})^2,\\
    \sfe_{i,3}&:=A^*_{i,3}, \qquad
    \sfe_{i,4} :=
    A^*_{i,2}(-A_{i,1}+A_{i,2}),
    \qquad
    \sfe_{i,5}:=-A^*_{i,1}A_{i,2}.
\end{align*}
By the definitions of \(\sfb_i,\sfc_{i,1},\sfc_{i,2}\) and \(\sfd_i\), we have that
\begin{equation*}
    \sfb_i+\sfe_{i,1}
    = A^*_{i,1}-A^*_{i,2},
    \qquad
    \sfc_{i,1}
    =
    A^*_{i,2}-A^*_{i,3},
   \qquad \sfc_{i,2}+\sfd_i
    =-\big(A^*_{i,1}-A^*_{i,2}\big)
     \big(A_{i,1}-A_{i,2}\big).
\end{equation*}
Using $\sfc_i=\sfc_{i,1}+\sfc_{i,2}$ and
\[
    \frac{A^*_{i,1}}{1+A_{i,1}}
    -A^*_{i,1}
    +A^*_{i,1}A_{i,1}
    =
    \frac{A^*_{i,1}}{1+A_{i,1}}(A_{i,1})^2 ,
\] we obtain that
\begin{equation*}
       \sfa_i-\sfb_i-\sfc_i-\sfd_i
    =\sfe_{i,1}+
    \frac{A^*_{i,1}}{1+A_{i,1}}
    -\big(A^*_{i,1}-A^*_{i,3}\big)
    +\big(A^*_{i,1}-A^*_{i,2}\big)
     \big(A_{i,1}-A_{i,2}\big) =\sum_{k=1}^5 \sfe_{i,k}.
\end{equation*}
We will show below that all the terms on the right-hand side are small with high probability.

    We first control the term $\mathsf e_{i,1}$. Note that for any $x$, $h_x$ is a standard normal variable, thus $\EE\big|\tanh(\eps\beta_c h_x)-\eps\beta_c h_x\big|\le C_1\eps^{3}$ for some constant $C_1\ge 0$. Therefore, we get that \begin{align}
        \EE|\sfe_{i,1}|\le \sum_{x\in \fB_i}2\EE\big|\tanh(\eps\beta_c h_x)-\eps\beta_c h_x\big|\le C_2\eps^{3}M^2.\nonumber
    \end{align}Applying Markov inequality and recalling $\eps =N^{-\frac{7}{8}}, M=\iota N$, we get that \begin{equation}
        \PP\Big(|\sfe_{i,1}|\ge (\eps M^{\frac{7}{8}})^{2.7} \Big)\le \frac{C_2\eps^{3}M^2}{(\eps M^{\frac{7}{8}})^{2.7}}\le  C_2\iota^{-\frac{29}{80}} N^{-\frac{5}{8}}.\label{eq: upper bound for the first error term}
    \end{equation} 

    Then we control the terms $e_{i,k}$ for $2\le k\le 5$.  Applying Corollary~\ref{cor: DHX input for specific rare event} with $\kappa_4=0.1$ to $A^*_{i,1}$ and $A_{i,1}$ yields that
    \begin{align}
        \PP\big(|A^*_{i,k}|\ge C_3 (\eps M^{\frac{7}{8}})^{0.9k}\big)&\le C_4^{-1}\exp\big(-C_4(\eps^{-1} M^{-\frac{7}{8}})^{0.1}\big),\label{eq: good external field for the second error term}\\\PP\big(|A_{i,k}|\ge C_3 (\eps M^{\frac{7}{8}})^{0.9k}\big)&\le C_4^{-1}\exp\big(-C_4(\eps^{-1} M^{-\frac{7}{8}})^{0.1}\big). \label{eq: good external field for the second error term 2}
    \end{align} On the event that $A_{i,k}\le C_1 (\eps M^{\frac{7}{8}})^{0.9k}$ and $A^*_{i,k}(x)\le C_1 (\eps M^{\frac{7}{8}})^{0.9k}$ for $1\le k\le 3$, it is straightforward that $e_{i,k}\le C_3 (\eps M^{\frac{7}{8}})^{2.7}$  for $2\le k\le 5$. Therefore, we get from \eqref{eq: good external field for the second error term} and \eqref{eq: good external field for the second error term 2} that \begin{equation}
        \PP\Big(|\sum_{k=2}^5\sfe_{i,k}|\ge 4C_3 (\eps M^{\frac{7}{8}})^{2.7} \Big)\le 6C_4^{-1}\exp\big(-C_4(\eps^{-1} M^{-\frac{7}{8}})^{0.1}\big).\label{eq: upper bound for the second error term}
    \end{equation}
 The desired result thus follows from combining \eqref{eq: upper bound for the first error term} and \eqref{eq: upper bound for the second error term} and choosing $N>0$ large enough depending on $\iota$.
\end{proof}
As a quick consequence of Lemma~\ref{prop: adding external field fluctuation} and Corollary~\ref{cor: DHX input for specific rare event}, we have the following bound on $\sfa_i$.
\begin{corollary}\label{cor: ai upper bound for adding external field}
    For any $\kappa_{6}\in(0,1)$, there exists a constant $C>0$ such that for any $1\le i\le n$, we have \begin{equation}
        \PP(|\sfa_i|\ge (\eps M^{\frac{7}{8}})^{1-\kappa_{6}})\le C^{-1}\exp\big(-C(\eps^{-1} M^{-\frac{7}{8}})^{\kappa_{6}}\big).\nonumber
    \end{equation}
\end{corollary}
\begin{proof}
    The proof follows directly from combining \eqref{eq: chaos expansion for cle singularity} and Corollary~\ref{cor: DHX input for specific rare event}, and we omit further details here.
\end{proof}

We now turn to the proof of Theorem ~\ref{thm: cle outermost singularity}. Let us first briefly explain how Proposition~\ref{prop: adding external field fluctuation} is used to show that the ratio
$\frac{\mu^+_{\lamn,\eps h}(\mathcal{E}_\tau)}
{\mu^+_{\lamn,0}(\mathcal{E}_\tau)}$
is close to 0. By Proposition~\ref{prop: adding external field fluctuation}, it suffices to analyze the product
$\prod_{i=1}^{n}(1+\sfa_i).$
Since typically \(|\sfa_i|\le C\eps M^{7/8}\), each \(\sfa_i\) is small, and hence we may approximate this product by
$\exp\left(\sum_{i=1}^{n}\sfa_i-\frac12\sum_{i=1}^{n}\sfa_i^2\right).$
Moreover, Proposition~\ref{prop: adding external field fluctuation} allows us to approximate
$\sum_{i=1}^{n}\sfa_i$ and $\sum_{i=1}^{n}\sfa_i^2$
by$\sum_{i=1}^{n}\sfb_i$ and $\sum_{i=1}^{n}\sfb_i^2,$
respectively. Here \(\sum_{i=1}^{n}\sfb_i\) is a centered Gaussian random variable whose variance satisfies $\sum_{i=1}^{n}\EE \sfb_i^2
\ge c_3(\eps M^{7/8})^2|\tau|,$
where \( |\tau|=\sum_{i=1}^n \tau_i \). Thus, for a good configuration $\tau$, this variance is much larger than 1. Consequently, the quadratic correction
$-\frac12\sum_{i=1}^{n}\sfb_i^2$
dominates the linear fluctuation \(\sum_{i=1}^{n}\sfb_i\) with high probability, forcing the above exponential, and hence the original ratio, to be close to 0.

Before diving into the core of the proof of Theorem~\ref{thm: cle outermost singularity}, in view of the above discussion,  let us start by proving the aforementioned bound on $\sum_{i=1}^{n}\sfb_i-\sum_{i=1}^{n}\frac{\sfb_i^2}{2}$.

\begin{lemma}\label{lem: martingale clt}
    Fix a good configuration $\tau\in\{0,1\}^n$. Let $\mathsf b_i$ be defined as in Proposition~\ref{prop: adding external field fluctuation}. There exists a constant $c_7>0$ such that \begin{equation}
        \PP\big(\sum_{i=1}^{n}(\frac{\sfb_i^2}{2}-\frac{\sigma_i^2}{4}-\sfb_i)\le c_7\eps^2M^{\frac{7}{4}}\tn\big)\le \exp(-c_7\eps^2M^{\frac{7}{4}}\tn).
    \end{equation}
\end{lemma}
\begin{proof}
We want to compute the exponential moment of $\sum_{i=1}^{n}(\frac{4\sfb_i+\sigma_i^2-2\sfb_i^2}{16})$ and then apply the Markov inequality. Recall the definition of $\mathcal{F}_i$ in Proposition~\ref{prop: adding external field fluctuation}. We first compute $\EE \big(e^{\frac{4\sfb_i+\sigma_i^2-2\sfb_i^2}{16}}\mid \mathcal{F}_{i-1}\big)$. Note that $\mathsf b_i$ is a Gaussian variable with mean $0$ and variance $\sigma_i^2$. Thus we calculate that \begin{align}
    \EE \big(e^{\frac{4\sfb_i+\sigma_i^2-2\sfb_i^2}{16}}\mid \mathcal{F}_{i-1}\big)&=\frac{1}{\sqrt{2\pi \sigma_i^2}}\int_{-\infty}^{\infty}e^{\frac{4x+\sigma_i^2-2x^2}{16}}\cdot e^{-\frac{x^2}{2\sigma_i^2}}dx=\frac{1}{\sqrt{1+\sigma_i^2/4}}\cdot e^{\frac{\sigma_i^2}{32+8\sigma_i^2}+\frac{\sigma_i^2}{16}}.\nonumber
\end{align} Note that $f(u)=-\log(1+u/4)/2+\frac{u}{32+8u}+\frac{u}{16}$ is decreasing for $u\in[0,1] .$ By Proposition~\ref{prop: adding external field fluctuation}, we get that $\sigma_i^2\le c_2\eps^2M^{\frac{7}{4}}\le 1.$ Therefore we get that $f(\sigma_i^2)\le f(0)=0$ and thus we have that \begin{equation} 
    \EE \big(e^{\frac{4\sfb_i+\sigma_i^2-2\sfb_i^2}{16}}\mid \mathcal{F}_{i-1}\big)\le 1.\label{eq: recursive relation for bi exponential bound 1}
\end{equation} In addition, if $\tau_i=1$, then we have $\sigma_i^2\ge c_3 \eps^2M^{\frac{7}{4}}$. Therefore we get that $f(\sigma_i^2)\le f(c_3 \eps^2M^{\frac{7}{4}})\le -C_1\eps^2M^{\frac{7}{4}}$ and hence we have that \begin{equation}
    \EE \big(e^{\frac{4\sfb_i+\sigma_i^2-2\sfb_i^2}{16}}\mid \mathcal{F}_{i-1}\big)\le \exp(-C_1\eps^2M^{\frac{7}{4}}).\label{eq: recursive relation for bi exponential bound 2}
\end{equation} Combining \eqref{eq: recursive relation for bi exponential bound 1} and \eqref{eq: recursive relation for bi exponential bound 2}, we get that 
    \begin{equation}
        \EE \exp\Big(\sum_{i=1}^{n}(\frac{4\sfb_i+\sigma_i^2-2\sfb_i^2}{16})\Big)\le \exp(-C_1\eps^2M^{\frac{7}{4}}\tn).\nonumber
    \end{equation}Choosing $c_7=\min\{1,C_1\}/4$ and applying the Markov inequality to $\exp\Big(\sum_{i=1}^{n}(\frac{4\sfb_i+\sigma_i^2-2\sfb_i^2}{16})\Big)$, we complete the proof of Lemma~\ref{lem: martingale clt}.
\end{proof}

\begin{definition}\label{def: def of good external field configuration event for cle singularity}
    Recall the definitions of  $\sfa_i(\tau),\sfb_i(\tau),\sfc_i(\tau), \sfd_i(\tau), \mathcal{F}_i(\tau)$ from Proposition~\ref{prop: adding external field fluctuation}. Recall that we have chosen $\kappa=10^{-5}.$ We define the event $\mathcal{A}_i\subset\mathbb{R}^{\lamn}\otimes\{0,1\}^n$ to be the collection of pairs $(h,\tau)$ such that $$|\sfa_i(\tau)-\sfb_i(\tau)-\sfc_i(\tau)-\sfd_i(\tau)|\le (\eps M^{\frac{7}{8}})^{2.7}\mbox{  and  }|\sfa_i(\tau)|\le (\eps M^{\frac{7}{8}})^{0.9}$$ and let $\mathcal{A}=\cap_{i=1}^n\mathcal{A}_i$.  In addition, we define the following events \begin{equation*}
    \begin{aligned}
        \mathcal{B}&=\{(h,\tau):\sum_{i=1}^{n}(\frac{\sfb_i^2}{2}-\sfb_i)\ge c_7\eps^2M^{\frac{7}{4}}\tn +\sum_{i=1}^n\frac{\EE\big(\sfb_i^2\mid \mathcal{F}_{i-1}\big)}{4} \},\\
    \mathcal{A}^*&=\{(h,\tau):|\sum_{i=1}^{n}(\sfa_i^2-\sfb_i^2)|\le 1\},
    \\
    \mathcal{C}&=\{(h,\tau):|\sum_{i=1}^{n}\sfc_i|\le 1\},~~~~
    \mathcal{D}=\Big\{(h,\tau):|\sum_{i=1}^{n}\sfd_i|\le \sqrt{\iota^{-2\kappa}\cdot \sum_{i=1}^{n}\sigma_i^2}\Big\}.
    \end{aligned}
\end{equation*}
\end{definition}
\begin{lemma}\label{lem: good external field and configuration pair for singular}
     There exists a constant $c_{5}>0$ such that for any good configuration $\tau$, we have \begin{equation}
\PP(\mathcal{A}\cap\mathcal{A}^*\cap\mathcal{B}\cap\mathcal{C}\cap\mathcal{D}\mid \tau)\ge 1-C_3\iota^{\kappa}.
\end{equation}
\end{lemma}
\begin{proof}
Fix a good configuration $\tau$. 
Applying the Markov inequality to $\sum_{i=1}^{n}\sfc_i$ and combining with Proposition \ref{prop: adding external field fluctuation} \ref{item: ci}, we get that  \begin{align}
    \PP(\mathcal{C}\mid \tau)&\ge 1-\EE\left[(\sum_{i=1}^{n}\sfc_i)^2\right]\ge 1-c_3\eps^4M^{\frac{7}{2}}\cdot n= 1-c_3\iota^{\frac{3}{2}}.\label{eq: good external field for ci 1}
\end{align}
Combined with Proposition \ref{prop: adding external field fluctuation} \ref{item: di}, \ref{item: ai-bi-ci}, Lemma~\ref{lem: martingale clt} and Corollary \ref{cor: ai upper bound for adding external field}, it yields that  \begin{equation}\label{eq: good external field for cle singularity 1}
\PP(\mathcal{A}\cap\mathcal{B}\cap\mathcal{C}\cap\mathcal{D}\mid \tau)\ge 1-C_2\iota^{\kappa}.
\end{equation} 
Next, we  control $(\mathcal{A}^*)^c\cap\mathcal{A}$. To this end, we  first bound $\EE\Big|\sum_{i=1}^{n}(\sfa_i^2-\sfb_i^2)\1_{\mathcal{A}} \Big|$ and then apply the Markov inequality.
By the Cauchy-Schwarz inequality, we obtain that \begin{align}
    \EE\Big|\sum_{i=1}^{n}(\sfa_i^2-\sfb_i^2)\1_{\mathcal{A}}\Big|&\le \EE\Big|\sum_{i=1}^{n}2\sfb_i(\sfa_i-\sfb_i)\1_{\mathcal{A}}\Big|+\EE\Big[\sum_{i=1}^{n}(\sfa_i-\sfb_i)^2\1_{\mathcal{A}}\Big]\nonumber\\&\le 2\sqrt{\EE\sum_{i=1}^{n}\sfb_i^2\cdot \EE\Big[\sum_{i=1}^{n}(\sfa_i-\sfb_i)^2\1_{\mathcal{A}}\Big]}+\EE\Big[\sum_{i=1}^{n}(\sfa_i-\sfb_i)^2\1_{\mathcal{A}}\Big].\label{eq: expansion for the second moment of ai}
\end{align}
Recalling the definition of $\mathcal{A}$, we get from the Cauchy-Schwarz inequality and Proposition~\ref{prop: adding external field fluctuation} that \begin{align}
\EE\Big[\sum_{i=1}^{n}(\sfa_i-\sfb_i)^2\1_{\mathcal{A}}\Big]&\le 3\EE\sum_{i=1}^{n}\sfc_i^2+3\EE\sum_{i=1}^{n}\sfd_i^2+ 3\sum_{i=1}^{n}\EE\Big[(\sfa_i-\sfb_i-\sfc_i-\sfd_i)^2\1_{\mathcal{A}}\Big]\nonumber\\&\le c_4\eps^4M^{\frac{7}{2}}n+c_4\eps^4M^{\frac{7}{2}}\cdot  \iota^{-\frac{7}{4}-\kappa_{1}}+3(\eps^2M^{\frac{7}{4}})^{2.7}\cdot n\le C_3\iota^{1.5}.\label{eq: second moment bound for ai}
\end{align} 
Plugging \eqref{eq: second moment bound for ai} into \eqref{eq: expansion for the second moment of ai} and combining with Lemma~\ref{lem: property for bi in adding external field} yields that \begin{align}
    \EE\Big|\sum_{i=1}^{n}(\sfa_i^2-\sfb_i^2)\1_{\mathcal{A}}\Big|&\le 2\sqrt{\eps^2M^{\frac{7}{4}}n\cdot C_3\iota^{1.5}}+C_3\iota^{1.5}\le C_4\sqrt{\iota}.\label{eq: second moment bound for ai final}
\end{align}
Applying the Markov inequality to \eqref{eq: second moment bound for ai final} and combining with \eqref{eq: good external field for cle singularity 1}, we get that
\begin{equation}\label{eq: good external field for cle singularity}
\PP(\mathcal{A}\cap\mathcal{A}^*\cap\mathcal{B}\cap\mathcal{C}\cap\mathcal{D}\mid \tau)\ge 1-C_3\iota^{\kappa}.
\end{equation}Thus we complete the proof of Lemma~\ref{lem: good external field and configuration pair for singular}.
\end{proof}

\begin{proof}[Proof of Theorem~\ref{thm: cle outermost singularity}]
    Note that for any external field, we have that \begin{equation}
        \frac{1}{2}\sum_{\tau\in\{0,1\}^n}\Big|\mu^+_{\lamn,0}(\mathcal{E}_\tau)-\mu^+_{\lamn,\eps h}(\mathcal{E}_\tau)\Big|=1-\sum_{\tau\in\{0,1\}^n}\min\{\mu^+_{\lamn,0}(\mathcal{E}_\tau),\mu^+_{\lamn,\eps h}(\mathcal{E}_\tau)\}.\label{eq: total variation expansion}
    \end{equation} So it suffices to prove that \begin{equation}
        \EE\sum_{\tau\in\{0,1\}^n}\min\{\mu^+_{\lamn,0}(\mathcal{E}_\tau),\mu^+_{\lamn,\eps h}(\mathcal{E}_\tau)\}\le C_1\iota^{\kappa}. \label{eq: total variation distance formula}
    \end{equation}
   We first split the set of configurations $\tau$ according to Lemma~\ref{lem: outermost loops number lower bound 1}:\begin{align}
         &\EE\sum_{\tau\in\{0,1\}^n}\min\{\mu^+_{\lamn,0}(\mathcal{E}_\tau),\mu^+_{\lamn,\eps h}(\mathcal{E}_\tau)\}\nonumber\\\le ~& \EE\sum_{\tn\ge c_1n^{\frac{15}{16}-\alpha}}\min\{\mu^+_{\lamn,0}(\mathcal{E}_\tau),\mu^+_{\lamn,\eps h}(\mathcal{E}_\tau)\}+\sum_{\tn< c_1n^{\frac{15}{16}-\alpha}}\mu^+_{\lamn,0}(\mathcal{E}_\tau)\nonumber\\\le~& \EE\sum_{\tn\ge c_1n^{\frac{15}{16}-\alpha}}\min\{\mu^+_{\lamn,0}(\mathcal{E}_\tau),\mu^+_{\lamn,\eps h}(\mathcal{E}_\tau)\}+c_1^{-1}\exp(-c_1n^{\frac{144\alpha-7}{112}}).\label{eq: total variation expansion 1}
    \end{align} 
    Now we fix a good configuration $\tau$, i.e., $\tn\ge c_1n^{\frac{15}{16}-\alpha}$. (Recall that $\alpha=\frac{7}{144}+10^{-5}$.)
Recall the events defined in Lemma~\ref{lem: good external field and configuration pair for singular}.
We expand $\log\Big(\frac{\mu^+_{\lamn,\eps h}(\mathcal{E}_\tau)}{\mu^+_{\lamn,0}(\mathcal{E}_\tau)}\Big)$ as follows: \begin{align}
    &\log\Big(\frac{\mu^+_{\lamn,\eps h}(\mathcal{E}_\tau)}{\mu^+_{\lamn,0}(\mathcal{E}_\tau)}\Big)= \sum_{i=1}^{n}\log(1+\sfa_i)\nonumber\\\le~&  \sum_{i=1}^{n}\sfb_i-\sum_{i=1}^{n} \frac{ \sfb_i^2}{2}+\Big|\sum_{i=1}^{n}\sfc_i\Big|+\Big|\sum_{i=1}^{n}\sfd_i\Big| +\sum_{i=1}^{n}\Big|\sfa_i-\sfb_i-\sfc_i-\sfd_i\Big|+\Big|\sum_{i=1}^{n}\frac{\sfa_i^2-\sfb_i^2}{2}\Big|\nonumber\\&+\Big|\sum_{i=1}^{n}\big(\log(1+\sfa_i)-\sfa_i+\frac{\sfa_i^2}{2}\big)\Big|.\label{eq: log rn derivative expansion singularity}
\end{align}Recall the definitions in Definition \ref{def: def of good external field configuration event for cle singularity}. It suffices to control the term $\Big|\sum_{i=1}^{n}\big(\log(1+\sfa_i)-\sfa_i+\frac{\sfa_i^2}{2}\big)\Big|$ in the right-hand of \eqref{eq: log rn derivative expansion singularity}. Assume that $(h,\tau)\in \mathcal{A}\cap\mathcal{A}^*\cap\mathcal{B}\cap\mathcal{C}\cap\mathcal{D}$.
Then we get that \begin{equation}\label{eq: good external field configuration pair bound}
    \begin{aligned}
        &\sum_{i=1}^{n}(\sfb_i-\frac{\sf\sfb_i^2}{2})\le -c_7\eps^2M^{\frac{7}{4}}\tn-\sum_{i=1}^n\frac{\sigma_i^2}{4} ,~~~~\Big|\sfa_i-\sfb_i-\sfc_i-\sfd_i\Big|\le (\eps M^{\frac{7}{8}})^{2.7},\forall 1\le i\le n,\\
    &|\sum_{i=1}^{n}(\sfa_i^2-\sfb_i^2)|\le 1,~~~~~|\sum_{i=1}^{n}\sfc_i|\le 1,~~~~~|\sum_{i=1}^{n}\sfd_i|\le \sqrt{\iota^{-2\kappa}\cdot \sum_{i=1}^{n}\sigma_i^2}.
    \end{aligned}
\end{equation}
Noting that $|\ln(1+x)-x+\frac{x^2}{2}|\le \frac{x^3}{3}$ and $\sfa_i\le (\eps M^{\frac{7}{8}})^{0.9}$, we obtain that \begin{align}
    \Big|\sum_{i=1}^{n}\big(\log(1+\sfa_i)-\sfa_i+\frac{\sfa_i^2}{2}\big)\Big|\le \sum_{i=1}^{n} \frac{|\sfa_i^3|}{3}\le (\eps M^{\frac{7}{8}})^{2.7} n\le 1.\label{eq: log rn derivative expansion ai third moment bound}
\end{align}
Plugging \eqref{eq: good external field configuration pair bound} and
\eqref{eq: log rn derivative expansion ai third moment bound} into \eqref{eq: log rn derivative expansion singularity} yields that \begin{align}
    \log\Big(\frac{\mu^+_{\lamn,\eps h}(\mathcal{E}_\tau)}{\mu^+_{\lamn,0}(\mathcal{E}_\tau)}\Big)&\le -c_7\eps^2M^{\frac{7}{4}}|\tau|-\sum_{i=1}^n \frac{\sigma_i^2}{4}+3+(\eps M^{\frac{7}{8}})^{2.7}\cdot n+\iota^{-\kappa}\sqrt{\sum_{i=1}^{n}\sigma_i^2}\nonumber\\&\le -C_2\eps^2M^{\frac{7}{4}}n^{\frac{15}{16}-\alpha}-\sum_{i=1}^n \frac{\sigma_i^2}{4}+3+(\eps M^{\frac{7}{8}})^{2.7}\cdot n+\iota^{-\kappa}\sqrt{\sum_{i=1}^{n}\sigma_i^2}.\label{eq: log RN lower bound 1}
\end{align}
Note that $\iota>0$ is chosen small enough and $\sum_{i=1}^n\sigma_i^2\ge \eps^2M^{\frac{7}{4}}\tn$, we get that $\sum_{i=1}^n \frac{\sigma_i^2}{4}\ge \iota^{-\kappa}\sqrt{\sum_{i=1}^{n}\sigma_i^2}$. Plugging into  $n=(\frac{N}{M})^{2}=\iota^{-2}$ and $\eps M^{\frac{7}{8}}=\iota^{\frac{7}{8}}$ into  \eqref{eq: log RN lower bound 1},
we get that \begin{equation}
    \log\Big(\frac{\mu^+_{\lamn,\eps h}(\mathcal{E}_\tau)}{\mu^+_{\lamn,0}(\mathcal{E}_\tau)}\Big)\le -C_3 \iota^{-\frac{1}{8}+2\alpha}.\label{eq: log RN lower bound}
\end{equation}
Now we are ready to upper-bound the left-hand side of \eqref{eq: total variation distance formula}. Summing \eqref{eq: log RN lower bound} over all good configurations, we get that \begin{align}
    & \EE\sum_{\tn\ge c_1n^{\frac{15}{16}-\alpha}}\min\{\mu^+_{\lamn,0}(\mathcal{E}_\tau),\mu^+_{\lamn,\eps h}(\mathcal{E}_\tau)\}\nonumber\\\le~& \PP\otimes\mu^+_{\lamn,0}\Big(\big(\mathcal{A}\cap\mathcal{A}^*\cap\mathcal{B}\cap\mathcal{C}\cap\mathcal{D}\big)^c\Big)+\sum_{(h,\tau)\in \mathcal{A}\cap\mathcal{A}^*\cap\mathcal{B}\cap\mathcal{C}\cap\mathcal{D}}\PP(h)\cdot \mu^+_{\lamn,\eps h}(\mathcal{E}_\tau)\nonumber\\ \stackrel{\eqref{eq: log RN lower bound}}{\le}&\PP\otimes\mu^+_{\lamn,0}\Big(\big(\mathcal{A}\cap\mathcal{A}^*\cap\mathcal{B}\cap\mathcal{C}\cap\mathcal{D}\big)^c\Big)+\sum_{(h,\tau)\in \mathcal{A}\cap\mathcal{A}^*\cap\mathcal{B}\cap\mathcal{C}\cap\mathcal{D}}\PP(h)\cdot e^{-C_3\iota^{-\frac{1}{8}+\alpha}}\times\mu^+_{\lamn,0}(\mathcal{E}_\tau)\nonumber\\ \stackrel{\eqref{eq: good external field for cle singularity}}{\le}&C_4\iota^{\kappa}+e^{-C_4 \iota^{-\frac{1}{8}+2\alpha}}.\label{eq: final calculation in cle singulartiy}
\end{align}
Combining with \eqref{eq: total variation expansion} and \eqref{eq: total variation expansion 1}, we finish the proof of Theorem~\ref{thm: cle outermost singularity} by noting that the right-hand side of \eqref{eq: final calculation in cle singulartiy} goes to $0$ as $\iota$ goes to $0$.
\end{proof}

\subsection{Proof of the central lemma}\label{sec: proof of DHX rare event}

In this section, we prove Lemma \ref{lem: DHX input for specific rare event}. From now on, we fix an $M$-box $\fB_i=\Lambda_M(u_i)$. For notational convenience, we will write $\Lambda_K$ to denote the $K$-box $\Lambda_K(u_i)$ except for the special case $K=N$, where $\lamn$ denotes the $N$-box centered at the origin $o.$

Note that on the left-hand side of \eqref{eq: DHX input for specific rare event}, we have the conditional spin correlation of $I$. Thus, let us now focus on the behavior of the conditional measure $\mu^+_{\lamn,\eps \hat{h}^i}(\cdot\mid\mathcal{E}_\tau)$ restricted to the $M$-box $\fB_i$. Our strategy is to condition on the configuration outside $\fB_i$ and then convert the conditioning on $\mathcal{E}_\tau$ into events related to the induced boundary condition on $\partial\Lambda_M$.
We now give a precise definition of the relevant boundary events.

\begin{definition}\label{def: boundary connectivity event}
Write
$\Sigma_M^c=\{-1,1\}^{\partial\Lambda_M\cup (\lamn\setminus\Lambda_M)}.$

For any $\xi\in \Sigma_M^c$, let
$
\mathcal P=\{x\in \partial \Lambda_M:\xi_x=1\}.
$
Let $\mathcal P_0,\mathcal P_1\subset \mathcal P$ be two disjoint subsets, and let $\mathcal U$ be a collection of mutually disjoint subsets of $\mathcal P$, each of which is also disjoint from $\mathcal P_0\cup \mathcal P_1$.

Define $\mathcal Y^{-}$ to be the event that $\mathcal P_0$ is not plus-connected to $\mathcal P_1$ inside $\Lambda_M$, and define $\mathcal Y^{+}$ to be the event that, for every $U\in \mathcal U$, the set $U$ is plus-connected to $\mathcal P_0$ inside $\Lambda_M$. We then write
$
\mathcal Y=\mathcal Y^{-}\cap \mathcal Y^{+}.
$

Next, let $\mathcal X^{-}$ be the event that there exists a $*$-minus circuit in the annulus
$
\Lambda_{M/2}\setminus \Lambda_{M/4},
$
and let $\mathcal X^{+}$ be the event that there exists a plus circuit in the same annulus which is plus-connected to $\mathcal P_0$. We define
$
\mathcal X=\mathcal X^{+}\cap \mathcal X^{-}.
$

Finally, set
$
\mathcal R=\mathcal Y\setminus \mathcal X,
\mathcal S=\mathcal X\cap \mathcal Y.
$
For simplicity, we use
$
X,X^{+},X^{-},Y,Y^{+},Y^{-},R,S\in\{0,1\}
$
to denote the indicator functions of the events
$
\mathcal X,\mathcal X^{+},\mathcal X^{-},\mathcal Y,\mathcal Y^{+},\mathcal Y^{-},\mathcal R,\mathcal S,
$
respectively.
\end{definition}
We now explain the motivation for the preceding definitions.

\begin{lemma}\label{lem: Etau conditioning reduction}
    Fix $i$, $\tau\in\{0,1\}^n$, and an exterior
extended-Ising configuration
$\xi\in \Sigma_M^c.$
Let
$\mathcal E_\tau(\xi):=\{\eta\in\{-1,1\}^{\Lambda_M\setminus\partial\Lambda_M}:\xi\oplus\eta\in \mathcal E_\tau\}.$
Then there exist deterministic boundary data
$$
    \mathcal P_0(\xi),\; \mathcal P_1(\xi)\subset \partial\Lambda_M,
    \qquad
    \mathcal U(\xi)=\{U_1(\xi),\ldots,U_m(\xi)\},
$$
satisfying the assumptions of Definition \ref{def: boundary connectivity event}, such that, with
$\mathcal Y_\xi,\mathcal X_\xi,\mathcal R_\xi,\mathcal S_\xi$ defined from these data as in Definition
\ref{def: boundary connectivity event}, the following identities hold:
\begin{equation}
    \mathcal E_\tau(\xi)=\mathcal R_\xi
    \quad\text{if }\tau_i=0,\mbox{and }\mathcal E_\tau(\xi)=\mathcal S_\xi
    \quad\text{if }\tau_i=1.
    \label{8.20a}
\end{equation}
Here, the equalities are identities of events in the inside
configuration $\eta$.  If $\mathcal E_\tau(\xi)$ is empty, both sides are
understood to be empty after the corresponding boundary data have been
defined; such exterior configurations do not contribute to conditional
expectations given $\mathcal E_\tau$.
\end{lemma}
\begin{remark}
    For notation convenience we will omit  $\xi$ from the notations $\mathcal Y_\xi,\mathcal X_\xi,\mathcal R_\xi,\mathcal S_\xi$ and $\mathcal P_0(\xi), \mathcal P_1(\xi), \mathcal{U}(\xi)$. In addition, we emphasize that $\mathcal{S}$ may be an empty set since $\mathcal P_0$ may be empty. Nevertheless, we will write the conditional operator $\langle\cdot \mid \mathcal{S}\rangle$ for convenience,  and we use the convention that $\langle\cdot \mid \mathcal{S}\rangle\equiv 0$.
\end{remark}
\begin{proof}[Proof of Lemma~\ref{lem: Etau conditioning reduction}]
    Let $\xi\in \Sigma_M^c$ be a configuration on $(\Lambda_M\setminus\partial\Lambda_M)^c$. After conditioning on $\xi$ together with the event $\mathcal E_\tau$, the relevant objects are as follows:
\begin{itemize}
    \item $\mathcal P_0$ is the set of plus spins that are connected to the boundary through $\Lambda_M^c$.
    \item $\mathcal P_1$ is the set of plus spins that should not be connected to the boundary.
    \item $\mathcal U$ is the collection of boundary traces of clusters that are supposed to be connected to the boundary, but fail to be so through $(\Lambda_M\setminus\partial\Lambda_M)^c$. Consequently, they must be connected to $\mathcal P_0$ inside $\Lambda_M\setminus\partial\Lambda_M$.
    \item $\mathcal X$ denotes the event that there exist both a plus circuit and a $*$-minus circuit in the annulus; the plus circuit is connected to $\mathcal P_0$, and hence may be viewed as the outermost such circuit relevant to our construction.
\end{itemize}

Since we condition on $\mathcal E_\tau$, it follows that for every $U\in \mathcal U$, the set $U$ is plus-connected to $\mathcal P_0$ inside $\Lambda_M\setminus\partial\Lambda_M$, while $\mathcal P_0$ is not plus-connected to $\mathcal P_1$. In other words, the event
$
\mathcal Y=\mathcal Y^{-}\cap \mathcal Y^{+}
$
occurs. Thus, after conditioning on the exterior configuration, the conditioning on $\mathcal E_\tau$ is encoded by the event $\mathcal Y$ inside $\Lambda_M$.  Moreover, by the definition of $\mathcal X$, the event $\mathcal R$ corresponds to the case $\tau_i=0$, whereas the event $\mathcal S$ corresponds to the case $\tau_i=1$.
\end{proof}

The following proposition is the main focus of this section. 
\begin{proposition}\label{prop:stronger-DHX}
     For any boundary condition $\xi\in\Sigma_M^c$ and any $\mathcal{P}_0,\mathcal{P}_1,\mathcal{U}$ defined as in Definition \ref{def: boundary connectivity event} such that $\mcc P_0\neq\emptyset$, there exists an absolute constant $c_7>0$ such that for any $I\subset\Lambda_M$ we have \begin{equation}\label{eq:extension-of-DHX-X=0}
        \Big|\langle\sigma^I\mid \mcc R\rangle^{\xi}_{\Lambda_M,0}\Big|\le c_7^{|I|} F(\bar\partial\Lambda_M,I);
    \end{equation}
\begin{equation}\label{eq:extension-of-DHX-X=1}
        \Big|\langle\sigma^I\mid\mcc S\rangle^{\xi}_{\Lambda_M,0}\Big|\le c_7^{|I|} F(\bar\partial\Lambda_M,I).
    \end{equation}
\end{proposition}

Since we will only focus on the case without an external field, we will omit  the field 0 in the subscript for convenience until the end of this subsection. 
We first explain how we obtain Lemma~\ref{lem: DHX input for specific rare event} from Proposition~\ref{prop:stronger-DHX}.

\begin{proof}[Proof of Lemma~\ref{lem: DHX input for specific rare event}]
    
Recall the discussion in Lemma~\ref{lem: Etau conditioning reduction}. For any boundary condition $\xi\in\Sigma_M^c$, if $\tau_i=0,$ then conditioned on the boundary condition $\xi$, we get that $\mathcal{E}_\tau$ is equivalent to $\mathcal{R}$. Thus we calculate \begin{equation}\label{eq: conditioning expansion in DHX rare event}
    \begin{aligned}
    \langle\sigma^I\mid \mcc R\rangle^{\xi}_{\Lambda_M,0}&=\frac{\langle\sigma^I\cdot  R\mid \xi\rangle^{+}_{\Lambda_N,\eps \hat{h}^i}}{\mu^{+}_{\lamn,\eps \hat{h}^i}(\mathcal{R}\mid \xi)}=\frac{\langle\sigma^I\cdot  \mathbf{1}_{\mathcal{E}_{\tau}}\mid \xi\rangle^{+}_{\Lambda_N,\eps \hat{h}^i}}{\mu^{+}_{\lamn,\eps \hat{h}^i}({\mathcal{E}_{\tau}}\mid \xi)}=\frac{\langle\sigma^I\cdot  \mathbf{1}_{\mathcal{E}_{\tau}}\mid \xi\rangle^{+}_{\Lambda_N,\eps \hat{h}^i}\cdot \mu^+_{\lamn,\eps \hat{h}^i}(\xi)}{\mu^{+}_{\lamn,\eps \hat{h}^i}({\mathcal{E}_{\tau}}\cap \xi)}.
    \end{aligned}
\end{equation} Summing \eqref{eq:extension-of-DHX-X=0} over all possible boundary conditions $\xi$ and combining with \eqref{eq: conditioning expansion in DHX rare event}, we get that \begin{align}
    c^{|I|} F(\bar\partial\Lambda_M,I)&\ge \sum_{\xi\in\Sigma_M^c}\mu^+_{\lamn,\eps \hat{h}^i}(\xi\mid\mathcal{E}_{\tau})\cdot \langle\sigma^I\mid \mcc R\rangle^{\xi}_{\Lambda_M,0}\nonumber\\&=\sum_{\xi\in\Sigma_M^c}\frac{\mu^+_{\lamn,\eps \hat{h}^i}(\xi\cap\mathcal{E}_{\tau})}{\mu^+_{\lamn,\eps \hat{h}^i}(\mathcal{E}_{\tau})}\cdot \frac{\langle\sigma^I\cdot  \mathbf{1}_{\mathcal{E}_{\tau}}\mid \xi\rangle^{+}_{\Lambda_N,\eps \hat{h}^i}\cdot \mu^+_{\lamn}(\xi)}{\mu^{+}_{\lamn,\eps \hat{h}^i}({\mathcal{E}_{\tau}}\cap \xi)}\nonumber\\&=\sum_{\xi\in\Sigma_M^c}\frac{\langle\sigma^I\cdot  \mathbf{1}_{\mathcal{E}_{\tau}}\mid \xi\rangle^{+}_{\Lambda_N,\eps \hat{h}^i}\cdot \mu^+_{\lamn,\eps \hat{h}^i}(\xi)}{\mu^+_{\lamn,\eps \hat{h}^i}(\mathcal{E}_{\tau})}=\langle\sigma^I\mid  \mathcal{E}_{\tau}\rangle^{+}_{\Lambda_N,\eps \hat{h}^i}.\nonumber
\end{align}Thus we finish the proof of Lemma~\ref{lem: DHX input for specific rare event} in the case $\tau_i=0$. The case $\tau_i=1$ follows similarly by replacing $\mathcal{R}$ with $\mathcal{S}$.
\end{proof}

The rest of this section is devoted to proving Proposition~\ref{prop:stronger-DHX}. From now on, we fix a subset $I\subset \Lambda_M$.
The derivations of the right-hand sides in \eqref{eq:extension-of-DHX-X=0} and \eqref{eq:extension-of-DHX-X=1} are quite similar to those in \cite[Section 4.2]{DHX23}. We therefore introduce the following notion of ``local balls'', which is similar to that in \cite[Section 4.2]{DHX23} but more convenient for the analysis in this paper.

\begin{definition}\label{def: distance and ball}
For any $I\subset \Lambda_M$ and any $x\in I$, recalling that
$\bpartial\Lambda_M=
\partial\Lambda_M
\cup \partial\Lambda_{M/2}(u_i)
\cup \partial\Lambda_{M/4}(u_i),$
define
\[
d_x
=
\min\Big\{
\dist(x,\bpartial\Lambda_M),
\min_{y\in I\setminus\{x\}}\dist(x,y)
\Big\}.
\]
 When $I=\{x\}$, we use the convention that $\min_{y\in I\setminus\{x\}}\dist(x,y)=\infty$.
	In addition, we write
	$$
	\mrr{Ball}(I):=\bigcup_{x\in I}\Lambda_{d_x/10}(x).
	$$
    Here and throughout, whenever the subscript of $\Lambda$ is not an integer, it is understood to be rounded down to the nearest integer. In particular, $\Lambda_{d_x/10}$ stands for $\Lambda_{\lfloor d_x/10\rfloor}$, and the same convention applies to quantities such as $\Lambda_{d_x/9}$.  Moreover, we adopt the convention that  $\Lambda_{d_x/10}(x)=\{x\}$ whenever $d_x/10<1$.
\end{definition}

One of the main tools in our analysis is the extended Ising measure; we refer the reader to Section~\ref{sec: extended ising} for its definition. Since the spin marginal of the extended Ising measure coincides with the classical Ising measure, we will frequently pass between the two frameworks, using notation that should be clear from the context.

\begin{definition}

For any extended graph $\bar G=(V,E)$, let
$
\Xi(\bar G)=\{-1,1\}^{V}\otimes \{-1,0,1\}^{E}
$
denote the space of extended Ising configurations on $\bar G$.
     Define
    $$
    \mcc C_x:=\{y\in\overline{\Lambda_M}: x\stackrel{\sigma_x}{\longleftrightarrow} y\}
    $$
    to be the connected component of $x$ with the same spin value as $\sigma_x$ inside $\overline{\Lambda_M}$, (i.e., the corresponding FK-Ising cluster together with all its vertices).
     
    We say that $\mathcal{C}_x$ is \textbf{local} if $\mathcal{C}_x\cap\partial\Lambda_{d_x/10}(x)=\emptyset$.
    Otherwise, we call it \textbf{non-local}.
\end{definition}

\begin{definition}
    For any $\mcc W\in \{\mcc X^+,\mcc X^-,\mcc Y^+,\mcc Y^-\}$ and any $x\in I$, we say that $\mathcal{C}_x$ is \textbf{excellent} to $\mcc W$ if it is local, and flipping the spin on this cluster does not change the value of $\mbb{I}_\mcc W$. In addition, we say that $\mathcal{C}_x$ is \textbf{excellent} if it is excellent to all of $\mcc X^+,\mcc X^-,\mcc Y^+,\mcc Y^-$ (and thus flipping $\mcc C_x$ does not change the values of $R$ and $S$). Moreover, let $\Ne$ denote the event that there exists no excellent cluster.
    We say that a cluster is \textbf{pivotal} to $\mcc X^+$ if it is local but not excellent to $\mcc X^+$, i.e., flipping the spin on this cluster changes the value of $X^+$. Similarly, we define clusters that are pivotal to $\mcc X^-$, $\mcc Y^+$, or $\mcc Y^-$.
\end{definition}

 If there exists an excellent cluster $\mathcal C_x$, 
flipping all spins in $\mathcal C_x$ changes the sign of $\sigma^I$ while leaving the value of $\mathcal R$ unchanged since
$\mathcal C_x\cap I=\{x\}$.
Therefore, each configuration contributing to
$\langle \mathcal R \sigma^I\rangle_{\Lambda_M}^{\xi}$
can be paired with its flipped counterpart, and the two contributions cancel exactly.
As a consequence, configurations containing at least one excellent cluster contribute zero to
$\langle \mathcal R \sigma^I\rangle_{\Lambda_M}^{\xi}$.
This observation is formalized in the next lemma, whose proof is identical to that of \cite[Lemma~4.19]{DHX23}.

\begin{lemma}\label{lem:cancellation-of-excellent-clusters in cle}
For any boundary condition $\xi\in\hat\Sigma_i$, we have
\begin{equation}
\langle R\sigma^I\rangle_{\Lambda_M}^\xi=\langle \mbf 1_{\Ne} R\sigma^I\rangle_{\overline{\Lambda_M}}^\xi,\quad
\langle S\sigma^I\rangle_{\Lambda_M}^\xi=\langle \mbf 1_{\Ne} S\sigma^I\rangle_{\overline{\Lambda_M}}^\xi.\label{eq:cancellation-of-excellent-cluster in cle}
\end{equation}
Here, the subscript $\overline{\Lambda_M}$ indicates that the expectation is taken with respect to the extended Ising measure. Moreover, on the left-hand side, the quantities $R\sigma^I$ and $S\sigma^I$ depend only on the spin configuration, so it is equivalent to work with the classical Ising measure.
\end{lemma}

In order to control the quantities on the right-hand sides of the equalities in \eqref{eq:cancellation-of-excellent-cluster in cle}, we need an upper bound on the event $\Ne$. Note that if $\mathcal{C}$ is not excellent, then either it is non-local or it is pivotal. By \eqref{eq: FK one arm event exponent}, we know that the probability that $\mathcal{C}_x$ is non-local decays as ${(\max\{d_x,1\})}^{-\frac{1}{8}}$. It therefore suffices to control the probability that a cluster is pivotal. Our main intuition is that if flipping a local FK-cluster $\mathcal{C}$ changes the value of $X^+$, then there should exist two disjoint plus clusters connecting $\mathcal{C}$ to the boundary of a larger box containing $\mathcal{C}$. The following BK-type inequality is designed to provide an upper bound on this probability. Let $\mu_{\Lambda_{M_2}\setminus\Lambda_{M_1}}^{\nu,\xi}$ denote the Ising measure on $\Lambda_{M_{2}}\setminus\Lambda_{M_{1}}$ with boundary condition $\nu$ on $\partial\Lambda_{M_1+1}$ and $\xi$ on $\partial\Lambda_{M_2}.$

\begin{lemma}\label{lem: BK for Ising weak four arm}
For any integers $0<M_1<M_2$, let $\cont(M_1,M_2)$ be the event that there exist two plus crossings in the annulus $\Lambda_{M_2}\setminus\Lambda_{M_1}$ such that their plus spin clusters, restricted to $\Lambda_{M_2}\setminus\Lambda_{M_1}$, are disjoint. Then there exists a constant $c_8>0$ such that for any integers $M_2>M_1>0$ and any boundary conditions $\xi,\nu$ on $\partial\Lambda_{M_2}$ and $\partial\Lambda_{M_1+1}$, respectively, we have
\begin{equation}
\mu_{\Lambda_{M_2}\setminus\Lambda_{M_1}}^{\nu,\xi}(\cont(M_1,M_2))\le c_8\Big(\frac{M_2}{M_1}\Big)^{-\frac{1}{4}}.\label{eq: BK for Ising weak four arm}
\end{equation}
\end{lemma}

\begin{proof}
The proof is similar to \cite[Lemma 4.14]{DHX23}, and we follow the notation used there.
    List $\partial \Lambda_{M_1+1}$ as $\{x_1,x_2,\cdots,x_n\}$. 
    For each $x$, if $\sigma_x=-1$,  let $\mathbf C_x = \emptyset$ and $ \mathbf B_x=\{x\}$; if $\sigma_x=1$,  let $\mathbf C_x$ be the plus cluster containing $x$ restricted to $\Lambda_{M_1,M_2}:=\Lambda_{M_2}\setminus\Lambda_{M_1}$. In addition, define
    $$\mathbf B_x=\partial_{\mathrm{ext}}\mathbf C_x\cap\Lambda_{M_1,M_2},$$
    where $\partial_{\mathrm{ext}}\mathbf C_x=\{u\notin \mathbf C_x:\exists\, v\in \mathbf C_x\text{ such that }\{u,v\}\in E\}$ is the exterior boundary of $\mathbf C_x.$ We say a cluster $\mathbf C_{x_i}$ is a plus crossing cluster if it connects $\partial \Lambda_{M_2}$ and $\partial \Lambda_{M_1+1}$.
    
    Define $\tau=\min\{i: \mathbf C_{x_i} \text{ is a plus crossing cluster}\}.$
    If there is no plus crossing cluster, we set $\tau=\infty$. 
    Define
    $$\mathbf B=\bigcup_{i=1}^{{\min\{\tau,n\}}} \mathbf B_{x_i}, \quad  \mathbf C=\bigcup_{i=1}^{{\min\{\tau,n\}}} \mathbf C_{x_i}.$$ Note that although $\partial_{\mathrm{ext}} \mathbf C\cap\Lambda_{M_1,M_2}\subset\mathbf B$, the inclusion may be strict (recall that $\mathbf B_x=\{x\}$ if $\mathbf C_x=\emptyset$). Note also that the points in $\mathbf B$ have spin $-1$. Let $\mathcal{S}(M_1,M_2)$ denote the event that there exists a plus spin cluster connecting $\partial \Lambda_{M_2}$ and $\partial \Lambda_{M_1+1}$ but containing no circuit in $\Lambda_{M_1,M_2}$.  In addition, for any deterministic set $\mathsf B,\mathsf C\subset\Lambda_{M_1,M_2},$ let $\mathcal{D}(M_1,M_2,\Lambda_{M_1,M_2}\setminus( \mathsf B\cup \mathsf C))$ denote the event that there exists a plus spin cluster connecting $\partial \Lambda_{M_2}$ and $\partial \Lambda_{M_1+1}$ in $\Lambda_{M_1,M_2}\setminus( \mathsf B\cup \mathsf C)$.

We claim that, on the event
$\cont(M_1,M_2)\cap\{\mathbf B=\mathsf B,\ \mathbf C=\mathsf C\},$
both $\mathcal S(M_1,M_2)$ and
$\mathcal D\bigl(M_1,M_2,\\\Lambda_{M_1,M_2}\setminus(\mathsf B\cup \mathsf C)\bigr)$
occur.
We first show that $\mathcal S(M_1,M_2)$ occurs. Suppose, to the contrary, that
$\mathcal S(M_1,M_2)$ does not occur. Since $\cont(M_1,M_2)$ occurs, there exists
at least one plus crossing cluster in $\Lambda_{M_1,M_2}$. By the failure of
$\mathcal S(M_1,M_2)$, every plus crossing cluster contains a circuit in
$\Lambda_{M_1,M_2}$. Let $\mathcal C$ be one such cluster. The circuit contained in
$\mathcal C$ separates $\partial\Lambda_{M_1+1}$ from $\partial\Lambda_{M_2}$.
Hence every path connecting $\partial\Lambda_{M_1+1}$ to $\partial\Lambda_{M_2}$
must intersect $\mathcal C$. In particular, every plus crossing intersects
$\mathcal C$, and therefore belongs to the same plus cluster as $\mathcal C$.
Thus, there is only one plus crossing cluster, contradicting the occurrence of
$\cont(M_1,M_2)$. This proves that $\mathcal S(M_1,M_2)$ occurs.

It remains to show that
$\mathcal D\bigl(M_1,M_2,\Lambda_{M_1,M_2}\setminus(\mathsf B\cup \mathsf C)\bigr)$
occurs. Since $\cont(M_1,M_2)$ occurs, we have $\tau\leq n$. By the definition of
$\tau$, the cluster $\mathbf C_{x_\tau}$ is the first plus cluster, in the prescribed
ordering of $\partial\Lambda_{M_1+1}$, which connects $\partial\Lambda_{M_1+1}$ to
$\partial\Lambda_{M_2}$. On the event $\{\mathbf B=\mathsf B,\ \mathbf C=\mathsf C\}$,
the set $\mathsf C$ consists of all plus clusters explored up to and including this
first plus crossing cluster, while $\mathsf B$ contains their exterior boundary in
$\Lambda_{M_1,M_2}$. Since $\cont(M_1,M_2)$ occurs, there is another plus crossing
cluster, distinct from $\mathbf C_{x_\tau}$. This cluster is disjoint from
$\mathsf C$ and cannot intersect $\mathsf B$, since all points of $\mathsf B$ have
spin $-1$. Therefore, this additional plus crossing is contained in
$\Lambda_{M_1,M_2}\setminus(\mathsf B\cup \mathsf C).$
Consequently,
$\mathcal D\bigl(M_1,M_2,\Lambda_{M_1,M_2}\setminus(\mathsf B\cup \mathsf C)\bigr)$
occurs.

In conclusion, we have the following relations \begin{align}
    \cont(M_1,M_2)\subset\bigcup\limits_{\mathsf B,\mathsf C\subset \Lambda_{M_1,M_2}} \{\mathbf B=\mathsf B,  \mathbf C=\mathsf C, \tau\leq n\}&\cap \mathcal{D}({M_1,M_2}, \Lambda_{M_1,M_2}\setminus( \mathsf B\cup \mathsf C)),\nonumber\\
    \text{and }\bigcup\limits_{ \mathsf B,\mathsf C\subset \Lambda_{M_1,M_2}} \{\mathbf B=\mathsf B,  \mathbf C=\mathsf C, \tau\leq n\}&\subset\mathcal{D}(M_1,M_2,\Lambda_{M_1,M_2}).\nonumber
\end{align} 
We point out that 
    in the unions above, some of the events of the form
     $\{\mathbf B=\mathsf B,  \mathbf C=\mathsf C, \tau\leq n\}$ are empty. In light of this, we say a pair of deterministic sets $(\mathsf B,\mathsf C)$ is valid if the event $\{\mathbf B=\mathsf B,  \mathbf C=\mathsf C, \tau\leq n\}$ is not empty, and there does not exist a connected circuit $\gamma$ in $\mathsf C$ that separates the inner and outer boundary of $\Lambda_{M_1,M_2}$. Then we get that\begin{align}
    \cont(M_1,M_2)\subset\bigcup\limits_{\text{valid }\mathsf B,\mathsf C\subset \Lambda_{M_1,M_2}} \{\mathbf B=\mathsf B,  \mathbf C=\mathsf C, \tau\leq n\}&\cap \mathcal{D}({M_1,M_2}, \Lambda_{M_1,M_2}\setminus( \mathsf B\cup \mathsf C)),\label{eq: inclusion relation in BK 1}\\
    \text{and }\bigcup\limits_{\text{valid }\mathsf B,\mathsf C\subset \Lambda_{M_1,M_2}} \{\mathbf B=\mathsf B,  \mathbf C=\mathsf C, \tau\leq n\}&\subset\mathcal{S}(\Lambda_{M_1,M_2}).\label{eq: inclusion relation in BK 2}
\end{align}
    In addition, for each valid pair $(\mathsf B, \mathsf C)$, the event
$\{\mathbf B=\mathsf B,  \mathbf C=\mathsf C, \tau\leq n\}$ is equivalent to $\{\mathsf B\subset \mathsf{Plus}^c, \mathsf C\subset \mathsf{Plus}\}$, where $\mathsf{Plus}$ denotes the set of plus spins.

Then we calculate that \begin{align}
    &\mu^{\nu,\xi}_{\Lambda_{M_1,M_2}}(\cont(M_1,M_2))\stackrel{\eqref{eq: inclusion relation in BK 1}}{\le}\sum_{\text{valid }\mathsf B,\mathsf C\subset \Lambda_{M_1,M_2}}\mu^{\nu,\xi}_{\Lambda_{M_1,M_2}}(\sigma|_{\mathsf B}=-1,\sigma|_{\mathsf C}=1)\nonumber\\
    &\times\mu^{\nu,\xi}_{\Lambda_{M_1,M_2}}\Big(\mathcal{D}(M_1,M_2,\Lambda_{M_1,M_2}\setminus(\mathsf B\cup \mathsf C))~\Big|~ \sigma|_{\mathsf B}=-1,\sigma|_{\mathsf C}=1\Big). \label{decomposition-of-exploration}
    \end{align}
To control the right-hand side of \eqref{decomposition-of-exploration}, we claim that the following inequalities hold for any valid pair $\mathsf{B},\mathsf{C}$,
\begin{align}
    \mu^{\nu,\xi}_{\Lambda_{M_1,M_2}}\Big(\mathcal{S}(M_1,M_2)\Big)&\le C_1(\frac{M_2}{M_1})^{-\frac{1}{8}},\label{eq: FK dual path one arm}\\
    \mu^{\nu,\xi}_{\Lambda_{M_1,M_2}}\Big(\mathcal{D}\big(M_1,M_2, \Lambda_{M_1,M_2}\setminus( \mathsf B\cup \mathsf C)\big)\Big)&~\Big|~ \sigma|_{\mathsf B}=-1,\sigma|_{\mathsf C}=1)\le C_2(\frac{M_2}{M_1})^{-\frac{1}{8}}.\label{eq: FK dual path one arm after exploration}
\end{align}
We first show how we proceed assuming \eqref{eq: FK dual path one arm} and \eqref{eq: FK dual path one arm after exploration}. Plugging \eqref{eq: FK dual path one arm} and \eqref{eq: FK dual path one arm after exploration} into \eqref{decomposition-of-exploration}, we get that \begin{align}
    \mu^{\nu,\xi}_{\Lambda_{M_1,M_2}}(\cont(M_1,M_2))&\stackrel{\eqref{eq: FK dual path one arm after exploration}}{\le}\sum_{\text{valid }\mathsf B,\mathsf C\subset \Lambda_{M_1,M_2}}\mu^{\nu,\xi}_{\Lambda_{M_1,M_2}}(\sigma|_{\mathsf B}=-1,\sigma|_{\mathsf C}=1)\times C_2(\frac{M_2}{M_1})^{-\frac{1}{8}}\nonumber\\
    &\stackrel{(*)}{\le} \mu^{\nu,\xi}_{\Lambda_{M_1,M_2}}\Big(\mathcal{S}(M_1,M_2)\Big)\times C_2(\frac{M_2}{M_1})^{-\frac{1}{8}} \stackrel{\eqref{eq: FK dual path one arm}}{\le}C_1C_2(\frac{M_2}{M_1})^{-\frac{1}{4}}. \label{eq: alternating weak four arm final}
\end{align} In $(*)$, we used \eqref{eq: inclusion relation in BK 2} and the fact that for valid pair $\mathsf B,\mathsf C\subset \Lambda_{M_1,M_2}$, the event $\{\sigma|_{\mathsf B}=-1,\sigma|_{\mathsf C}=1\}$ is equivalent to $\{\mathbf B=\mathsf B,  \mathbf C=\mathsf C, \tau\leq n\}$.
Next, we turn to the proofs of \eqref{eq: FK dual path one arm} and \eqref{eq: FK dual path one arm after exploration}.
\end{proof}
\begin{proof}[Proof of \eqref{eq: FK dual path one arm}]
Recall that $\mathcal{S}(M_1,M_2)$ denotes the event that there exists a plus spin cluster connecting $\partial \Lambda_{M_2}$ and $\partial \Lambda_{M_1+1}$ but containing no circuit in $\Lambda_{M_1,M_2}$.
    Let $C_1>0$ be a sufficiently large constant to be chosen later, and let $M_3=C_1M_1$ and $M_4=\frac{M_2}{C_1}$. We further restrict that $M_4>M_3$ (otherwise $M_2<C_1^2M_1$ and Lemma~\ref{lem: BK for Ising weak four arm} holds for free by taking $c_8>C_1^8.$).
In order to prove \eqref{eq: FK dual path one arm}, we first apply Corollary~\ref{cor: spatial mixing of Ising model} to remove the dependence on the boundary conditions $\xi,\nu.$   By Corollary~\ref{cor: spatial mixing of Ising model}, we obtain that  \begin{align}
    \mu^{\nu,\xi}_{\Lambda_{M_1,M_2}}\Big(\mathcal{S}(M_1,M_2)\Big)&~~~\le \mu^{\nu,\xi}_{\Lambda_{M_1,M_2}}\Big(\mathcal{S}(M_3,M_4)\Big)\nonumber\\&\stackrel{\text{Cor }~\ref{cor: spatial mixing of Ising model}}{\le} \Big(\frac{1}{1-2cC_1^{-\frac{1}{8}}}\Big)\mu^{-,-}_{\Lambda_{M_1,M_2},0}\Big(\mathcal{S}(M_3,M_4)\Big).\label{eq: boundary elimination in BK}
\end{align}
Let $\mathcal{D}^*(M_3,M_4)$ denote the event that there exists an FK dual path in $\Lambda_{M_3,M_4}$ connecting $\partial\Lambda_{M_3+1}$ and $\partial\Lambda_{M_4}$.  If $\mathcal{D}^*(M_3,M_4)$ does not occur, then there exists an FK cluster $\mathcal{C}$ containing a circuit in $\Lambda_{M_3,M_4}$ and thus any spin crossing in $\Lambda_{M_3,M_4}$ will intersect $\mathcal{C}$. Thus, every plus crossing cluster in $\Lambda_{M_3,M_4}$ contains a circuit in $\Lambda_{M_3,M_4}$. 
Therefore, we get that \begin{equation}
    \mathcal{S}(M_3,M_4)\subset\mathcal{D}^*(M_3,M_4).\label{eq: Ising crossing change to FK crossing}
\end{equation}
Combined with duality, we get that \begin{align}
    \mu^{-,-}_{\Lambda_{M_1,M_2}}\Big(\mathcal{S}(M_3,M_4)\Big)\le \phi^{\w,\w}_{\Lambda_{M_1,M_2}}\Big(\mathcal{D}^{*}(M_3,M_4)\Big)&=\phi^{\f,\f}_{\Lambda_{M_1,M_2}}\Big(\partial\Lambda_{M_3+1}\leftrightarrow\partial\Lambda_{M_4}\Big)
    .\label{eq: FK dual one arm with free boundary}
\end{align} 
Let us now upper-bound the right-hand side of \eqref{eq: FK dual one arm with free boundary}. The FKG and the CBC properties yield that \begin{align}
    \phi^{\f}_{\Lambda_{M_2}}(o\leftrightarrow \partial\Lambda_{M_4})&~\ge \phi^{\f}_{\Lambda_{M_2}}\big(o\leftrightarrow \partial\Lambda_{M_3},~\partial\Lambda_{M_3+1}\leftrightarrow\partial\Lambda_{M_4}\big)\nonumber\\&\stackrel{\mathrm{FKG}}{\ge} \phi^{\f}_{\Lambda_{M_2}}\big(o\leftrightarrow \partial\Lambda_{M_3}\big)\cdot \phi^{\f}_{\Lambda_{M_2}}\big(\partial\Lambda_{M_3+1}\leftrightarrow\partial\Lambda_{M_4}\big)\nonumber\\&\stackrel{\mathrm{CBC}}{\ge }\phi^{\f}_{\Lambda_{M_3}}\big(o\leftrightarrow \partial\Lambda_{M_3}\big)\cdot \phi^{\f,\f}_{\Lambda_{M_1,M_2}}\big(\partial\Lambda_{M_3+1}\leftrightarrow\partial\Lambda_{M_4}\big).\label{eq: FK annulus crossing probability step one}
\end{align} Plugging the one-arm event estimate \eqref{eq: FK one arm event exponent} into \eqref{eq: FK annulus crossing probability step one}, we get that \begin{equation}
\phi^{\f,\f}_{\Lambda_{M_1,M_2}}\big(\partial\Lambda_{M_3+1}\leftrightarrow\partial\Lambda_{M_4}\big)\le \frac{C_2M_4^{-\frac{1}{8}}}{C_3M_3^{-\frac{1}{8}}}=\frac{C_2}{C_3}(\frac{M_4}{M_3})^{-\frac{1}{8}}.\label{eq: FK dual one arm with free boundary 1} 
\end{equation}
Combining \eqref{eq: boundary elimination in BK} with \eqref{eq: FK dual one arm with free boundary} and \eqref{eq: FK dual one arm with free boundary 1} and then choosing $C_1>0$ large enough such that $2cC_1^{-\frac{1}{8}}\le \frac12$, we finish the proof of \eqref{eq: FK dual path one arm}. 
\end{proof}
\begin{proof}[Proof of \eqref{eq: FK dual path one arm after exploration}]
Recall that $\mathcal{D}(M_1,M_2,\Lambda_{M_1,M_2}\setminus( \mathsf B\cup \mathsf C))$ denotes the event that there exists a plus spin cluster connecting $\partial \Lambda_{M_2}$ and $\partial \Lambda_{M_1+1}$ in $\Lambda_{M_1,M_2}\setminus( \mathsf B\cup \mathsf C)$.
    To prove \eqref{eq: FK dual path one arm after exploration}, we first fix a valid pair $\mathsf B,\mathsf C\subset \Lambda_{M_1,M_2}$. Similar to the proof of \eqref{eq: FK dual path one arm}, we choose $M_3=C_1M_1$ and $M_4=\frac{M_2}{C_1}$ for some $C_1>0$ to be chosen later and further restrict that $M_4>M_3$.
     Since $\mathcal{D}\big(M_3,M_4,\Lambda_{M_3,M_4} \setminus(\mathsf{B}\cup\mathsf{C})\big)$ is an increasing event, the DMP and CBC properties give that  \begin{align}
&\mu^{\nu,\xi}_{\Lambda_{M_1,M_2}}\Big(\mathcal{D}\big(M_1,M_2, \Lambda_{M_1,M_2}\setminus( \mathsf B\cup \mathsf C)\big)~\Big|~ \sigma|_{\mathsf B}=-1, \sigma|_{\mathsf C}=1\Big)\nonumber\\\stackrel{\text{DMP}}{=}&\mu^{\nu,\xi}_{\Lambda_{M_1,M_2}}\Big(\mathcal{D}\big(M_1,M_2, \Lambda_{M_1,M_2}\setminus( \mathsf B\cup \mathsf C)\big)~\Big|~ \sigma|_{\mathsf B}=-1\Big)\nonumber\\\stackrel{\text{CBC}}{\le}&\mu^{\nu,\xi}_{\Lambda_{M_1,M_2}}\Big(\mathcal{D}\big(M_3,M_4,\Lambda_{M_3,M_4}\setminus(\mathsf{B}\cup\mathsf{C})\big)~\Big|~\sigma|_{\mathsf B\cap \Lambda_{M_3,M_4}}=-1\Big)\label{eq: FK dual one arm CBC to a smaller region}.
\end{align}
Note that both $\mathcal{D}\big(M_3,M_4,\Lambda_{M_3,M_4}\setminus(\mathsf{B}\cup\mathsf{C})\big)$ and $\sigma_{\mathsf B\cap \Lambda_{M_3,M_4}}=-1$ are supported on $\Lambda_{M_3,M_4}$.  Corollary \ref{cor: spatial mixing of Ising model} implies that \begin{align}
&\mu^{\nu,\xi}_{\Lambda_{M_1,M_2}}\Big(\mathcal{D}\big(M_3,M_4,\Lambda_{M_3,M_4}\setminus(\mathsf{B}\cup\mathsf{C})\big)~\Big|~\sigma|_{\mathsf B\cap \Lambda_{M_3,M_4}}=-1\Big)\nonumber\\\le~&\frac{1}{1-2cC_1^{-\frac{1}{8}}} \mu^{-,-}_{\Lambda_{M_{1},M_{2}}}\Big(\mathcal{D}\big(M_3,M_4,\Lambda_{M_3,M_4}\setminus(\mathsf{B}\cup\mathsf{C})\big)~\Big|~\sigma|_{\mathsf B\cap \Lambda_{M_3,M_4}}=-1\Big).\label{eq: FK dual one arm CBC to a smaller region 2}
\end{align}As in the proof of \eqref{eq: FK dual path one arm}, we define $\mathcal{D}^*\big(M_3,M_4,\Lambda_{M_3,M_4} \setminus(\mathsf{B}\cup\mathsf{C})\big)$ to be the event that there exists a FK dual path connecting $\partial\Lambda_{M_3+1}$ and $\partial\Lambda_{M_4}$ in $\Lambda_{M_3,M_4}\setminus(\mathsf{B}\cup\mathsf{C}).$ By the same argument as  in \eqref{eq: Ising crossing change to FK crossing}, we get that under minus boundary condition on $\mathsf B\cap \Lambda_{M_3,M_4}$,
\begin{equation}
    \mathcal{D}\big(M_3,M_4,\Lambda_{M_3,M_4} \setminus(\mathsf{B}\cup\mathsf{C})\big)\subset\mathcal{D}^*\big(M_3,M_4,\Lambda_{M_3,M_4} \setminus(\mathsf{B}\cup\mathsf{C})\big).\nonumber
\end{equation} Thus, we have that  \begin{align}
&\mu^{-,-}_{\Lambda_{M_{1},M_{2}}}\Big(\mathcal{D}\big(M_3,M_4,\Lambda_{M_3,M_4}\setminus(\mathsf{B}\cup\mathsf{C})\big)~\Big|~\sigma|_{\mathsf B\cap \Lambda_{M_3,M_4}}=-1\Big)\nonumber\\\le~& \phi^{\w,\w}_{\Lambda_{M_{1},M_{2}}\setminus(\mathsf{B}\cup\mathsf{C})}\Big(\mathcal{D}^*\big(M_3,M_4,\Lambda_{M_3,M_4}\setminus(\mathsf{B}\cup\mathsf{C})\big)\Big)\label{eq: FK dual one arm CBC to a smaller region 3}.
\end{align} By duality and  the CBC property of the FK measure, we get that \begin{align}
&\phi^{\w,\w}_{\Lambda_{M_1,M_2}\setminus(\mathsf{B}\cup\mathsf{C})}\Big(\mathcal{D}^*\big(M_3,M_4,\Lambda_{M_3,M_4}\setminus(\mathsf{B}\cup\mathsf{C})\big)\Big)\nonumber=   \phi^{\f,\f}_{\Lambda_{M_1,M_2}\setminus(\mathsf{B}\cup\mathsf{C})}\Big(\partial\Lambda_{M_4}\leftrightarrow\partial\Lambda_{M_3+1}\Big)\\\stackrel{\text{CBC}}{\le }&\phi^{\f,\f}_{\Lambda_{M_{1},M_{2}}}\Big(\partial\Lambda_{M_4}\leftrightarrow\partial\Lambda_{M_3+1}\big)\Big)\stackrel{\eqref{eq: FK dual one arm with free boundary 1}}{\le} C_4(\frac{M_4}{M_3})^{-\frac{1}{8}}.\label{eq: FK dual one arm CBC to a smaller region 4}
\end{align} Combining \eqref{eq: FK dual one arm CBC to a smaller region}, \eqref{eq: FK dual one arm CBC to a smaller region 2} and \eqref{eq: FK dual one arm CBC to a smaller region 3} with \eqref{eq: FK dual one arm CBC to a smaller region 4}, we get \eqref{eq: FK dual path one arm after exploration}.
\end{proof}

The following lemma controls the probability of $\mcc C_x$ being pivotal.

\begin{lemma}\label{lem: pivotal event bound for cle outermost}
    There exists a constant $c_8>0$ such that for any $\mcc W\in\{\mcc X^+,\mcc X^-,\mcc Y^+,\mcc Y^-\}$, any $x\in I$, and any boundary condition $\xi$ on $\overline{\Lambda_{d(x)/10}(x)^c}$, 
    we have \begin{equation}
        \bar\mu^{\xi}_{\Lambda_{d(x)/10}(x)}(\mathcal{C}_x \textit{ is pivotal to }\mcc W)\leq c_8{(\max\{d(x),1\})}^{-\frac{1}{8}}.\label{eq: pivotal for cle}
    \end{equation}
\end{lemma}
\begin{proof} We first deal with the case that $\mcc W=\mcc X^+.$ As explained before Lemma \ref{lem: BK for Ising weak four arm}, we want to construct a $\cont$ event to control the probability on the left-hand side of \eqref{eq: pivotal for cle}.
Since $\mathcal{C}_x$ is local, we can assume that  $\mathcal{C}_x\cap\partial\Lambda_{2^t}(x)=\emptyset$ and  $\mathcal{C}_x\cap\partial\Lambda_{2^{t-1}}(x)\neq\emptyset$, 
    for some integer $t\leq \lceil\log_2(d(x)/10)\rceil$. 
Let $\sigma\in\{-1,1\}^{\Lambda_M}$ denote the configuration associated with $s_x=1$ and $\hat\sigma\in\{-1,1\}^{\Lambda_M}$ denote the configuration associated with $s_x=-1$, where $s_x$ is the sign of the cluster $\mcc C_x$. Therefore we have 
    $$\sigma|_{ \Lambda_M\setminus\Lambda_{d(x)/10}(x)}=\hat\sigma|_{ \Lambda_M\setminus\Lambda_{d(x)/10}(x)}=\xi, \quad  \sigma|_{ \Lambda_M\setminus\Lambda_{2^t}(x)}=\hat\sigma|_{ \Lambda_M\setminus\Lambda_{2^t}(x)}.$$
Let $\mathsf{Plus}(\sigma)$ denote the set of plus spins in $\sigma$. Then $\mathsf{Plus}(\hat\sigma)=\mathsf{Plus}(\sigma)\setminus\mcc C_x.$ Recall $\mcc P_0$ from Definition~\ref{def: boundary connectivity event}. Let $\mathsf{CP}$ denote the connected component of $\mcc P_0$ in $\mathsf{Plus}(\sigma).$ Then the connected component of $\mcc P_0$ in $\mathsf{Plus}(\hat\sigma)=\mathsf{Plus}(\sigma)\setminus\mcc C_x$ is a subset of $\mathsf{CP}\setminus \mcc C_x.$

 Recall that  $\mcc X^+$ is the event that there exists a plus circuit in $\Lambda_{M/2}\setminus\Lambda_{M/4}$ that connects to $\mathcal{P}_0$. 
    Since $\mcc X^+$ is an increasing event, flipping the sign of $\mathcal{C}_x$ can affect $X^+$ only if $\mcc X^+$ holds when $s_x=1$ but fails when $s_x=-1$. Therefore, $\mathsf{Plus}(\sigma)$ contains a plus circuit in $\Lambda_{M/2}\setminus\Lambda_{M/4}$  connected to $\mcc P_0,$ hence  $\mathsf{CP}$ contains a plus circuit in $\Lambda_{M/2}\setminus\Lambda_{M/4}$. But we also require that 
    $\mathsf{Plus}(\hat\sigma)$ does not contain such circuit,
    so there are two possible cases:

    \begin{itemize}
    \item At least one point $v\in\mathsf{CP}\setminus\mcc C_x$ is not connected to $\mcc P_0$. Assume $\pi=[v=v_0,v_1,\cdots,v_L]$ is a path in $\mathsf{CP}$ from $v_0$ to $\mcc P_0$. (If multiple such paths exist, select the first one according to a pre-fixed order.)  Therefore we know $\pi\not\subset \mathsf{CP}\setminus \mcc C_x.$ Thus $\pi\cap\Lambda_{2^t}(x)\neq \emptyset.$
    This motivates us to define the following entrance time and exit time through $\Lambda_{2^t}(x)$:
     $$l_2=\min\{1\leq l\leq L: v_l\in\partial\Lambda_{2^t}(x)\},\mbox{ and }  l_1=\max\{1\leq l<l_2: v_l\in\partial\Lambda_{d(x)/10}(x)\};$$
     $$l_3=\max\{1\leq l\leq L: v_l\in\partial\Lambda_{2^t}(x)\},\mbox{ and }  l_4=\min\{l_3\leq l\leq L: v_l\in\partial\Lambda_{d(x)/10}(x)\}.$$
     Here we recall that $2^t\leq d(x)/10$. From the above definitions, we see that $1\leq l_1 < l_2 < l_3 < l_4 \leq L$. We further define $$\pi_1=\{v_{l_1},v_{l_1+1},\cdots,v_{l_2}\}, \quad \pi_2=\{v_{l_3}, v_{l_3+1},\cdots,v_{l_4}\},$$
    which correspond to the `first' and `last' plus crossings of $\pi$ in the annulus $\Lambda_{d(x)/10}(x)\setminus\Lambda_{2^t}{(x)}.$
    We claim that $\pi_1$ and $\pi_2$ are not connected through plus spins in $\mathsf{CP}\cap\big(\Lambda_{d(x)/10}(x)\setminus\Lambda_{2^t}{(x)}\big) .$ 
    Otherwise, this yields a path from $v$ to $\mcc P_0$ that avoids $\Lambda_{2^t}(x)$. Hence, $v$ is connected to $\mcc P_0$ in $\mathsf{CP}\setminus \mcc C_x$, contradicting the definition of $v$.
 See Figure~\ref{fig: two disjoint plus crossing} for an illustration.

        \item Every point in $\mathsf{CP}\setminus\mcc C_x$ is connected to $\mcc P_0,$ then $\mcc X^+$ fails only if $\mathsf{CP}\setminus\mcc C_x$ contains no plus circuit in $\Lambda_{M/2}\setminus\Lambda_{M/4}$.
        Let $\gamma$ denote the plus circuit in $\mathsf{CP}\cap (\Lambda_{M/2}\setminus\Lambda_{M/4}).$ (If multiple such circuits exist, select the first one according to a pre-fixed order.) Assume 
        $\gamma=[v_0,v_1,v_2,\cdots,v_m=v_0]$ such that $v_{i-1}\sim v_i$ for $i=1,2,\cdots,m.$  Since $v_0$ could be chosen arbitrarily in the circuit, we could assume that $v_0\not\in\Lambda_{d(x)/10}(x).$ The rest is similar to the first case.
    \end{itemize}
\begin{figure}[htb]
    \centering
     \begin{minipage}[t]{0.49\linewidth}
        \begin{tikzpicture}[scale=0.75]
    \draw (-3.5,-3.5)rectangle (3.5,3.5);
    \draw (-2,-2) rectangle (2,2);
    \draw[dotted] (0,0) .. controls (-1,1) and (-0.5,-1.5) .. (-1,-2);
    \draw[dotted] (0,0) .. controls (0.3,-1) and (0.7,-1.5) .. (1.1,1);
    \draw[dotted] (1.1,1) .. controls (1.5,0) and (1.8,1.5) .. (1.9,1.1);
    \draw[dotted] (1.1,1) .. controls (1.3,1.3) and (0.8,1.7) .. (1.1,1.9);
    \draw[ ] (-3.5,-3) .. controls (-3,-0.9) and (-2.5,-1.3) .. (-2,-1.8);
    \draw[ ] (-2,-1.8) .. controls (-1,-2) and (1,1.8) .. (2,1.5);
    \draw[ ] (2,1.5) .. controls (2.5,1.4) .. (3.5,2.5);
    \draw[ ] (3,2.4) .. controls (2.8,1) .. (0,-3);
    \draw[ ] (-3,-3.3) .. controls (-2.8,-1) .. (-2.5,0);
    \node [name=start1] at (3.5,3.5){ };
    \node [name=end1] at (3.5,4.5){$\Lambda_{d(x)}(x)$ };
    \draw [->] (start1)--(end1) ;
    \node [name=start2] at (2,2){ };
    \node [name=end2] at (2.8,2.8){ $\Lambda_t(x)$};
    \node at (0.2,0.2) {$\cC_{x}^+$};
    \draw [->] (start2)--(end2) ;
    \end{tikzpicture}
    \end{minipage}\hfill
    \begin{minipage}[t]{0.49\linewidth}
        \begin{tikzpicture}[scale=0.75]    
    \draw (-3.5,-3.5)rectangle (3.5,3.5);
    \draw (-2,-2) rectangle (2,2);
    \draw[dotted] (0,0) .. controls (-1,1) and (-0.5,-1.5) .. (-1,-2);
    \draw[dotted] (0,0) .. controls (0.3,-1) and (0.7,-1.5) .. (1.1,1);
    \draw[dotted] (1.1,1) .. controls (1.5,0) and (1.8,1.5) .. (1.9,1.1);
    \draw[dotted] (1.1,1) .. controls (1.3,1.3) and (0.8,1.7) .. (1.1,1.9);
    \draw[color= blue,very thick] (-3.5,-3) .. controls (-3,-0.9) and (-2.5,-1.3) .. (-2,-1.8);
    \draw[fill=black] (-3.5,-3) circle(0.2em);
    \node[above left] at (-3.5,-3) {$x_{\ell_4}$};
    \draw[fill=black] (-2,-1.8) circle(0.2em);
    \node[above left] at (-2,-1.8) {$x_{\ell_3}$};
   \draw[ ] (-2,-1.8) .. controls (-1,-2) and (1,1.8) .. (2,1.5);
    \draw[color=red, very thick] (2,1.5) .. controls (2.5,1.4) .. (3.5,2.5);
    \draw[fill=black] (2,1.5) circle(0.2em);
    \node[above right] at (2,1.5) {$x_{\ell_2}$};
    \draw[fill=black] (3.5,2.5) circle(0.2em);
    \node[above right] at (3.5,2.5) {$x_{\ell_1}$};
    
    \node [color=red] at (3.3,1.8) {$\mcc P_1$};
    \node [color=blue] at (-3.2,-3.2) {$\mcc P_2$};
    \draw (3,2.4) .. controls (2.8,1) .. (0,-3);
     \draw[ ] (-3,-3.3) .. controls (-2.8,-1) .. (-2.5,0);
    \node [name=start1] at (3.5,3.5){ };
    \node [name=end1] at (3.5,4.5){$\Lambda_{d(x)}(x)$ };
    \draw [->] (start1)--(end1) ;
    \node [name=start2] at (2,2){ };
    \node [name=end2] at (2.8,2.8){ $\Lambda_t(x)$};
    \node at (0.2,0.2) {$\cC_{x}^+$};
    \draw [->] (start2)--(end2) ;
    \end{tikzpicture}
    \end{minipage}
    \caption{Illustration of two disjoint plus crossings. Dotted lines: the FK-cluster $\cC_{x}^+$. Solid lines: the pluses in $\Lambda_{d(x)}(x)$. Colored lines: 
    the first and last plus crossings of $\mcc P$ in $\Lambda_{t,d(x)}(x)$.}
    \label{fig: two disjoint plus crossing}
\end{figure}
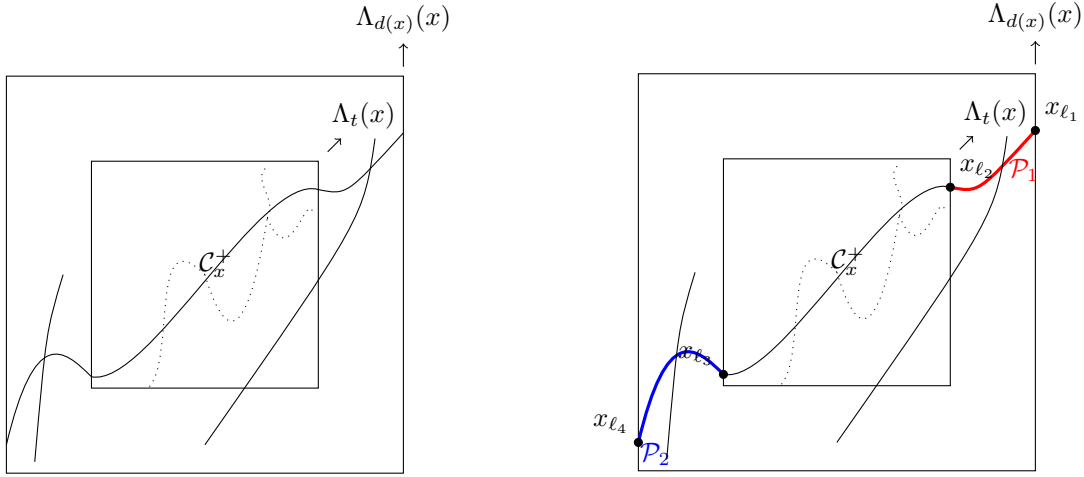


    
Let  $\mathtt L_t$ be the event that $\{\mathcal{C}_x\cap\partial\Lambda_{2^t}(x)=\emptyset\}\cap \{ \mathcal{C}_x\cap\partial\Lambda_{2^{t-1}}(x)\neq\emptyset\}$. We conclude from the above discussion that 
    \begin{equation}\{\mathcal{C}_x\text{ is pivotal to } \mcc X^+\}\subset\bigcup_{t=1}^{\lceil\log_2(d(x)/10)\rceil}\mathtt L_t
    \cap \cont(2^t, d(x)/10).\label{eq:decomposition-of-pivotal cle for plus circuit}\end{equation} 
     Let $\sB_t$ denote the collection of boundary conditions on $\partial\Lambda_{2^t}(x)$, i.e., $\sB_t=\{1,-1\}^{\partial\Lambda_{2^t}(x)}$.
Here, notice that $\mathtt L_t$ is measurable with respect to the configurations on {$\Lambda_{2^t}$}. 
 For $\nu\in \sB_t$, let $$\Psi_t(\nu):=\mu^{\xi}_{\Lambda_{d(x)/10}(x)}\Big(\sigma|_{\partial\Lambda_{2^t}(x)}=\nu\,
 \Big).$$ Applying the DMP and Lemma \ref{lem: BK for Ising weak four arm}, we obtain that 
    \begin{align}
&\sum_{t=1}^{\lceil\log_2(d(x)/10)\rceil}\bar\mu^{\xi}_{\Lambda_{d(x)/10}(x)}(\mathtt L_t\cap \cont(2^t,d(x)/10))\nonumber\\ {=}& \sum_{t=1}^{\lceil\log_2(d(x)/10)\rceil}\sum_{\nu\in \sB_t}\Psi_t(\nu)\bar\mu_{\Lambda_{2^t}(z)}^{\nu}(\mathtt L_t)\times\mu_{\Lambda_{2^t,d(x)/10}(z)}^{\nu,\xi}\big(\cont(2^t,d(x)/10)\big)\nonumber~~~\text{(by DMP)}\\ {\le} &\sum_{t=1}^{\lceil\log_2(d(x)/10)\rceil}\left[\sum_{\nu\in \sB_t}\Psi_t(\nu)\times\bar\mu_{\Lambda_{2^t}(z)}^{\nu}(\mathtt L_t)\right]\times C_1(\frac{d(x)}{2^t})^{-\frac{1}{4}}\nonumber~~~(\text{by Lemma \ref{lem: BK for Ising weak four arm}}),
    \end{align}
where $\mu_{\Lambda_{2^t,d(x)/10}(z)}^{\nu,\xi}(\cont(2^t,d(x)/10)):=1$ if $2^t\ge d(x)/10$. 
Applying the total law of probability and the change from $\bar\mu^{\xi}_{\Lambda_{d(x)/10}(x)}$ to $\phi^{\xi}_{\Lambda_{d(x)/10}(x)}$, we obtain that \begin{align}  &\sum_{t=1}^{\lceil\log_2(d(x)/10)\rceil}\left[\sum_{\nu\in \sB_t}\Psi_t(\nu)\times\bar\mu_{\Lambda_{2^t}(z)}^{\nu}(\mathtt L_t)\right]\times C_1(\frac{d(x)}{2^t})^{-\frac{1}{4}}\nonumber\\ =~&\sum_{t=1}^{\lceil\log_2(d(x)/10)\rceil}\phi_{\Lambda_{d(x)/10}(x)}^{\xi}(\mathtt L_t)\times C_1(\frac{d(x)}{2^t})^{-\frac{1}{4}}\le C_2\sum_{t=1}^{\lceil\log_2(d(x)/10)\rceil}2^{-\frac{t}{8}}\times (\frac{d(x)}{2^t})^{-\frac{1}{4}}\le C_3 d(x)^{-\frac{1}{8}}.\nonumber
\end{align}
Thus, we complete the proof of the 
Lemma~\ref{lem: pivotal event bound for cle outermost} for $\mcc W=\mcc X^+.$

Now we explain how to deal with the other cases when $\mcc W\in\{\mcc X^-,\mcc Y^+,\mcc Y^-\}.$ To this end, we claim that \eqref{eq: pivotal for cle} actually holds for
a large class of events $\mcc W$.
Let $\mss P_0$ denote the set of events that two disjoint subsets of $\partial\Lambda_M\cup\Lambda_{M/2}\cup\Lambda_{M/4}$ are connected by plus spins, i.e.,
$$\mss P_0=\big\{\{A{\leftrightarrow} B\}:A,B\subset\bpartial\Lambda_M=\partial\Lambda_M\cup\partial\Lambda_{M/2}\cup\partial\Lambda_{M/4}, A\cap B=\emptyset \big\}.$$
Moreover, define
$$\mss P=\{\bigcap_{i=1}^n\mcc A_i:n<\infty, \mcc A_i\in\mss P_0 \mbox{ for }1\leq i\leq n\}$$
to be the finite intersection of events in $\mss P_0.$

Note that all the events in $\mss P_0$ are increasing, and therefore so are events in $\mss P.$ For any $\mcc W=\bigcap_{i=1}^n \mcc A_n\in\mss P,$ where $\mcc A_i=\{A_i{\leftrightarrow} B_i\}$,
if $\mcc C_x$ is pivotal to $\mcc W$,  then $\mcc W$ holds with $s_x=1$ and fails with $s_x=0.$ 
Let $\mcc A_j$ be the first one (according to some fixed order on $\mss P_0$) among events $\mcc A_1$ to $\mcc A_n$ such that $\mcc C_x$ is pivotal to $\mcc A_j$ and $\gamma$ is the first path (according to some fixed order on paths) from $A_j$ to $B_j$ such that all the spins on $\gamma$ are plus if $s_x=1$. We could follow the arguments (and notations) above to show that the event $\mathsf{Con}_2(2^t,d(x)/10)$ occurs and thus the desired bound holds.

Now recall that 
$$(\mcc X^-)^c=\{\partial\Lambda_{M/2}{\leftrightarrow}\partial\Lambda_{M/4}\},\quad \mcc Y^+=\bigcap_{U\in\mcc U}\{U{\leftrightarrow}\mcc P_0\},\quad  (\mcc Y^-)^c= \{\mcc P_1{\leftrightarrow}\mcc P_0\} ,$$
therefore, $(\mcc X^-)^c,\mcc Y^+,(\mcc Y^-)^c\in\mss P.$ Noting that by definition $\mcc C_x$ being pivotal to $\mcc W$ is  equivalent to $\mcc C_x$ being pivotal to $\mcc W^c$ completes the proof.
\end{proof}
Lemma~\ref{lem: pivotal event bound for cle outermost} combined with the DMP property yields the following Corollary.
\begin{corollary}\label{cor: noex upper bound in cle singular}
    Let $\Xi_I=\{-1,1\}^{V(\mrr{Ball}(I))}\otimes \{-1,0,1\}^{E(\mrr{Ball}(I))}$ be the extended-Ising configuration space on $\overline{\mrr{Ball}(I)}.$ In addition, we  define $\Xi_I'\subset \Xi_I$ be the collection of extended-Ising configuration $\bar\eta$ on $\overline{\mrr{Ball}(I)}$ such that there exists a configuration $\bar\sigma\in \mathtt{NoEx}$ satisfying $\bar\sigma|_{\overline{\mrr{Ball}(I)}}=\bar\eta$. Then, we have 
    \begin{equation}\label{eq: noex upper bound in cle singular}
        \bar\mu^{\xi}_{\Lambda_{M}}(\Xi_I')\le c^{|I|} F(\bpartial\mathsf C_i,I).
    \end{equation}
\end{corollary}
Combining Lemma \ref{lem:cancellation-of-excellent-clusters in cle} and Corollary \ref{cor: noex upper bound in cle singular}, we will get upper bounds for $\langle R\sigma^I\rangle_{\Lambda_{M}}^\xi$ and $\langle S\sigma^I\rangle_{\Lambda_{M}}^\xi$. In order to extend to the desired upper bounds for the conditional expectation $\langle \sigma^I\mid \mcc R\rangle_{\Lambda_{M}}^\xi$ and $\langle \sigma^I\mid \mcc S\rangle_{\Lambda_{M}}^\xi$, we also need to control the probabilities of $\mcc R$ and $\mcc S$ under different  configurations in $\mrr{Ball}(I)$.
This leads to the following Lemma~\ref{lem: two rare events upper bound with conditioning} that we claim is sufficient to 
 to obtain Proposition~\ref{prop:stronger-DHX}.

\begin{lemma}\label{lem: two rare events upper bound with conditioning}
Let $I\subset \Lambda_{M}$ be a subset of $\Lambda_{M}.$ There exists a constant $c_9>0$ such that 
for any boundary condition $\bar\eta_I$ on $\overline{\mrr{Ball}(I)}$ and any configuration $\xi$, we have
    \begin{equation}\label{eq: two rare events upper bound with conditioning-X=0}
\mu^{\xi,\bar\eta_I}_{\Lambda_{M}\setminus{\mrr{Ball}(I)}}(\mcc R)\leq c_9^{|I|}\mu^{\xi}_{\Lambda_{M}}(\mcc R).
\end{equation}
    \begin{equation}\label{eq: two rare events upper bound with conditioning-X=1}
\mu^{\xi,\bar\eta_I}_{\Lambda_{M}\setminus{\mrr{Ball}(I)}}(\mcc S)\leq c_9^{|I|}\mu^{\xi}_{\Lambda_{M}}(\mcc S).
\end{equation}
Here, notice that although $\bar\eta_I$ is an extended-Ising configuration, it naturally works as a boundary condition for the Ising measure.
\end{lemma}

\begin{proof}[Proof of Proposition~\ref{prop:stronger-DHX} given Lemma~\ref{lem: two rare events upper bound with conditioning}]
We just prove \eqref{eq:extension-of-DHX-X=0} since \eqref{eq:extension-of-DHX-X=1} is essentially the same.
    Recall the definition of $\Xi_I'$ from Corollary~\ref{cor: noex upper bound in cle singular} and recall that $R=\mathbf{1}_{\mcc R}$. We get from Lemma~\ref{lem:cancellation-of-excellent-clusters in cle} that
\begin{align}\label{eq:proof-of-Etau-3}
	\Big|\langle \sigma^I\cdot R\rangle_{\Lambda_{M}}^{\xi}\Big|=&\Big|\langle \sigma^I\cdot R\cdot \mbf 1_{\mathtt{NoEx}}\rangle_{\Lambda_{M}}\Big|\le \langle R\cdot \mbf 1_{\mathtt{NoEx}}\rangle_{\Lambda_{M}}\nonumber \\
	\leq & \sum_{\bar\eta_I\in\Xi_I'}\bar\mu^{\xi}_{\Lambda_{M}}(\bar\eta_I)\cdot \mu^{\xi,\bar\eta_I}_{\Lambda_{M}\setminus{\mrr{Ball}(I)}}(\mcc R).
\end{align}

Combining Lemma~\ref{lem: two rare events upper bound with conditioning} and Corollary~\ref{cor: noex upper bound in cle singular} with \eqref{eq:proof-of-Etau-3}, we have that
\begin{align}
	\Big|\langle \sigma^I\cdot R\rangle_{\Lambda_{M}}^{\xi}\Big|&\leq C^{|I|}\mu^{\xi}_{\Lambda_{M}}(\mcc R)\sum_{\eta_I\in\Xi_I'}\bar\mu^{\xi}_{\Lambda_{M}}(\eta_I)\le C^{|I|}\langle  R\rangle_{\Lambda_{M}}^{\xi}\cdot F(\bar\partial\Lambda_{M},I),\label{eq:byproduct}
\end{align}
This completes the proof by recalling the definition that 
$\langle \sigma^I\mid \mcc R\rangle_{\Lambda_{M}}^{\xi}=\frac{\langle \sigma^I\cdot R\rangle_{\Lambda_{M}}^{\xi}}{\langle R\rangle_{\Lambda_{M}}^{\xi}}.$
 \end{proof}

We now turn to the proof of Lemma~\ref{lem: two rare events upper bound with conditioning}, which asserts that conditioning on the configuration inside $\mathrm{Ball}(I)$ does not affect the probability of the events $\mcc R$ and $\mcc S$ too much. 
Here are two key difficulties: first, the two events are neither increasing nor decreasing, making it really subtle to do some kind of FKG argument; second, the probabilities of $\mcc R$ and $\mcc S$ could be very close to 0 or 1 for some choice of $\mcc P_0,\mcc P_1, \mcc U$, in which case the measures conditioned on $\mathcal{R}$ or on $\mathcal{S}$ could behave badly.

To circumvent these difficulties, we begin with a much simpler setting: the following lemma quantifies the influence of $\eta_I$ on plus-crossing events.

\begin{lemma}\label{lem:Ball-influence-on-crossing}
    Let $\mcc V_1$ denote the event that there exists a plus crossing connecting  $\partial\Lambda_{M/2}$ and $\partial\Lambda_{M/4+1}$, and
    $\mcc V_2$ denote the event that there exists a plus circuit in $\Lambda_{M/2}\setminus\Lambda_{M/4}$.  Then there exists a constant $c_{10}>0$ such that for any configurations $\eta_I$ on $\mathrm{Ball}(I)$ and any boundary condition $\xi\in\Sigma_M^c$, we have 
    \begin{equation}
\mu_{\Lambda_M\setminus\mathrm{Ball}(I)}^{\xi,\eta
        _I}(\mcc V_i)\geq c_{10}^{|I|} \quad \mbox{for } i=1,2.
    \end{equation}
\end{lemma}

\begin{proof}
For any $x\in I$, let $\mcc A_x$ denote the event that  there exists a plus circuit in the annulus  $\Lambda_{\frac{d_x}{8}}(x)\setminus\Lambda_{\frac{d_x}{9}}(x)$ and let $\mcc A=\cap_{x\in I }\mcc A_x.$
 By FKG and RSW, we have that
    \begin{equation}
      \mu_{\Lambda_M\setminus\mathrm{Ball}(I)}^{\xi,\eta
        _I}(\mcc A)\geq \prod_{x\in I} \mu_{\Lambda_M\setminus\mathrm{Ball}(I)}^{\xi,\eta
        _I}(\mcc A_x)\ge C_1^{|I|}.\label{eq: plus circuit event probability O}
    \end{equation}  On the event $\mcc A$, we can explore the innermost plus circuit in each annulus $\Lambda_{\frac{d_x}{8}}(x)\setminus\Lambda_{\frac{d_x}{9}}(x)$ and denote this collection of circuits by $\gamma=\{\gamma_x\}_{x\in I}.$
        Let $\mcc O_\gamma$ denote the event that the exploration outcome is $\gamma$ and
        $\mss O$ denote all the possible exploration outcomes $\gamma$. 
        Then we have that 
\begin{equation}
    \label{eq: plus circuit event O in each box decomposition}
    \begin{aligned}
        \mu_{\Lambda_M\setminus\mathrm{Ball}(I)}^{\xi,\eta
        _I}(\mcc A)&=\sum_{\gamma\in\mss O} \mu_{\Lambda_M\setminus\mathrm{Ball}(I)}^{\xi,\eta
        _I}(\mcc O_\gamma),\\\mbox{ and }\mu_{\Lambda_M\setminus\mathrm{Ball}(I)}^{\xi,\eta
        _I}(\mcc V_i\cap\mcc A)
        &=\sum_{\gamma\in\mss O}\mu_{\Lambda_M\setminus\mathrm{Ball}(I)}^{\xi,\eta_I}(\mcc V_i\cap\mcc O_\gamma).
    \end{aligned}
\end{equation}
 Let $\Gamma=\cup_{x\in I}\Gamma_x$ where $\Gamma_x$ is the
        interior of $\gamma_x$.
        Note that on the event $\mcc O_\gamma,$ all the spins on $\gamma$ are positive. In addition, $\mcc V_i$ depends only on the connectivity of plus spins and thus, conditioned on $\{\sigma_v=1,\forall v\in\cup_{x\in I}\gamma_x\},$ it does not depend on the configuration in $\Gamma\setminus(\cup_{x\in I}\gamma_x)$. Then we get that \begin{equation}
            \mu_{\Lambda_M\setminus\mathrm{Ball}(I)}^{\xi,\eta
        _I}(\mcc V_i\mid \mcc O_\gamma)=\mu_{\Lambda_M\setminus\Gamma}^{\xi,+}(\mcc V_i)= \mu_{\Lambda_M}^{\xi}(\mcc V_i\mid \sigma_v=1,\forall v\in\cup_{x\in I}\gamma_x)\label{eq: plus circuit boundary reduction in O}
        \end{equation} where the equality comes from the fact that $\mu_{\Lambda_M\setminus\Gamma}^{\xi,+}$ and $\mu_{\Lambda_M}^{\xi}(\cdot\mid \sigma_v=1,\forall v\in\cup_{x\in I}\gamma_x)$ have the same law on $\Lambda_M\setminus\Gamma.$
        FKG and RSW yield that $$\mu_{\Lambda_M}^{\xi}(\mcc V_i\mid \{\sigma_v=1,\forall v\in\cup_{x\in I}\gamma_x\}\stackrel{\mathrm{FKG}}{\geq}\mu_{\Lambda_M}^{\xi}(\mcc V_i)\stackrel{\mathrm{RSW}}{\geq} C_2.$$ 
        Plugging into \eqref{eq: plus circuit boundary reduction in O}, we get that \begin{equation}
            \mu_{\Lambda_M\setminus\mathrm{Ball}(I)}^{\xi,\eta
        _I}(\mcc V_i\mid \mcc O_\gamma)\ge C_2.\label{eq: Vi prob conditioned on plus circuit O}
        \end{equation}
        Combining \eqref{eq: plus circuit event probability O}, \eqref{eq: plus circuit event O in each box decomposition} and \eqref{eq: Vi prob conditioned on plus circuit O}, we obtain that
        \begin{align*}
        \mu_{\Lambda_M\setminus\mathrm{Ball}(I)}^{\xi,\eta
        _I}(\mcc V_i)\geq ~&\mu_{\Lambda_M\setminus\mathrm{Ball}(I)}^{\xi,\eta
        _I}(\mcc V_i\cap\mcc A)
        \stackrel{\eqref{eq: plus circuit event O in each box decomposition}}{=}\sum_{\gamma\in\mss O}\mu_{\Lambda_M\setminus\mathrm{Ball}(I)}^{\xi,\eta_I}(\mcc V_i\cap\mcc O_\gamma)\\=~&\sum_{\gamma\in\mss O}\mu_{\Lambda_M\setminus\mathrm{Ball}(I)}^{\xi,\eta_I}(\mcc O_\gamma)\mu_{\Lambda_M\setminus\mathrm{Ball}(I)}^{\xi,\eta_I}(\mcc V_i\mid \mcc O_\gamma)\\
        \stackrel{\eqref{eq: Vi prob conditioned on plus circuit O}}{\ge} &C_2\sum_{\gamma\in\mss O}\mu_{\Lambda_M\setminus\mathrm{Ball}(I)}^{\xi,\eta_I}(\mcc O_\gamma)\stackrel{\eqref{eq: plus circuit event O in each box decomposition}}{=}  C_2\mu_{\Lambda_M\setminus\mathrm{Ball}(I)}^{\xi,\eta
        _I}(\mcc A)\stackrel{\eqref{eq: plus circuit event probability O}}{\ge} C_2\cdot C_1^{|I|}.\qedhere
        \end{align*}
\end{proof}

To deal with the events $\mcc X$ and $\mcc Y,$ we start by exploring the clusters $\mcc P_0$ and $\mcc P_1.$

\begin{definition}\label{def:explore-plus-spin-clusters}
Let $\mcc C_0$ denote the plus spins cluster of $\mcc P_0$ in $\Lambda_M.$
Then all the spins on $\partial_{\mathrm{ext}}\mcc C_0\cap \Lambda_M$ are minus. Let $\mss C_0\subset 2^{\Lambda_M}$ denote all the possible choices of $\mcc C_0$, i.e., $A\in\mss C_0$ if and only if 
$\mu^{\xi}(\mcc C_0=A)>0.$
 Let $\mss F_0=\sigma\left(\{\mcc C_0=A\}: A\in\mss C_0\right)$ denote the $\sigma$-algebra generated by $\mcc C_0.$ Recall that $\mathcal Y^{-}$ is the event that $\mathcal P_0$ is not plus-connected to $\mathcal P_1$ inside $\Lambda_M$, and  $\mathcal Y^{+}$ is the event that, for every $U\in \mathcal U$, the set $U$ is plus-connected to $\mathcal P_0$ inside $\Lambda_M$.
 Then we have $\mcc Y^+,\mcc Y^-\in\mss F_0$ since they are all decided by $\mcc C_0.$
 Similarly, we can define $\mcc C_1$ and $\mss F_1$ to be the plus spins cluster of $\mcc P_1$ in $\Lambda_M$ and the $\sigma$-algebra generated by $\mcc C_1$.
\end{definition}

Now the intuition is if $\mcc C_0$ is large and the event $\mcc Y^-$ holds, i.e. $\mcc P_0$ is not connected to $\mcc P_1,$ then $\mcc C_1$ is unlikely to be also large.
To be precise, we have the following lemma.

\begin{lemma}\label{lem:explore-plus-spin-clusters}
   
    There exists a universal constant $c_{11}>0$ such that for any increasing event $\mcc I\in\mss F_0$ and any boundary condition $\xi$ we have 
    \begin{equation}
        \mu_{\Lambda_M}^{\xi}(\mcc C_1\cap\Lambda_{0.95M}=\emptyset\mid \mcc I\cap\mcc Y^-)\geq c_{11}.
    \end{equation}
    In particular, we could choose $\mcc I=\mcc Y^+$ so that $\mcc I\cap\mcc Y^-=\mcc Y.$ We keep the flexibility of choosing $\mcc I$ for later convenience.
\end{lemma}

\begin{proof}

    Since $\mcc I\cap\mcc Y^-\in\mss F_0,$ we can define $\hat{\mss C_0}\subset\mss C_0$  to be the collection of $A$ such that $\{\mcc C_0=A\}\subset\mcc I\cap\mcc Y^-$. Then we have that
    \begin{equation}
        \mu^{\xi}_{\Lambda_{M}}(\{\mcc C_1\cap\Lambda_{0.95M}=\emptyset\}\cap\mcc I\cap\mcc Y^-)=\sum_{A\in \hat{\mss{C}_0}}\mu^{\xi}_{\Lambda_{M}}(\mathcal{C}_0=A) \cdot \mu^{\xi}_{\Lambda_{M}}(\mcc C_1\cap\Lambda_{0.95M}=\emptyset\mid \mcc C_0=A).\label{eq: event E localize 1 with interface}
    \end{equation}
    Recall the discussions before \eqref{eq: plus circuit boundary reduction in O}. We want to apply a similar argument to the event $\{\mcc C_0=A\}.$ Note that all the spins on $\partial_{\mathrm{ext}}\mcc C_0\cap \Lambda_M=\partial_{\mathrm{ext}}A\cap \Lambda_M$ are minus and $\mcc C_1$ is the plus spin cluster, thus conditioned on $\{\sigma_v=-1,\forall v\in\partial_{\mathrm{ext}}A\cap \Lambda_M\},$ the event $\mcc C_1\cap\Lambda_{0.95M}=\emptyset$ does not depend on the configuration in $A\setminus\partial A$.
    Therefore, for the same reason as \eqref{eq: plus circuit boundary reduction in O}, we get that 
\begin{equation}
    \mu^{\xi}_{\Lambda_{M}}(\mcc C_1\cap\Lambda_{0.95M}=\emptyset\mid \mcc C_0=A)=\mu^{\xi}_{\Lambda_{M}}(\mcc C_1\cap\Lambda_{0.95M}=\emptyset\mid \sigma_v=-1,\forall v\in\partial_{\mathrm{ext}}A\cap \Lambda_M).\label{eq: event E localize 2 with interface}
\end{equation} 
Now, both the events $\{\mcc C_1\cap\Lambda_{0.95M}=\emptyset\}$ and $\{\sigma_v=-1,\forall v\in\partial_{\mathrm{ext}}A\cap \Lambda_M\}$ are decreasing, we can use FKG and RSW to obtain that
\begin{equation}
\mu^{\xi}_{\Lambda_{M}}(\mcc C_1\cap\Lambda_{0.95M}=\emptyset\mid \sigma_v=-1,\forall v\in\partial_{\mathrm{ext}}A\cap \Lambda_M)\stackrel{\mathrm{FKG}}{\geq} \mu^{\xi}_{\Lambda_{M}}(\mcc C_1\cap\Lambda_{0.95M}=\emptyset)\geq \mu^{\xi}_{\Lambda_{M}}(\mcc A^-)\stackrel{\mathrm{RSW}}{\ge} C_1,\nonumber
\end{equation} 
where $\mathcal{A}^-$ denotes the event that there exists a minus circuit in $\Lambda_{M}\setminus\Lambda_{0.95M}$, which implies $\{\mcc C_1\cap\Lambda_{0.95M}=\emptyset\}$. 
\end{proof}

The utility of Lemma~\ref{lem:explore-plus-spin-clusters} is that once $\mcc C_1\cap \Lambda_{0.95M}=\emptyset,$ we have some distance to take advantage of the mixing property of the Ising model.  As a consequence, we could bound the probability of $\mcc V_1$ conditioning on $\mcc Y.$ Recall that $\mcc V_1$ denotes the event that there exists a plus crossing connecting $\partial\Lambda_{M/2}$ and $\partial\Lambda_{M/4}$, therefore $\mcc V_1$ and $\mcc X^-$ cannot hold simultaneously, i.e.
$$\mcc V_1\subset \{X=0\},
\quad (\mcc V_1\cap \mcc Y)\subset  \mcc R.$$
\begin{lemma}\label{lem:bound-plus-crossing-given-Y}
    There exists a universal constant $c_{12}>0$ such that for any boundary condition $\xi$ we have that
    $$\mu^{\xi}_{\Lambda_M}(\mcc V_1\mid \mcc Y)\geq c_{12}. $$
\end{lemma}

\begin{proof}
    Let $\mcc A'=\{\mcc C_1\cap \Lambda_{0.95M}=\emptyset\}$.
    Note that the events $\mcc A'$ and $\mcc Y^-$ are decided by $\mcc C_1,$ i.e. $\mcc A',\mcc Y^-\in \mss F_1$ (recall the definition of $\mss F_1$ from Definition~\ref{def:explore-plus-spin-clusters}). Let $\hat{\mss C}_1$ denote the set of $A$ such that 
    $\{\mcc C_1=A\}\subset \mcc A'\cap \mcc Y^-,$
    then we have that
    \begin{align}\label{eq:bound-plus-crossing-given-Y-1}
        \mu^{\xi}_{\Lambda_M}(\mcc V_1\cap \mcc Y\cap \mcc A')
        =\sum_{A\in\hat{\mss C_1}}\mu^{\xi}_{\Lambda_M}(\{\mcc C_1=A\})\mu^{\xi}_{\Lambda_M}(\mcc V_1\cap \mcc Y^+\mid \mcc C_1=A).
    \end{align}
    Similarly, we also have that \begin{equation}\label{eq:bound-plus-crossing-given-Y-5}
        \mu_{\Lambda_M}^{\xi}(\mcc A'\cap\mcc Y)=\sum_{A\in\hat{\mss C_1}}\mu^{\xi}_{\Lambda_M}(\{\mcc C_1=A\})\mu^{\xi}_{\Lambda_M}(\mcc Y^+\mid \{\mcc C_1=A\}).
    \end{equation}
    Recall that all the spins on $\partial_{\mathrm{ext}}\mcc C_1\cap\Lambda_M$ are minus. By DMP and FKG, we have that
    \begin{equation}\label{eq:bound-plus-crossing-given-Y-2}
        \mu^{\xi}_{\Lambda_M}(\mcc V_1\cap \mcc Y^+\mid \mcc C_1=A)\stackrel{\mathrm{DMP}}{=}\mu^{\xi,-}_{\Lambda_M\setminus A}(\mcc V_1\cap \mcc Y^+)\stackrel{\mathrm{FKG}}{\geq} \mu^{\xi,-}_{\Lambda_M\setminus A}(\mcc V_1)\mu^{\xi,-}_{\Lambda_M\setminus A}(\mcc Y^+).
    \end{equation}
    Here, $(\xi,-)$ denotes the boundary condition which is $\xi$ on $\partial\Lambda_M$ and which is minus on $\partial_{\mathrm{ext}}\mcc C_1\cap\Lambda_M.$ 

Note that $\mcc V_1$ is measurable with respect to the configuration inside $\Lambda_{M/2}$, while every $A\in\hat{\mathscr C}_1$ is contained in $\Lambda_M\setminus\Lambda_{0.95M}$. Hence, by the mixing property  of Lemma~\ref{lem: spatial mixing of Ising model},
    \begin{equation}\label{eq:bound-plus-crossing-given-Y-3}
        \mu^{\xi,-}_{\Lambda_M\setminus A}(\mcc V_1)\geq C_1\mu^{\xi}_{\Lambda_M}(\mcc V_1).
    \end{equation}
    Plugging \eqref{eq:bound-plus-crossing-given-Y-3} into \eqref{eq:bound-plus-crossing-given-Y-2} and then into \eqref{eq:bound-plus-crossing-given-Y-1}, we obtain that
\begin{align}
     \mu^{\xi}_{\Lambda_M}(\mcc V_1\cap \mcc Y\cap \mcc A')&~\geq~ \mu^{\xi,-}_{\Lambda_M\setminus A}(\mcc V_1)\sum_{A\in\hat{\mss C_1}}\mu^{\xi}_{\Lambda_M}(\{\mcc C_1=A\})\mu^{\xi,-}_{\Lambda_M\setminus A}(\mcc Y^+)\nonumber\\&~\geq~C_1\mu^{\xi}_{\Lambda_M}(\mcc V_1)\sum_{A\in\hat{\mss C_1}}\mu^{\xi}_{\Lambda_M}(\{\mcc C_1=A\})\mu^{\xi,-}_{\Lambda_M\setminus A}(\mcc Y^+).\label{eq: V1 Y separation}
\end{align}
    For the same reason as in \eqref{eq: plus circuit boundary reduction in O}, we get that $\mu^{\xi,-}_{\Lambda_M\setminus A}(\mcc Y^+)=\mu^{\xi}_{\Lambda_M}(\mcc Y^+\mid \{\mcc C_1=A\}).$ Plugging into \eqref{eq: V1 Y separation} and combining with \eqref{eq:bound-plus-crossing-given-Y-5}, we get that
    \begin{align}
        \mu^{\xi}_{\Lambda_M}(\mcc V_1\cap \mcc Y\cap \mcc A')
        &~\ge C_1\mu^{\xi}_{\Lambda_M}(\mcc V_1)\sum_{A\in\hat{\mss C_1}}\mu^{\xi}_{\Lambda_M}(\{\mcc C_1=A\})\mu^{\xi}_{\Lambda_M}(\mcc Y^+\mid \{\mcc C_1=A\})\nonumber\\
        &~=C_1\mu^{\xi}_{\Lambda_M}(\mcc V_1)\mu_{\Lambda}^{\xi}(\mcc A'\cap\mcc Y)\geq C_2\mu^{\xi}_{\Lambda_M}(\mcc V_1)\mu_{\Lambda}^{\xi}(\mcc Y),\label{eq:bound-plus-crossing-given-Y-4}
    \end{align}where the last inequality comes from choosing $\mcc I=\mcc Y^+$ in Lemma~\ref{lem:explore-plus-spin-clusters} and recalling that $\mcc A'=\{\mcc C_1\cap \Lambda_{0.95M}=\emptyset\}$. By RSW, we get that $\mu^{\xi}_{\Lambda_M}(\mcc V_1)\geq C_3$. Plugging into \eqref{eq:bound-plus-crossing-given-Y-4} and noting that $\mu^{\xi}_{\Lambda_M}(\mcc V_1\cap \mcc Y)\geq 
        \mu^{\xi}_{\Lambda_M}(\mcc V_1\cap \mcc Y\cap \mcc A')$ completes the proof of Lemma \ref{lem:bound-plus-crossing-given-Y}.
\end{proof}

\begin{remark}
    We point out that \eqref{eq:bound-plus-crossing-given-Y-4} is an FKG-type inequality up to a constant. In other words, Lemma~\ref{lem:explore-plus-spin-clusters} provides the possibility to show some kind of FKG results for events neither increasing nor decreasing. This exploration approach may also prove applicable to other problems with complex dependencies beyond the current framework.
\end{remark}

Now we claim that we have almost finished the proof of \eqref{eq: two rare events upper bound with conditioning-X=0}, given that the following lemma holds.

\begin{lemma}\label{lem: two rare event probability with interface and ball}
    Let $I\subset \Lambda_{M}$ be a subset of $\Lambda_{M}.$ There exists a constant $c_{13}>0$ such that 
for any boundary condition $\eta_I$ on $\mrr{Ball}(I)$ and any configuration $\xi\in\{-1,1\}^{(\lamn\setminus\Lambda_M)\cup\partial\Lambda_M}$,  we have that
    \begin{align}
	\mu_{\Lambda_M}^{\xi,\eta_I}(\mcc Y)\le c_{13}^{|I|}\mu_{\Lambda_M}^{\xi}(\mcc Y).\label{eq: two rare event probability with interface and ball}
\end{align}
\end{lemma}

\begin{proof}[Proof of \eqref{eq: two rare events upper bound with conditioning-X=0} given Lemma \ref{lem: two rare event probability with interface and ball}.]Recall that $\mcc R=\mcc Y\cap \{X=0\}.$
    \begin{align*}
        ~&\mu_{\Lambda_M}^{\xi,\eta_I}(\mcc R)=\mu_{\Lambda_M}^{\xi,\eta_I}(\{X=0\}\cap\mcc Y)\leq \mu_{\Lambda_M}^{\xi,\eta_I}(\mcc Y)
        \stackrel{(*)}{\leq} c_{13}^{|I|}\mu_{\Lambda_M}^{\xi}(\mcc Y)\stackrel{(**)}{\leq} c_{12}\cdot c_{13}^{|I|}\mu_{\Lambda_M}^{\xi}(\mcc Y\cap \mcc V_1).
    \end{align*}
    Here, (*) uses Lemma~\ref{lem: two rare event probability with interface and ball} and (**) uses Lemma~\ref{lem:bound-plus-crossing-given-Y}. Combining with the fact that $\mcc V_1\subset\{X=0\}$, we get that \begin{equation*}
        \mu_{\Lambda_M}^{\xi,\eta_I}(\mcc R)\leq c_{12}\cdot c_{13}^{|I|}\mu_{\Lambda_M}^{\xi}(\mcc Y\cap \{X=0\})\leq c_{9}^{|I|}\mu_{\Lambda_M}^{\xi}(\mcc R)\qedhere
    \end{equation*}
    by choosing $c_9$ relying on $c_{12}$ and $c_{13}$.
\end{proof}
The proof of Lemma~\ref{lem: two rare event probability with interface and ball} relies on an induction process. We first introduce the following decomposition of $\eta_I$.

\begin{definition}\label{def:decompose-Ball}
    Fix $I=\{x_1,x_2,\cdots,x_k\}\subset \Lambda_M$ and recall the definition of $d_x$ and 
$\mrr{Ball}(I)$
from
Definition~\ref{def: distance and ball}.
For $0\leq p\leq k,$ let 
$$\mrr{Ball}(I,0)=\emptyset,\quad\mrr{Ball}(I,p)=\bigcup_{j=1}^p\Lambda_{d_{x_j}/10}(x_j).$$ Then define 
$\eta_I^p$ to be the configuration that $\eta_I$ restricts on ${\eta_I}|_{_{\mrr{Ball}(I,p)}}$.
Here, for clarity, we point out that by definition 
$$\mrr{Ball}(I,k)=\mrr{Ball}(I)\mbox{ but }\mrr{Ball}(I,p)\neq \mrr{Ball}(\{x_1,x_2,\cdots,x_p\})$$
for general $p$ since the definition of $d_{x_j}$ could also depend on $\{x_{p+1},\cdots,x_k\}$.
\end{definition}

\begin{proof}[Proof of Lemma~\ref{lem: two rare event probability with interface and ball}]

With the notations in Definition~\ref{def:decompose-Ball}, it suffices to show that for $1\le p\le k$
\begin{equation}
    \label{eq: two rare event probability with interface and ball-induction step} \mu_{\Lambda_M}^{\xi,\eta_I^p}(\mcc Y)\le c_9\mu_{\Lambda_M}^{\xi,\eta_I^{p-1}}(\mcc Y).
\end{equation}

For simplicity, from now on we fix $1\le p\le k$ and write $$x=x_p, B=\Lambda_{d_{x}/10}(x), \eta=\eta_I^{p-1}, \eta^B=\eta_I^p, B^{s}=\Lambda_{sd_{x}/10}(x)\mbox{ for }s=8,9,10,$$
until the end of this proof.

Now for the measures $\mu_{\Lambda_M}^{\xi,\eta^B}$ and $\mu_{\Lambda_M}^{\xi,\eta},$ we  explore the configurations outside $B^{10}.$
That is, let $\hat\Sigma=\{-1,1\}^{(\Lambda_M\setminus B^{10})\cup\partial B^{10}}$ be the configuration set, then we have 
\begin{align}\label{eq:explore-outside-10B}
\mu_{\Lambda_M}^{\xi,\eta^B}(\mcc Y)
=\sum_{\hat\xi\in\hat\Sigma}\mu_{\Lambda_M}^{\xi,\eta^B}(\hat\xi)\mu_{ B^{10}}^{\hat\xi,\eta^B}(\mcc Y)
=\sum_{\hat\xi\in\hat\Sigma}\mu_{\Lambda_M}^{\xi,\eta^B}(\hat\xi)\mu_{ B^{10}}^{\hat\xi,\eta^B}( \mcc Y)
=\sum_{\hat\xi\in\hat\Sigma}\mu_{\Lambda_M}^{\xi,\eta^B}(\hat\xi)\mu_{ B^{10}}^{\hat\xi,\eta^B}(\hat{\mathcal{Y}}),
\end{align}
and a similar expansion holds if we replace $\eta^B$ by $\eta$.
Here $\hat{\mcc Y}$ is defined in the following sense:
let $\hat{\mcc P}_0$ denote the ``push forward" of $\mcc P_0$ on $\partial B^{10}$ given the information $\hat\xi\in\hat\Sigma$, that is, the
set of points on $\partial{B^{10}}$ that is connected to $\mcc P_0$ in $\hat\xi$ via plus spins, we can define $\hat{\mcc P}_1, \hat U$ similarly. Then we can define the events of $\hat{\mcc Y}^+$ and $\hat{\mcc Y}^-$ to be the push forward of $\mcc Y^+$ and $\mcc Y^-,$ i.e. $\hat{\mcc Y}^+$ is the event that every component of $\hat{\mcc U}$ is connected to $\hat{\mcc P}_0$ and $\hat{\mcc Y}^-$ is the event that $\hat{\mcc P}_1$ is not connected to $\hat{\mcc P}_0$.
Finally, we define $\hat{\mcc Y}=\hat{\mcc Y}^+\cap \hat{\mcc Y}^-.$

Now we define two events 
\begin{align*}
  &\mcc A_1:=\{\mbox{the connected component of }\hat{\mcc P}_1 \mbox{ does not intersect with } B^9\},\\
  &\mcc A_2:=\{\mbox{there exsits a plus spin circuit in }B^9\setminus B^8\},
\end{align*}
and claim that there exists a constant $C>0$ such that the following inequalities hold for any $\xi\in \Sigma_M^c$ and $\hat\xi\in\hat\Sigma$:
\begin{align}
&\mu_{\Lambda_M}^{\xi,\eta}(\hat\xi)\geq C\mu_{\Lambda_M}^{\xi,\eta^B}(\hat\xi);\label{proof of 177-0}\\
   & \mu_{B^{10}}^{\hat\xi,\eta^B}(\mcc A_1\mid\hat{\mcc Y})\geq C;\label{proof of 177-1}\\
   &\mu_{ B^{10}}^{\hat\xi,\eta^B}(\mcc A_2\mid \mcc A_1\cap \hat{\mcc Y})\geq C;\label{proof of 177-2}\\
   &\mu_{ B^{10}}^{\hat\xi}(\mcc A_2\cap \hat{\mcc Y})\geq C\mu_{ B^{10}}^{\hat\xi,\eta^B}(\mcc A_2\cap \hat{\mcc Y})\label{proof of 177-3}.
\end{align}

Now we quickly explain why these inequalities hold.
 \eqref{proof of 177-1} follows directly from Lemma~\ref{lem:explore-plus-spin-clusters}. Combining \eqref{eq:bound-plus-crossing-given-Y-4} with RSW, we get that \begin{equation}
     \mu_{ B^{10}}^{\hat\xi,\eta^B}(\mcc A_2\mid \mcc A_1\cap \hat{\mcc Y})\stackrel{\eqref{eq:bound-plus-crossing-given-Y-4}}{\ge}\frac{C_1\mu_{ B^{10}}^{\hat\xi,\eta^B}(\mcc A_2)\mu_{ B^{10}}^{\hat\xi,\eta^B}(\hat{\mcc{ Y}})}{\mu_{ B^{10}}^{\hat\xi,\eta^B}( \mcc A_1\cap \hat{\mcc Y})}\stackrel{\mbox{RSW}}{\ge}\frac{C_2\mu_{ B^{10}}^{\hat\xi,\eta^B}(\hat{\mcc{ Y}})}{\mu_{ B^{10}}^{\hat\xi,\eta^B}(\mcc A_1\cap  \hat{\mcc Y})}\ge C_2.
 \end{equation}Thus we finish the proof of \eqref{proof of 177-2}.
Both \eqref{proof of 177-0} and \eqref{proof of 177-3} come from the mixing property of Lemma~\ref{lem: spatial mixing of Ising model}. \eqref{proof of 177-0}  claims that the configuration inside $B$ does not influence the configuration outside $B^{10}$ too much. To prove \eqref{proof of 177-3}, it suffices to show that the event $\mcc A_2\cap \hat{\mcc Y}$ relies only on the configuration in $(B^8)^c$. We prove this by contradiction. Assume that  $\sigma_1,\sigma_2$ are two configurations that coincide on $(B^8)^c$, such that $\sigma_1\in \mcc A_2\cap\hat{\mcc Y}$ but $\sigma_2$ does not. Then $\sigma_1$ contains a plus circuit $\gamma$ in $B^9\setminus B^8$, so does $\sigma_2.$ Now, given the existence of $\gamma,$ whether $\hat{\mcc P_0}$ is connected to $\hat{\mcc P_1}$ or $\hat{\mcc U}$ is decided outside $\gamma,$ on which region $\sigma_1$ coincides with $\sigma_2.$ Hence we arrive at the contradiction since $\sigma_2\notin \mcc A_2\cap\hat{\mcc Y}$.

We are now ready to prove \eqref{eq: two rare event probability with interface and ball-induction step}:
\begin{align*}
~~&\mu_{\Lambda_M}^{\xi,\eta_I^{p-1}}(\mcc Y)
\stackrel{\eqref{eq:explore-outside-10B}}{=}\sum_{\hat\xi\in\hat\Xi}\mu_{\Lambda_M}^{\xi,\eta}(\hat\xi)\mu_{ B^{10}}^{\hat\xi,\eta}(\hat{\mathcal{Y}})
\stackrel{\eqref{proof of 177-0}}{\geq} C\sum_{\hat\xi\in\hat\Xi}\mu_{\Lambda_M}^{\xi,\eta^B}(\hat\xi)\mu_{ B^{10}}^{\hat\xi,\eta}(\mcc A_2\cap\hat{\mathcal{Y}})\\
\stackrel{\eqref{proof of 177-3}}{\geq} &C^2
\sum_{\hat\xi\in\hat\Xi}\mu_{\Lambda_M}^{\xi,\eta^B}(\hat\xi)\mu_{ B^{10}}^{\hat\xi,\eta^B}(\mcc A_2\cap\hat{\mathcal{Y}})
\geq C^2
\sum_{\hat\xi\in\hat\Xi}\mu_{\Lambda_M}^{\xi,\eta^B}(\hat\xi)\mu_{ B^{10}}^{\hat\xi,\eta^B}(\mcc A_1\cap\mcc A_2\cap\hat{\mathcal{Y}})\\
\stackrel{\eqref{proof of 177-2}}{\geq} &C^3
\sum_{\hat\xi\in\hat\Xi}\mu_{\Lambda_M}^{\xi,\eta^B}(\hat\xi)\mu_{ B^{10}}^{\hat\xi,\eta^B}(\mcc A_1\cap\hat{\mathcal{Y}})
\stackrel{\eqref{proof of 177-1}}{\geq} C^4
\sum_{\hat\xi\in\hat\Xi}\mu_{\Lambda_M}^{\xi,\eta^B}(\hat\xi)\mu_{ B^{10}}^{\hat\xi,\eta^B}(\hat{\mathcal{Y}})
\stackrel{\eqref{eq:explore-outside-10B}}{=}C^4\mu_{\Lambda_M}^{\xi,\eta_I^p}(\mcc Y).
\end{align*}
Choosing $c_{13}\geq C^{-4}$ completes the proof of Lemma~\ref{lem: two rare event probability with interface and ball}. 
\end{proof}


Next, we prove \eqref{eq: two rare events upper bound with conditioning-X=1}. The main difference between the proofs for \eqref{eq: two rare events upper bound with conditioning-X=0} and \eqref{eq: two rare events upper bound with conditioning-X=1} is that $\{X=0\}$ is an event with positive probability, thus we can apply Lemma~\ref{lem:bound-plus-crossing-given-Y} and then \eqref{eq: two rare events upper bound with conditioning-X=0} is reduced to consider only the event $\mcc Y$. To deal with the event $\{X=1\}=\mcc X$, it is possible that $\mcc P_0$ is small such that the probability of $\mcc X$ is close to 0. Therefore, we need a more careful analysis. We want to explore the outermost circuit in $\Lambda_{M/2}\setminus\Lambda_{M/4}$ and then show that the location of the outermost circuit will not influence the probability of the event $\mcc S=\mcc X\cap\mcc Y$ too much. We now elaborate as follows.
Let $\gamma$ be a plus-minus  circuit and $\Gamma$ be the region enclosed by $\gamma$. Note that whether a circuit $\gamma$ is the outermost circuit is equivalent to whether $\gamma$ is connected to $\mathcal{P}_0$. In order to explore $\gamma$ from inside, we order all plus-minus circuits in
$\Lambda_{M/2}\setminus\Lambda_{M/4}$ according to inclusion, starting from $\partial\Lambda_{M/4}$.
For a plus-minus circuit $\gamma$, let $\tau(\gamma)$ be its rank in this ordering. Since plus-minus circuits are disjoint, this ordering is well-defined. Let $\mu_{\Lambda_M}^{\xi}(\gamma)$ be the probability that $\gamma$ is  the innermost plus-minus circuit and let $\mu_{\Lambda_M}^{\xi}(\gamma,\tau(\gamma)=k)$ be the probability that $\gamma$ is the $k$th circuit in $\Lambda_{M/2}\setminus\Lambda_{M/4}$ from $\partial\Lambda_{M/4}$. Then we get that \begin{align}
    \mu_{\Lambda_M}^{\xi}(\mcc X\cap\mcc Y)&\ge \sum_{\gamma}\mu_{\Lambda_M}^{\xi}(\gamma)\cdot \mu_{\Lambda_M\setminus\Gamma}^{\xi,+}(\mathcal{Y}\cap\{\gamma\leftrightarrow\mathcal{P}_0\});\label{eq: two rare event exploring the interface}\\
    \mu_{\Lambda_M}^{\xi,\eta_I}(\mcc X\cap\mcc Y)&=\sum_{k=1}^{\infty}\sum_{\gamma}\mu_{\Lambda_M}^{\xi,\eta_I}(\gamma,\tau(\gamma)=k)\cdot \mu_{\Lambda_M\setminus\Gamma}^{\xi,\eta_I,+}(\mathcal{Y}\cap\{\gamma\leftrightarrow\mathcal{P}_0\}).\label{eq: two rare event exploring the interface with number}
\end{align}
Here, the ``$+$" in the superscript denotes the fact that all the spins on $\partial_{\mathrm{ext}}\Gamma\supset \partial(\Lambda_{M}\setminus\Gamma)\setminus\partial\Lambda_M$ are plus given $\gamma$ is a plus-minus circuit.

To start with, we control $\mu_{\Lambda_M}^{\xi}(\gamma)$ and $\mu_{\Lambda_M}^{\xi,\eta_I}(\gamma,\tau(\gamma)=k)$.
\begin{lemma}\label{lem: interface exploration probability with number}
    Let $I\subset \Lambda_{M}$ be a subset of $\Lambda_{M}.$ There exists a constant $c_{14}>0$ such that 
for any boundary condition $\eta_I$ on $\mrr{Ball}(I)$ and any configuration $\xi\in\hat\Sigma_i$, we have
    \begin{align}
	\sum_{\gamma}\mu^{\xi}_{\Lambda_{M}}(\gamma)&\ge c_{14};\label{eq: interface exploration probability}\\
    \sum_{\gamma}\mu^{\xi,\eta_I}_{\Lambda_{M}\setminus{\mrr{Ball}(I)}}(\gamma,\tau(\gamma)=k)&\leq (1-c_{14}^{|I|})^{k+1}.\label{eq: interface exploration probability with number}
\end{align}
\end{lemma}

\begin{proof}
    The proof of \eqref{eq: interface exploration probability} is rather direct. Note that 
    \begin{align*}
    \sum_{\gamma}\mu^{\xi}_{\Lambda_{M}}(\gamma)   \geq\mu_{\Lambda_M}^\xi(\mbox{there exists a plus-minus circuit in }\Lambda_{M/2}\setminus\Lambda_{M/4})
    \end{align*}
since the summation is taken over all possible choices of $\gamma.$ Moreover, a sufficient condition for the existence of such $\gamma$ is that there exists a plus circuit in $\Lambda_{M/2}\setminus\Lambda_{3M/8}$ and a minus circuit in $\Lambda_{3M/8}\setminus\Lambda_{M/4}$, which occurs with positive probability by the RSW inequality of two dimensional critical Ising model.
    
Then we focus on  proving \eqref{eq: interface exploration probability with number} via induction in $k$.
For $k=1,$ notice that 
\begin{align*}
~&\sum_{\gamma}\mu^{\xi,\eta_I}_{\Lambda_{M}\setminus{\mrr{Ball}(I)}}(\gamma,\tau(\gamma)=1)\\
=~&\mu^{\xi,\eta_I}_{\Lambda_{M}\setminus{\mrr{Ball}(I)}}(\mbox{there exists a plus-minus circuit in }\Lambda_{M/2}\setminus\Lambda_{M/4})\\
\leq ~& 1-\mu^{\xi,\eta_I}_{\Lambda_{M}\setminus{\mrr{Ball}(I)}}(\partial\Lambda_{M/2}\leftrightarrow\partial\Lambda_{M/4})
\leq 1-c_{10}^{|I|},
\end{align*}
where the last inequality uses Lemma~\ref{lem:Ball-influence-on-crossing}.

For the induction step, we first define $\mu^{\xi,\eta_I}_{\Lambda_{M}\setminus{\mrr{Ball}(I)}}(\gamma',\tau(\gamma')=k,\gamma,\tau(\gamma)=k+1)$ to be the probability that $\gamma'$ is the $k$-th circuit and $\gamma$ is the $(k+1)$-th circuit. Note that 
\begin{align}
    \sum_{\gamma}\mu^{\xi,\eta_I}_{\Lambda_{M}\setminus{\mrr{Ball}(I)}}(\gamma,\tau(\gamma)=k+1)
=\sum_{\gamma',\gamma}\mu^{\xi,\eta_I}_{\Lambda_{M}\setminus{\mrr{Ball}(I)}}(\gamma',\tau(\gamma')=k,\gamma,\tau(\gamma)=k+1).\label{eq: interface exploration probability induction step}
\end{align} Let $\Gamma'$ denote the region enclosed by $\gamma'$ and let $\mathcal{A}(\Gamma')$ denote the event that there exists a plus-minus circuit in $\Lambda_{M/2}\setminus(\Gamma'\cup\mrr{Ball}(I))$, then we get from DMP that \begin{align}
    &\sum_{\gamma',\gamma}\mu^{\xi,\eta_I}_{\Lambda_{M}\setminus{\mrr{Ball}(I)}}(\gamma',\tau(\gamma')=k,\gamma,\tau(\gamma)=k+1)\nonumber\\=~&\sum_{\gamma'}\mu^{\xi,\eta_I}_{\Lambda_{M}\setminus{\mrr{Ball}(I)}}(\gamma',\tau(\gamma')=k)\mu^{\xi,\eta_I}_{\Lambda_{M}\setminus{(\mrr{Ball}(I)\cup\Gamma')}}\big(\mathcal{A}(\Gamma')\big)\nonumber\\
    \leq ~&\sum_{\gamma'}\mu^{\xi,\eta_I}_{\Lambda_{M}\setminus{\mrr{Ball}(I)}}(\gamma',\tau(\gamma')=k)\mu^{\xi,\eta_I}_{\Lambda_{M}\setminus{\mrr{Ball}(I)}}\big(\mathcal{A}(\Gamma')\big).\label{eq: interface exploration probability induction step 1}
\end{align} 
Let $\mathcal{B}$ denote the event that there exists a plus spin crossing from $\partial\Lambda_{M/2}$ to $\partial\Lambda_{M/4}$.
Note that $\mathcal{A}(\Gamma')\subset\mathcal{A}(\Lambda_{M/4})$ and by planar duality $\mathcal{A}(\Lambda_{M/4})\subset\mathcal{B}^c$. We get from Lemma~\ref{lem:Ball-influence-on-crossing} that \begin{equation}
    \mu^{\xi,\eta_I}_{\Lambda_{M}\setminus{\mrr{Ball}(I)}}\big(\mathcal{A}(\Gamma')\big)\le \mu^{\xi,\eta_I}_{\Lambda_{M}\setminus{\mrr{Ball}(I)}}\big(\mathcal{B}^c\big)\le 1-c_{10}^{|I|}.\nonumber
\end{equation}Combined with \eqref{eq: interface exploration probability induction step 1}, it yields that \begin{align*}
&\sum_{\gamma',\gamma}\mu^{\xi,\eta_I}_{\Lambda_{M}\setminus{\mrr{Ball}(I)}}(\gamma',\tau(\gamma')=k,\gamma,\tau(\gamma)=k+1)\nonumber\\\le~&\sum_{\gamma'}\mu^{\xi,\eta_I}_{\Lambda_{M}\setminus{\mrr{Ball}(I)}}(\gamma',\tau(\gamma')=k)(1-c_{10}^{|I|})
\leq (1-c_{10}^{|I|})^{k+1}.
\end{align*}
Plugging into \eqref{eq: interface exploration probability induction step} and 
choosing $c_{14}$ depending on $c_{10}$
completes the proof. 
\end{proof}

To compare the right-hand sides of \eqref{eq: two rare event exploring the interface} and \eqref{eq: two rare event exploring the interface with number}, we also need to compare $\mu_{\Lambda_M\setminus\Gamma_1}^{\xi,+}(\mathcal{Y}\cap\{\gamma_1\leftrightarrow\mathcal{P}_0\})$ with $\mu_{\Lambda_M\setminus\Gamma_2}^{\xi,+}(\mathcal{Y}\cap\{\gamma_2\leftrightarrow\mathcal{P}_0\})$ and $\mu_{\Lambda_M\setminus\Gamma}^{\xi,+}(\mathcal{Y}\cap\{\gamma\leftrightarrow\mathcal{P}_0\})$ with $\mu_{\Lambda_M\setminus\Gamma}^{\xi,\eta_I,+}(\mathcal{Y}\cap\{\gamma\leftrightarrow\mathcal{P}_0\})$. We state the corresponding inequalities in the following two lemmas.

\begin{lemma}\label{lem: two rare event probability with different interface explored}
    There exists a constant $c_{15}>0$ such that the following holds for 
  any configuration $\xi\in\hat\Sigma=\{-1,1\}^{(\lamn\setminus\Lambda_M)\cup\partial\Lambda_M}$
    \begin{align}
	\mu_{\Lambda_M\setminus\Gamma_1}^{\xi,+}(\mathcal{Y}\cap\{\gamma_1\leftrightarrow\mathcal{P}_0\})\le c_{15}\mu_{\Lambda_M\setminus\Gamma_2}^{\xi,+}(\mathcal{Y}\cap\{\gamma_2\leftrightarrow\mathcal{P}_0\}).\label{eq: two rare event probability with different interface explored}
\end{align}
\end{lemma}
\begin{lemma}\label{lem: two rare event probability with interface and ball-X=1}
    Let $I\subset \Lambda_{M}$ be a subset of $\Lambda_{M}.$ There exists a constant $c_{16}>0$ such that 
for any boundary condition $\eta_I$ on $\mrr{Ball}(I)$ and any configuration $\xi\in\hat\Sigma_i$,  we have
    \begin{align}
	\mu_{\Lambda_M\setminus\Gamma}^{\xi,\eta_I,+}(\mcc Y\cap\{\gamma\leftrightarrow\mcc P_0\})\le c_{16}^{|I|}\mu_{\Lambda_M\setminus\Gamma}^{\xi,+}(\mcc Y\cap\{\gamma\leftrightarrow \mcc P_0\}).\label{eq: two rare event probability with interface and ball-X=1}
\end{align}
\end{lemma}
We are now ready to prove \eqref{eq: two rare events upper bound with conditioning-X=1}.
\begin{proof}[Proof of \eqref{eq: two rare events upper bound with conditioning-X=1}] Fix a circuit $\gamma_0\subset\Lambda_{M/2}\setminus\Lambda_{M/4}$ and let $\Gamma_0$ be the region enclosed by $\gamma_0$.

Combining Lemmas~\ref{lem: two rare event probability with different interface explored} and \ref{lem: interface exploration probability with number} with \eqref{eq: two rare event exploring the interface}, we obtain that 
\begin{align}
    \mu_{\Lambda_M}^{\xi}(\mcc X\cap\mcc Y)&\ge \sum_{\gamma}\mu_{\Lambda_M}^{\xi}(\gamma)\cdot \mu_{\Lambda_M\setminus\Gamma}^{\xi,+}(\mathcal{Y}\cap\{\gamma\leftrightarrow\mathcal{P}_0\})~~~\text{(by \eqref{eq: two rare event exploring the interface})}\nonumber\\&\ge \sum_{\gamma}\mu_{\Lambda_M}^{\xi}(\gamma)\cdot \frac{1}{c_{15}}\mu_{\Lambda_M\setminus\Gamma_0}^{\xi,+}(\mathcal{Y}\cap\{\gamma_0\leftrightarrow\mathcal{P}_0\})~~~\text{(by Lemma~\ref{lem: two rare event probability with different interface explored})}\nonumber\\&\ge \frac{c_{14}}{c_{15}}\cdot \mu_{\Lambda_M\setminus\Gamma_0}^{\xi,+}(\mathcal{Y}\cap\{\gamma_0\leftrightarrow\mathcal{P}_0\})~~~\text{(by Lemma~\ref{lem: interface exploration probability with number})}.\label{eq: two rare event exploring the interface lower bound 1}
\end{align}
Combining Lemmas~\ref{lem: two rare event probability with interface and ball-X=1},~\ref{lem: two rare event probability with different interface explored} and \ref{lem: two rare event probability with interface and ball} with \eqref{eq: two rare event exploring the interface with number}, we get that \begin{align}
    \mu_{\Lambda_M}^{\xi,\eta_I}(\mcc X\cap\mcc Y)&\le \sum_{k=0}^{\infty} \sum_{\gamma}\mu_{\Lambda_M}^{\xi,\eta_I}(\gamma,\tau(\gamma)=k)\cdot \mu_{\Lambda_M\setminus\Gamma}^{\xi,\eta_I,+}(\mathcal{Y}\cap\{\gamma\leftrightarrow\mathcal{P}_0\})~~~\text{(by \eqref{eq: two rare event exploring the interface with number})}\nonumber\\&\le \sum_{k=0}^{\infty} \sum_{\gamma}\mu_{\Lambda_M}^{\xi,\eta_I}(\gamma,\tau(\gamma)=k)\cdot c_{16}^{|I|}\mu_{\Lambda_M\setminus\Gamma}^{\xi,+}(\mathcal{Y}\cap\{\gamma\leftrightarrow\mathcal{P}_0\})~~~\text{(by Lemma~\ref{lem: two rare event probability with interface and ball-X=1}))}\nonumber\\&\le\sum_{k=0}^{\infty} \sum_{\gamma}\mu_{\Lambda_M}^{\xi,\eta_I}(\gamma,\tau(\gamma)=k)\cdot c_{15} c_{16}^{|I|}\mu_{\Lambda_M\setminus\Gamma_0}^{\xi,+}(\mathcal{Y}\cap\{\gamma_0\leftrightarrow\mathcal{P}_0\})~~~\text{(by Lemma~\ref{lem: two rare event probability with different interface explored})}\nonumber\\&\le c_{15} c_{16}^{|I|}\cdot \sum_{k=0}^{\infty}(1-c_{14}^{|I|})^k\cdot \mu_{\Lambda_M\setminus\Gamma_0}^{\xi,+}(\mathcal{Y}\cap\{\gamma_0\leftrightarrow\mathcal{P}_0\})~~~\text{(by Lemma~\ref{lem: two rare event probability with interface and ball})}\nonumber\\&= c_{15} (\frac{c_{16}}{c_{14}})^{|I|}\cdot \mu_{\Lambda_M\setminus\Gamma_0}^{\xi,+}(\mathcal{Y}\cap\{\gamma_0\leftrightarrow\mathcal{P}_0\}).\label{eq: two rare event exploring the interface lower bound 2}
\end{align} Combining \eqref{eq: two rare event exploring the interface lower bound 1} and \eqref{eq: two rare event exploring the interface lower bound 2}, we finish the proof of \eqref{eq: two rare events upper bound with conditioning-X=1}.
\end{proof}
\begin{proof}[Proof of Lemma~\ref{lem: two rare event probability with different interface explored}]
Recall Definition~\ref{def:explore-plus-spin-clusters} and recall that $\mathcal{Y}=\mathcal{Y}^-\cap \mathcal{Y}^+.$ Choosing 
$\mcc I=\mcc Y^+\cap\{\gamma_2\leftrightarrow\mcc P_0\}$ in Lemma~\ref{lem:explore-plus-spin-clusters}, we have that
\begin{equation}        \mu^{\xi,+}_{\Lambda_{M}\setminus\Gamma_2}(\mathcal{C}_1\cap \Lambda_{0.95M}=\emptyset\mid \mathcal{Y}\cap\{\gamma_2\leftrightarrow\mathcal{P}_0\})\ge C_1.\label{eq: P_1 cluster trapped lower bound with single condition}
    \end{equation}  
Next, we want to explore the cluster $\mathcal{C}_1$ in $\Lambda_M$. Note that by exploring $\mathcal{C}_1$, we already see whether the event $\mathcal{Y}^-\cap \{\mathcal{C}_1\cap\Lambda_{0.95M}=\emptyset\}$ occurs. Thus we define $\mathfrak{C}_1$ to be the collection of $A$ such that  $\{\mathcal{C}_1=A\}$ is possible and $\{\mathcal{C}_1=A\}$ implies   $\mathcal{Y}^-\cap \{\mathcal{C}_1\cap \Lambda_{0.95M}=\emptyset\}$. Then we have 
\begin{align}  &~\mu^{\xi,+}_{\Lambda_{M}\setminus\Gamma_1}( \mathcal{Y}\cap\{\gamma_1\leftrightarrow\mathcal{P}_0\}\cap\{\mcc C_1\cap\Lambda_{0.95M}=\emptyset\})\nonumber\\=&\sum_{A\in\mathfrak{C}_1} \mu^{\xi,+}_{\Lambda_{M}\setminus\Gamma_1}(\mathcal{C}_1=A)\times\mu^{\xi,+,-}_{\Lambda_{M}\setminus(\Gamma_1\cup A)}(\mathcal{Y}^+\cap\{\gamma_1\leftrightarrow\mathcal{P}_0\}),\label{eq: cluster 1 exploration in two rare event lower bound 1}\\
&~\mu^{\xi,+}_{\Lambda_{M}\setminus\Gamma_2}( \mathcal{Y}\cap\{\gamma_2\leftrightarrow\mathcal{P}_0\}\cap\{\mcc C_1\cap\Lambda_{0.95M}=\emptyset\})\nonumber\\=&\sum_{A\in\mathfrak{C}_1} \mu^{\xi,+}_{\Lambda_{M}\setminus\Gamma_2}(\mathcal{C}_1=A)\times\mu^{\xi,+,-}_{\Lambda_{M}\setminus(\Gamma_2\cup A)}(\mathcal{Y}^+\cap\{\gamma_2\leftrightarrow\mathcal{P}_0\}).\label{eq: cluster 1 exploration in two rare event lower bound 2}
\end{align}
Here, for any region $\Gamma$, we use $\mu^{\xi,+,-}_{\Lambda_{M}\setminus(\Gamma\cup A)}$ to denote the Gibbs measure on $\Lambda_{M}\setminus(\Gamma\cup A)$ with plus boundary condition on $\partial_{\mathrm{ext}}\Gamma\cap{\partial(\Lambda_M\setminus(\Gamma\cup A))}$ and minus boundary condition on $\partial_{\mathrm{ext}} A\cap{\partial(\Lambda_M\setminus(\Gamma\cup A))}$.

Note that $\Gamma_1\subset\Lambda_{M/2}$ and $\Gamma_2\subset\Lambda_{M/2}$.
By the mixing property Lemma~\ref{lem: spatial mixing of Ising model}, we get that for any $A\subset\Lambda_M\setminus\Lambda_{0.95M}$ that \begin{equation}
    \mu^{\xi,+}_{\Lambda_{M}\setminus\Gamma_1}(\mathcal{C}_1=A)\ge C_2\mu^{\xi,+}_{\Lambda_{M}\setminus\Gamma_2}(\mathcal{C}_1=A).\label{eq: mixing property in two rare events given different region}
\end{equation}
Let $\mathcal{A}$ denote the event that there exists a plus circuit in $\Lambda_{0.8M}\setminus\Lambda_{3M/4}$ which is connected to $\gamma_1$. Then by RSW, we get that \begin{equation}
    \mu^{\xi,+,-}_{\Lambda_{M}\setminus(\Gamma_1\cup A)}(\mathcal{A})\ge C_3.\label{eq: plus circuit existence in interface comparison}
\end{equation}
Moreover, we decompose the event $\mcc A$ into the disjoint union of $\mcc A_{\hat\gamma}, \hat\gamma\in\mss S,$ where $\mcc A_{\hat\gamma}$ is the event that $\hat\gamma$ is the innermost plus circuit connected to $\gamma_1$ in $\Lambda_{0.8M}\setminus\Lambda_{3/4M}$  and $\mss S$ is the set of all possible $\hat\gamma.$
Noting that $\gamma_2\subset\Lambda_{M/2}\subset\hat\Gamma$, where $\hat\Gamma$ is the interior of $\hat\gamma,$ it follows from DMP and CBC that 
\begin{align}
& \mu^{\xi,+,-}_{\Lambda_{M}\setminus(\Gamma_1\cup A)}(\mathcal{Y}^+\cap\{\gamma_1\leftrightarrow\mathcal{P}_0\}\mid\mathcal{A}_{\hat\gamma}) \stackrel{\mathrm{DMP}}{=}\mu^{\xi,+,-}_{\Lambda_{M}\setminus(\hat\Gamma\cup A)}(\mathcal{Y}^+\cap\{\hat\gamma\leftrightarrow\mathcal{P}_0\})\nonumber\\
    \stackrel{\mathrm{CBC}}{\geq}
    &\mu^{\xi,+,-}_{\Lambda_{M}\setminus(\Gamma_2\cup A)}(\mathcal{Y}^+\cap\{\hat\gamma\leftrightarrow\mathcal{P}_0\})
    \geq 
    \mu^{\xi,+,-}_{\Lambda_{M}\setminus(\Gamma_2\cup A)}(\mathcal{Y}^+\cap\{\gamma_2\leftrightarrow\mathcal{P}_0\}).
    \label{eq: plus boundary condition connectivity comparison inequality}
\end{align}

Combining  \eqref{eq: plus circuit existence in interface comparison} with \eqref{eq: plus boundary condition connectivity comparison inequality}, we get that \begin{align}
    &\mu^{\xi,+,-}_{\Lambda_{M}\setminus(\Gamma_1\cup A)}(\mathcal{Y}^+\cap\{\gamma_1\leftrightarrow\mathcal{P}_0\})\nonumber\\
    \ge ~&\sum_{\hat\gamma\in\mss S}\mu^{\xi,+,-}_{\Lambda_{M}\setminus(\Gamma_1\cup A)}(\mathcal{Y}^+\cap\{\gamma_1\leftrightarrow\mathcal{P}_0\}\mid\mathcal{A}_{\hat\gamma})\cdot \mu^{\xi,+,-}_{\Lambda_{M}\setminus(\Gamma_1\cup A)}(\mathcal{A}_{\hat\gamma})\nonumber\\\ge~ & \mu^{\xi,+,-}_{\Lambda_{M}\setminus(\Gamma_2\cup A)}(\mathcal{Y}^+\cap\{\gamma_2\leftrightarrow\mathcal{P}_0\})\sum_{\hat\gamma\in\mss S}\mu^{\xi,+,-}_{\Lambda_{M}\setminus(\Gamma_1\cup A)}(\mathcal{A}_{\hat\gamma})\nonumber\\ =~& \mu^{\xi,+,-}_{\Lambda_{M}\setminus(\Gamma_2\cup A)}(\mathcal{Y}^+\cap\{\gamma_2\leftrightarrow\mathcal{P}_0\})\mu^{\xi,+,-}_{\Lambda_{M}\setminus(\Gamma_1\cup A)}(\mathcal{A})\nonumber\\
    \geq ~& C_3\mu^{\xi,+,-}_{\Lambda_{M}\setminus(\Gamma_2\cup A)}(\mathcal{Y}^+\cap\{\gamma_2\leftrightarrow\mathcal{P}_0\}).\label{eq: plus boundary condition connectivity comparison inequality 2}
\end{align}Plugging \eqref{eq: mixing property in two rare events given different region} and \eqref{eq: plus boundary condition connectivity comparison inequality 2} into \eqref{eq: cluster 1 exploration in two rare event lower bound 1} and \eqref{eq: cluster 1 exploration in two rare event lower bound 2}, we get that \begin{align}
    &\mu^{\xi,+}_{\Lambda_{M}\setminus\Gamma_1}( \mathcal{Y}\cap\{\gamma_1\leftrightarrow\mathcal{P}_0\}\cap\{\mcc C_1\cap\Lambda_{0.95M}=\emptyset\})\nonumber\\ \ge~& C_2\sum_{A\in\mathfrak{C}_1} \mu^{\xi,+}_{\Lambda_{M}\setminus\Gamma_2}(\mathcal{C}_1=A)\times C_3 \mu^{\xi,+,-}_{\Lambda_{M}\setminus(\Gamma_2\cup A)}(\mathcal{Y}^+\cap\{\gamma_1\leftrightarrow\mathcal{P}_0\})\nonumber \\=~& C_2C_3
\mu^{\xi,+}_{\Lambda_{M}\setminus\Gamma_2}( \mathcal{Y}\cap\{\gamma_2\leftrightarrow\mathcal{P}_0\}\cap\{\mcc C_1\cap\Lambda_{0.95M}=\emptyset\}).\label{eq: two rare event different region comparison conditioned on trapped}
\end{align}
Combining \eqref{eq: two rare event different region comparison conditioned on trapped} with \eqref{eq: P_1 cluster trapped lower bound with single condition}, we get that \begin{align}
    &\mu^{\xi,+}_{\Lambda_{M}\setminus\Gamma_1}(\mathcal{Y}\cap\{\gamma_1\leftrightarrow\mathcal{P}_0\})\ge \mu^{\xi,+}_{\Lambda_{M}\setminus\Gamma_1}( \mathcal{Y}\cap\{\gamma_1\leftrightarrow\mathcal{P}_0\}\cap\{\mathcal{P}_1\not\leftrightarrow \Lambda_{0.95M}\})\nonumber\\ \ge& C_2C_3\mu^{\xi,+}_{\Lambda_{M}\setminus\Gamma_2}( \mathcal{Y}\cap\{\gamma_2\leftrightarrow\mathcal{P}_0\}\cap\{\mathcal{P}_1\not\leftrightarrow \Lambda_{0.95M}\})\ge C_1C_2C_3 
    \mu^{\xi,+}_{\Lambda_{M}\setminus\Gamma_2}( \mathcal{Y}\cap\{\gamma_2\leftrightarrow\mathcal{P}_0\}).
\end{align}Thus we finish the proof of Lemma~\ref{lem: two rare event probability with different interface explored}.
\end{proof}

\begin{proof}[Proof of Lemma \ref{lem: two rare event probability with interface and ball-X=1}]
    The proof is very similar to that of Lemma~\ref{lem: two rare event probability with interface and ball}. So we just provide a sketch and point out the differences here.
    With the notations in Definition~\ref{def:decompose-Ball} and the proof of Lemma~\ref{lem: two rare event probability with interface and ball}, it suffices to show that $$\mu_{\Lambda\setminus\Gamma}^{\xi,\eta_I^p,+}(\mcc Y\cap\{\gamma\leftrightarrow\mcc P_0\})\leq c_9\mu_{\Lambda\setminus\Gamma}^{\xi,\eta_I^{p-1},+}(\mcc Y\cap\{\gamma\leftrightarrow\mcc P_0\}).$$
    As in the proof of Lemma~\ref{lem: two rare event probability with interface and ball}, we want to explore the configuration $\hat\xi$ outside $B^{10}$. Recall the definitions of  the push forward sets and events $\hat{\mcc P_0},\hat{\mcc P_1},\hat{\mcc U}$ and $\hat{\mcc Y}$ from Lemma~\ref{lem: two rare event probability with interface and ball}. We further define the push forward of $\{\gamma\leftrightarrow\mcc P_0\},$ to be
    $$\hat{\mcc T}=\{\gamma\leftrightarrow\hat{\mcc P_0}\mbox{ in }B^{10}\}\cup\{\gamma\leftrightarrow\hat{\mcc P_0}\mbox{ in }\hat\xi\}.$$
    As in the proof of Lemma~\ref{lem: two rare event probability with interface and ball}, we only have to show that for some universal constant
    $C>0$, the following inequalities hold:
\begin{align}
&\mu_{\Lambda_M\setminus\Gamma}^{\xi,+}(\hat\xi)\geq C\mu_{\Lambda_M\setminus\Gamma}^{\xi,\eta^B,+}(\hat\xi);\label{proof of 538-0}\\
   & \mu_{B^{10}\setminus\Gamma}^{\hat\xi,\eta^B,+}(\mcc A_1'\mid\hat{\mcc Y}\cap\hat{\mcc T})\geq C;\label{proof of 538-1}\\
   &\mu_{ B^{10}\setminus\Gamma}^{\hat\xi,\eta^B,+}(\mcc A_2'\mid \mcc A_1'\cap \hat{\mcc Y}\cap\hat{\mcc T})\geq C;\label{proof of 538-2}\\
   &\mu_{ B^{10}\setminus\Gamma}^{\hat\xi,+}(\mcc A_2'\cap \hat{\mcc Y}\cap\hat{\mcc T})\geq C\mu_{ B^{10}\setminus\Gamma}^{\hat\xi,\eta^B,+}(\mcc A_2'\cap \hat{\mcc Y}\cap \hat{\mcc T})\label{proof of 538-3},
\end{align}
where $\mcc A_1',\mcc A_2'$ are the events $\mcc A_1,\mcc A_2$ (recall the definitions from the proof of Lemma~\ref{lem: two rare event probability with interface and ball}) restricted in $B^{10}\setminus\Gamma,$ i.e.,
\begin{align*}
  \mcc A_1':=\{&\mbox{the connected component of }\hat{\mcc P}_1 \mbox{ does not intersect with } B^9 \mbox{ in }B^{10}\setminus\Gamma\},\\
  \mcc A_2':=\{&\mbox{there exists } S\subset B^9\setminus (B^8\cup\Gamma)\mbox{ such that }\sigma_v=1, \forall v\in S\\
  &\mbox{ and any path from }\partial B^9\mbox{ to }B^8 \mbox{ must intersect with } S \}.
\end{align*}

The reasons why \eqref{proof of 538-0} to \eqref{proof of 538-2} hold are the same as those for \eqref{proof of 177-0} to \eqref{proof of 177-2}. Next, we show \eqref{proof of 538-3}: similar to \eqref{proof of 177-3}, it still comes from the mixing property of Lemma~\ref{lem: spatial mixing of Ising model} and the fact that whether $\mcc A_2'\cap \hat{\mcc Y}\cap \hat{\mcc T}$ holds only relies on the configuration outside $B^8$. The only difference is that we now have a new event $\hat{\mcc T},$ but note that $\gamma$ is a circuit in $\Lambda_{M/2}\setminus\Lambda_{M/4},$ therefore either $\gamma$ lies outside $(B^8)^c$, or it intersects the set $S$ from $\mcc A_2'.$ In both cases, the connecting event does not depend on the configurations inside $B^8.$
\end{proof}

\appendix

\section{Mixing of critical Ising model}
In this section, we prove the following spatial mixing result for the critical planar Ising model that was first introduced in \cite[Cor 5.2]{armExponentsIsing}. 

\begin{lemma}\label{lem: spatial mixing of Ising model}
    Let $K<N$ be two integers.  Then there exists an absolute constant $c>0$ such that for any event $\mathcal{E}$ supporting on $\Lambda_K$ and any $\xi_1,\xi_2$ be two boundary conditions on $\partial\Lambda_N$, we have \begin{equation}
        \Big|\frac{\mu^{\xi_1}_{\Lambda_N}(\mathcal{E})}{\mu^{\xi_2}_{\Lambda_N}(\mathcal{E})}-1\Big|\le c(\frac{N}{K})^{-\frac{1}{8}}.\label{eq: spatial mixing of Ising model}
    \end{equation}
\end{lemma}
Applying Lemma \ref{lem: spatial mixing of Ising model} twice gives the following Corollary.
\begin{corollary}\label{cor: spatial mixing of Ising model}
    Let $M_1<M_3<M_4<M_2$ be four integers. Let $\mathcal{E}$ be an event supported on $\Lambda_{M_3,M_4}$. Then there exists an absolute constant $c>0$ such that for any boundary conditions $\xi,\xi'$ on $\partial\Lambda_N$ and any boundary condition $\nu,\nu'$ on $\partial\Lambda_M$, we have\begin{equation}
        \Big|\frac{\mu^{\nu,\xi}_{\Lambda_{M_1,M_2}}(\mathcal{E})}{\mu^{\nu',\xi'}_{\Lambda_{M_1,M_2}}(\mathcal{E})}-1\Big|\le c(\frac{M_3}{M_1})^{-\frac{1}{8}}+c(\frac{M_2}{M_4})^{-\frac{1}{8}}.\label{eq: spatial mixing of Ising model in annulus}
    \end{equation}
\end{corollary}
In order to prove Lemma~\ref{lem: spatial mixing of Ising model}, we first introduce the disagreement Ising model. Recall the definition of the extended Ising model in Section~\ref{sec: extended ising}.

For a finite graph $G$, let $\bar{\sigma}^+$ and $\bar{\sigma}^-$ be two extended Ising configurations on $\bar G$. The \textbf{pre-disagreement set} of $(\bar \sigma^+, \bar \sigma^-)$ is defined as
\begin{equation*}
\mathcal{D}:=\mathcal{D}(\bar{\sigma}^+,\bar{\sigma}^-) = \{u \in \bar{G} : \bar{\sigma}_u^+>\bar{\sigma}_u^- \}.
\end{equation*} For $u\in \bar G$, we call it a pre-disagreement  if it is in the pre-disagreement set. We also call $u$ an agreement if $\bar{\sigma}_u^+=\bar{\sigma}_u^-.$ 

\begin{proof}[Proof of Lemma~\ref{lem: spatial mixing of Ising model}]
	For the proof, we use the disagreement Ising model. We first prove that for any configuration $\eta$ on $\Lambda_K$, we have \begin{equation}
	    \frac{\mu_{\Lambda_N}^{\xi_1}(\sigma|_{\Lambda_K}=\eta)}{\mu_{\Lambda_N}^{\xi_2}(\sigma|_{\Lambda_K}=\eta)}\le 1+C (\frac{N}{K})^{-\frac{1}{8}}.\label{eq: spatial mixing step 0}
	\end{equation}
	Now we define the joint partition function for $\bar\mu_{\Lambda_N\setminus\Lambda_K}^{\xi_1,\eta}$ and $\bar\mu_{\Lambda_N}^{\xi_2}$ as
	$$\bar{\mathcal Z}(\xi_1,\eta;\xi_2)=\bar{\mathcal Z}_{\Lambda_N\setminus\Lambda_K}^{\xi_1,\eta}\cdot \bar{\mathcal Z}^{\xi_2}_{\Lambda_N}=\sum_{\bar\sigma\in\Xi(\Lambda_N\setminus\Lambda_K)}\sum_{\bar\sigma'\in\Xi(\Lambda_N)}\exp\left[-\beta_c\left(\bar H^{\xi_1,\eta}(\bar\sigma)+H^{\xi_2}(\bar\sigma')\right)\right],$$ where $\Xi(\cdot)$ denotes the configuration space for the extended Ising model;
	similarly we could define $\bar{\mathcal Z}(\xi_2,\eta;\xi_1).$ We also define the joint partition function for $\bar\mu_{\Lambda_N}^{\xi_1}$ and $\bar\mu_{\Lambda_N}^{\xi_2}$ as $$\bar{\mathcal Z}(\xi_1;\xi_2)=\bar{\mathcal Z}_{\Lambda_N}^{\xi_1}\cdot \bar{\mathcal Z}^{\xi_2}_{\Lambda_N}=\sum_{\bar\sigma\in\Xi(\Lambda_N)}\sum_{\bar\sigma'\in\Xi(\Lambda_N)}\exp\left[-\beta_c\left(\bar H^{\xi_1}(\bar\sigma)+H^{\xi_2}(\bar\sigma')\right)\right],$$

	Note that 
	$$\mu_{\Lambda_N}^{\xi_1}(\sigma|_{\Lambda_K}=\eta)=\frac{\bar{\mathcal Z}(\xi_1,\eta;\xi_2)}{\bar{\mathcal Z}(\xi_1;\xi_2)},\quad \mu_{\Lambda_N}^{\xi_2}(\sigma|_{\Lambda_K}=\eta)=\frac{\bar{\mathcal Z}(\xi_2,\eta;\xi_1)}{\bar{\mathcal Z}(\xi_2;\xi_1)}.$$
	In addition, we have by symmetry
	$$\bar{\mathcal Z}(\xi_1;\xi_2)=\bar{\mathcal Z}(\xi_2;\xi_1).$$
	
	As a consequence, 
	$$\frac{\mu_{\Lambda_N}^{\xi_1}(\sigma|_{\Lambda_K}=\eta)}{\mu_{\Lambda_N}^{\xi_2}(\sigma|_{\Lambda_K}=\eta)}=\frac{\bar{\mathcal Z}(\xi_1,\eta;\xi_2)}{\bar{\mathcal Z}(\xi_2,\eta;\xi_1)}.$$
	
	Now, let 
	\begin{align*}
		\mathcal A:=\{(\bar\sigma,\bar\sigma')\in\Xi(\Lambda_N\setminus\Lambda_K)\otimes\Xi(\Lambda_N):&\mbox{ there exists an agreement *-circuit}\\  & \mbox{ separating }\partial\Lambda_N \mbox{ and }\Lambda_K \}.
	\end{align*}
	Then, by planar duality, we get that
		\begin{align*}
		\mathcal A^c=\{(\bar\sigma,\bar\sigma')\in\bar\Sigma(\Lambda_N\setminus\Lambda_K)\otimes\bar\Sigma(\Lambda_N):&\mbox{ there exists an pre-disagreement path }\\  & \mbox{ connecting }\partial\Lambda_N \mbox{ and }\Lambda_K \}.
	\end{align*}
	Now we define
	$$\bar{\mathcal Z}(\xi_1,\eta;\xi_2)[\mathcal A]:=\sum_{(\bar\sigma,\bar\sigma')\in\mathcal A}\exp\left[-\beta_c\left(\bar H^{\xi_1,\eta}(\bar\sigma)+H^{\xi_2}(\bar\sigma')\right)\right],$$
	and similarly $\bar{\mathcal Z}(\xi_2,\eta;\xi_1)[\mathcal A].$
	By the swapping identity of \cite[Lemma 3.2]{AHP20}, we have 
	$$\bar{\mathcal Z}(\xi_1,\eta;\xi_2)[\mathcal A]=\bar{\mathcal Z}(\xi_2,\eta;\xi_1)[\mathcal A].$$
	Moreover, we get from CBC that 
	\begin{align*}
		&\left|\frac{\bar{\mathcal Z}(\xi_1,\eta;\xi_2)[\mathcal A]}{\bar{\mathcal Z}(\xi_1,\eta;\xi_2)}-1\right|=\left|\bar{\mu}^{\xi_1,\eta}_{{\Lambda_N\setminus\Lambda_K}}\otimes \bar{\mu}^{\xi_2}_{{\Lambda_N}}(\mathcal A)-1\right|=\bar{\mu}^{\xi_1,\eta}_{{\Lambda_N\setminus\Lambda_K}}\otimes \bar{\mu}^{\xi_2}_{{\Lambda_N}}(\mathcal A^c)\le \bar{\mu}^{+/-,+/-}_{{\Lambda_N\setminus\Lambda_K}}(\mathcal A^c).
	\end{align*}Combined with \cite[(4.2)]{DHX23} and RSW, we get that \begin{align}
	    \left|\frac{\bar{\mathcal Z}(\xi_1,\eta;\xi_2)[\mathcal A]}{\bar{\mathcal Z}(\xi_1,\eta;\xi_2)}-1\right|\le\bar{\mu}^{+/-,+/-}_{{\Lambda_N\setminus\Lambda_K}}(\mathcal A^c)= 1-\left(\frac{1-\phi_{\Lambda_N\setminus\Lambda_K}^{w,w}(\partial\Lambda_N\leftrightarrow\partial\Lambda_K)}{1+\phi_{\Lambda_N\setminus\Lambda_K}^{w,w}(\partial\Lambda_N\leftrightarrow\partial\Lambda_K)}\right)^2
		\leq C\left(\frac{K}{N}\right)^{\frac{1}{8}}.\label{eq: spatial mixing step 1}
	\end{align}Applying \eqref{eq: spatial mixing step 1} twice, we get that \begin{align}
	    \frac{\bar{\mathcal Z}(\xi_1,\eta;\xi_2)}{\bar{\mathcal Z}(\xi_2,\eta;\xi_1)}=\frac{\bar{\mathcal Z}(\xi_1,\eta;\xi_2)}{\bar{\mathcal Z}(\xi_1,\eta;\xi_2)[\mathcal{A}]}\cdot \frac{\bar{\mathcal Z}(\xi_2,\eta;\xi_1)[\mathcal{A}]}{\bar{\mathcal Z}(\xi_2,\eta;\xi_1)}\le \Big(1+C\left(\frac{K}{N}\right)^{\frac{1}{8}}\Big)^2.\nonumber
	\end{align}Thus we finish the proof of \eqref{eq: spatial mixing step 0}. To finish the proof of \eqref{eq: spatial mixing of Ising model}, we get from \eqref{eq: spatial mixing step 0} again that \begin{equation}
	     \frac{\mu_{\Lambda_N}^{\xi_1}(\sigma|_{\Lambda_K}=\eta)}{\mu_{\Lambda_N}^{\xi_2}(\sigma|_{\Lambda_K}=\eta)}\ge \frac{1}{1+C (\frac{N}{K})^{-\frac{1}{8}}}\ge 1-C (\frac{N}{K})^{-\frac{1}{8}}.\label{eq: spatial mixing step 2}
	\end{equation}Combining \eqref{eq: spatial mixing step 0} and \eqref{eq: spatial mixing step 2} gives the desired result.
    
\end{proof}

\section{Estimates on the probability of two-arms events at criticality}

In this Appendix, we collect a few lemmas giving upper bounds on the probability of two-arm events in the critical Ising model. These lemmas are useful to control the probability that the critical interface gets close to a point, or to multiple points, and are instrumental in the proof of Theorem \ref{theorem-single-interface}.

\begin{lemma}\label{lemma-one-pt-estimate-interface-discrete}
    Let $0 < \rho \ll 1$. There exist $\delta_0 > 0$ and an absolute constant $C>0$ such that for any $0 < \delta < \delta_0$, any $x \in \delta \mathbb{Z}^2$ and any $r > 0$,
    \begin{equation*}
        \mu_{\delta, \Lambda_{\delta}}^{\pm}[\dist(x,\gamma_{\delta}) \leq r] \leq C \bigg( \frac{r}{\dist(x,\partial \Lambda_{\delta})}\bigg)^{\frac{5}{8}-\rho}
    \end{equation*}
    where $\Lambda_{\delta} \subset \delta\mathbb{Z}^2$ is any square centered at $x$ and where under $\mu_{\delta, \Lambda_{\delta}}^{\pm}$, the Ising model in $\Lambda_{\delta}$ has $-1$ boundary conditions on the bottom side of $\Lambda_{\delta}$ and $+1$ boundary conditions everywhere else.
\end{lemma}

\begin{proof}
     Let $0 < \rho \ll 1$, $r > 0$ and $x \in \delta \mathbb{Z}^2$. Let $\Lambda_{\delta} \subset \mathbb{Z}^2$ be a square centered at $x$. For $R > 0$, let us set $\Lambda_{R,\delta}(x) := \Lambda_{R}(x) \cap \delta \mathbb{Z}^2$, where $\Lambda_{R}(x)$ is the Euclidean square of side-length $2R$ centered at $x$. For $0 < R_1 < R_2$, we say that there is an arm of sign $+$, respectively $-$, in the annulus $\Lambda_{R_{2},\delta}(x) \setminus \Lambda_{R_{1},\delta}(x)$ if there is a nearest neighbor path joining $\partial \Lambda_{R_{1},\delta}(x)$ to $\partial \Lambda_{R_{2},\delta}(x)$ along which all vertices have sign $+1$, respectively $-1$. For $0 < R_1 < R_2$ and $x \in \delta \mathbb{Z}^2$, we denote by $A_2(x;R_1,R_2)$ the event that there exist two arms of opposite sign crossing $\Lambda_{R_{2},\delta}(x) \setminus \Lambda_{R_{1},\delta}(x)$. It is then easy to see that
     \begin{equation*}
         \mu_{\delta, \Lambda_{\delta}}^{\pm}[\dist(x,\gamma_{\delta}) \leq r] \leq \mu_{\delta, \Lambda_{\delta}}^{\pm}[A_2(x;r,\dist(x,\partial \Lambda_{\delta})].
     \end{equation*}
     By \cite[Theorem~1.2]{armExponentsIsing}, there exist $\delta_{0} > 0$ and an absolute constant $C>0$ such that
     \begin{equation*}
         \mu_{\delta, \Lambda_{\delta}}^{\pm}[A_2(x;r,\dist(x,\partial \Lambda_{\delta})] \leq C\bigg( \frac{r}{\dist(x,\partial \Lambda_{\delta})}\bigg)^{\frac{5}{8}-\rho},
     \end{equation*}
     which concludes the proof of the lemma.
\end{proof}

Moreover, using mixing results for the critical Ising model, one can use the one-point estimate of Lemma \ref{lemma-one-pt-estimate-interface-discrete} to obtain an upper bound on the probability that the interface in the critical Ising model gets close to $k$ distinct points, as established in the following lemma.

\begin{lemma}\label{lemma-k-pts-estimate-interface-discrete}
    Let $0 < \rho \ll 1$. There exist $\delta_0 > 0$ and an absolute constant $C>0$ such that for any $0 < \delta < \delta_0$, any $(x_1,\dots,x_k) \in D_{\delta}$ and any $(r_1,\dots,r_k) \in (0,\infty)$,
    \begin{equation*}
        \mu_{\delta}^{\pm}\bigg[ \bigcap_{j=1}^{k} \{ \dist(x_j, \gamma_{\delta}) \leq r_j \}\bigg] \leq C^k \prod_{j=1}^{k} \bigg(\frac{r_j}{L(x_j) \wedge d_{\delta}(x_j)}\bigg)^{\frac{5}{8}-\rho}
    \end{equation*}
    where we have set $d_{\delta}(x):= \dist(x,\partial D_{\delta})$ and $L(x_j):= \frac{1}{2} \min_{p: j \neq p} \vert x_j - x_p \vert$.
\end{lemma}

\begin{proof}
    Let $0 < \rho \ll 1$, $k \in \mathbb{N}$, $(x_1,\dots,x_k) \in D_{\delta}^k$ and $(r_1,\dots,r_k) \in (0,\infty)^k$. Using the same notations as in the proof of Lemma \ref{lemma-one-pt-estimate-interface-discrete}, we have that
    \begin{equation} \label{ineq-visit-interface-two-arm-event}
        \mu_{\delta}^{\pm}\bigg[ \bigcap_{j=1}^{k} \{ \dist(x_j, \gamma_{\delta}) \leq r_j \}\bigg] \leq \mu_{\delta}^{\pm}\bigg[ \bigcap_{j=1}^{k} A_{2}(x_j;r_j,d_{\delta}(x_j)) \bigg].
    \end{equation}
    We now would like to use the Markov property of the Ising model to decouple the events 
    \begin{equation*}
        \{A_{2}(x_j;r_j,d_{\delta}(x_j))\}_{1 \leq j \leq k}.
    \end{equation*} 
    However, the outer squares $\Lambda_{d_{\delta}(x_j),\delta}(x)$ may overlap each other. We therefore consider smaller outer squares to obtain from \eqref{ineq-visit-interface-two-arm-event} that
    \begin{equation*}
        \mu_{\delta}^{\pm}\bigg[ \bigcap_{j=1}^{k} \{ \dist(x_j, \gamma_{\delta}) \leq r_j \}\bigg] \leq \mu_{\delta}^{\pm}\bigg[ \bigcap_{j=1}^{k} A_{2}(x_j;r_j,d_{\delta}(x_j) \wedge L(x_j)) \bigg],
    \end{equation*}
    where $L(x_j)$ is defined as in the statement of the lemma. Using now the Markov property of the Ising model in the squares $(\Lambda_{\tilde L(x_j),\delta}(x))_j$, we obtain that
    \begin{equation*}
        \mu_{\delta}^{\pm}\bigg[ \bigcap_{j=1}^{k} \{ \dist(x_j, \gamma_{\delta}) \leq r_j \}\bigg] \leq \mu_{\delta}^{\pm} \bigg[ \mu_{\Lambda_{\tilde L(x_j),\delta}(x)}^{\sigma_{\partial \Lambda_{\tilde L(x_j),\delta}(x)}}\big[ A_{2}(x_j;r_j,\frac{1}{2}d_{\delta}(x_j) \wedge \frac{1}{4}L(x_j)) \big] \bigg]
    \end{equation*}
    where we have set $\tilde L(x_j) = d_{\delta}(x_j) \wedge L(x_j)$. \cite[Corollary~5.2]{armExponentsIsing} guarantees that when $\beta =\beta_c$, changing the boundary conditions does not change by more than a multiplicative constant the probability of events depending on edges that are at macroscopic distance from the boundary. This implies that there exists a constant $c>0$ such that $\PP_{\delta}^{\pm}$-almost surely, for any $x_j$,
    \begin{equation} \label{mixing-2-arm-event}
        \mu_{\Lambda_{\tilde L(x_j),\delta}(x)}^{\sigma_{\partial \Lambda_{\tilde L(x_j),\delta}(x)}}\big[ A_{2}(x_j;r_j,\frac{1}{2}d_{\delta}(x_j) \wedge \frac{1}{4}L(x_j)) \big] \leq c\mu_{\Lambda_{\tilde L(x_j),\delta}(x)}^{\pm}\big[ A_{2}(x_j;r_j,\frac{1}{2}d_{\delta}(x_j) \wedge \frac{1}{4}L(x_j)) \big]
    \end{equation}
    where under $\PP_{\Lambda_{\tilde L(x_j),\delta}(x)}^{\pm}$, the Ising model has boundary conditions $-1$ on the bottom side of $\partial \Lambda_{\tilde L(x_j),\delta}(x)$ and $+1$ boundary conditions elsewhere. Lemma \ref{lemma-one-pt-estimate-interface-discrete} provides an upper bound on the right-hand side of \eqref{mixing-2-arm-event}, which combined with \eqref{ineq-visit-interface-two-arm-event}, yields that there exist $\delta_0>0$ (independent of $k$ and $(x_1,\dots,x_k)$) and an absolute constant $C>0$ such that
    \begin{align*}
        \mu_{\delta}^{\pm}\bigg[ \bigcap_{j=1}^{k} \{ \dist(x_j, \gamma_{\delta})\leq r_j \}\bigg] &\leq c^k \prod_{j=1}^{k} \mu_{\Lambda_{\tilde L(x_j),\delta}(x)}^{\pm}\big[ A_{2}(x_j;r_j,\frac{1}{2}d_{\delta}(x_j) \wedge \frac{1}{4}L(x_j)) \big] \\
        & \leq C^k c^k \bigg( \prod_{j=1}^{k} \frac{r_j}{\frac{1}{2}d_{\delta}(x_j) \wedge \frac{1}{4}L(x_j))} \bigg)^{\frac{5}{8}-\rho}.
    \end{align*}
    This directly implies the lemma.
\end{proof}

\printbibliography
\end{document}